\documentclass[11pt,twoside,reqno]{amsart}
\usepackage{import}
\usepackage{anysize}
\usepackage{amsmath}
\usepackage{amsthm}
\usepackage{amsmath,amscd}
\usepackage[utf8]{inputenc}
\usepackage{amssymb}
\usepackage{stmaryrd}
\usepackage{wasysym}
\usepackage{mathrsfs}
\usepackage[pagebackref=true]{hyperref} %gives page numbers in references, see https://tex.stackexchange.com/questions/183702/formatting-back-references-in-bibliography-bibtex
\renewcommand*{\backref}[1]{}
\renewcommand*{\backrefalt}[4]{({\tiny%
   \ifcase #1 Not cited.%
         \or Cited on page~#2.%
         \else Cited on pages #2.%
   \fi%
   })}
\usepackage{cleveref}
\usepackage{graphicx}
\usepackage{enumitem}
\usepackage{relsize} %for making things bigger in math mode, see https://tex.stackexchange.com/questions/490066/bigger-equation-in-text-mode-math
\usepackage{dsfont}
\usepackage{soul}
\usepackage{filecontents}
\hypersetup{colorlinks=true,linkcolor=black,citecolor=black}
\usepackage{color}
\usepackage[all]{xy}
\usepackage{float}
\usepackage{bm}
\usepackage{mathtools}
\usepackage{thmtools}
\usepackage{setspace}
\usepackage{comment}
\usepackage{tikz}
\usepackage{ytableau}
\usepackage{mathdots}
\usepackage[myheadings]{fullpage}

\numberwithin{equation}{section}
\newcommand\mtop{1in}
\newcommand\mbottom{1in}
\newcommand\mleft{1in}
\newcommand\mright{1in}
\usepackage[top = \mtop, bottom = \mbottom, left = \mleft, right=\mright]{geometry}

\newtheorem{thm}{Theorem}[section]
\newtheorem{example}[thm]{Example}
\newtheorem{prop}[thm]{Proposition}

\newtheorem{lemma}[thm]{Lemma}
\newtheorem{cor}[thm]{Corollary}

\theoremstyle{definition}
\newtheorem{defi}{Definition}

\newtheorem{rmk}{Remark}

\usepackage{scalerel,stackengine}
\stackMath
\newcommand\reallywidehat[1]{%
\savestack{\tmpbox}{\stretchto{%
  \scaleto{%
    \scalerel*[\widthof{\ensuremath{#1}}]{\kern-.6pt\bigwedge\kern-.6pt}%
    {\rule[-\textheight/2]{1ex}{\textheight}}%WIDTH-LIMITED BIG WEDGE
  }{\textheight}% 
}{0.5ex}}%
\stackon[1pt]{#1}{\tmpbox}%
}
\DeclareSymbolFont{bbold}{U}{bbold}{m}{n}
\DeclareSymbolFontAlphabet{\mathbbold}{bbold}

\makeatletter
\def\@tocline#1#2#3#4#5#6#7{\relax
  \ifnum #1>\c@tocdepth % then omit
  \else
    \par \addpenalty\@secpenalty\addvspace{#2}%
    \begingroup \hyphenpenalty\@M
    \@ifempty{#4}{%
      \@tempdima\csname r@tocindent\number#1\endcsname\relax
    }{%
      \@tempdima#4\relax
    }%
    \parindent\z@ \leftskip#3\relax \advance\leftskip\@tempdima\relax
    \rightskip\@pnumwidth plus4em \parfillskip-\@pnumwidth
    #5\leavevmode\hskip-\@tempdima
      \ifcase #1
       \or\or \hskip 1em \or \hskip 2em \else \hskip 3em \fi%
      #6\nobreak\relax
    \hfill\hbox to\@pnumwidth{\@tocpagenum{#7}}\par% <---- \dotfill -> \hfill
    \nobreak
    \endgroup
  \fi}
\makeatother

\newcommand{\fixed}[1]{\textcolor{green}{~\\ \textbf{\large #1\normalsize}}\\}

\makeatletter
\newcommand{\subalign}[1]{%
  \vcenter{%
    \Let@ \restore@math@cr \default@tag
    \baselineskip\fontdimen10 \scriptfont\tw@
    \advance\baselineskip\fontdimen12 \scriptfont\tw@
    \lineskip\thr@@\fontdimen8 \scriptfont\thr@@
    \lineskiplimit\lineskip
    \ialign{\hfil$\m@th\scriptstyle##$&$\m@th\scriptstyle{}##$\hfil\crcr
      #1\crcr
    }%
  }%
}
\makeatother

\DeclarePairedDelimiter{\abs}{\lvert}{\rvert} %a reminder: you have to type \abs*{...} for the \left and \right to take effect and actually give you the size you want
\DeclarePairedDelimiter{\floor}{\lfloor}{\rfloor}

\newcommand{\R}{\mathbb{R}}
\newcommand{\Z}{\mathbb{Z}}
\newcommand{\Q}{\mathbb{Q}}
\newcommand{\N}{\mathbb{N}}
\newcommand{\C}{\mathbb{C}}
\newcommand{\F}{\mathbb{F}}
\newcommand{\D}{\mathbb{D}}
\newcommand{\T}{\mathbb{T}}

\newcommand{\E}{\mathbb{E}}

\newcommand{\mc}{\mathcal}

\newcommand{\bbone}{\mathbbold{1}}
\newcommand{\la}{\lambda}
\newcommand{\La}{\Lambda}

\newcommand{\eps}{\epsilon}
\renewcommand{\Re}[1]{\text{Re}(#1)}
\renewcommand{\Re}{\operatorname{Re}}
\renewcommand{\Im}{\text{Im}}

\newcommand{\pfrac}[2]{\left(\frac{#1}{#2}\right)}

\newcommand{\ot}{\otimes}

\newcommand{\lan}{\left\langle}
\newcommand{\ran}{\right\rangle}

\newcommand{\tth}{^{th}}

\newcommand{\tnu}{{\tilde{\nu}}}

\newcommand{\tl}{{\tilde{\lambda}}}
\newcommand{\tN}{\tilde{N}}

\renewcommand{\L}{\Lambda}

\newcommand{\tLL}{\tilde{L}}
\newcommand{\Y}{\mathbb{Y}}

\newcommand{\tf}{\tilde{f}}

\newcommand{\ba}{\mathbf{a}}

\newcommand{\bx}{\mathbf{x}}
\newcommand{\by}{\mathbf{y}}
\newcommand{\bi}{\mathbf{i}}
\newcommand{\cP}{\mathbb{Y}}

\renewcommand{\vec}[1]{\boldsymbol{#1}}
\DeclareMathOperator{\Geom}{Geom}

\DeclareMathOperator{\len}{len}

\DeclareMathOperator{\Res}{Res}

\DeclareMathOperator{\coker}{coker}
\DeclareMathOperator{\Sig}{Sig}

\DeclareMathOperator{\tG}{\tilde{\Gamma}}

\DeclareMathOperator{\const}{const}
\DeclareMathOperator{\SN}{SN}

\DeclareMathOperator{\diag}{diag}

\DeclareMathOperator{\Mat}{Mat}
\DeclareMathOperator{\val}{val}

\DeclareMathOperator{\corank}{corank}
\DeclareMathOperator{\rank}{rank}

\DeclareMathOperator{\Law}{Law}
\DeclareMathOperator{\Exp}{Exp}

\newcommand{\sqbinom}[2]{\begin{bmatrix}#1\\ #2\end{bmatrix}}

\newcommand{\Pois}{\mathcal{S}}

\newcommand{\bz}{\bar{z}}
\newcommand{\mbz}{\mathbf{z}}

\newcommand{\Sur}{\operatorname{Sur}}

\newcommand{\GL}{\mathrm{GL}}

\newcommand{\U}{\mathrm{U}}

\newcommand{\cL}{\mathcal{L}}

\newcommand{\tzeta}{\tilde{\zeta}}
\renewcommand{\l}{\lambda}

\renewcommand{\fixed}[1]{} %uncomment this to make the green fixed notes disappear
\title{Local limits in $p$-adic random matrix theory}

\author{Roger Van Peski}
\date{\today}

\begin{document}

\begin{abstract}
We study the distribution of singular numbers of products of certain classes of $p$-adic random matrices, as both the matrix size and number of products go to $\infty$ simultaneously. In this limit, we prove convergence of the local statistics to a new random point configuration on $\mathbb{Z}$, defined explicitly in terms of certain intricate mixed $q$-series/exponential sums. This object may be viewed as a nontrivial $p$-adic analogue of the interpolating distributions of Akemann-Burda-Kieburg \cite{akemann2019integrable}, which generalize the sine and Airy kernels and govern limits of complex matrix products. Our proof uses new Macdonald process computations and holds for matrices with iid additive Haar entries, corners of Haar matrices from $\mathrm{GL}_N(\mathbb{Z}_p)$, and the $p$-adic analogue of Dyson Brownian motion studied in \cite{van2023p}.
\end{abstract}

\maketitle

\tableofcontents

\section{Introduction}

\subsection{Preface.} This work concerns analogues for $p$-adic random matrices of local limit results in classical random matrix theory over $\R$ and $\C$. The singular values of an $N \times N$ random complex matrix $A$ are a random collection of points on $\R_{\geq 0}$ which may look something like \Cref{fig:cplx_svs}, and \emph{local limits} refer to asymptotics at the scale of individual singular values: examples include the limiting probability there is no singular value in a given interval, or the distribution of the spacing between a given pair of singular values. The study of such local limits for eigenvalues of real and complex random matrices goes back to nuclear physics and the work of Dyson, Gaudin, Mehta, Wigner and others \cite{dyson1962statistical,mehta1960density,wig1,wig3} and extends to modern work proving their universality, see the textbook by Erd{\H o}s-Yau \cite{erdHos2017dynamical} and references therein.

\begin{figure}[H] 
\begin{center}
\begin{tikzpicture}[
dot/.style = {circle, fill, minimum size=#1,
              inner sep=0pt, outer sep=0pt},
dot/.default = 3pt  % size of the circle diameter 
                    ] %dot commands taken from https://tex.stackexchange.com/questions/445946/how-set-tikz-circle-radius-in-nodecircle
% \draw[latex-latex] (0,-2) -> (0,2); %edit here for the y axis                    
\draw[latex-latex] (-1,0) -> (14.5,0); %edit here for the x axis
%\draw[latex-latex] (0,.2) -- (0,-.2);
\foreach \x in {-1,0,1,2,3,4,5,6,7,8,9,10,11,12,13,14} % edit here for the numbers
\node at (\x,-.5) {$\x$};
\node at (0,1) {Hard edge};
\node at (6,1) {Bulk};
\node at (13,1) {Soft edge};
\foreach \x in {13.6712,12.8271,12.4924,12.0374,11.5206,11.2559,11.1163,10.6883,10.2807,10.122,9.62816,9.24257,8.90799,8.4278,8.30592,7.90767,7.70532,7.4304,7.305,6.85874,6.75724,6.69621,6.31229,6.00649,5.601,5.31064,5.13248,4.8562,4.6758,4.50427,4.15579,4.01573,3.74676,3.58766,3.38101,3.21973,2.88521,2.7185,2.40246,2.17775,2.10557,1.92448,1.63954,1.50133,1.23,0.994764,0.807994,0.537315,0.334268,0.115778} 
\node[dot] at (\x,0) {};
\end{tikzpicture}
\end{center}
\caption{The singular values of a typical $50 \times 50$ matrix with standard complex Gaussian entries, generated on Mathematica.}\label{fig:cplx_svs}
\end{figure}
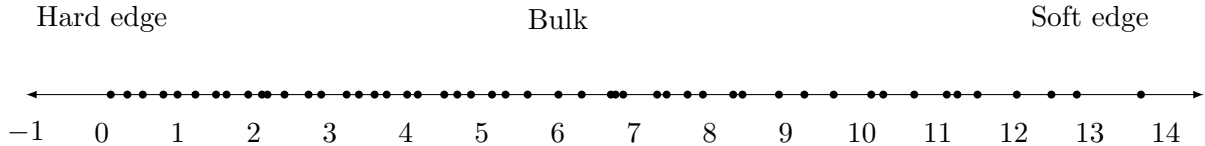

The limiting local statistics near the left and right edges were shown by Forrester \cite{forrester1993spectrum} to be governed by the \emph{Bessel kernel} and the \emph{Airy kernel} respectively. In the bulk (away from the edges) they are governed by the \emph{sine kernel}, found originally by Dyson \cite{dyson1962statistical} in the Hermitian setting. The left and right edges are usually called the \emph{hard edge} and \emph{soft edge}, since nonnegativity of singular values imposes a hard lower bound on the former. 

In the complex setup, a matrix $A \in \Mat_N(\C)$ has singular value decomposition $A=UDV$ where $U,V \in \U(N)$ and $D$ is diagonal with entries given by the singular values in decreasing order. In the setting of integer matrices, a matrix $A \in \Mat_N(\Z)$ similarly has Smith normal form $A=U\diag(a_1,\ldots,a_N)V$ where $ U,V \in \GL_N(\Z)$ and the $a_i$ are a weakly decreasing\footnote{In fact, $a_i|a_{i-1}$.} sequence of nonnegative integers which also parametrize the matrix's cokernel
\begin{equation}
\coker(A) := \Z^N/A\Z^N \cong \bigoplus_{i=1}^N \Z/a_i\Z,
\end{equation}
an abelian group. Many works such as Clancy-Kaplan-Leake-Payne-Wood \cite{clancy2015cohen}, Wood \cite{wood2015random,wood2017distribution,wood2018cohen}, M\'esz\'aros \cite{meszaros2020distribution,meszaros2023cohen}, Nguyen-Wood \cite{nguyen2022random,nguyen2022local}, and Lee \cite{lee2022universality,lee2023joint} study cokernels of random matrices over the integers\footnote{It is worth clarifying that in this literature, `local' typically refers to local fields such as $\Q_p$ as opposed to global fields such as $\Q$, and `local statistics' in \cite{nguyen2022local} refers to those which only depend on the localization of the cokernel at a finite collection of primes. Our usage of `local' in this work, which instead follows that of the real/complex random matrix theory literature, is unrelated.} $\Z$, motivated by problems in number theory, graph theory and topology where random or pseudorandom integral matrices appear.

It is in fact easier to work with random matrices over the $p$-adic integers\footnote{We encourage a reader less familiar with the $p$-adic setting to read the first few paragraphs of background in \Cref{sec:p-adic} before continuing.}, and computations in the $p$-adic setting are used to find the limit distributions which occur (after taking a product over multiple primes) in the integer setting. A matrix $A \in \Mat_N(\Q_p)$ similarly has Smith normal form $A=U\diag(p^{\SN(A)_1},\ldots,p^{\SN(A)_N})V$ for $U,V \in \GL_N(\Z_p)$ and \emph{singular numbers} $\SN(A)_1 \geq \ldots \geq \SN(A)_N \in \Z$. These are the analogues of the negative logarithms of singular values, because the $p$-adic norm $|p^m|_p = p^{-m}$ of a high power of $p$ is small.

The most classic $p$-adic random matrix ensemble, introduced by Friedman-Washington \cite{friedman-washington}, is given by $A \in \Mat_N(\Z_p)$ with iid entries distributed by the Haar probability measure on the additive group $\Z_p$. This measure is the analogue of the Gaussian in this context---see for instance \cite[Remark 2]{van2020limits}---and is characterized by the fact that $A \pmod{p^k}$ is a uniform element of the finite set $\Mat_N(\Z/p^k\Z)$ for any $k$.

Unlike the complex setting, most singular numbers are $0$ asymptotically. Representing the singular numbers $\SN(A)_1 \geq \SN(A)_2 \geq \ldots \geq \SN(A)_N$ of such a matrix graphically as a Young diagram, they might look like \Cref{fig:1matrix_sns}, and in fact the probability that those exact singular numbers appear is
\begin{equation}
\lim_{N \to \infty} \Pr(\SN(A) = (4,2,2,1,0,\ldots,0)) = \prod_{i \geq 1}(1-p^{-i}) \frac{p^{-27}}{(1-p^{-1})^3(1-p^{-2})}
\end{equation}
by \cite[Proposition 1]{friedman-washington}. This contrasts markedly with \Cref{fig:cplx_svs}, where the singular values are all distinct. In fact, \cite{friedman-washington} showed that the cokernel $\coker(A)$ converges as $N \to \infty$ to the so-called \emph{Cohen-Lenstra distribution} on finite abelian $p$-groups, observed slightly earlier in the study of distributions of class groups \cite{cohen-lenstra}. The answer found there was subsequently shown by Wood \cite{wood2015random} to be universal for matrices with iid entries from any sufficiently nondegenerate distribution. Hence \cite{friedman-washington} already contains complete information on the local limits of a single addtive Haar matrix. Convergence of the cokernel, and the corresponding fact that almost all singular numbers are $0$ in the limit, is not just a property of this example: subsequent work studied many different ensembles which likewise have limiting cokernel distributions. We mention Bhargava-Kane-Lenstra-Poonen-Rains \cite{bhargava2013modeling}, Kovaleva \cite{kovaleva2020note}, Lipnowski-Sawin-Tsimmerman \cite{lipnowski2020cohen}, Cheong-Huang \cite{cheong2021cohen} -Kaplan \cite{cheong2022generalizations} and -Yu \cite{cheong2023cokernel}, as well as the above works on integer matrices which yield corresponding results in the $p$-adic setting as well. 

\begin{figure}[H] 
\centering 
\includegraphics[scale=1.2]{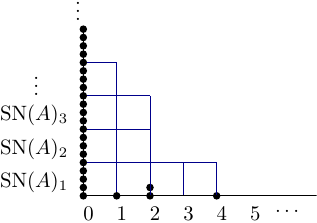}
\caption{A possible tuple $(4,2,2,1,0,\ldots,0)$ of singular numbers $(\SN(A)_1,\ldots,\SN(A)_N)$ for large $A \in \Mat_N(\Z_p)$ with additive Haar measure, represented as a Young diagram with row lengths corresponding to singular numbers, and as a collection of points on $\Z_{\geq 0}$ along the axis. Here we stack points that are at the same location, a situation which of course does not arise in pictures of spectra of generic Hermitian matrices such as \Cref{fig:cplx_svs}.}\label{fig:1matrix_sns}
\end{figure}

The limit at one edge is thus nontrivial and dependent on the matrix ensemble, but away from that edge all singular numbers are $0$ in these examples. This nontrivial edge corresponds to the hard edge because large powers of $p$ are small in the $p$-adic norm. Hence the bulk and soft edge local limits, the analogues of the sine and Airy kernel in the $p$-adic setting, are uninteresting: a deterministic collection of infinitely many points at $0$.

However, when one begins taking products of $p$-adic random matrices, the situation changes. For a product $A_\tau \cdots A_2 A_1$ of random matrices in $\Mat_N(\C)$, there are two natural parameters to vary in the limit: $N$, the matrix size, and $\tau$, the number of products. For matrices $A,B \in \Mat_N(\Z_p)$, the singular numbers of $AB$ are bounded below by those of $A$ (and of $B$), and multiplicativity of the determinant shows
\begin{equation}
\sum_{i=1}^N \SN(AB)_i = \sum_{i=1}^N \SN(A)_i + \sum_{i=1}^N \SN(B)_i.
\end{equation}
Hence the singular numbers of a product $A_\tau \cdots A_1$, with $A_i \in \Mat_N(\Z_p)$ additive Haar, will become larger as $\tau$ increases, looking like \Cref{fig:many_matrix_sns} instead of \Cref{fig:1matrix_sns}. For fixed $\tau$, an exact distribution was computed in \cite[Corollary 3.4]{van2020limits}, further simplified in \cite[Theorem 1.4]{vanpeski2021halllittlewood}, and shown to be universal in the $N \to \infty$ limit for matrices with generic iid entries in \cite{nguyen2022universality}. For fixed $N$ and $\tau \to \infty$, the singular numbers were shown in \cite{van2020limits} to have independent Gaussian fluctuations asymptotically. However, to our knowledge the singular numbers of integer or $p$-adic matrices have not been studied in the regime where both $N$ and $\tau$ go to $\infty$.

\begin{figure}[H]
\centering 
\includegraphics[scale=1.2]{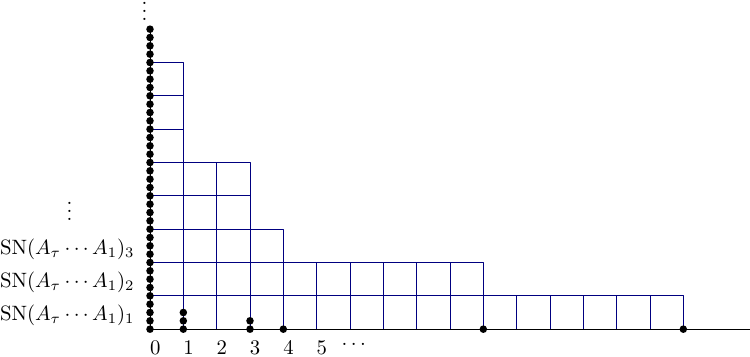}
\caption{A possible tuple of singular numbers for a product of additive Haar $A_i \in \Mat_N(\Z_p)$, represented as a Young diagram and as a collection of points on $\Z_{\geq 0}$ along the axis.} \label{fig:many_matrix_sns}
\end{figure}

For complex random matrices, by contrast, this regime is well studied and the singular values are known to exhibit interesting generalizations of single-matrix local statistics. Singular values of random matrix products have been studied since the 1950s, initially by Bellman \cite{bellman1954limit} and then by Furstenberg-Kesten \cite{furstenberg1960products}. They have since appeared in the study of disordered systems in statistical physics, dynamical systems, convolutional neural networks, and wireless communication channels; see for instance Crisanti-Paladin-Vulpiani \cite{crisanti2012products}, Cohen-Newman \cite{cohen1984stability}, Hanin-Nica \cite{hanin2020products}, and Akemann-Kieburg-Wei \cite{akemann2013singular} respectively, and the references therein. When $N$ is fixed and $\tau \to \infty$, the logarithms of the singular values exhibit independent Gaussian fluctuations in the limit, as shown for the largest singular value by Furstenberg-Kesten \cite{furstenberg1960products} and for all by Akemann-Burda-Kieburg \cite{akemann2019integrable,akemann2020universality} and Liu-Wang-Wang \cite{liu2018lyapunov}. In the opposite regime when $\tau$ is fixed and $N \to \infty$, Liu-Wang-Zhang \cite{liu2016bulk} showed that the singular values in the bulk are governed by the sine kernel as with eigenvalues of a single matrix. In the intermediate regime where $\tau,N \to \infty$ and $\tau/N$ converges to a constant, a new universal limit object appears. This distribution interpolates between the two aforementioned ones; it first appeared for matrices with iid Gaussian entries in work of Akemann-Burda-Kieburg \cite{akemann2019integrable,akemann2020universality} in the physics literature, and in the mathematics literature was shown by Liu-Wang-Wang \cite{liu2018lyapunov}. 

In this work, we treat the joint $N,\tau \to \infty$ regime in $p$-adic random matrix theory. We explicitly compute the limiting joint distribution of the number of singular numbers at $0,1,2,\ldots$, which completely characterizes the local statistics near $0$ since all singular numbers are nonnegative. The same limit appears in all of our examples, giving us strong reason to suspect universality. Unlike the trivial all-$0$ limit for the bulk/soft edge of a single matrix which was mentioned above, the limit distribution we find is highly nontrivial. We first state some consequences of these results for random matrices over finite fields, which even in this setting are new and which give a good sense of the general versions. 

\begin{rmk}\label{rmk:hard_edge}
It is also natural to ask about the limits of the largest singular numbers (the rows in \Cref{fig:many_matrix_sns}) as $N,\tau \to \infty$. For fixed $N$ it was shown in \cite{van2020limits} that these rows grow linearly in $\tau$ with Gaussian fluctuations of scale $\sqrt{\tau}$. We are very confident that this also holds when $N,\tau \to \infty$ and can be established by similar analysis to that of \cite{van2020limits}, though to our knowledge this has not appeared in the literature.
\end{rmk}

\subsection{Matrix products over $\F_p$.} For $A \in \Mat_N(\Z_p)$, the number of nonzero singular numbers of $A$ (the length of the first column in \Cref{fig:many_matrix_sns}) is exactly the corank of the reduction $A \pmod{p} \in \Mat_N(\Z/p\Z)$. Roughly, we find in \Cref{thm:F_p_intro} that for uniformly random iid $A_i^{(N)} \in \Mat_N(\F_p)$,
\begin{equation}
\corank(A_s^{(N)} A_{s-1}^{(N)} \cdots A_1^{(N)}) \approx \log_p s + \text{(finite-order fluctuations)}
\end{equation}
as $N,s \to \infty$ in such a way that the $\log_p s$ is far from both $0$ and $N$, as the corank must of course lie in $[0,N]$. The fluctuations have a limit, and because they are finite-order this limit is a $\Z$-valued random variable which we now define. Here and later, we use the $q$-Pochhammer notation $(a;t)_n = \prod_{i=1}^n (1-at^{i-1})$ for $n \in \Z_{\geq 0} \cup \{\infty\}$.

\begin{defi}\label{def:L1}
For any parameters $t \in (0,1)$ and $\chi \in \R_{>0}$, define the $\Z$-valued random variable $\cL_{t,\chi}^{(1)}$ by 
\begin{equation}\label{eq:L1}
\Pr(\cL_{t,\chi}^{(1)} = x) = \frac{1}{(t;t)_\infty} \sum_{m \geq 0} e^{-\chi t^{x-m}}  \frac{(-1)^m  t^{\binom{m}{2}}}{(t;t)_m}.
\end{equation}
\end{defi}

We are not aware of any appearance of $\cL_{t,\chi}^{(1)}$ in previous literature, so let us give some basic comments on the definition. The terms in the sum \eqref{eq:L1} go to $0$ quite rapidly as $m$ increases due to the $e^{-\chi t^{x-m}}$ factor, so it is in fact very easy to approximate the probabilities numerically. While this makes it easy to check numerically that the probabilities sum to $1$, actually proving this amounts to the quite nontrivial algebraic identity
\begin{equation}\label{eq:sum_to_1_identity}
\sum_{x \in \Z}  \sum_{m \geq 0} e^{-\chi t^{x-m}}  \frac{(-1)^m  t^{\binom{m}{2}}}{(t;t)_m} = (t;t)_\infty.
\end{equation}
Even the fact that the left hand side of \eqref{eq:sum_to_1_identity} is independent of $\chi$ seems quite unintuitive. It is however manifest by translating $x$ in \eqref{eq:L1} that
\begin{equation}\label{eq:obvious_invariance_k=1}
\cL_{t,t\chi}^{(1)} = \cL_{t,\chi}^{(1)} - 1
\end{equation}
in distribution. We explain the probabilistic origins of this invariance, and its generalization \eqref{eq:obvious_invariance} from $\F_p$ to $\Z_p$, in the companion paper \cite{van2023+dynamical}, but for now we content ourselves with the algebraic observation. %\fixed{fill in specific place once intro is written} 

For fixed $t$, the $\chi^{-1} t^\Z$-valued random variables $\chi^{-1} t^{-\cL_{t,\chi}^{(1)}}$ solve the indeterminate Stieltjes moment problem
\begin{equation}
\E[X^m] = \frac{t^{-\binom{m+1}{2}}(t;t)_m}{m!}, m = 0,1,2,\ldots
\end{equation}
for $X$ on $\R_{\geq 0}$ (see \Cref{thm:k=1_t_moments}). Their supports $\chi^{-1} t^\Z$ are distinct as $\chi$ ranges over e.g. $[t,1)$ and \eqref{eq:obvious_invariance_k=1} shows that 
\begin{equation}
\chi^{-1} t^{-\cL_{t,\chi}^{(1)}} = (t\chi)^{-1} t^{-\cL_{t,t\chi}^{(1)}}
\end{equation}
in distribution, so the random variables $\chi^{-1} t^{-\cL_{t,\chi}^{(1)}}$ are parametrized by the circle $\R/\Z$ through $\log_t \chi$. Such explicit solutions to indeterminate moment problems are quite rare; the first indeterminate moment problem to be solved completely was that of the continuous $q^{-1}$-Hermite polynomials, which was solved by Ismail and Masson \cite{ismail1994q} (see also \cite[Chapter 21]{ismail2005classical} for subsequent developments), and appeared later in an interacting particle system context in work of Borodin-Corwin \cite{borodin2020dynamic}. For us, $\cL_{t,\chi}^{(1)}$ appears naturally in random matrix theory through the following result, for which we first give some standard notation.

\begin{defi}\label{def:metric}
For two Borel measures $M_1,M_2$ on a set $S$ with discrete $\sigma$-algebra, we denote by 
\begin{equation}
D_\infty(M_1,M_2) := \sup_{x \in S} |M_1(\{x\}) - M_2(\{x\})|
\end{equation}
the $\ell_\infty$ distance between them. When $X_1,X_2$ are two $S$-valued random variables with laws $M_1,M_2$, we abuse notation and write $D_\infty(X_1,X_2) := D_\infty(M_1,M_2)$. Note that the set $S$ is implicit in the notation, and we will often use the same notation $D_\infty$ for different sets $S$.
\end{defi}

% \begin{defi}\label{def:nearest_integer}
% We define the nearest integer function
% \begin{equation}
% \near{x} = \operatorname{argmin}_{y \in \Z} |x-y|.
% \end{equation}
% It does not matter for any results what convention is chosen for half-integers.
% \end{defi}
We also use $\{x\}$ for the fractional part $x - \floor{x}$ when $x$ is real.

\begin{cor}\label{thm:F_p_intro}
Fix $p$ prime, and for each $N \in \mathbb{Z}_{\geq 1}$ let $A_i^{(N)}, i \geq 1$ be iid uniform elements of $\Mat_N(\F_p)$. Let $(s_{N})_{N \geq 1}$ be any sequence of natural numbers such that $s_{N}$ and $N-\log_p s_{N}$ both go to $\infty$ as $N \rightarrow \infty$. Then the sequence of $\Z$-valued random variables
\begin{equation}
\label{eq:shift_corank_def}
X_N^{(1)} := \corank(A_{s_N}^{(N)} \cdots A_2^{(N)}A_1^{(N)}) - \floor{\log_p s_N}, N=1,2,\ldots
\end{equation}
satisfies
\begin{equation}
\lim_{N \to \infty} D_\infty(X_N^{(1)},\cL_{p^{-1},p^{\{\log_p s_N\}}/(p-1)}^{(1)}) = 0.%\corank(A_{s_N}^{(N)} \cdots A_2^{(N)}A_1^{(N)}) - \floor{\log_p s_N},\cL_{p^{-1},p^{\{\log_p s_N\}}/(p-1)}^{(1)}) = 0.
\end{equation}
\end{cor}

% \begin{rmk}\label{rmk:preimage_doesn't_matter}
% Note that we did not specify which preimage $\zeta$ to choose, but choosing a different one simply translates the left hand side of \eqref{eq:F_p_limit_intro} by an integer and multiplies the parameter of $\cL_{1,t,\cdot}$ by an integer power of $t=p^{-1}$ on the right hand side, which in light of \eqref{eq:obvious_invariance_k=1} leaves \eqref{eq:F_p_limit_intro} invariant. However, it was necessary to choose a subsequence of the $s_{N}$ such that $\log _{t} s_{N_{j}}$ converges in $\mathbb{R} / \mathbb{Z}$, because $\corank(A_{s_{N_{j}}}^{(N_{j})} \cdots A_{1}^{(N_{j})})$ is an integer. 
% \end{rmk}
In other words, the shifted corank $X_N^{(1)}$ is asymptotically close to the random variables in the family $\cL^{(1)}_{p^{-1},\chi}, \chi \in \R_{>0}$, but the parameter $\chi$ must be chosen differently depending on $N$. \Cref{thm:F_p_intro} is an immediate corollary of {\Cref{thm:matrix_product_bulk_metric_intro}} below. Interestingly, the limit is truly nonunique and this is not just a technical feature: as $\chi$ varies, the probabilities $\Pr(\cL^{(1)}_{t,\chi} = x)$ given in \eqref{eq:L1} are not related to one another in any simple way apart from \eqref{eq:obvious_invariance_k=1}, though they do depend continuously on $\chi$. This nonuniqueness of the limit may be viewed as coming the fact that the fluctuations lie on a discrete lattice but the shift $\log_p s_N$ is in general not an integer. The only way of obtaining a traditional convergence statement seems to be by passing to subsequences for which $\{\log_p s_N\}$ converges, as is done in \Cref{thm:matrix_product_bulk}. 

The complexity of the limit object $\cL_{t,\chi}^{(1)}$ in \Cref{thm:F_p_intro}, even in the seemingly simple problem over $\F_p$, suggests that the local statistics over $\Z_p$ (the other columns of \Cref{fig:many_matrix_sns}) may yield interesting distributions if they can be computed at all. This computation is the main result of this paper, and we carry it out not only for additive Haar matrices (\Cref{thm:matrix_product_bulk_metric_intro}) but also for $N \times N$ corners matrices distributed by the Haar measure on $\GL_{N+D_N}(\Z_p)$ (\Cref{thm:corner_product_bulk_metric_intro}), and for an interacting particle system argued in \cite{van2023p} to be the natural $p$-adic analogue of Dyson Brownian motion (\Cref{thm:matrix_stat_dist}). The limit object is the same in all of these cases and hence most likely universal.

\subsection{Matrix product local limits over $\Z_p$.} For any weakly decreasing tuple of integers $\la$, we write $\la_j' := \#\{i: \la_i \geq j\}$. Hence $\la_1'$ is the number positive parts $\la_i$, and for any $A \in \Mat_N(\Z_p)$ we have 
\begin{equation}\label{eq:sn_and_corank}
\SN(A)_1' = \#\{i:\SN(A)_i \geq 1\} = \rank(\coker(A)/p\coker(A)) = \corank(A \pmod{p}).
\end{equation}
The other conjugate parts $\SN(A)_2',\SN(A)_3',\ldots$, which correspond to the lengths of the second, third, etc. columns in \Cref{fig:1matrix_sns}, are the ranks of the $\F_p$-vector spaces
\begin{equation}
\SN(A)_i' = \rank(p^{i-1}\coker(A)/p^i\coker(A)).
\end{equation}
In light of \eqref{eq:sn_and_corank}, we have already seen the limiting distribution of $\SN(A)_1'$ in \Cref{thm:F_p_intro}, and the result below generalizes this to the joint distribution of $\SN(A)_1',\SN(A)_2',\ldots$. Because 
\begin{equation}
\#\{i:\SN(A)_i = 0\} = N - \SN(A)_1'
\end{equation}
and
\begin{equation}
\#\{i: \SN(A)_i = j\} = \SN(A)_j' - \SN(A)_{j+1}',
\end{equation}
the joint distribution of $\SN(A)_1',\ldots,\SN(A)_k'$ is equivalent to the joint distribution of the number of singular numbers at $0,1,\ldots,k-1$. Hence \Cref{thm:matrix_product_bulk_metric_intro} below gives the complete local statistics of the singular numbers near $0$.

Below we use the notation 
\begin{equation}\label{eq:def_sig}
\Sig_k := \{(x_1,\ldots,x_k) \in \Z^k: x_1 \geq \ldots \geq x_k\}
\end{equation}
for the set of integer signatures of length $k$, where we allow $k=\infty$. The limit object in our theorems is a random element 
\begin{equation}
\cL_{t,\chi} = (\cL^{(1)}_{t,\chi},\cL^{(2)}_{t,\chi},\ldots)
\end{equation}
of $\Sig_\infty$, where the first coordinate $\cL^{(1)}_{t,\chi}$ is the random integer which appeared above. We define $\cL_{t,\chi}$ explicitly in \Cref{thm:stat_dist_1pt} by specifying the distribution of its $\Sig_k$-valued marginals $(\cL^{(i)}_{t,\chi})_{1 \leq i \leq k} $ for every $k$ through an explicit formula for $\Pr(\cL_{k,t,\chi}=\vec{L})$ valid for any $\vec{L} \in \Sig_k$. These formulas require enough symmetric function notation from \Cref{sec:macdonald_background} that we will not give them in full here, but rather postpone them to \Cref{thm:stat_dist_1pt}---to which the interested reader is invited to skip---and continue to the random matrix limit results.

%Whenever we speak of convergence in distribution of $\Sig_\infty$-valued random variables, we mean with respect to the product topology on $\Z^\infty$ where each $\Z$ factor has the discrete topology, or equivalently convergence in distribution of all finite collections of coordinates.

\begin{thm}\label{thm:matrix_product_bulk_metric_intro}
Fix $p$ prime, and for each $N \in \mathbb{Z}_{\geq 1}$ let $A_{i}^{(N)}, i \geq 1$ be iid matrices with iid entries distributed by the additive Haar measure on $\mathbb{Z}_{p}$. Let $(s_{N})_{N \geq 1}$ be a sequence of natural numbers such that $s_{N}$ and $N-\log _{p} s_{N}$ both go to $\infty$ as $N \rightarrow \infty$. Then the random integers
\begin{equation}
X_N^{(j)} := \mathrm{SN}(A_{s_{N}}^{(N)} \cdots A_{1}^{(N)})_{j}'-\floor{\log _{p}(s_N)}
\end{equation}
are asymptotically close to the ones $\cL_{p^{-1},p^{\{\log_p s_N\}}/(p-1)}^{(j)}$ in the sense that 
\begin{equation}\label{eq:haar_metric_cvg}
\lim_{N \to \infty} D_\infty((X_N^{(j)})_{1 \leq j \leq k}, (\cL^{(j)}_{p^{-1},p^{\{\log_p s_N\}}/(p-1)})_{1 \leq j \leq k}) = 0
\end{equation}
for every $k \in \Z_{\geq 1}$. Here $D_\infty$ is the metric of \Cref{def:metric} on the set $\Sig_k$, and $\cL^{(j)}_{t,\chi}$ is as defined in \Cref{thm:stat_dist_1pt}.
\end{thm}

\Cref{thm:F_p_intro} is simply the $k=1$ case of the above result, and we use the same notation $\{x\}:= x-\floor{x}$ defined there. Another natural probability measure on $p$-adic matrices is the Haar probability measure on the compact group $\GL_N(\Z_p)$. Its singular numbers are all $0$, but those of an $m \times m$ submatrix are nontrivial. This measure was studied in \cite{van2020limits}, and the following result generalizes \Cref{thm:matrix_product_bulk_metric_intro} to this case as well.

\begin{thm}\label{thm:corner_product_bulk_metric_intro}
Fix $p$ prime, and for each $N \in \mathbb{Z}_{\geq 1}$ let $D_N \in \Z_{\geq 1}$ be an integer and $A_{i}^{(N)}, i \geq 1$ be the top-left $N \times N$ corners of iid matrices distributed by the Haar probability measure on $\GL_{N+D_N}(\Z_p)$. Let $(s_{N})_{N \geq 1}$ be a sequence of natural numbers such that $s_{N}$ and $N-\log _{p} s_{N}$ both go to $\infty$ as $N \rightarrow \infty$. Then the random integers%Let $(s_{N_{j}})_{j \geq 1}$ be any subsequence for which $-\log _{p} ((1-p^{-D_{N_j}})s_{N_{j}})$ converges in $\mathbb{R} / \mathbb{Z}$, and let $\zeta$ be any preimage in $\mathbb{R}$ of this limit. Then
\begin{equation}
X_N^{(j)} :=\mathrm{SN}(A_{s_{N}}^{(N)} \cdots A_{1}^{(N)})_{j}'-\floor{\log_{p}((1-p^{-D_{N}})s_{N})}
\end{equation}
satisfy
\begin{equation}
\lim_{N \to \infty} D_\infty((X_N^{(j)})_{1 \leq j \leq k}, (\cL^{(j)}_{p^{-1},p^{\{\log_p ((1-p^{-D_N})s_N)\}}/(p-1)})_{1 \leq j \leq k}) = 0
\end{equation}
for every $k \in \Z_{\geq 1}$. Here $D_\infty$ is the metric of \Cref{def:metric} on the set $\Sig_k$, and $\cL^{(j)}_{t,\chi}$ is as defined in \Cref{thm:stat_dist_1pt}.
\end{thm}

In addition to \Cref{thm:matrix_product_bulk_metric_intro} and \Cref{thm:corner_product_bulk_metric_intro}, we find in \Cref{thm:matrix_stat_dist} that $\cL_{t,\chi}$ is also a limit of another stochastic process. This process $\Pois^{(N)}(\tau)$ is in continuous time $\tau$, and may be explicitly described as a reflected Poisson random walk on a Weyl chamber. It was further realized in \cite{van2023p} as coming from `Poissonized' $p$-adic matrix products in continuous time, and argued to be a natural $p$-adic analogue of multiplicative Dyson Brownian motion on $\GL_N(\C)$.

% \begin{rmk}\label{rmk:L_parameter}
% The fact that the final parameter in $\cL_{t,\cdot}$ is $p^{-\zeta}/(p-1)$ in both \Cref{thm:matrix_product_bulk} and \Cref{thm:corner_product_bulk} above (and in \Cref{thm:matrix_stat_dist_metric} below) is simply a choice of normalization in the limit; by choosing different values of $\zeta$ one may obtain $\cL_{p^{-1},\chi}$ for any $\chi \in \R_{>0}$.
% \end{rmk}

\begin{rmk}\label{rmk:dynamical}
It is also natural to consider dynamical limits of $\SN(A^{(N)}_s \cdots A^{(N)}_1), s \geq 0$ across multiple times, beyond the $1$-point distribution computed above. We do this in a companion work \cite{van2023+dynamical}.
\end{rmk}

Let us give some comments on $\cL_{t,\chi}$ in lieu of the full details in \Cref{thm:stat_dist_1pt}. First, its coordinates $\cL_{t,\chi}^{(j)}$ satisfy the natural generalization of \eqref{eq:obvious_invariance_k=1}, namely 
\begin{equation}\label{eq:obvious_invariance}
\cL_{t, t \chi}^{(j)} = \cL_{t, \chi}^{(j)}-1% - (\underbrace{1, \ldots, 1}_{k \text { times }})
\end{equation}
in joint distribution, for any finite collection of $j$'s. For general $k$, the law of the marginals $\cL_{k,t,\chi} := (\cL_{t,\chi}^{(j)})_{1 \leq j \leq k}$ is defined by a series expansion formula \eqref{eq:limit_rv_res_formula} generalizing \eqref{eq:L1}, which in general takes the form
\begin{equation}
\Pr(\cL_{k,t,\chi} = (L_1,\ldots,L_k)) =\frac{1}{(t;t)_\infty} \sum_{m \geq 0} e^{-\chi t^{L_k - m}} C_{m;L_1,\ldots,L_k}(t,\chi)
\end{equation}
for $C_{m;L_1,\ldots,L_k}(t,\chi)$ which are polynomials in $\chi$ with coefficients in $\Q(t)$. For example, when $k=2$ we have the following formula, where we use notation 
\begin{equation}
\sqbinom{a}{b}_q := \frac{(q;q)_a}{(q;q)_b (q;q)_{a-b}}
\end{equation}
for the $q$-binomial coefficient.

\begin{cor}\label{thm:k=2_case}
In the notation of \Cref{thm:stat_dist_1pt}, for any $L \in \Z$ and $x \in \Z_{\geq 0}$,
\begin{multline}
\Pr(\cL_{2,t,\chi} = (L+x,L)) =\frac{1}{(t;t)_\infty}  \sum_{m \geq 0} e^{-\chi t^{L-m}} (-1)^m t^{m^2+(x-1)m+\binom{x}{2}} \\ 
\times \sum_{i=0}^x \frac{(-1)^{x-i}}{(t;t)_{x-i}} \sqbinom{m+i}{i}_t \left(\frac{(t^{L-m}\chi)^{i+m}}{(i+m)!} + \frac{(t^{L-m}\chi)^{i+m-1}\bbone(i+m \geq 1)}{(i+m-1)!}\right)
\end{multline}
\end{cor}

\Cref{thm:k=2_case} is deduced straightforwardly from the general formula \Cref{thm:stat_dist_1pt} in \Cref{sec:examples}. In general, the coefficients $C_{m;L_1,\ldots,L_k}(t,\chi)$ are defined in terms of Hall-Littlewood polynomials with specializations featuring both $\alpha$ and Plancherel parameters, for which we again refer to \Cref{sec:macdonald_background}, and may be computed in terms of standard Young tableaux, see \Cref{sec:examples} for a worked example when $k=3$. An equivalent formula for $\cL_{k,t,\chi}$, given in \eqref{eq:limit_rv_int_formula}, is a certain multiple contour integral.

\subsection{Techniques.} Of the works on complex matrix products mentioned above, the vast majority use the determinantal point process structure of matrix products. We are not aware of any such structure for $p$-adic matrices. However, Borodin-Gorin-Strahov \cite{borodin2018product} found that singular values of complex matrix products are also a degeneration of the Macdonald processes introduced by Borodin-Corwin \cite{borodin2014macdonald}. Many techniques for analyzing such processes exist, which were applied to complex matrix products in subsequent work of Ahn \cite{ahn2019fluctuations}, Ahn-Strahov \cite{ahn2020product}, Ahn and the author \cite{ahn2023lyapunov}, and Gorin-Sun \cite{gorin2018gaussian}. The work \cite{van2020limits} found analogous relations between $p$-adic matrix products and Macdonald processes. These convert our matrix asymptotics into statements about Hall-Littlewood measures, which are probability measures defined in terms of Hall-Littlewood polynomials. Explicitly these are symmetric polynomials 
\begin{equation}\label{eq:hl_explicit_intro}
P_\lambda(x_1,\ldots,x_n;0,t) :=  \frac{1}{v_\lambda(t)} \sum_{\sigma \in S_n} \sigma\left(x_1^{\lambda_1} \cdots x_n^{\lambda_n} \prod_{1 \leq i < j \leq n} \frac{x_i-tx_j}{x_i-x_j}\right),
\end{equation}
where $\sigma$ acts by permuting the variables and $v_\lambda(t)$ is the normalizing constant such that the $x_1^{\lambda_1} \cdots x_n^{\lambda_n}$ term has coefficient $1$. By \cite[Corollary 3.4]{van2020limits}, we have for iid additive Haar $A_i \in \Mat_N(\Z_p)$ that
\begin{equation}\label{eq:uniform_hl_intro}
\Pr(\SN(A_s^{(N)} \cdots A_1^{(N)}) = \la) = \frac{P_\la(1,t,\ldots,t^{N-1};0,t)Q_\la(\overbrace{t, \ldots, t}^{s \text { times }},\overbrace{t^2, \ldots, t^2}^{s \text { times }},\ldots)}{\Pi_{0,t}(1,t,t^2,\ldots;\underbrace{t, \ldots, t}_{s \text { times }},\underbrace{t^2, \ldots, t^2}_{s \text { times }},\ldots)}
\end{equation}
with $t=1/p$, where $Q_\la$ is a certain constant multiple of $P_\la$ and $\Pi_{0,t}(\cdots)$ is a normalizing constant, both defined in \Cref{sec:macdonald_background}. The same result \cite[Corollary 3.4]{van2020limits} implies a similar expression for the $\GL_N(\Z_p)$ corners in \Cref{thm:corner_product_bulk_metric_intro}, and a different one \cite[Proposition 3.4]{van2022q} yields such an expression for $\Pois^{(N)}$, see \Cref{thm:hl_meas_matrices} and \Cref{thm:HL_poisson}.

The resulting Hall-Littlewood process asymptotics are still highly nontrivial. For the Gaussian limits at fixed matrix size shown in \cite{van2020limits}, explicit formulas for Hall-Littlewood polynomials gave an explicit sampling algorithm for the singular numbers of matrix products which was then amenable to direct probabilistic analysis. In the present work, by contrast, proving (even guessing) the limit required nontrivial work at the level of symmetric functions.

The approach we take is based on expectations of the form 
\begin{equation}\label{eq:skew_observable_intro}
\E_\la\left[\frac{P_{\la/\mu}(1,t,\ldots;0,t)}{P_\la(1,t,\ldots;0,t)}\right],
\end{equation}
where $\la$ is distributed by a Hall-Littlewood measure such as \eqref{eq:uniform_hl_intro} and $P_{\la/\mu}$ is a skew Hall-Littlewood polynomial as defined in \eqref{eq:def_skewP}. For symmetric functions associated to vertex models, of which the Hall-Littlewood polynomials are a limit, such observables have been used in works such as Borodin-Wheeler \cite[(4.14)]{borodin2020observables}.

For fixed $\eta$, one may express the prelimit probabilities
\begin{equation}
\Pr((\la_1',\ldots,\la_k') = (\eta_1,\ldots,\eta_k))
\end{equation}
as a certain linear combination over different $\mu$ of the expectations \eqref{eq:skew_observable_intro}, see \Cref{thm:prob_from_observables}. The fact that reasonable expressions exist for these probabilities is in general related to the Markovian projection property of Hall-Littlewood processes, see e.g. Borodin-Bufetov-Wheeler \cite{borodin2016between}, and in the random matrix setting comes because these probabilities only depend on the random matrix modulo $p^k$. We use these formulas to obtain contour integral formulas for these probabilities, given in \Cref{thm:use_torus_product}, and then take asymptotics of this integral to obtain convergence of the resulting probabilities directly. These prelimit formulas hold for a slightly different Hall-Littlewood measure than in \eqref{eq:uniform_hl_intro}, which is given by the $N \to \infty$ limit of \eqref{eq:uniform_hl_intro}; at the level of random matrices, this should be thought of as first sending the matrix size to infinity, then the number of products. The asymptotic analysis is done in \Cref{sec:alpha}, following an argument in \Cref{sec:stationary_law_hl} for a slightly simpler Hall-Littlewood measure where the varying specialization of \eqref{eq:uniform_hl_intro} is replaced by a so-called Plancherel specialization. These computations are the main technical component of the paper. 

Deducing the limit \Cref{thm:matrix_product_bulk_metric_intro}, where $N$ is sent to infinity in tandem with the number $s$ of products, requires additional probabilistic arguments. In the case of \Cref{thm:matrix_product_bulk_metric_intro} and \Cref{thm:corner_product_bulk_metric_intro} these use properties of the Hall-Littlewood process sampling algorithm given in \cite[Proposition 5.3]{van2020limits}. The proofs of \Cref{thm:matrix_product_bulk_metric_intro} and \Cref{thm:corner_product_bulk_metric_intro}, like the statements, are essentially the same. The proof of the related result for $\Pois^{(N)}$, \Cref{thm:matrix_stat_dist}, is structurally exactly the same as those but some must be done in parallel due to slightly different formulas. Since these are actually slightly easier to work with, we give most arguments first for the cases relevant to \Cref{thm:matrix_stat_dist}.

\begin{rmk}\label{rmk:relation_to_sawin_wood}
When $\la$ is distributed as the singular numbers of a random matrix $A \in \Mat_N(\Z_p)$, it is equivalent to the cokernel
\begin{equation}
\coker(A) := \Z_p^N/A\Z_p^N \cong \bigoplus_{i=1}^N \Z/p^{\la_i}\Z.
\end{equation}
We noticed later that in this setting, the expectation \eqref{eq:skew_observable_intro} is exactly the \emph{$p$-group moment}, the expected number of surjective group homomorphisms
\begin{equation}
\E\left[\#\Sur\left(\coker(A),\bigoplus_i \Z/p^{\mu_i}\Z\right)\right]
\end{equation}
(up to a constant $Q_\mu(t,t^2,\ldots;0,t)^{-1}$ depending only on $\mu$, see \cite[Proposition 6.2]{nguyen2022universality}). Such moments of abelian groups are the main tool used to study the cokernels of integer and $p$-adic matrices in the works mentioned in the Preface. In fact, in retrospect our computations provide an alternative proof, via symmetric function theory, of the formulas for probabilities in terms of moments given recently by Sawin-Wood \cite{sawin2022moment} in the case of finite abelian groups; the details will appear in a forthcoming paper \cite{van2023+symmetric}. 
\end{rmk}

\begin{rmk}\label{rmk:triangular matrices}
The techniques we develop here also apply, with some additional work, in the problem of the limiting distribution of nilpotent Jordan blocks (equivalently, ranks of powers) of a random $n \times n$ strictly upper-triangular matrix $A$ over $\F_q$, studied by Kirillov \cite{kirillov1995variations} and Borodin \cite{borodin1995limit,borodin1999lln}. The same random variable $\cL_{t,\chi}$, with $t=q^{-1}$ and $\chi$ depending on $n$, also appears in these limits \cite{van2023+rank}.
\end{rmk}

\begin{rmk}
Though we have stated them over $\Z_p$, our results hold with $\Z_p$ replaced by the ring of integers of any non-Archimedean local field with finite residue field of size $q$, and the parameter $t=p^{-1}$ in formulas replaced by $q^{-1}$. This is because the only result we need in order to translate to Hall-Littlewood measures is \Cref{thm:hl_meas_matrices}, which holds in this generality, see \cite[Remark 4]{van2020limits}.
\end{rmk}

\subsection{Problems with the moment problem.} We conclude by mentioning issues with another plausible approach to our convergence results. Most previous literature on Macdonald process asymptotics proceeds by analyzing contour integral formulas for certain `$q$-moment' observables of the random partition $\la$. In the Hall-Littlewood case, these are the so-called joint $t$-moments 
\begin{equation}
\E[t^{-\kappa_1 \la_1' - \kappa_2 \la_2' - \ldots - \kappa_k \la_k'}]
\end{equation}
of Hall-Littlewood measures such as \eqref{eq:uniform_hl_intro}, where $\kappa_1 \geq \ldots \geq \kappa_k \geq 0$ are integers. These contour integral formulas were given in the Hall-Littlewood case by Bufetov-Matveev \cite{bufetov2018hall}, following similar formulas of Borodin-Corwin \cite{borodin2014macdonald} for a different case of Macdonald processes. Such contour integral formulas are amenable to asymptotics: in \cite{borodin2014macdonald} and many subsequent works---for example, Borodin-Gorin \cite{borodin2015general}, Borodin-Corwin-Ferrari \cite{borodin2014free,borodin2018anisotropic}, Dimitrov \cite{dimitrov2018kpz}---convergence of these moments is enough to deduce convergence in distribution and identify the limit object. 

In our setting, unfortunately, the resulting moment problem is indeterminate, as we saw above for $\cL_{t,\chi}^{(1)}$. This was discussed at the end of \cite[Section 5]{van2022q}, where the limiting $t$-moments 
\begin{equation}
\lim_{\substack{\tau \to \infty \\ \tau \in t^{\Z+\zeta}}} \E[t^{m \cdot (\Pois^{(\infty)}(\tau)_1' - \log_{t^{-1}}\tau)}]
\end{equation}
of the particle system $\Pois^{(\infty)}$ (see \Cref{def:poisson_walks}) were computed, and similarly for joint moments of $\Pois^{(\infty)}(\tau)_1',\ldots,\Pois^{(\infty)}(\tau)_k'$. These moments were found to be independent of $\zeta \in \R$, but a limiting random variable
\begin{equation}\label{eq:limit_pois_infty}
\lim_{\substack{\tau \to \infty \\ \tau \in t^{\Z+\zeta}}} \Pois^{(\infty)}(\tau)_1' - \log_{t^{-1}}\tau
\end{equation}
must take values in $\Z+\zeta$. Hence the limiting moments, which agree for different $\zeta$, cannot determine a unique random variable. Without this, it was not clear how to establish the convergence \eqref{eq:limit_pois_infty}, and we left this and the corresponding joint statement for $\Pois^{(\infty)}(\tau)_1',\ldots,\Pois^{(\infty)}(\tau)_k'$ as a conjecture. This paper settles the conjecture, see \Cref{sec:indet_gen_k} for further discussion. We do not use the $t$-moment observables at all, and other than the shared starting point of the observables \eqref{eq:skew_observable_intro}, the asymptotic analysis outlined in the previous subsection seems fairly different from other works of which we are aware (as does the discrete limit $\cL_{t,\chi}$ itself).

\subsection{Outline.} We give background on $p$-adic numbers and matrices over them in \Cref{sec:p-adic}, and on symmetric functions and Macdonald polynomials in \Cref{sec:macdonald_background}. In \Cref{sec:stationary_law_hl} we compute limiting probabilities for a certain Hall-Littlewood measure, yielding the integral formulas for the distribution of $\cL_{k,t,\chi}$. We further derive the summation formulas in \Cref{sec:residues}, and show these formulas define a valid random variable in \Cref{sec:tightness_and_limit}---see \Cref{thm:stat_dist_1pt}. In \Cref{sec:indeterminate} we discuss the indeterminate Stieltjes moment problem associated with $\cL_{t,\chi}^{(1)}$, and the joint version. We compute the $k=1,2$ cases of the residue formula explicitly in \Cref{sec:examples}, as well as an example for $k=3$. \Cref{sec:alpha} gives analogues of these results for limits from another class of Hall-Littlewood processes, those with one infinite principal and one repeated alpha specialization, which is needed for the matrix product results in \Cref{thm:matrix_product_bulk_metric_intro} and \Cref{thm:corner_product_bulk_metric_intro}. Finally, in \Cref{sec:pois_mat_prod} we deduce those theorems from the previously established asymptotics, by the probabilistic sampling algorithm mentioned above. In \Cref{appendix:existing_work} we give additional detail on the existing work on complex random matrix products and the analogies between it and the present work. In \Cref{sec:dbm_pois_appendix} we give an alternate proof of a result of \cite{van2023p} using Hall-Littlewood machinery.

\textbf{Acknowledgments.} I am deeply grateful to Alexei Borodin for invaluable conversations, encouragement, and advice throughout this project, and for feedback on the exposition. I also wish to thank Amol Aggarwal, Andrew Ahn, Gernot Akemann, Ivan Corwin, Evgeni Dimitrov, Maurice Duits, Vadim Gorin, and Melanie Matchett Wood for various helpful discussions, and the anonymous referees for helpful comments which improved the exposition (and in particular led to cleaner statements of the main theorems). This paper is based on PhD thesis work \cite{van2023asymptotics} which was supported by an NSF Graduate Research Fellowship under Grant No. 1745302, and the paper itself was written (and some results added) while supported by the European Research Council (ERC), Grant Agreement No. 101002013.

\section{$p$-adic matrix background} \label{sec:p-adic}

We begin with a few paragraphs of background which are essentially quoted from \cite{van2020limits}, and are a condensed version of the exposition in \cite[Section 2]{evans2002elementary}, to which we refer any reader desiring a more detailed introduction to $p$-adic numbers geared toward a probabilistic viewpoint. Fix a prime $p$. Any nonzero rational number $r \in \Q^\times$ may be written as $r=p^k (a/b)$ with $k \in \Z$ and $a,b$ coprime to $p$. Define $|\cdot|: \Q \to \R$ by setting $|r|_p = p^{-k}$ for $r$ as before, and $|0|_p=0$. Then $|\cdot|_p$ defines a norm on $\Q$ and $d_p(x,y) :=|x-y|_p$ defines a metric. We additionally define $\val_p(r)=k$ for $r$ as above and $\val_p(0) = \infty$, so $|r|_p = p^{-\val_p(r)}$. We define the \emph{field of $p$-adic numbers} $\Q_p$ to be the completion of $\Q$ with respect to this metric, and the \emph{$p$-adic integers} $\Z_p$ to be the unit ball $\{x \in \Q_p : |x|_p \leq 1\}$. It is not hard to check that $\Z_p$ is a subring of $\Q_p$. We remark that $\Z_p$ may be alternatively defined as the inverse limit of the system $\ldots \to \Z/p^{n+1}\Z \to \Z/p^n \Z \to \cdots \to \Z/p\Z \to 0$, and that $\Z$ naturally includes into $\Z_p$. 

$\Q_p$ is noncompact but is equipped with a left- and right-invariant (additive) Haar measure; this measure is unique if we normalize so that the compact subgroup $\Z_p$ has measure $1$. The restriction of this measure to $\Z_p$ is the unique Haar probability measure on $\Z_p$, and is explicitly characterized by the fact that its pushforward under any map $r_n:\Z_p \to \Z/p^n\Z$ is the uniform probability measure. For concreteness, it is often useful to view elements of $\Z_p$ as `power series in $p$' $a_0 + a_1 p + a_2 p^2 + \ldots$, with $a_i \in \{0,\ldots,p-1\}$; clearly, these specify a coherent sequence of elements of $\Z/p^n\Z$ for each $n$. The Haar probability measure then has the alternate explicit description that each $a_i$ is iid uniformly random from $\{0,\ldots,p-1\}$. Additionally, $\Q_p$ is isomorphic to the ring of Laurent series in $p$, defined in exactly the same way.

Similarly, $\GL_N(\Q_p)$ has a unique left- and right-invariant measure for which the total mass of the compact subgroup $\GL_N(\Z_p)$ is $1$. The restriction of this measure to $\GL_N(\Z_p)$, which we denote by $M_{Haar}(\GL_N(\Z_p))$, pushes forward to $\GL_N(\Z/p^n\Z)$ and is the uniform measure on the finite group $\GL_N(\Z/p^n\Z)$. This gives an alternative characterization of the Haar measure. 

The following standard result, \Cref{prop:smith}, is sometimes known either as Smith normal form or Cartan decomposition.

\begin{prop}\label{prop:smith}
Let $n \leq m$. For any nonsingular $A \in \Mat_{n \times m}(\Q_p)$, there exist $U \in \GL_n(\Z_p), V \in \GL_m(\Z_p)$ such that $UAV = \diag_{n \times m}(p^{\la_1},\ldots,p^{\la_n})$ where $\la \in \Sig_n$ is an integer signature (see \eqref{eq:def_sig} for the definition). Furthermore, there is a unique such $n$-tuple $\la$. The parts $\la_i$ are known as the \emph{singular numbers} of $A$, and we denote them by $\SN(A) = \la = (\la_1,\ldots,\la_n)$.
\end{prop}

We will often write $\diag_{n \times N}(p^\la)$ for $\diag_{n \times N}(p^{\la_1},\ldots,p^{\la_n})$, and also omit the dimensions $n \times N$ when they are clear from context. We note also that for any $\la \in \Sig_N$, the orbit $\GL_N(\Z_p) \diag(p^\la) \GL_N(\Z_p)$ is compact. The restriction of the Haar measure on $\GL_N(\Q_p)$ to such a double coset, normalized to be a probability measure, is the unique $\GL_N(\Z_p) \times \GL_N(\Z_p)$-invariant probability measure on $\GL_N(\Q_p)$ with singular numbers $\la$, and all $\GL_N(\Z_p) \times \GL_N(\Z_p)$-probability measures are convex combinations of these for different $\la$. These measures may be equivalently described as $U \diag(p^{\la_1},\ldots,p^{\la_N}) V$ where $U,V$ are independently distributed by the Haar probability measure on $\GL_N(\Z_p)$. More generally, if $n \leq m$ and $U \in \GL_n(\Z_p), V \in \GL_m(\Z_p)$ are Haar distributed and $\mu \in \Sig_n$, then $U \diag_{n \times m}(p^\mu) V$ is invariant under $\GL_n(\Z_p) \times \GL_m(\Z_p)$ acting on the left and right, and is the unique such bi-invariant measure with singular numbers given by $\mu$.

\section{Symmetric function background}\label{sec:macdonald_background}

We denote by $\cP$ the set of all integer partitions $(\la_1,\la_2,\ldots)$, i.e. sequences of nonnegative integers $\la_1 \geq \la_2 \geq \cdots$ which are eventually $0$. These may be represented by Ferrers diagrams as in \Cref{fig:arms_and_legs}.

We call the integers $\la_i$ the \emph{parts} of $\la$, set $\la_i' = \#\{j: \la_j \geq i\}$, and write $m_i(\la) = \#\{j: \la_j = i\} = \la_i'-\la_{i+1}'$. We write $\len(\la)$ for the number of nonzero parts, and denote the set of partitions of length $\leq n$ by $\cP_n$. We write $\mu \prec \la$ or $\la \succ \mu$ if $\la_1 \geq \mu_1 \geq \la_2 \geq \mu_2 \geq \cdots$, and refer to this condition as \emph{interlacing}. A stronger partial order is defined by containment of Ferrers diagrams, which we write as $\mu \subset \la$, meaning $\mu_i \leq \la_i$ for all $i$. Finally, we denote the partition with all parts equal to zero by $\emptyset$. 

We denote by $\La_n$ the ring $\C[x_1,\ldots,x_n]^{S_n}$ of symmetric polynomials in $n$ variables $x_1,\ldots,x_n$. For a symmetric polynomial $f$, we will often write $f(\bx)$ for $f(x_1,\ldots,x_n)$ when the number of variables is clear from context. We will also use the shorthand $\bx^\la:= x_1^{\la_1} x_2^{\la_2} \cdots x_n^{\la_n}$ for $\la \in \cP_n$. A simple $\C$-basis for $\La_n$ is given by the \emph{monomial symmetric polynomials} $\{m_\la(\bx): \la \in \Y_n\}$ defined by 
\[
m_\la(\bx) = \sum_{\sigma \in S_n/\text{Stab}(\la)} \sigma(\bx^\la)
\]
where $\sigma$ acts by permuting the variables, and $\text{Stab}(\la)$ is the subgroup of permutations such that $\sigma(\bx^\la) = \bx^\la$. It is also a very classical fact that the power sum symmetric polynomials 
\[p_k(\bx) = \sum_{i=1}^n x_i^k, k =1,\ldots,n\]
are algebraically independent and algebraically generate $\La_n$, and so by defining 
\begin{equation*}
    p_\la(\bx) := \prod_{i \geq 1} p_{\la_i}(\bx)
\end{equation*}
for $\la \in \Y$ with $\la_1 \leq n$, we have that $\{p_\la(\bx): \la_1 \leq n\}$ forms another basis for $\La_n$. 

Another special basis for $\La_n$ is given by the \emph{Macdonald polynomials} $P_\la(\bx;q,t)$, which depend on two additional parameters $q$ and $t$ which may in general be complex numbers, though in probabilistic contexts we take $q,t \in (-1,1)$. Our first definition of them requires a certain scalar product on $\La_n$.

\begin{defi}[{\cite[Chapter VI, (9.10)]{mac}}]\label{def:torus_product}
For polynomials $f,g \in \La_n$, define
\begin{equation}
\label{eq:torus_product}
\lan f, g \ran_{q,t;n}' := \frac{1}{n! (2 \pi \bi)^n} \int_{\T^n} f(z_1,\ldots,z_n) \overline{g(z_1,\ldots,z_n)} \prod_{1 \leq i \neq j \leq n} \frac{(z_iz_j^{-1};q)_\infty}{(tz_iz_j^{-1};q)_\infty} \prod_{i=1}^n \frac{dz_i}{z_i},
\end{equation}
where $\T$ denotes the unit circle with usual counterclockwise orientation, and to avoid confusion we clarify that the product is over $\{(i,j) \in \Z: 1 \leq i,j \leq n, i \neq j\}$.
\end{defi}

\begin{defi}\label{def:mac_poly_torus}
The \emph{Macdonald symmetric polynomials} $P_\la(x_1,\ldots,x_n;q,t), \la \in \Y_n$ are defined by the following two properties:
\begin{enumerate}
\item They are `monic' and upper-triangular with respect to the $m_\la(\bx)$ basis, in the sense that they expand as 
\begin{equation}
P_\la(\bx;q,t) = m_\la(\bx) + \sum_{\mu < \la} R_{\la \mu}(q,t) m_\mu(\bx)
\end{equation}
where $<$ denotes the lexicographic order on partitions.
\item They are orthogonal with respect to $\lan \cdot, \cdot \ran_{q,t;n}'$. 
\end{enumerate}
\end{defi}

These conditions \emph{a priori} overdetermine the set $\{P_\la(\bx;q,t): \la \in \Y_n\}$, and it is a theorem which follows from \cite[VI (4.7)]{mac} that the Macdonald symmetric polynomials do indeed exist. It is then also clear that they form a basis for $\Y_n$. We note that this paper will only ever consider the two special cases of Macdonald polynomials when either $q$ or $t$ is set to $0$, which are called \emph{Hall-Littlewood polynomials} and \emph{$q$-Whittaker polynomials} respectively.

\begin{defi}\label{def:Q}
For $\la \in \Y_n$, the \emph{dual Macdonald polynomial} $Q_\la$ is given by 
\begin{equation}
Q_\la(x_1,\ldots,x_n;q,t) := b_\la(q,t) P_\la(x_1,\ldots,x_n;q,t)
\end{equation}
where $b_\la$ is defined as follows (see \cite[p339, (6.19)]{mac}). We associate to $\la$ its Ferrers diagram as in \Cref{fig:arms_and_legs}. The boxes in the diagram correspond to pairs $(i,j)$ with $1 \leq i \leq \la_j', 1 \leq j \leq \la_i$. For such a box $s=(i,j)$, we define its \emph{arm-length} $a(s)$ and \emph{leg-length} $\ell(s)$ by the horizontal (resp. vertical) distance from $s$ to the edge of the diagram as in \Cref{fig:arms_and_legs}, explicitly
\begin{align}
a(s) &= \la_i - j\\ 
\ell(s) &= \la_j' - i.
\end{align}
The formula is then
\begin{equation}
\label{eq:blam_formula}
b_\la(q,t) = \prod_{s \in \la} \frac{1-q^{a(s)}t^{\ell(s)+1}}{1-q^{a(s)+1}t^{\ell(s)}}
\end{equation}
where the product is over boxes $s$ inside the diagram of $\lambda$.
\end{defi}

\begin{figure}
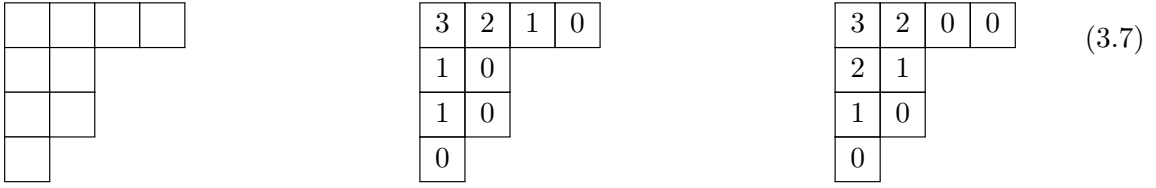

\begin{equation}
\begin{ytableau} *(white) & *(white) & *(white) & *(white)  \\ *(white) & *(white) \\ *(white) & *(white) \\ *(white)  \end{ytableau} \quad \quad \quad \quad \quad \quad \quad \quad \begin{ytableau} 3 & 2 & 1 & 0  \\ 1 & 0\\ 1 & 0 \\ 0  \end{ytableau}  \quad \quad \quad \quad \quad \quad \quad \quad \begin{ytableau} 3 & 2 & 0 & 0  \\ 2 & 1 \\ 1 & 0 \\ 0  \end{ytableau} 
\end{equation}
\caption{The Ferrers diagram of $\la = (4,2,2,1)$ (left), with $a(s)$ and $\ell(s)$ listed for each box $s$ (middle and right respectively).}\label{fig:arms_and_legs}
\end{figure}

The constant multiples $b_\la(q,t)$ are chosen so that the $Q_\la(\bx;q,t)$ form a dual basis to the $P_\la(\bx;q,t)$ with respect to a different scalar product which is related to (a renormalized version of) $\lan \cdot, \cdot \ran_{q,t;n}'$ in the $n \to \infty$ limit, see \cite[Chapter VI, (9.9)]{mac}. Because that scalar product is not necessary for our arguments while the explicit formula for $b_\la(q,t)$ is, we have given the rather unmotivated definition of $Q_\la$ above. While on the subject, we record a computation needed later.

\begin{lemma}\label{thm:blam_computation}
In the $q$-Whittaker specialization,
\begin{equation}
b_\la(q,0) = \prod_i \frac{1}{(q;q)_{\la_i - \la_{i+1}}}.
\end{equation}
\end{lemma}
\begin{proof}
Because $\ell(s) \geq 0$, the numerator of \eqref{eq:blam_formula} is always $1$ when $t=0$, so 
\begin{equation}
b_\la(q,0) = \prod_{\substack{s \in \la \\ \ell(s) = 0}} \frac{1}{1-q^{a(s)+1}} = \prod_i \frac{1}{(q;q)_{\la_i - \la_{i+1}}}.
\end{equation}
\end{proof}

Because the $P_\la$ form a basis for the vector space of symmetric polynomials in $n$ variables, there exist symmetric polynomials $P_{\la/\mu}(x_1,\ldots,x_{n-k};q,t) \in \La_{n-k}$ indexed by $\la \in \Y_{n+k}, \mu \in \Y_n$ which are defined by
\begin{equation}\label{eq:def_skewP}
    P_\la(x_1,\ldots,x_{n+k};q,t) = \sum_{\mu \in \Y_n} P_{\la/\mu}(x_{n+1},\ldots,x_{n+k};q,t) P_\mu(x_1,\ldots,x_n;q,t).
\end{equation}
It follows easily from \eqref{eq:def_skewP} that for any $1 \leq d \leq k-1$,
\begin{equation}\label{eq:gen_branch}
    P_{\la/\mu}(x_1,\ldots,x_k;q,t) = \sum_{\nu \in \Y_d} P_{\la/\nu}(x_{d+1},\ldots,x_k;q,t) P_{\nu/\mu}(x_1,\ldots,x_d;q,t).
\end{equation}
We define $Q_{\la/\mu}$ by \eqref{eq:def_skewP} with $Q$ in place of $P$, and it is similarly clear that \eqref{eq:gen_branch} holds for $Q$. An important property of (skew) Macdonald polynomials for probabilistic purposes is the \emph{Cauchy identity} below.

\begin{prop}\label{thm:finite_cauchy}
Let $\nu, \mu \in \Y$. Then
\begin{multline}\label{eq:finite_cauchy}
    \sum_{\kappa \in \Y} P_{\kappa/\nu}(x_1,\ldots,x_n;q,t)Q_{\kappa/\mu}(y_1,\ldots,y_m;q,t) \\
    = \prod_{\substack{1 \leq i \leq n \\ 1 \leq j \leq m}} \frac{(tx_iy_j;q)_\infty }{(x_iy_j;q)_\infty} \sum_{\la \in \Y} Q_{\nu/\la}(y_1,\ldots,y_m;q,t) P_{\mu/\la}(x_1,\ldots,x_n;q,t).
\end{multline}
\end{prop}

The above identity should be interpreted as an identity of formal power series in the variables, after expanding the $1/(1-q^\ell x_iy_j)$ factors as geometric series. It may be seen as partial motivation for the definition of the $Q$ polynomials earlier: the constant factors there are the ones needed for such an identity to hold. For later convenience we set
\begin{equation}\label{eq:def_cauchy_kernel}
     \Pi_{q,t}(\bx;\by) := \prod_{\substack{1 \leq i \leq n \\ 1 \leq j \leq m}} \frac{(tx_iy_j;q)_\infty }{(x_iy_j;q)_\infty} = \exp\left(\sum_{\ell = 1}^\infty \frac{1-t^\ell}{1-q^\ell}\frac{1}{\ell}p_\ell(\bx)p_\ell(\by)\right),
\end{equation}
where the second equality in \eqref{eq:def_cauchy_kernel} is not immediate but is shown in \cite{mac}. 

The skew Macdonald polynomials may also be made explicit, which is needed for later computations.

\begin{defi}\label{def:psi_varphi}
For $\la, \mu \in \Y$ with $\mu \subset \la$, let $f(u) := (tu;q)_\infty/(qu;q)_\infty$ and define 
\begin{equation}\label{eq:pbranch}
    \psi_{\la/\mu}(q,t) := \bbone(\mu \prec \la) \prod_{1 \leq i \leq j \leq \len(\la) } \frac{f(t^{j-i}q^{\mu_i-\mu_j})f(t^{j-i}q^{\la_i-\la_{j+1}})}{f(t^{j-i}q^{\la_i-\mu_j})f(t^{j-i}q^{\mu_i-\la_{j+1}})}
\end{equation}
and
\begin{equation}\label{eq:qbranch}
    \varphi_{\la/\mu}(q,t) := \bbone(\mu \prec \la) \prod_{1 \leq i \leq j \leq \len(\mu)} \frac{f(t^{j-i}q^{\la_i-\la_j})f(t^{j-i}q^{\mu_i-\mu_{j+1}})}{f(t^{j-i}q^{\la_i-\mu_j})f(t^{j-i}q^{\mu_i-\la_{j+1}})}.
\end{equation}
When $q$ and $t$ are clear from context we will often write $\psi_{\la/\mu}$ and $\varphi_{\la/\mu}$ without arguments.
\end{defi}

\begin{prop}[{\cite[VI.7, (7.13) and (7.13')]{mac}}] \label{thm:branching_formulas}
For $\la \in \Y_n, \mu \in \Y_{n-k}$, we have
\begin{equation}\label{eq:skewP_branch_formula}
    P_{\la/\mu}(x_1,\ldots,x_k;q,t) = \sum_{\mu = \la^{(0)} \prec \la^{(1)} \prec \cdots \prec \la^{(k)}= \la} \prod_{i=1}^{k-1} x_i^{|\la^{(i+1)}|-|\la^{(i)}|}\psi_{\la^{(i+1)}/\la^{(i)}}
\end{equation}
and
\begin{equation}\label{eq:skewQ_branch_formula}
    Q_{\la/\mu}(x_1,\ldots,x_k;q,t) = \sum_{\mu = \la^{(0)} \prec \la^{(1)} \prec \cdots \prec \la^{(k)}=\la} \prod_{i=1}^{k-1} x_i^{|\la^{(i+1)}|-|\la^{(i)}|}\varphi_{\la^{(i+1)}/\la^{(i)}}.
\end{equation}
\end{prop}

The explicit forms of the Hall-Littlewood and $q$-Whittaker special cases of the formulas in \Cref{def:psi_varphi} will be useful, and are easy to establish by direct computation from \Cref{def:psi_varphi}. 

\begin{lemma}\label{thm:hl_qw_branch_formulas}
Let $\la,\mu \in \Y$ with $\mu \prec \la$. In the Hall-Littlewood case $q=0$ the formulas of \Cref{def:psi_varphi} specialize to
\begin{align}\label{eq:HL_branch}
\begin{split}
\psi_{\la/\mu}(0,t) &= \prod_{\substack{i > 0\\ m_i(\mu) = m_i(\la)+1}} (1-t^{m_i(\mu)})\\ 
\varphi_{\la/\mu}(0,t) &= \prod_{\substack{i > 0\\ m_i(\la) = m_i(\mu)+1}} (1-t^{m_i(\la)}).
\end{split}
\end{align}
In the $q$-Whittaker case $t=0$ they specialize to 
\begin{align}\label{eq:qw_branch}
\begin{split}
\psi_{\la/\mu}(q,0) &= \prod_{i=1}^{\len(\mu)} \sqbinom{\la_i-\la_{i+1}}{\la_i-\mu_i}_q\\ 
\varphi_{\la/\mu}(q,0) &= \frac{1}{(q;q)_{\la_1-\mu_1} }\prod_{i=1}^{\len(\la)-1} \sqbinom{\mu_i-\mu_{i+1}}{\mu_i-\la_{i+1}}_q.
\end{split}
\end{align}
For pairs $\mu,\la$ with $\mu \not \prec \la$, all of the above are $0$.
\end{lemma}

We will require two, in a sense orthogonal, extensions of Macdonald polynomials: symmetric Laurent polynomials in finitely many variables, and symmetric functions---informally, symmetric polynomials in infinitely many variables. 

\subsection{Symmetric Laurent polynomials.} We wish to extend the indices of Macdonald polynomials in $n$ variables from the set $\Y_n$ of partitions of length at most $n$ to the set $\Sig_n$ of signatures of length $n$, where we recall $\Sig_n := \{(x_1,\ldots,x_n) \in \Z^n: x_1 \geq \ldots \geq x_n\}$. Recall that for signatures $\mu \in \operatorname{Sig}_{k-1}, \lambda \in \operatorname{Sig}_{k}$ we write
$$
\mu \prec \lambda \Longleftrightarrow \lambda_{1} \geq \mu_{1} \geq \lambda_{2} \geq \ldots \geq \mu_{k-1} \geq \lambda_{k},
$$
$|\mu|=\sum_{i} \mu_{i}$, and $\mu-(d[k-1])=\left(\mu_{1}-d, \ldots, \mu_{k-1}-d\right)$.

\begin{lemma}\label{thm:P_shift_property}
Let $\la \in \Y_n, \mu \in \Y_{n-k}$, and $d \in \Z_{\geq 0}$. Then 
\begin{equation}
P_{(\la+(d[n]))/(\mu+(d[n-k]))}(x_1,\ldots,x_k;q,t) = (x_1 \cdots x_k)^d P_{\la/\mu}(x_1,\ldots,x_k;q,t).
\end{equation}
\end{lemma}
\begin{proof}
The claim follows from \eqref{eq:skewP_branch_formula} together with the observation from the explicit formula \eqref{eq:pbranch} that for $\la^{(i)}$ as in \eqref{eq:skewP_branch_formula},
\begin{equation}\label{eq:psi_invariant}
\psi_{\la^{(i+1)}/\la^{(i)}}(q,t) = \psi_{(\la^{(i+1)}+(d[n-k+i+1]))/(\la^{(i)}+(d[n-k+i]))}.
\end{equation}
\end{proof}

This motivates the following.

\begin{defi}\label{def:mac_laurent}
For any $\la \in \Sig_n, \nu \in \Sig_{n-1}$ we define
\begin{equation}\label{eq:pbranch_sig}
    \psi_{\la/\nu}(q,t) :=  \prod_{1 \leq i \leq j \leq n-1} \frac{f(t^{j-i}q^{\nu_i-\nu_j})f(t^{j-i}q^{\la_i-\la_{j+1}})}{f(t^{j-i}q^{\la_i-\nu_j})f(t^{j-i}q^{\nu_i-\la_{j+1}})}
\end{equation}
and for $\la \in \Sig_n, \mu \in \Sig_{n-k}$ define 
\begin{equation}\label{eq:skewP_branch_formula_sig}
    P_{\la/\mu}(x_1,\ldots,x_k;q,t) := \sum_{\mu = \la^{(0)} \prec \la^{(1)} \prec \cdots \prec \la^{(k)}= \la} \prod_{i=1}^{k-1} x_i^{|\la^{(i+1)}|-|\la^{(i)}|}\psi_{\la^{(i+1)}/\la^{(i)}}.
\end{equation}
\end{defi}

\begin{cor}\label{thm:signature_shift}
Let $\la \in \Sig_n, \mu \in \Sig_{n-k}$, and $d \in \Z$. Then 
\begin{equation}
P_{(\la+(d[n]))/(\mu+(d[n-k]))}(x_1,\ldots,x_k;q,t) = (x_1 \cdots x_k)^d P_{\la/\mu}(x_1,\ldots,x_k;q,t).
\end{equation}
\end{cor}
\begin{proof}
Same as \Cref{thm:P_shift_property}.
\end{proof}

\begin{rmk}\label{rmk:don't_sig_Q}
We have not stated \Cref{thm:P_shift_property} and \Cref{def:mac_laurent} for the dual Macdonald polynomials, for the simple reasons that (1) the naive versions do not hold, and (2) we do not need this for our proofs. In fact, $\varphi_{\la/\nu}(q,t)$ and the skew polynomials $Q_{\la/\mu}(\bx;q,t)$ must be reinterpreted in a more nontrivial way in order to make such statements true, see \cite[Section 2.1]{van2020limits}.
\end{rmk}

\begin{rmk}
Note that for $\la \in \Sig_n, \mu \in \Sig_{n-1}$ the formula for $\psi_{\la/\mu}(q,0)$ in \Cref{thm:hl_qw_branch_formulas} continues to makes sense, while for the Hall-Littlewood case one should instead interpret
\begin{equation}\label{eq:hl_sig_pbranch}
\psi_{\la/\mu}(0,t) = \prod_{\substack{m_i(\mu) = m_i(\la)+1}} (1-t^{m_i(\mu)})
\end{equation}
(without the $i > 0$ restriction in the product) in order for the translation-invariance property \eqref{eq:psi_invariant} to hold. If $\la \in \Sig_n^+$ (and hence $\mu \in \Sig_{n-1}^+$ by interlacing) then both \eqref{eq:HL_branch} and \eqref{eq:hl_sig_pbranch} give the same result.
\end{rmk}

The defining orthogonality property of Macdonald polynomials also extends readily to their Laurent versions.

\begin{prop}\label{thm:laurent_orthogonality}
If $\la,\mu \in \Sig_n$ and $\la \neq \mu$, then 
\begin{equation}\label{eq:laurent_orthogonality_P}
\lan P_\la(\mbz;q,t), P_\mu(\mathbf{z};q,t) \ran'_{q,t;n} = 0.
\end{equation}
\end{prop}
\begin{proof}
Let $D \in \Z$ be such that $\la + (D[n])$ and $\mu+(D[n])$ both lie in $\Sig_n^+$. Then 
\begin{equation}
\lan P_{\la + (D[n])}(\mbz;q,t), P_{\mu+(D[n])}(\mbz;q,t) \ran'_{q,t;n} = 0
\end{equation} 
by the defining orthogonality property of Macdonald polynomials. However,
\begin{align}
\begin{split}
P_{\la + (D[n])}(\mbz;q,t) \overline{P_{\mu + (D[n])}(\mbz)} &= (z_1 \cdots z_n)^D P_\la(\mbz;q,t) \overline{(z_1 \cdots z_n)^D P_\mu(\mbz;q,t)} \\ 
&= P_\la(\mbz;q,t)\overline{P_\mu(\mbz;q,t)}
\end{split}
\end{align}
for any $z_1,\ldots,z_n \in \T$, so 
\begin{equation}
\lan P_{\la + (D[n])}(\mbz;q,t), P_{\mu+(D[n])}(\mbz;q,t) \ran'_{q,t;n} = \lan P_\la(\mbz;q,t), P_\mu(\mbz;q,t) \ran'_{q,t;n},
\end{equation} 
which completes the proof.
\end{proof}

We record also the following naive bound, useful for later analysis.

\begin{lemma}\label{thm:complex_hl_bound}
Let $\la \in \Sig_n$, $q,t \in (-1,1)$, and $z_1,\ldots,z_n \in \C^\times$. Then
\begin{equation}
|P_\la(z_1,\ldots,z_n;q,t)| \leq P_\la(|z_1|,\ldots,|z_n|;q,t).
\end{equation}
\end{lemma}
\begin{proof}
Follows by expanding $P_\la$ via the branching rule \Cref{def:mac_laurent}, noting that the coefficient of each monomial is nonnegative since $q,t \in (-1,1)$, and applying the triangle inequality. Note that if $z_1,\ldots,z_n$ were allowed to be $0$ and $\la$ were a signature with negative parts, the Macdonald Laurent polynomials would not be well-defined.
\end{proof}

\subsection{Symmetric functions.} It is often convenient to consider symmetric polynomials in an arbitrarily large or infinite number of variables, which we formalize as follows, heavily borrowing from the introductory material in \cite{van2022q}. One has a chain of maps
\[
\cdots \to \La_{n+1} \to \La_n \to \La_{n-1} \to \cdots \to 0
\]
where the map $\La_{n+1} \to \La_n$ is given by setting $x_{n+1}$ to $0$.
In fact, writing $\La_n^{(d)}$ for symmetric polynomials in $n$ variables of total degree $d$, one has 
\[
\cdots \to \La_{n+1}^{(d)} \to \La_n^{(d)} \to \La_{n-1}^{(d)} \to \cdots \to 0
\]
with the same maps. The inverse limit $\La^{(d)}$ of these systems may be viewed as symmetric polynomials of degree $d$ in infinitely many variables. From the ring structure on each $\La_n$ one gets a natural ring structure on $\La := \bigoplus_{d \geq 0} \La^{(d)}$, and we call this the \emph{ring of symmetric functions}. Because $p_k(x_1,\ldots,x_{n+1}) \mapsto p_k(x_1,\ldots,x_n)$ and $m_\la(x_1,\ldots,x_{n+1}) \mapsto m_\la(x_1,\ldots,x_n)$ (for $n \geq \len(\la)$) under the natural map $\La_{n+1} \to \La_n$, these families of symmetric polynomials define symmetric functions $p_k, m_\la \in \La$. An equivalent definition of $\La$ is $\Lambda := \C[p_1,p_2,\ldots]$ where $p_i$ are indeterminates; under the natural map $\Lambda \to \Lambda_n$ one has $p_i \mapsto p_i(x_1,\ldots,x_n)$.

The Macdonald polynomials satisfy a consistency property 
\begin{equation}
P_\la(x_1,\ldots,x_n,0;q,t) = P_\la(x_1,\ldots,x_n;q,t)
\end{equation}
for any $\la \in \Y$ (and similarly for the dual and skew polynomials). Hence here exist \emph{Macdonald symmetric functions}, denoted $P_\la,Q_\la$ as well, such that $P_\la \mapsto P_\la(\bx;q,t)$ under the natural map $\Lambda \to \Lambda_n$. Macdonald symmetric functions satisfy the skew Cauchy identity
\begin{multline}\label{eq:infinite_cauchy}
    \sum_{\kappa \in \Y} P_{\kappa/\nu}(\bx;q,t)Q_{\kappa/\mu}(\by;q,t) \\
    = \exp\left(\sum_{\ell = 1}^\infty \frac{1-t^\ell}{1-q^\ell} \frac{1}{\ell}p_\ell(\bx)p_\ell(\by)\right) \sum_{\la \in \Y} Q_{\nu/\la}(\by;q,t) P_{\mu/\la}(\bx;q,t).
\end{multline}
Here $P_{\kappa/\nu}(\bx;q,t)$ is an element of $\La$, a polynomial in $p_1(\bx),p_2(\bx),\ldots \in \La$, and summands such as $P_{\kappa/\nu}(\bx;q,t)Q_{\kappa/\mu}(\by;q,t)$ are interpreted as elements of a ring $\La \ot \La$ and both sides interpreted as elements of a completion thereof.

To get a probability measure on $\cP$ from the skew Cauchy identity, we would like homomorphisms $\phi: \La \to \C$ which take $P_\la$ and $Q_\la$ to $\R_{\geq 0}$---here we recall that we take $q,t \in (-1,1)$. Simply plugging in nonnegative real numbers for the variables in \eqref{eq:finite_cauchy} works, but does not yield all of them. However, a full classification of such homomorphisms, called \emph{Macdonald nonnegative specializations} of $\La$, was conjectured by Kerov \cite{kerov1992generalized} and proven by Matveev \cite{matveev2019macdonald}. We describe them now: they are associated to triples of $\{\alpha_n\}_{n \geq 1}, \{\beta_n\}_{n \geq 1},\tau$ (the Plancherel parameter) such that $\tau \geq 0$, $0 \leq \alpha_n,\beta_n $ for all $ n \geq 1$, and $\sum_n \alpha_n, \sum_n \beta_n < \infty$. These are typically called usual (or alpha) parameters, dual (or beta) parameters, and the Plancherel parameter\footnote{The terminology `Plancherel' comes from the fact that in the case $q=t$ where the Macdonald polynomials reduce to Schur polynomials, this specialization is related to (the poissonization of) the Plancherel measure on irreducible representations of the symmetric group $S_N$, see \cite{borodin2017representations}.} respectively. Given such a triple, the corresponding specialization is defined by 
\begin{align}\label{eq:p_specs}
\begin{split}
    p_1 &\mapsto \sum_{n \geq 1} \alpha_n + \frac{1-q}{1-t}\left(\tau + \sum_{n \geq 1} \beta_n\right) \\
    p_k &\mapsto \sum_{n \geq 1} \alpha_n^k + (-1)^{k-1}\frac{1-q^k}{1-t^k}\sum_{n \geq 1} \beta_n^k \quad \quad \text{ for all }k \geq 2.
\end{split}
\end{align}
Note that the above formula defines a specialization for arbitrary tuples of reals $\alpha_n,\beta_n$ and $\tau$ satisfying convergence conditions, but it will not in general be nonnegative.

\begin{defi}\label{def:spec_notation}
For the specialization $\theta$ defined by the triple $\{\alpha_n\}_{n \geq 1}, \{\beta_n\}_{n \geq 1}, \tau$, we write
\begin{equation}\label{eq:spec_argument_notation}
P_\la(\alpha(\alpha_1,\alpha_2,\ldots),\beta(\beta_1,\beta_2,\ldots),\gamma(\tau);q,t) := P_\la(\theta;q,t) := \theta(P_\la)
\end{equation}
and similarly for skew and dual Macdonald polynomials. Likewise, for any other specialization $\phi$ defined by parameters $\{\alpha'_n\}_{n \geq 1}, \{\beta'_n\}_{n \geq 1}, \tau'$, we let
\begin{align}\label{eq:spec_cauchy_kernel}
\begin{split}
   \Pi_{q,t}(\alpha(\alpha_1,\ldots),\beta(\beta_1,\ldots),\gamma(\tau);\alpha(\alpha'_1,\ldots),\beta(\beta'_1,\ldots),\gamma(\tau')) &:= \Pi_{q,t}(\theta;\phi) \\ 
   &:= \exp\left(\sum_{\ell = 1}^\infty \frac{1-t^\ell}{1-q^\ell}\frac{1}{\ell}\theta(p_\ell)\phi(p_\ell)\right). 
\end{split}
\end{align}
We will omit the $\alpha(\cdots)$ in notation if all alpha parameters are zero for the given specialization, and similarly for $\beta$ and $\gamma$.
\end{defi}

We refer to a specialization as 
\begin{itemize}
    \item \emph{pure alpha} if $\tau$ and all $\beta_n, n \geq 1$ are $0$.
    \item \emph{pure beta} if $\tau$ and all $\alpha_n, n \geq 1$ are $0$.
    \item \emph{Plancherel} if all $\alpha_n,\beta_n, n \geq 1$ are $0$.
\end{itemize}

On Macdonald polynomials these act as follows. 

\begin{prop}\label{thm:specialize_mac_poly}
Let $\la,\mu \in Y$ and $c_1,\ldots,c_n \in \R_{\geq 0}$. Then
\begin{align}\label{eq:spec_mac_pol}
\begin{split}
P_\la(\alpha(c_1,\ldots,c_n);q,t) &= P_\la(c_1,\ldots,c_n;q,t) \\ 
Q_\la(\alpha(c_1,\ldots,c_n);q,t) &= Q_\la(c_1,\ldots,c_n;q,t) \\ 
P_\la(\beta(c_1,\ldots,c_n);q,t) &= Q_{\la'}(c_1,\ldots,c_n;t,q) \\ 
Q_\la(\beta(c_1,\ldots,c_n);q,t) &= P_{\la'}(c_1,\ldots,c_n;t,q),
\end{split}
\end{align}
where in each case the left hand side is a specialized Macdonald symmetric function while the right hand side is a Macdonald polynomial with real numbers plugged in for the variables. Furthermore,
\begin{equation}\label{eq:alpha_gamma_limit}
P_\la(\gamma(\tau);q,t) = \lim_{D \to \infty} P_\la\left(\tau \cdot \frac{1-q}{1-t} \frac{1}{D}[D];q,t\right)
\end{equation}
and similarly for $Q$.
\end{prop}
The alpha case of \eqref{eq:spec_mac_pol}, and \eqref{eq:alpha_gamma_limit}, are straightforward from \eqref{eq:p_specs}. The $\beta$ case follows from properties of a certain involution on $\La$, see \cite[Chapter VI]{mac}, and explains the terminology `dual parameter'.

The pure alpha specialization $\alpha(u,ut,\ldots,ut^{n-1})$, often referred as a \emph{principal specialization}, produces simple factorized expressions for Macdonald polynomials. For brevity we give only the Hall-Littlewood case which is needed later. Let
\begin{equation}
    n(\la) := \sum_{i=1}^n (i-1)\la_i = \sum_{x \geq 1} \binom{\la_x'}{2}.
\end{equation}
The following formula is standard and may be easily derived from \eqref{eq:hl_explicit_intro}, and also follows directly from \cite[Ch. III.2, Ex. 1]{mac}.

\begin{prop}[Principal specialization formula]\label{thm:hl_principal_formulas}
For $\la \in \Y_n$,
\begin{align*}
    P_\la(u,ut,\ldots,ut^{n-1};0,t) &= u^{|\la|} t^{n(\la)} \frac{(t;t)_n}{(t;t)_{n-\len(\la)}\prod_{i \geq 1} (t;t)_{m_i(\la)}}.
\end{align*}
\end{prop}

When the principal specialization is infinite, nice formulas for the principally specialized skew Hall-Littlewood polynomials were shown in \cite{kirillov1998new}. The phrasing below follows {\cite[Theorem 3.3]{vanpeski2021halllittlewood}}.

\begin{prop}\label{thm:skew_hl_principal}
For any $\la,\mu \in \Y$, 
\begin{equation}
P_{\la/\mu}(u,ut,\ldots;0,t) = \bbone(\mu \subset \la) u^{|\la|-|\mu|} \prod_{x \geq 1} t^{\binom{\la_x'-\mu_x'}{2}} \frac{(t;t)_{\mu_x' - \la_{x+1}'}}{(t;t)_{\mu_x'-\la_x'} (t;t)_{\mu_x'-\mu_{x+1}'}}.
\end{equation}
\end{prop}

We note that for any nonnegative specializations $\theta,\phi$ with 
\begin{equation}\label{eq:finiteness_cauchy}
    \sum_{\la \in \Y}P_\la(\theta;q,t)Q_\la(\phi;q,t) < \infty,
\end{equation}
the specialized Cauchy identity
\begin{equation}\label{eq:specialized_cauchy}
    \sum_{\kappa \in \Y} P_{\kappa/\nu}(\theta;q,t)Q_{\kappa/\mu}(\phi;q,t) \\
    =\Pi_{q,t}(\theta;\psi) \sum_{\la \in \Y} Q_{\nu/\la}(\phi;q,t) P_{\mu/\la}(\theta;q,t).
\end{equation}
holds by applying $\theta \ot \phi$ to \eqref{eq:infinite_cauchy}. Similarly, we have a version of the branching rule for specializations: for $\la,\mu \in \Y$, 
\begin{equation}\label{eq:specialization_branch}
P_{\la/\mu}(\phi,\phi';q,t) = \sum_{\nu \in \Y: \mu \subset \nu \subset \la} P_{\la/\nu}(\phi;q,t)P_{\nu/\mu}(\phi';q,t),
\end{equation}
see e.g. \cite[(2.24)]{borodin2014macdonald}; the same holds with $Q$ in place of $P$. Here by $P_{\la/\mu}(\phi,\phi';q,t)$ we simply mean $\phi(P_{\la/\mu}) + \phi'(P_{\la/\mu})$. We will often refer to \eqref{eq:specialization_branch}, for either $P$ or $Q$ functions, as the branching rule.

\subsection{Macdonald processes.} One obtains probability measures on sequences of partitions using \eqref{eq:specialized_cauchy} as follows.

\begin{prop}\label{thm:specs_cancel}
Let $u \in \R$ and $\theta$ be the (not a priori nonnegative) specialization defined by $\alpha$ parameters $u,ut,ut^2,\ldots$ and $\beta$ parameter $-u,-uq,-uq^2,\ldots$. Then $\theta$ is the zero specialization, i.e. for any $\la,\mu \in \Y$
\begin{equation}
P_{\la/\mu}(\theta;q,t) = \bbone(\la=\mu).
\end{equation}
\end{prop}
\begin{proof}
By \eqref{eq:p_specs}, $\theta(p_i) = 0$ for all $i \geq 1$, and $P_{\la/\mu}$ is a polynomial in the $p_i$ with no constant term unless $\la=\mu$.
\end{proof}

\subsection{Hall-Littlewood processes and random matrices.} One obtains probability measures on sequences of partitions using \eqref{eq:specialized_cauchy} as follows.

\begin{defi}\label{def:mac_proc}
Let $\theta$ and $\phi_1,\ldots,\phi_k$ be Macdonald-nonnegative specializations such that each pair $\theta,\phi_i$ satisfies \eqref{eq:finiteness_cauchy}. The associated \emph{ascending Macdonald process} is the probability measure on sequences $\la^{(1)},\ldots,\la^{(k)}$ given by 
\[
\Pr(\la^{(1)},\ldots,\la^{(k)}) = \frac{Q_{\la^{(1)}}(\phi_1;q,t) Q_{\la^{(2)}/\la^{(1)}}(\phi_2;q,t) \cdots Q_{\la^{(k)}/\la^{(k-1)}}(\phi_k;q,t) P_{\la^{(k)}}(\theta;q,t)}{\prod_{i=1}^k \Pi_{q,t}(\phi_i;\theta)}.
\]
\end{defi}

The $k=1$ case of \Cref{def:mac_proc} is a measure on partitions, referred to as a \emph{Macdonald measure}. Instead of defining joint distributions all at once as above, one can define Markov transition kernels on $\cP$.

\begin{defi}\label{def:cauchy_dynamics}
Let $\theta,\phi$ be Macdonald nonnegative specializations satisfying \eqref{eq:finiteness_cauchy} and $\la$ be such that $P_\la(\theta) \neq 0$. The associated \emph{Cauchy Markov kernel} is defined by 
\begin{equation}\label{eq:def_HL_cauchy_dynamics}
    \Pr(\la \to \nu) = Q_{\nu/\la}(\phi) \frac{P_\nu(\theta)}{P_\la(\theta) \Pi(\phi; \theta)}.
\end{equation}
\end{defi}

It is clear that the ascending Macdonald process above is nothing more than the joint distribution of $k$ steps of a Cauchy Markov kernel with specializations $\phi_i,\theta$ at the $i\tth$ step and initial condition $\emptyset$. In this work we will only refer to the $q=0$ case, where the Macdonald polynomials are Hall-Littlewood polynomials and the corresponding measure (resp. process) is called a Hall-Littlewood measure (resp. process). The following special case is relevant, so we give it its own notation.

\begin{defi}\label{def:lambda_hl_planch}
For any $n \in \N \cup \{\infty\}$, we denote by $\la^{(n)}(\tau), \tau \geq 0$ the stochastic process on $\cP_n$ in continuous time $\tau$ with finite-dimensional marginals given by the Hall-Littlewood process
\begin{equation}\label{eq:hlproc_planch_def}
\Pr(\la^{(n)}(\tau_i) = \la(i) \text{ for all }i=1,\ldots,k) =  \frac{\left(\prod_{j=1}^k Q_{\la(j)/\la(j-1)}(\gamma(\tau_j-\tau_{j-1});0,t) \right) P_{\la(k)}(1,t,\ldots,t^{n-1};0,t)}{\exp\left(\frac{\tau_k(1-t^n)}{1-t}\right)}
\end{equation}
for each sequence of times $0 \leq \tau_1 \leq \tau_2 \leq \cdots \leq \tau_k$ and $\la(1),\ldots,\la(k)\in \cP_n$, where in the product we take the convention $\tau_0=0$ and $\la(0)$ is the zero partition. More generally, for $\nu \in \Sig_n$ we denote by $\la^{(n,\nu)}(\tau)$ the same process started at initial condition $\nu$, i.e. with marginals defined by $1/P_\nu(1,\ldots,t^{n-1};0,t)$ times the expression in \eqref{eq:hlproc_planch_def} where we instead take $\la(0) = \nu$. 
\end{defi}

Though we use the notation $\la^{(n,\nu)}$ in proofs to avoid dealing with extra normalizations, this is the same `$p$-adic Dyson Brownian motion' mentioned in the Introduction, and has an explicit description which we give now.
 
\begin{defi}
    \label{def:poisson_walks}
    For $n \in \N \cup \{\infty\}$ and $\mu \in \Sig_{n}$ and $t \in (0,1)$, we define the stochastic process $\Pois^{\mu,n}(\tau) = (\Pois^{\mu,n}_1(\tau),\ldots,\Pois^{\mu,n}_n(\tau))$ on $\Sig_n$ as follows. For each $1 \leq i \leq n$, $\Pois^{\mu,n}_i$ has an exponential clock of rate $t^i$, and when $\Pois^{\mu,n}_i$'s clock rings, $\Pois^{\mu,n}_i$ increases by $1$ if the resulting $n$-tuple is still weakly decreasing. If not, then $\Pois^{\mu,n}_{i-d}$ increases by $1$ instead and $\Pois^{\mu,n}_i$ remains the same, where $d \geq 0$ is the smallest index so that the resulting tuple is weakly decreasing. In the case of trivial initial condition we will often write $\Pois^{(n)}$ for $\Pois^{(0[n]),n}$. 
\end{defi}

\begin{figure}[htbp]
\begin{center}
\includegraphics[scale=1]{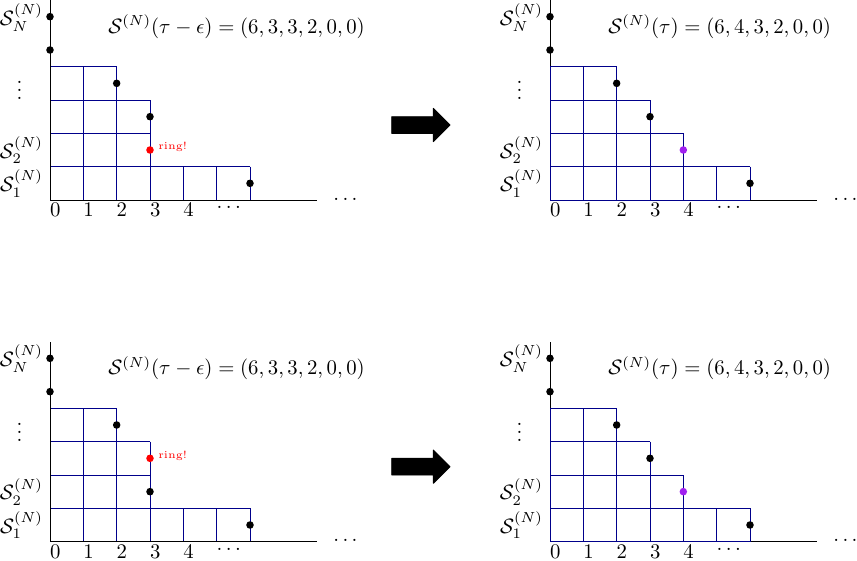}
\end{center}
\caption{A depiction of the process $\Pois^{(N)}(\tau) = (\Pois_1^{(N)}(\tau),\ldots,\Pois_N^{(N)}(\tau))$ as a Young diagram in English notation when $N=6$. If the clock associated to $\Pois^{(N)}_2 = 2$ rings at time $\tau$ and the process was previously in state $(6,3,3,2,0,0)$ (i.e. $\Pois^{(N)}(\tau-\eps) = (6,3,3,2,0,0)$ for all sufficiently small $\eps > 0$), then $\Pois^{(\infty)}_2 $ increases by $1$ and so $\Pois^{(N)}(\tau) = (6,4,3,2,0,0)$ (above). If instead the clock associated to $\Pois^{(N)}_3 = 2$ rings at time $\tau$ and the process was previously in state $(6,3,3,2,0,0)$, then $\Pois^{(\infty)}_2 $ increases instead by $1$ and so still $\Pois^{(N)}(\tau) = (6,4,3,2,0,0)$ (below), demonstrating the reflection condition.}\label{fig:poisson_walks_horiz}
\end{figure}

When $t=1/p$, the process $\Pois^{(N)}$ arises as the singular numbers of certain $N \times N$ $p$-adic matrix products with exponential waiting times between them \cite[Theorem 1.2]{van2023p}. It was further argued in that work that $\Pois^{(N)}$ is the natural $p$-adic analogue of the \emph{multiplicative Dyson Brownian motion}, which governs the singular values of a matrix $A \in \GL_N(\C)$ evolving under multiplicative Brownian motion. See \cite{van2023p} (from which \Cref{def:poisson_walks} and \Cref{fig:poisson_walks_horiz} are essentially reproduced) for details. It is the same as the process of \Cref{def:poisson_walks} from the Introduction up to time change:

\begin{prop}\label{thm:HL_poisson}
Let $n \in \N \cup \{\infty\}$ and $\nu \in \Sig_n$, and let $\la^{(n,\nu)},\Pois^{\nu,n}$ be as in \Cref{def:lambda_hl_planch} and \Cref{def:poisson_walks} respectively. Then
\begin{equation}
\la^{(n,\nu)}(\tau) = \Pois^{\nu,n}\left(\frac{1-t^n}{t}\tau\right)
\end{equation}
in distribution as stochastic processes.
\end{prop}
\begin{proof}
The result is stated in \cite[Proposition 3.4]{van2022q} in the case of trivial initial condition and proven by verifying equality of Markov generators, which also shows the case of general initial condition automatically.
\end{proof}

Hall-Littlewood processes also relate to matrix products as follows:

\begin{prop}\label{thm:hl_meas_matrices}
Let $D > N > 0$ be integers and $A_i,i \geq 1$ be $N \times N$ corners of iid Haar-distributed elements of $\GL_D(\Z_p)$. Then $\SN(A_s A_{s-1} \cdots A_1)$ is a Hall-Littlewood process given by 
\begin{equation}
\Pr(\SN(A_{s+1} \cdots A_1) = \nu | \SN(A_s \cdots A_1) = \la) =  \frac{Q_{\nu/\la}(t,\ldots,t^{D-N};0,t)P_\nu(1,\ldots,t^{N-1};0,t)}{P_\la(1,\ldots,t^{N-1};0,t)\Pi_{0,t}(t,\ldots,t^{D-N};1,\ldots,t^{N-1})}
\end{equation}
for $s \in \Z_{\geq 0}$. If instead the $A_i$ are iid with iid entries drawn from the additive Haar measure on $\Z_p$, then 
\begin{equation}
\Pr(\SN(A_{s+1} \cdots A_1) = \nu | \SN(A_s \cdots A_1) = \la) =  \frac{Q_{\nu/\la}(t,t^2,\ldots;0,t)P_\nu(1,\ldots,t^{N-1};0,t)}{P_\la(1,\ldots,t^{N-1};0,t)\Pi_{0,t}(t,t^2,\ldots;1,\ldots,t^{N-1})}.
\end{equation}
\end{prop}
\begin{proof}
This is the special case of \cite[Corollary 3.4]{van2020limits} corresponding to iid matrices.
\end{proof}

\section{The limit of the Plancherel/principal Hall-Littlewood measure}\label{sec:stationary_law_hl}

In this section we compute explicit contour integral formulas, given in \Cref{thm:hl_stat_dist}, for the limiting distribution of conjugate parts of the half-infinite Poisson walk $\Pois^{\nu,\infty}(\tau)$. These formulas are valid for suitably small initial conditions $\nu$, which we hope to make use of in future work. Because our methods come from Macdonald processes, we will state things in terms of the Hall-Littlewood process $\la^{(\infty,\nu)}(\tau)$ of \Cref{def:lambda_hl_planch}, but this is the same as $\Pois^{\nu,\infty}(\tau)$ up to time change by \Cref{thm:HL_poisson}. 

Let us briefly outline the proof before giving details. We define a family of observables $f_\mu, \mu \in \Y_k$ of a random partition $\kappa$ by
\begin{equation}
f_\mu(\kappa) = \frac{P_{\kappa/\mu'}(1,t,\ldots; 0,t)}{P_\kappa(1,t,\ldots;0,t)}.
\end{equation}
These have nice expectations when $\kappa$ is distributed by the Hall-Littlewood measure we are interested in, but also may be explicitly inverted to yield a different form, given in \Cref{thm:prob_from_observables}, for the prelimit probability we wish to take asymptotics of. Though this expression may look \emph{a priori} more complicated, it is suitable for converting to contour integrals (\Cref{thm:use_torus_product}), and hence for asymptotic analysis. 

\subsection{The main theorem.} In this subsection, we prove the following result, modulo some analytic lemmas which we defer to the next subsection.

\begin{thm}\label{thm:hl_stat_dist}
Let $\la(\tau) = \la^{(\infty,\emptyset)}(\tau)$ as defined in \Cref{def:lambda_hl_planch}, and fix $k \in \Z_{\geq 2}$ and $\zeta \in \R$. Then for any integers $L_1 \geq \ldots \geq L_k$,
\begin{multline}
\label{eq:explicit_hl_stat_dist_formula}
\lim_{\substack{\tau \to \infty \\ \log_{t^{-1}}(\tau) + \zeta \in \Z}} \Pr(\la_i'(\tau) - \log_{t^{-1}}(\tau) = L_i +\zeta \text{ for all }1 \leq i \leq k) \\ 
=  \frac{(t;t)_\infty^{k-1}}{k! (2 \pi \bi)^k} \prod_{i=1}^{k-1} \frac{t^{\binom{L_i-L_k}{2}}}{(t;t)_{L_i-L_{i+1}}} \int_{\tG^k} e^{\frac{t^{L_k+\zeta}}{1-t}(w_1+\ldots+w_k)} \frac{\prod_{1 \leq i \neq j \leq k} (w_i/w_j;t)_\infty}{\prod_{i=1}^k (-w_i^{-1};t)_\infty (-tw_i;t)_{\infty}} \\ 
\times \sum_{j=0}^{L_{k-1}-L_k} t^{\binom{j+1}{2}} \sqbinom{L_{k-1}-L_k}{j}_t  P_{(L_1-L_k,\ldots,L_{k-1}-L_k,j)}(w_1^{-1},\ldots,w_k^{-1};t,0)  \prod_{i=1}^k \frac{dw_i}{w_i}
\end{multline}
with contour 
\begin{equation}
\tG := \{x + \bi: x \leq 0 \} \cup \{x - \bi:  x \leq 0\} \cup \{x+\bi y: x^2+y^2=1, x > 0\}
\end{equation}
in usual counterclockwise orientation. When $k=1$, the equality \eqref{eq:explicit_hl_stat_dist_formula} holds with the sum over $j$ replaced by
\begin{equation}\label{eq:k=1_interp}
\sum_{j=0}^\infty t^{\binom{j+1}{2}} \frac{1}{(t;t)_j}  P_{(j)}(w_1^{-1};t,0).
\end{equation}
Furthermore, if $\nu(\tau), \tau \in t^{\Z+\zeta}$ is a sequence of partitions such that 
\begin{equation}\label{eq:technical_restriction}
 \nu_1'(\tau) \leq \log_{t^{-1}} \tau - 2\pfrac{k}{\log t^{-1}}^2(\log \log \tau)^2
\end{equation}
for all sufficiently large $\tau$, then the same result (in either the $k=1$ or $k \geq 2$ case) holds with $\la(\tau) = \la^{(\infty,\nu(\tau))}(\tau)$. 
\end{thm}

\begin{figure}[htbp]
\begin{center}
\begin{tikzpicture}[scale=1.5]
  % Draw the axes
  \draw[<->] (0,-2) -- (0,2) node[above] {$\Im(w_i)$};
  \draw[<->] (-6,0) -- (2,0) node[above] {$\Re(w_i)$};

  % Draw the half circle
  \draw[thick] (0,-1) arc (-90:90:1);

  % Draw the horizontal lines
  \draw[thick] (-6,1) node[left] {$\cdots$} -- (0,1); %node[right,yshift=3mm] {$$};
  \draw[thick] (-6,-1) node[left] {$\cdots$} -- (0,-1); % node[right,yshift=-3mm] {(0,-1)};

\end{tikzpicture}
\caption{The contour $\tG$ in $\C$.
}\label{fig:tG}
\end{center}
\end{figure}
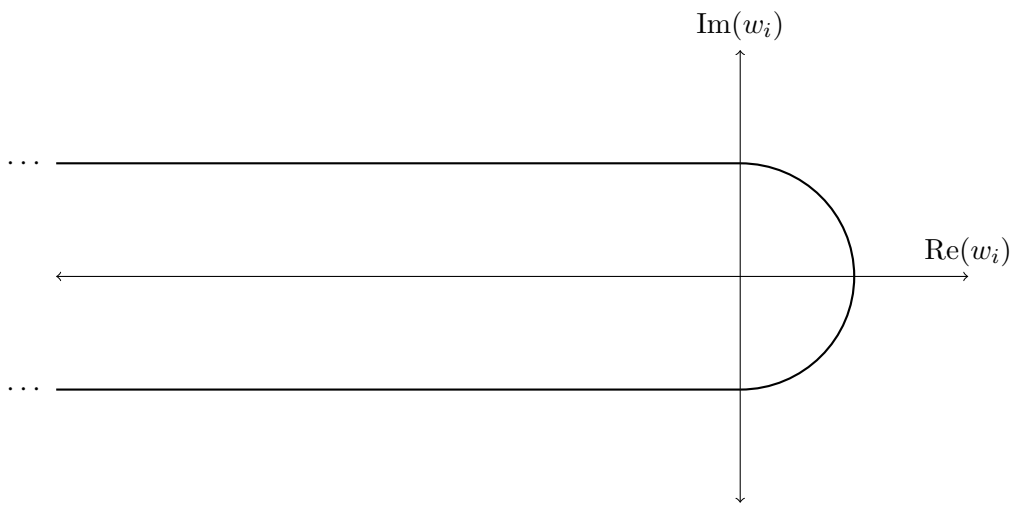

\begin{rmk}\label{rmk:k=1_interp}
The sum in \eqref{eq:k=1_interp} is equal to $(-tw_1^{-1};t)_\infty$ by the $q$-binomial theorem. One should view this sum as what one naively obtains in the general $k\geq 2$ form \eqref{eq:explicit_hl_stat_dist_formula} by taking $k=1$ and substituting $L_{k-1} = \infty$ and 
\begin{equation}
\sqbinom{\infty}{j}_t := \frac{1}{(t;t)_j}.
\end{equation}
\end{rmk}

\begin{rmk}
The reason for the parameter $\zeta$ is that for general $\tau$, $\la_i'(\tau) - \log_{t^{-1}}(\tau)$ will not be an integer but rather lie on some shift of the integer lattice, and it is necessary to consider a sequence of $\tau$ where $\la_i'(\tau) - \log_{t^{-1}}(\tau)$ all lie on the same shift $\Z+\zeta$ of the integer lattice in order to have any hope of a $T \to \infty$ limiting distribution. Though we have shifted by $\zeta$ to obtain a $\Z$-valued random variable, one might just as well not do so and define the limit as a $\Z+\zeta$-valued random variable. 
\end{rmk}

\begin{rmk}\label{rmk:hypotheses_suboptimal}
We do not use the case of nontrivial initial condition $\nu$ in \Cref{thm:hl_stat_dist} at all in subsequent results in this paper, but chose to prove it for potential future applications. The restriction \eqref{eq:technical_restriction} intuitively ensures that the initial condition $\nu$ is sufficiently far from the observation point at $\approx \log_{t^{-1}} \tau$ that there is time for the process to relax to its stationary distribution. We believe that \Cref{thm:hl_stat_dist} continues to hold under the weaker condition that $\log_{t^{-1}} \tau - \nu_1'(\tau) \to \infty$ at any rate, and it is easy to see that this condition is necessary to obtain a limit distribution supported on $\Z$ rather than some bounded-below interval $\Z_{\geq c}$. The bound $(\log \log \tau)^2$ is nonetheless quite good, and arises as a technical condition in certain error bounds on contour integrals in the proof. We discuss in more detail why it is technically necessary for our arguments in \Cref{rmk:why_hypotheses_suboptimal}.
\end{rmk}

It is also clear that the formula on the right hand side of \eqref{eq:explicit_hl_stat_dist_formula} is invariant under replacing $(L_1,\ldots,L_k) \mapsto (L_1+1,\ldots,L_k+1)$ and $\zeta \mapsto \zeta-1$, which it should be since the left hand side obviously has this invariance.

\begin{rmk}\label{rmk:cplx_tau}
For \Cref{thm:prob_from_observables} and \Cref{thm:use_torus_product}, we allow any $\tau \in \C$, in which case the distribution of $\la(\tau)$ is a complex-valued measure on $\Y$ rather than a positive probability measure. We write expectations and probabilities with respect to this measure without comment, e.g.
\begin{equation}
\Pr((\lambda_1'(\tau),\ldots,\lambda_k'(\tau)) = \eta) = \frac{1}{\Pi_{0,t}(1,t,\ldots;\gamma(\tau))}\sum_{\substack{\la \in \Y: \\ (\la_1',\ldots,\la_k') = \eta}} P_\la(1,t,\ldots;0,t)Q_\la(\gamma(\tau);0,t),
\end{equation}
which will in general not be a nonnegative real number. We will not need the case of general complex $\tau$ in this work, but show it because it will be useful in a subsequent one \cite{van2023+rank} and the proof is essentially the same as in the case of positive real $\tau$ when the above expression is a bona fide probability.
\end{rmk}

\begin{lemma}\label{thm:prob_from_observables}
Fix $\nu \in \Y$ and set $\la(\tau) = \la^{(\infty,\nu)}(\tau)$ as defined in \Cref{def:lambda_hl_planch}. Then for any $k \in \Z_{\geq 1}$ and $\eta \in \Y_k$,
\begin{multline}\label{eq:prob_inv_formula}
\Pr((\lambda_1'(\tau),\ldots,\lambda_k'(\tau)) = \eta) \\ 
= P_{\eta'}(1,t,\ldots;0,t)\sum_{\mu \in \Y_k} \sum_{\kappa \in \Y_k} \frac{P_{\nu/\kappa}(1,t,\ldots;0,t)}{P_\nu(1,t,\ldots;0,t)} Q_{\mu'/\kappa}(\gamma(\tau);0,t) P_{\mu'/\eta'}(\beta(-1);0,t).
\end{multline}
\end{lemma}

At first glance it might seem that we could simplify the expression in \eqref{eq:prob_inv_formula} still further by applying the skew Cauchy identity to 
\begin{equation}
\sum_{\mu \in \Y_k} Q_{\mu'/\kappa}(\gamma(\tau);0,t) P_{\mu'/\eta'}(\beta(-1);0,t)
\label{eq:almost_cauchy}
\end{equation}
to obtain a finite sum. However, this is slightly false: the Cauchy identity would only apply if the sum were over $\mu \in \Y$ rather than $\Y_k$. The fact that our sum is written over $\Y_k$ rather than $\Y$ is not purely cosmetic: if $\eta'$ has length $k$ then there will be $\mu \in \Y_{k+1}$ with $\mu' \succ \kappa, \mu' \succ \eta'$ (for $\kappa$ of length $\leq k$) and consequently nonzero values of $Q_{\mu'/\kappa}(\gamma(\tau);0,t) P_{\mu'/\eta'}(\beta(-1);0,t)$ for such $\mu \in \Y \setminus \Y_k$. Hence our sum is different from the one appearing the Cauchy identity, and in our view this fact is largely responsible for the fact that both the computations and the final formula in this section do not bear much resemblance to the previously studied asymptotics of Macdonald processes of which we are aware. Understanding the sum \eqref{eq:almost_cauchy} was a key difficulty in the computations, as it is not dominated by one or a small collection of terms. While $P_{\mu'/\eta'}(\beta(-1);0,t)$ is explicit by the branching rule \Cref{thm:hl_qw_branch_formulas}, the branching rule yields a much more complicated sum formula for $Q_{\mu'/\kappa}(\gamma(\tau);0,t)$, the number of terms of which grows superexponentially in $|\mu|$ with no clear way to separate into a main term and subleading terms as $T \to \infty$. It turns out, however, that after reexpressing $Q_{\mu'/\kappa}(\gamma(\tau);0,t)$ using the torus scalar product (\Cref{def:torus_product}) there are surprising simplifications, yielding an expression which is finally suitable to asymptotics and is given in the next lemma.

\begin{lemma}\label{thm:use_torus_product}
Keep the notation of \Cref{thm:prob_from_observables}, and let $\tnu = (\max(\nu_1,k),\max(\nu_2,k),\ldots) = (\nu_1',\ldots,\nu_k',0,\ldots)'$. Then
\begin{align}
\begin{split}
&\Pr((\lambda_1'(\tau),\ldots,\lambda_k'(\tau)) = \eta) = \frac{(t;t)_\infty^{k-1}}{k! (2 \pi \bi)^k  \prod_{i=1}^{k-1} (t;t)_{\eta_i - \eta_{i+1}}} \int_{c\T^k} e^{\frac{ \tau}{1-t}(z_1+\ldots+z_k)}t^{\sum_{i=1}^k \binom{\eta_i}{2}} \\ 
&  \times\frac{P_\tnu(\beta(z_1,\ldots,z_k), 1,t,\ldots;0,t)}{P_\tnu(1,t,\ldots;0,t)} \prod_{1 \leq i \neq j \leq k} (z_i/z_j;t)_\infty \\ 
&\times \sum_{j=0}^{\eta_{k-1}-\eta_k} t^{j(\eta_k+1)} t^{\binom{j}{2}} \sqbinom{\eta_{k-1}-\eta_k}{j}_t \frac{P_{\eta+j \vec{e_k}}(z_1^{-1},\ldots,z_k^{-1};t,0)}{\prod_{i=1}^k (-z_i^{-1};t)_\infty}  \prod_{i=1}^k \frac{dz_i}{z_i},
\end{split}
\end{align}
where $\T$ denotes the unit circle with counterclockwise orientation, $c \in \R_{>1}$ is arbitrary, and if $k=1$ we interpret the expression as in \Cref{thm:hl_stat_dist}.
\end{lemma}

The rest of the section consists of proofs of the above statements.

\begin{proof}[Proof of \Cref{thm:prob_from_observables}]
By the explicit formulas \Cref{thm:hl_principal_formulas} and \Cref{thm:skew_hl_principal}, for any $\kappa \in \Y, \mu \in \Y_k$ we have
\begin{equation}
\label{eq:hl_quotient}
\frac{P_{\kappa/\mu'}(1,t,\ldots; 0,t)}{P_\kappa(1,t,\ldots;0,t)} = \prod_{i=1}^k t^{\binom{\kappa_i' - \mu_i}{2} - \binom{\kappa_i'}{2}} (t^{1+\kappa_i'-\mu_i};t)_{\mu_i - \mu_{i+1}}.
\end{equation}
By the skew Cauchy identity \eqref{eq:specialized_cauchy} we also have
\begin{align}
\label{eq:E_observable}
\begin{split}
\E\left[\frac{P_{\la(\tau)/\mu'}(1,t,\ldots; 0,t)}{P_{\la(\tau)}(1,t,\ldots;0,t)}\right] &= \frac{\sum_{\la \in \Y} Q_{\la/\nu}(\gamma(\tau);0,t)P_{\la/\mu'}(1,t,\ldots;0,t) }{\Pi_{0,t}(1,t,\ldots;\gamma(\tau)) P_\nu(1,t,\ldots;0,t)}\\ 
&= \sum_{\kappa \in \Y_k} \frac{P_{\nu/\kappa}(1,t,\ldots;0,t)}{P_\nu(1,t,\ldots;0,t)} Q_{\mu'/\kappa}(\gamma(\tau);0,t).
\end{split}
\end{align}
However, an important property of the observable \eqref{eq:hl_quotient} is that it depends on the first $k$ parts of $\kappa'$ but not on the others. In other words, setting $\tl = (\la_1',\ldots,\la_k')' \in \Y$, we have 
\begin{equation}
\frac{P_{\la/\mu'}(1,t,\ldots; 0,t)}{P_\la(1,t,\ldots;0,t)}  = \frac{P_{\tl/\mu'}(1,t,\ldots; 0,t)}{P_{\tl}(1,t,\ldots;0,t)}.
\end{equation}
Hence by the branching rule \eqref{eq:specialization_branch} and \Cref{thm:specs_cancel}, 
\begin{align}\label{eq:obs_to_indicator}
\begin{split}
\sum_{\mu \in \Y_k} \frac{P_{\la/\mu'}(1,t,\ldots;0,t)}{P_\la(1,t,\ldots;0,t)} P_{\mu'/\eta'}(\beta(-1);0,t) &= \sum_{\mu \in \Y_k} \frac{P_{\tl/\mu'}(1,t,\ldots;0,t)}{P_{\tl}(1,t,\ldots;0,t)} P_{\mu'/\eta'}(\beta(-1);0,t) \\ 
&=\frac{P_{\tl/\eta'}(\beta(-1),1,t,\ldots;0,t)}{P_\tl(1,t,\ldots;0,t)} \\ 
& =  \frac{\bbone(\lambda_1'(\tau),\ldots,\lambda_k'(\tau)) = \eta)}{P_{\eta'}(1,t,\ldots;0,t)},
\end{split}
\end{align}
where we are using the fact that because $\tl' \in \Y_k$, each $\mu' \prec \tl'$ appearing in the branching rule automatically satisfies $\mu_1' \leq \tl_1'$ so $\mu \in \Y_k$.

Hence by taking expectations on both sides of \eqref{eq:obs_to_indicator}, we obtain
\begin{equation}\label{eq:prob_to_hl}
\text{LHS\eqref{eq:prob_inv_formula}} = P_{\eta'}(1,t,\ldots;0,t) \E\left[\sum_{\mu \in \Y_k} \frac{P_{\la(\tau)/\mu'}(1,t,\ldots;0,t)}{P_{\la(\tau)}(1,t,\ldots;0,t)} P_{\mu'/\eta'}(\beta(-1);0,t)\right].
\end{equation}
To apply \eqref{eq:E_observable} we wish to commute the expectation and the sum, i.e. we wish to show
\begin{multline}\label{eq:sum_exp_commute}
\frac{P_{\eta'}(1,t,\ldots;0,t)}{\Pi_{0,t}(1,t,\ldots;\gamma(\tau))} \sum_{\la \supset \mu'} \sum_{\mu \in \Y_k} P_{\la/\mu'}(1,t,\ldots;0,t) Q_{\la}(\gamma(\tau);0,t) P_{\mu'/\eta'}(\beta(-1);0,t) \\ 
= \frac{P_{\eta'}(1,t,\ldots;0,t)}{\Pi_{0,t}(1,t,\ldots;\gamma(\tau))} \sum_{\mu \in \Y_k} \sum_{\la \supset \mu'} P_{\la/\mu'}(1,t,\ldots;0,t) Q_{\la}(\gamma(\tau);0,t) P_{\mu'/\eta'}(\beta(-1);0,t)
\end{multline}
where the left hand side of \eqref{eq:sum_exp_commute} is just \eqref{eq:prob_to_hl} written out, and we have restricted the summation over $\la$ to $\la \supset \mu'$, since these are exactly the partitions for which $P_{\la/\mu'}(1,t,\ldots;0,t)$ is nonzero. Viewing both sides of \eqref{eq:sum_exp_commute} as formal power series in $\tau$ for the moment, and recalling that $Q_\la(\gamma(\tau);0,t)$ is a monomial of degree $|\la|$ in $\tau$, the restriction that $\la \supset \mu'$ guarantees that the coefficient of each power $\tau^\ell$ is a sum of finitely many terms, so both sides define the same formal power series. Because this formal power series is actually convergent for any $\tau \in \C$ by \eqref{eq:prob_to_hl}, this shows \eqref{eq:sum_exp_commute} for complex $\tau$. Hence
\begin{align}
\begin{split}
\text{RHS\eqref{eq:prob_to_hl}} &= P_{\eta'}(1,t,\ldots;0,t)\sum_{\mu \in \Y_k} P_{\mu'/\eta'}(\beta(-1);0,t)\E\left[\frac{P_{\la(\tau)/\mu'}(1,t,\ldots;0,t)}{P_{\la(\tau)}(1,t,\ldots;0,t)} \right] \\
&= P_{\eta'}(1,t,\ldots;0,t)\sum_{\mu \in \Y_k} \sum_{\kappa \in \Y_k} \frac{P_{\nu/\kappa}(1,t,\ldots;0,t)}{P_\nu(1,t,\ldots;0,t)} Q_{\mu'/\kappa}(\gamma(\tau);0,t) P_{\mu'/\eta'}(\beta(-1);0,t) \\ 
\end{split}
\end{align}
(the two sums commute because the sum over $\kappa$ is actually over $\kappa \subset \nu$ and hence finite). This completes the proof.
\end{proof}

\begin{proof}[Proof of \Cref{thm:use_torus_product}]
Our starting point is \Cref{thm:prob_from_observables}, which yields
\begin{multline}\label{eq:use_observable_inversion}
\Pr((\lambda_1'(\tau),\ldots,\lambda_k'(\tau)) = \eta') \\ 
= P_{\eta'}(1,t,\ldots;0,t)\sum_{\mu, \kappa \in \Y_k} \frac{P_{\nu/\kappa}(1,t,\ldots;0,t)}{P_\nu(1,t,\ldots;0,t)} Q_{\mu'/\kappa}(\gamma(\tau);0,t) P_{\mu'/\eta'}(\beta(-1);0,t).
\end{multline}
The only place $\nu$ appears above is
\begin{equation}
\frac{P_{\nu/\kappa}(1,t,\ldots;0,t)}{P_\nu(1,t,\ldots;0,t)},
\end{equation}
which is independent of $\nu_{k+1}',\nu_{k+2}',\ldots$ by \eqref{eq:hl_quotient} (with $\nu,\kappa$ substituted for $\kappa,\mu'$). This independence also follows from the fact that the projection of the stochastic process $\la(\tau)$ to $(\lambda_1'(\tau),\ldots,\lambda_k'(\tau))$ is Markovian. Hence the right hand side of \eqref{eq:use_observable_inversion} is the same upon replacing $\nu$ by $\tnu$, and to keep notation sanitary we use $\nu$ below and simply assume without loss of generality that $\nu_{k+1}'=0$. 

We now reexpress $Q_{\mu'/\kappa}(\gamma(\tau);0,t)$ using the torus scalar product. The specific case of the skew Cauchy identity with specializations $\gamma(\tau), \beta(z_1,\ldots,z_k)$ is
\begin{align}
\label{eq:cauchy_for_toral}
\begin{split}
\sum_{\mu \in \Y} Q_{\mu}(z_1,\ldots,z_k; t,0) Q_{\mu'/\kappa}(\gamma(\tau);0,t) &= \Pi_{0,t}(\gamma(\tau); \beta(z_1,\ldots,z_k))Q_{\kappa'}(z_1,\ldots,z_k;t,0) \\ 
&= e^{\frac{ \tau}{1-t}(z_1+\ldots+z_k)}Q_{\kappa'}(z_1,\ldots,z_k;t,0).
\end{split}
\end{align}
Since the polynomials $P_\la(z_1,\ldots,z_k;t,0)$ are orthogonal with respect to $\lan \cdot, \cdot \ran'_{t,0;k}$ and the $Q_\la$ are proportional to them by \Cref{def:Q}, \eqref{eq:cauchy_for_toral} together with the defining orthogonality property of Macdonald polynomials yields
\begin{equation}
\label{eq:int_for_planch_Q}
Q_{\mu'/\kappa}(\gamma(\tau);0,t) = \frac{\lan e^{\frac{ \tau}{1-t}(z_1+\ldots+z_k)}Q_{\kappa'}(z_1,\ldots,z_k;t,0), P_\mu(z_1,\ldots,z_k; t,0) \ran'_{t,0;k}}{\lan Q_\mu(z_1,\ldots,z_k; t,0), P_\mu(z_1,\ldots,z_k; t,0) \ran'_{t,0;k}}.
\end{equation}
By the definition of the proportionality constants $b_\la(q,t)$, 
\begin{multline}
\lan Q_\mu(z_1,\ldots,z_k; t,0), P_\mu(z_1,\ldots,z_k; t,0) \ran'_{t,0;k} \\ 
= b_\mu(t,0) \lan P_\mu(z_1,\ldots,z_k; t,0), P_\mu(z_1,\ldots,z_k; t,0) \ran'_{t,0;k}.
\end{multline}
By \Cref{thm:blam_computation},
\begin{equation}
\label{eq:bla_special_case}
b_\mu(t,0) = \prod_{i=1}^k \frac{1}{(t;t)_{\mu_i - \mu_{i+1}}}.
\end{equation}
By substituting $(t,0)$ for $(q,t)$ in \cite[(2.8)]{borodin2014macdonald}, we have 
\begin{equation}
\label{eq:prod_PP}
\lan P_\mu(z_1,\ldots,z_k; t,0), P_\mu(z_1,\ldots,z_k; t,0) \ran'_{t,0;k} = \prod_{i=1}^{k-1} \frac{(t;t)_{\mu_i - \mu_{i+1}}}{(t;t)_{\infty}}.
\end{equation}
Putting these together, the denominator in \eqref{eq:int_for_planch_Q} is 
\begin{equation}\label{eq:QP_IP}
\lan Q_\mu(z_1,\ldots,z_k; t,0), P_\mu(z_1,\ldots,z_k; t,0) \ran'_{t,0;k} = \frac{1}{(t;t)_{\mu_k} (t;t)_\infty^{k-1}}.
\end{equation}
Expressing the $Q_{\mu'/\kappa}(\gamma(\tau);0,t)$ in \eqref{eq:use_observable_inversion} via \eqref{eq:int_for_planch_Q} and substituting the definition of the torus scalar product for the numerator and \eqref{eq:QP_IP} for the denominator in \eqref{eq:int_for_planch_Q}, we obtain
\begin{multline}
\label{eq:prelimit_prob_int}
\Pr((\lambda_1'(\tau),\ldots,\lambda_k'(\tau)) = \eta) \\ 
=\frac{(t;t)_\infty^{k-1}}{k! (2 \pi \bi)^k} P_{\eta'}(1,t,\ldots;0,t) \sum_{\mu, \kappa \in \Y_k} \frac{P_{\nu/\kappa}(1,t,\ldots;0,t)}{P_\nu(1,t,\ldots;0,t)} (t;t)_{\mu_k} Q_{\mu/\eta}(-1;t,0) \\ 
\times \int_{\T^k} e^{\frac{ \tau}{1-t}(z_1+\ldots+z_k)} Q_{\kappa'}(z_1,\ldots,z_k;t,0) P_\mu(\bz_1,\ldots,\bz_k;t,0) \prod_{1 \leq i \neq j \leq k} (z_i/z_j;t)_\infty \prod_{i=1}^k \frac{dz_i}{z_i}.
\end{multline}
We wish to commute the sum and integral in \eqref{eq:prelimit_prob_int}, so we must check that Fubini's theorem applies. We first use the fact that $\bz = z^{-1}$ on $\T$ to rewrite the integrand as function analytic away from $0$ and $\infty$ and then expand the contours to $c\T$ to obtain
\begin{multline}
\int_{\T^k} e^{\frac{ \tau}{1-t}(z_1+\ldots+z_k)} P_\mu(\bz_1,\ldots,\bz_k;t,0) \prod_{1 \leq i \neq j \leq k} (z_i/z_j;t)_\infty \prod_{i=1}^k \frac{dz_i}{z_i} \\ 
= \int_{c\T^k} e^{\frac{ \tau}{1-t}(z_1+\ldots+z_k)} P_\mu(z_1^{-1},\ldots,z_k^{-1};t,0) \prod_{1 \leq i \neq j \leq k} (z_i/z_j;t)_\infty \prod_{i=1}^k \frac{dz_i}{z_i}
\end{multline}
where $c > 1$ may be arbitrary. Now 
\begin{align}\label{eq:sum_int_fubini}
\begin{split}
& \sum_{\mu,\kappa \in \Y_k} (t;t)_{\mu_k} P_{\nu/\kappa}(1,t,\ldots;0,t) |Q_{\mu/\eta}(-1;t,0)|\\ 
& \times \left|\int_{c\T^k} \prod_{1 \leq i \neq j \leq k} (z_i/z_j;t)_\infty   e^{\frac{ \tau}{1-t}(z_1+\ldots+z_k)}  Q_{\kappa'}(z_1,\ldots,z_k;t,0)  P_\mu(z_1^{-1},\ldots,z_k^{-1};t,0) \prod_{i=1}^k \frac{dz_i}{z_i}\right| \\ 
&\leq \sum_{\mu,\kappa \in \Y_k} (t;t)_{\mu_k} P_{\nu/\kappa}(1,t,\ldots;0,t)Q_{\mu/\eta}(1;t,0) \left( (2\pi)^k Q_{\kappa'}(c[k];t,0)  e^{\frac{|\tau|k c}{1-t}}P_\mu(c^{-1}[k];t,0) (-1;t)_\infty^{k^2-k}\right) \\ 
&\leq (2\pi)^k e^{\frac{|\tau|k c}{1-t}}(-1;t)_\infty^{k^2-k}  \left(\sum_{\mu \in \Y} Q_{\mu/\eta}(1;t,0) P_\mu(c^{-1}[k];t,0) \right) \left(\sum_{\kappa \in \Y} P_{\nu/\kappa}(1,t,\ldots;0,t) P_\kappa(\beta(c[k]);0,t) \right) \\ 
&= (2 \pi)^k (-1;t)_\infty^{k^2-k} e^{\frac{|\tau|k c}{1-t}} \Pi_{t,0}(1;c^{-1}[k]) P_\nu(\beta(c[k]),1,t,\ldots;0,t) < \infty.
\end{split}
\end{align}
Here we have applied \Cref{thm:complex_hl_bound} and trivial bounds to the integrand, then used the branching rule and Cauchy identity; note it is important that we have expanded the contours, as the sum in the Cauchy identity would diverge if the variables of $P_\mu$ were $1$ rather than $c^{-1} < 1$. By \eqref{eq:sum_int_fubini}, Fubini's theorem applies to \eqref{eq:prelimit_prob_int} (note that we must apply multiple times as we additionally split the sum over $\mu,\kappa$ into two sums below), hence
\begin{align}\label{eq:sum_inside_int}
\begin{split}
&\text{RHS\eqref{eq:prelimit_prob_int}} = \frac{(t;t)_\infty^{k-1}}{k! (2 \pi \bi)^k} P_{\eta'}(1,t,\ldots;0,t) \int_{c\T^k} \left(\sum_{\kappa \in \Y_k}  \frac{P_{\nu/\kappa}(1,t,\ldots;0,t)}{P_\nu(1,t,\ldots;0,t)}Q_{\kappa'}(z_1,\ldots,z_k;t,0)\right) \\ 
&\times  \left(\sum_{\mu \in \Y_k} (t;t)_{\mu_k} Q_{\mu/\eta}(-1;t,0)  P_\mu(z_1^{-1},\ldots,z_k^{-1};t,0)\right) e^{\frac{ \tau}{1-t}(z_1+\ldots+z_k)} \prod_{1 \leq i \neq j \leq k} (z_i/z_j;t)_\infty \prod_{i=1}^k \frac{dz_i}{z_i}  
\end{split}
\end{align}
Since $\len(\nu) \leq k$, by the branching rule
\begin{equation}\label{eq:nu_kappa_branch}
\sum_{\kappa \in \Y_k}  \frac{P_{\nu/\kappa}(1,t,\ldots;0,t)}{P_\nu(1,t,\ldots;0,t)}Q_{\kappa'}(z_1,\ldots,z_k;t,0) = \frac{P_\nu(\beta(z_1,\ldots,z_k), 1,t,\ldots;0,t)}{P_\nu(1,t,\ldots;0,t)}.
\end{equation}
By \Cref{thm:hl_principal_formulas} and the definition of $\eta = \eta(\tau)$ in terms of the $L_i$,
\begin{equation}
\label{eq:princ_for_integral}
P_{\eta'}(1,t,\ldots;0,t) = \frac{t^{\sum_{i=1}^k \binom{\eta_i}{2}}}{(t;t)_{\eta_k} \prod_{i=1}^{k-1} (t;t)_{L_i - L_{i+1}}}.
\end{equation}
Let us bring the $(t;t)_{\eta_k}^{-1}$ factor inside the sum in \eqref{eq:sum_inside_int} and evaluate the resulting sum
\begin{equation}\label{eq:not_quite_cauchy}
\sum_{\mu \in \Y_k} \frac{(t;t)_{\mu_k}}{(t;t)_{\eta_k}} Q_{\mu/\eta}(-1;t,0) P_\mu(z_1^{-1},\ldots,z_k^{-1};t,0).
\end{equation}
By the $q$-binomial theorem, 
\begin{equation}
\label{eq:qbinom_for_int}
\frac{(t;t)_{\mu_k}}{(t;t)_{\eta_k}} = (1-t^{\eta_k+1}) \cdots (1-t^{\mu_k}) = \sum_{j=0}^{\mu_k - \eta_k} (-t^{\eta_k+1})^j t^{\binom{j}{2}} \sqbinom{\mu_k - \eta_k}{j}_t,
\end{equation}
and we note that we can replace the sum by one over all $j \geq 0$ since the $q$-binomial coefficient will be $0$ when $j > \mu_k - \eta_k$. We will use the identity\footnote{We observed \eqref{eq:update_qw} through explicit examples and are not aware of any broader context for it in symmetric function theory, though this would certainly be interesting if it exists.} 
\begin{equation}
\label{eq:update_qw}
\sqbinom{\mu_k - \eta_k}{j}_t Q_{\mu/\eta}(-1;t,0) = \sqbinom{\eta_{k-1} - \eta_k}{j}_t (-1)^j Q_{\mu/(\eta+j \vec{e_k})}(-1;t,0)
\end{equation}
to simplify \eqref{eq:not_quite_cauchy}, but first we prove \eqref{eq:update_qw}. As before, in the case $k=1$ we interpret \eqref{eq:update_qw} by taking $\eta_{k-1}=\infty$ and 
\begin{equation}
\sqbinom{\infty}{j}_t = \frac{1}{(t;t)_j}
\end{equation}
to obtain
\begin{equation}
\label{eq:update_qw_k=1}
\sqbinom{\mu_1 - \eta_1}{j}_t Q_{(\mu_1)/(\eta_1)}(-1;t,0) = \frac{ (-1)^j}{(t;t)_j} Q_{(\mu_1)/(\eta_1+j)}(-1;t,0),
\end{equation}
which follows immediately from the branching rule \Cref{thm:hl_qw_branch_formulas}, so we will prove the $k\geq 2$ case. By the explicit branching rule \Cref{thm:hl_qw_branch_formulas},
\begin{equation}
\label{eq:use_qw_branch_1}
Q_{\mu/\eta}(-1;t,0) = (-1)^{|\mu/\eta|} \frac{1}{(t;t)_{\mu_1 - \eta_1}} \sqbinom{\eta_1 - \eta_{2}}{\eta_1 - \mu_{2}}_t \cdots \sqbinom{\eta_{k-2} - \eta_{k-1}}{\eta_{k-2} - \mu_{k-1}}_t \sqbinom{\eta_{k-1} - \eta_{k}}{\eta_{k-1} - \mu_{k}}_t
\end{equation}
while for any $j$ such that $\eta+j\vec{e_k} = (\eta_1,\ldots,\eta_{k-1},\eta_k+j) \prec \mu$ we have
\begin{equation}
\label{eq:use_qw_branch_2}
Q_{\mu/(\eta+j\vec{e_k})}(-1;t,0) = (-1)^{|\mu/\eta|-j} \frac{1}{(t;t)_{\mu_1 - \eta_1}} \sqbinom{\eta_1 - \eta_{2}}{\eta_1 - \mu_{2}}_t \cdots \sqbinom{\eta_{k-2} - \eta_{k-1}}{\eta_{k-2} - \mu_{k-1}}_t \sqbinom{\eta_{k-1} - \eta_{k}-j}{\eta_{k-1} - \mu_{k}}_t.
\end{equation}
By writing out the $q$-factorials on both sides and cancelling a pair it is elementary to check that
\begin{equation}
\label{eq:trade_qbinom}
\sqbinom{\mu_k - \eta_k}{j}_t \sqbinom{\eta_{k-1}- \eta_k}{\eta_{k-1} - \mu_k}_t = \sqbinom{\eta_{k-1} - \eta_k - j}{\eta_{k-1} - \mu_k}_t \sqbinom{\eta_{k-1} - \eta_k}{j}_t.
\end{equation}
Now \eqref{eq:update_qw} follows by combining \eqref{eq:use_qw_branch_1}, \eqref{eq:use_qw_branch_2} and \eqref{eq:trade_qbinom}. 

By \eqref{eq:update_qw} (or \eqref{eq:update_qw_k=1}, if $k=1$) and the Cauchy identity, 
\begin{align}\label{eq:cauchy_by_force}
\begin{split}
&\text{\eqref{eq:not_quite_cauchy}} = \sum_{j=0}^{\infty} (-t^{\eta_k+1})^j t^{\binom{j}{2}} \sum_{\mu \in \Y_k} \sqbinom{\mu_k - \eta_k}{j}_t Q_{\mu/\eta}(-1;t,0) P_\mu(z_1^{-1},\ldots,z_k^{-1};t,0) \\ 
&= \sum_{j=0}^{\infty} (-t^{\eta_k+1})^j t^{\binom{j}{2}} \sqbinom{\eta_{k-1}-\eta_k}{j}_t (-1)^j \sum_{\mu \in \Y_k} Q_{\mu/(\eta+j\vec{e_k})}(-1;t,0) P_\mu(z_1^{-1},\ldots,z_k^{-1};t,0) \\ 
&= \sum_{j=0}^{\eta_{k-1}-\eta_k} t^{j(\eta_k+1)} t^{\binom{j}{2}} \sqbinom{\eta_{k-1}-\eta_k}{j}_t P_{\eta+j \vec{e_k}}(z_1^{-1},\ldots,z_k^{-1};t,0) \Pi_{t,0}(-1;z_1^{-1},\ldots,z_k^{-1}).
\end{split}
\end{align}
Substituting \eqref{eq:nu_kappa_branch}, \eqref{eq:princ_for_integral} and \eqref{eq:cauchy_by_force} into \eqref{eq:sum_inside_int} and replacing $\Pi_{t,0}$ by its explicit product formula \eqref{eq:spec_cauchy_kernel} yields 
\begin{align}
\label{eq:int_after_cauchy}
\begin{split}
&\text{RHS\eqref{eq:sum_inside_int}} = \frac{(t;t)_\infty^{k-1}}{k! (2 \pi \bi)^k  \prod_{i=1}^{k-1} (t;t)_{\eta_i - \eta_{i+1}}} \int_{c\T^k} e^{\frac{ \tau}{1-t}(z_1+\ldots+z_k)}t^{\sum_{i=1}^k \binom{\eta_i}{2}} \prod_{1 \leq i \neq j \leq k} (z_i/z_j;t)_\infty \\ 
&\times  \frac{P_\nu(\beta(z_1,\ldots,z_k), 1,t,\ldots;0,t)}{P_\nu(1,t,\ldots;0,t)}\sum_{j=0}^{\eta_{k-1}-\eta_k} t^{j(\eta_k+1)} t^{\binom{j}{2}} \sqbinom{\eta_{k-1}-\eta_k}{j}_t \frac{P_{\eta+j \vec{e_k}}(z_1^{-1},\ldots,z_k^{-1};t,0)}{\prod_{i=1}^k (-z_i^{-1};t)_\infty}  \prod_{i=1}^k \frac{dz_i}{z_i},
\end{split}
\end{align}
where if $k=1$ we interpret as in the theorem statement. This completes the proof.
\end{proof}

In the below proof we will assume the same modification as before to the sum over $j$ in the $k=1$ case without comment.

\begin{proof}[Proof of \Cref{thm:hl_stat_dist}]
Write $\eta(\tau) = (L_1 + \log_{t^{-1}}(\tau) + \zeta, \ldots, L_k + \log_{t^{-1}}(\tau) + \zeta)$. Then 
\begin{equation}\label{eq:eta_difference}
\eta_{i}-\eta_j = L_{i}-L_j
\end{equation}
for each $i,j$, so by \Cref{thm:use_torus_product}
\begin{align}
\label{eq:int_after_cauchy2}
\begin{split}
&\Pr(\la_i'(\tau) = L_i + \log_{t^{-1}}(\tau) + \zeta \text{ for all }1 \leq i \leq k) = \frac{(t;t)_\infty^{k-1}}{k! (2 \pi \bi)^k  \prod_{i=1}^{k-1} (t;t)_{L_i - L_{i+1}}}  \\
&\times \int_{c\T^k} e^{\frac{ \tau}{1-t}(z_1+\ldots+z_k)}t^{\sum_{i=1}^k \binom{\eta_i(\tau)}{2}} \prod_{1 \leq i \neq j \leq k} (z_i/z_j;t)_\infty \frac{P_{\nu(\tau)}(\beta(z_1,\ldots,z_k), 1,t,\ldots;0,t)}{P_{\nu(\tau)}(1,t,\ldots;0,t)} \\ 
&\times  \sum_{j=0}^{L_{k-1}-L_k} t^{j(\eta_k(\tau)+1)} t^{\binom{j}{2}} \sqbinom{L_{k-1}-L_k}{j}_t \frac{P_{\eta(\tau)+j \vec{e_k}}(z_1^{-1},\ldots,z_k^{-1};t,0)}{\prod_{i=1}^k (-z_i^{-1};t)_\infty}  \prod_{i=1}^k \frac{dz_i}{z_i}.
\end{split}
\end{align}
(Technically the above requires that $\tau$ is large enough so $\eta(\tau) \in \Y_k$, which is true as long as $L_k + \log_{t^{-1}}(\tau) + \zeta \geq 0$). We wish to take a limit as $\tau \to \infty$ of the above expression, so to remove the $\tau$-dependence inside the exponential we make a change of variables to $w_i = t^{-\eta_k(\tau)}z_i = \tau t^{-L_k-\zeta} z_i$. For later convenience we also set 
%An alternate idea for how to get a better bound, which one might want to write somewhere or sic an REU student on: residue-expand the integral at this step, and then you have a product of qW polynomials $P_{\nu'}$ and $P_{\eta(\tau)+j \vec{e_k}}$ integrated against the Plancherel (with some additional beta from the residue stuff) Cauchy kernel. Just take that inner product. Now one wants to argue that the Plancherel specialization in $P_\nu$ helps us by giving something that's very small for exactly the terms in the branching rule expansion where the bound that we used is bad. Idea is that a $k \times n$ rectangle or something close to it, with $n \gg k$, has planch spec given by roughly $T^{nk} \frac{(n/k)!^k}{n!}$ by the hook length formula, and the term with factorials is very small.
\begin{align}
\label{eq:g_quot}
\begin{split}
g_{ \tau}(w_1,\ldots,w_k) &:= \frac{P_{\nu(\tau)}(\beta(t^{\eta_k(\tau)}w_1,\ldots,t^{\eta_k(\tau)}w_k), 1,t,\ldots;0,t)}{P_{\nu(\tau)}(1,t,\ldots;0,t)}.
\end{split}
\end{align}
This yields
\begin{align}
\label{eq:int_var_change}
\begin{split}
&\text{RHS\eqref{eq:int_after_cauchy2}} = \frac{(t;t)_\infty^{k-1}}{k! (2 \pi \bi)^k  \prod_{i=1}^{k-1} (t;t)_{L_i - L_{i+1}}} \int_{ct^{-\eta_k(\tau)}\T^k} e^{\frac{t^{L_k+\zeta}}{1-t}(w_1+\ldots+w_k)}g_{ \tau}(w_1,\ldots,w_k)  \\  
&\times \sum_{j=0}^{L_{k-1}-L_k} t^{j(\eta_k(\tau)+1)+\binom{j}{2}} \sqbinom{L_{k-1}-L_k}{j}_t (t^{-\eta_k(\tau)})^{|\eta(\tau)|+j} P_{\eta(\tau)+j \vec{e_k}}(w_1^{-1},\ldots,w_k^{-1};t,0) \\ 
&\times  \frac{\prod_{1 \leq i \neq j \leq k} (w_i/w_j;t)_\infty \prod_{i=1}^k t^{\binom{\eta_i(\tau)}{2}}}{{\prod_{i=1}^k (-t^{-\eta_k(\tau)}w_i^{-1};t)_\infty}}  \prod_{i=1}^k \frac{dw_i}{w_i},
\end{split}
\end{align}
where we have used that $P_{\eta+j\vec{e_k}}$ is homogeneous of degree $|\eta|+j$.

By the elementary identity 
\begin{equation}
\label{eq:favorite_binomial_split}
\binom{a+b}{2} = \binom{a}{2} + \binom{b}{2} + ab
\end{equation}
and \eqref{eq:eta_difference} we have 
\begin{equation}\label{eq:eta_i_to_k}
\binom{\eta_i}{2} = \binom{\eta_k}{2} + \binom{L_i-L_k}{2} + (\eta_i-\eta_k)\eta_k.
\end{equation}
Additionally, by \Cref{thm:signature_shift} and \eqref{eq:eta_difference}, 
\begin{equation}
\label{eq:shift_by_tau}
P_{\eta+j \vec{e_k}}(w_1^{-1},\ldots,w_k^{-1};t,0) = (w_1 \cdots w_k)^{-\eta_k}P_{(L_1-L_k,\ldots,L_{k-1}-L_k,j)}(w_1^{-1},\ldots,w_k^{-1};t,0).
\end{equation}
Substituting \eqref{eq:eta_i_to_k} for $1 \leq i \leq k$ and \eqref{eq:shift_by_tau} into \eqref{eq:int_var_change} yields%typo from v1: "and similarly simplifying the $t$ exponent in the $\kappa$ term yields"
\begin{align}
\label{eq:int_w_2}
\begin{split}
&\text{RHS\eqref{eq:int_var_change}} = \frac{(t;t)_\infty^{k-1}}{k! (2 \pi \bi)^k} \prod_{i=1}^{k-1} \frac{t^{\binom{L_i-L_k}{2}}}{(t;t)_{L_i-L_{i+1}}} \int_{ct^{-\eta_k(\tau)}\T^k} e^{\frac{t^{L_k+\zeta}}{1-t}(w_1+\ldots+w_k)} \prod_{1 \leq i \neq j \leq k} (w_i/w_j;t)_\infty \\ 
&\times \prod_{i=1}^k \frac{w_i^{-\eta_k(\tau)}t^{\binom{\eta_k(\tau)}{2}+(\eta_i(\tau)-\eta_k(\tau))\eta_k(\tau)}t^{-\eta_k(\tau)\eta_i(\tau)}}{{(-t^{-\eta_k(\tau)}w_i^{-1};t)_\infty}}g_{ \tau}(w_1,\ldots,w_k) \\ 
&\times \sum_{j=0}^{L_{k-1}-L_k} t^{j+\binom{j}{2}} \sqbinom{L_{k-1}-L_k}{j}_t P_{(L_1-L_k,\ldots,L_{k-1}-L_k,j)}(w_1^{-1},\ldots,w_k^{-1};t,0)  \prod_{i=1}^k \frac{dw_i}{w_i}.
\end{split}
\end{align}
Noting that 
\begin{equation}
\frac{w_i^{-\eta_k(\tau)}t^{\binom{\eta_k(\tau)}{2}+(\eta_i(\tau)-\eta_k(\tau))\eta_k(\tau)}t^{-\eta_k(\tau)\eta_i(\tau)}}{{(-t^{-\eta_k(\tau)}w_i^{-1};t)_\infty}} = \frac{1}{(-w_i^{-1};t)_\infty (-tw_i;t)_{\eta_k(\tau)}}
\end{equation}
and shifting contours yields
\begin{align}
\label{eq:int_w_3}
\begin{split}
&\text{RHS\eqref{eq:int_w_2}} = \frac{(t;t)_\infty^{k-1}}{k! (2 \pi \bi)^k} \prod_{i=1}^{k-1} \frac{t^{\binom{L_i-L_k}{2}}}{(t;t)_{L_i-L_{i+1}}} \int_{\Gamma(\tau)^k} e^{\frac{t^{L_k+\zeta}}{1-t}(w_1+\ldots+w_k)} \frac{\prod_{1 \leq i \neq j \leq k} (w_i/w_j;t)_\infty}{\prod_{i=1}^k (-w_i^{-1};t)_\infty (-tw_i;t)_{\eta_k(\tau)}} \\ 
&\times g_{ \tau}(w_1,\ldots,w_k)\sum_{j=0}^{L_{k-1}-L_k} t^{\binom{j+1}{2}} \sqbinom{L_{k-1}-L_k}{j}_t  P_{(L_1-L_k,\ldots,L_{k-1}-L_k,j)}(w_1^{-1},\ldots,w_k^{-1};t,0)  \prod_{i=1}^k \frac{dw_i}{w_i}
\end{split}
\end{align}
where 
\begin{multline}
\Gamma(\tau) := \{x+\bi y: x^2+y^2=1, x > 0\}  \cup \{x + \bi: -t^{-\eta_k(\tau)-1/2} < x \leq 0 \} \\ \cup \{x - \bi: -t^{-\eta_k(\tau)-1/2} < x \leq 0 \} \cup \{-t^{-\eta_k(\tau)-1/2} + \bi y: -1 \leq y \leq 1\},
\end{multline}
see \Cref{fig:Gamma1_decomp}. For the asymptotics, we will decompose the integration contour into a main term contour $\Gamma_1(\tau)$ and error term contour $\Gamma_2(\tau)$. First define 
\begin{equation}\label{eq:def_xi}
\xi(\tau) := \pfrac{2+\tfrac{1}{k}}{\log t^{-1}} \log \log \tau.
\end{equation}
Here we have chosen the constant in front of $\log \log \tau$ in \eqref{eq:def_xi} somewhat arbitrarily so that as $\tau \to \infty$ the limits
\begin{equation}\label{eq:xi_small}
2\pfrac{k}{\log t^{-1}}^2 (\log \log \tau)^2 - \frac{k^2-k}{2}\xi(\tau)^2 - \text{const} \cdot \xi(\tau) \to \infty 
\end{equation}
and 
\begin{equation}\label{eq:xi_large}
\xi(\tau) - \pfrac{2}{\log t^{-1}} \log \log \tau \to \infty 
\end{equation}
hold, as these are needed to control certain error terms in the proof below. We then decompose $\Gamma(\tau)$ as
\begin{align}\label{eq:contours}
\begin{split}
\Gamma(\tau) &= \Gamma_1(\tau) \cup \Gamma_2(\tau) \\ 
\Gamma_1(\tau) &= \{x + \bi: -t^{-\xi(\tau)} < x \leq 0 \} \cup \{x - \bi: -t^{-\xi(\tau)}  < x \leq 0\} \cup \{x+\bi y: x^2+y^2=1, x > 0\} \\ 
\Gamma_2(\tau) &= \{-t^{-\eta_k(\tau)-1/2} + \bi y: -1 \leq y \leq 1\} \cup \{x + \bi: -t^{-\eta_k(\tau)-1/2}< x \leq  -t^{-\xi(\tau)}  \} \\ 
&  \cup \{x - \bi: -t^{-\eta_k(\tau)-1/2}< x \leq  -t^{-\xi(\tau)} \}.
\end{split}
\end{align}

\begin{figure}[htbp]
\begin{center}
\begin{tikzpicture}[scale=1.5]
\def\b{2};
\def\t{.5};
\def\d{-5.656}; %this is -t^{-2.5} but i was having trouble with pow(...) for some reason...
  % Draw the axes
  \draw[<->] (0,-2) -- (0,2) node[above] {$\Im(w_i)$};
  \draw[<->] (-6,0) -- (2,0) node[above] {$\Re(w_i)$};

  % Draw the half circle
  \draw[thick,blue] (0,-1) arc (-90:90:1);

  % Draw the horizontal lines
  \draw[thick,blue] (-\b,1) node[left,black,yshift=3mm,xshift=10mm] {$-t^{-\xi(\tau)}+\bi$} -- (0,1) ;
  \draw[thick,blue] (-\b,-1) node[left,black,yshift=-3mm,xshift=10mm] {$-t^{-\xi(\tau)}-\bi$} -- (0,-1) ;

    \draw[thick,red] (\d,1)  -- (-\b,1);% node[right,yshift=3mm] {(0,1)};
    \draw[thick,red] (\d,-1)  -- (-\b,-1);% node[right,yshift=-3mm] {(0,-1)};

    \draw[thick,red] (\d,1) node[above,black] {$-t^{-\eta_k(\tau)-1/2}+\bi$} -- (\d,-1) node[below,black] {$-t^{-\eta_k(\tau)-1/2}-\bi$};

\fill (-\b,1) circle (1pt);
\fill (-\b,-1) circle (1pt);

 % to draw poles, loop through powers of t
  \foreach \n in {-2,...,20} {
    % Calculate x-coordinate
    \pgfmathsetmacro\x{-pow(\t,\n)}
    % Draw a dot at coordinate (\x,0)
    \fill (\x,0) circle (.5pt);
  }

\end{tikzpicture}
\caption{The contour $\Gamma(\tau)$ decomposed as in \eqref{eq:contours}, with $\Gamma_1(\tau)$ in blue and $\Gamma_2(\tau)$ in red, and the poles of the integrand at $w_i = -t^\Z$ shown.
}\label{fig:Gamma1_decomp}
\end{center}
\end{figure}
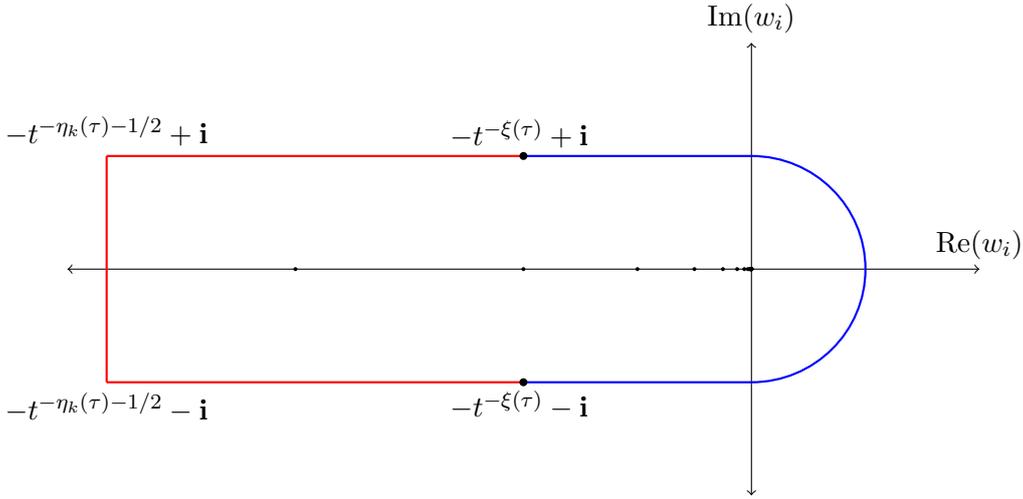

We further define another error term contour
\begin{equation}
\Gamma_3(\tau) = \{x + \bi:x \leq  -t^{- \xi(\tau)}\} \cup \{x - \bi: x \leq -t^{- \xi(\tau)}  \},
\end{equation}
so that 
\begin{equation}\label{eq:tG_decomp}
\Gamma_1(\tau) \cup \Gamma_3(\tau) = \{x + \bi: x \leq 0 \} \cup \{x - \bi:  x \leq 0\} \cup \{x+\bi y: x^2+y^2=1, x > 0\} = \tG
\end{equation}
\begin{figure}[htbp]
\begin{center}
\begin{tikzpicture}[scale=1.5]
\def\b{2}
  % Draw the axes
  \draw[<->] (0,-2) -- (0,2) node[above] {$\Im(w_i)$};
  \draw[<->] (-6,0) -- (2,0) node[above] {$\Re(w_i)$};

  % Draw the half circle
  \draw[thick,blue] (0,-1) arc (-90:90:1);

  % Draw the horizontal lines
  \draw[thick,blue] (-\b,1) node[left,black,yshift=3mm,xshift=10mm] {$-t^{-\xi(\tau)}+\bi$} -- (0,1) ;
  \draw[thick,blue] (-\b,-1) node[left,black,yshift=-3mm,xshift=10mm] {$-t^{-\xi(\tau)}-\bi$} -- (0,-1) ;

    \draw[thick,green] (-6,1) node[left,black] {$\cdots$} -- (-\b,1);% node[right,yshift=3mm] {(0,1)};
    \draw[thick,green] (-6,-1) node[left,black] {$\cdots$} -- (-\b,-1);% node[right,yshift=-3mm] {(0,-1)};

\fill (-\b,1) circle (1pt);
\fill (-\b,-1) circle (1pt);

\end{tikzpicture}
\caption{The contour $\tG$ decomposed as in \eqref{eq:tG_decomp}, with $\Gamma_1(\tau)$ in blue and $\Gamma_3(\tau)$ in green.
}\label{fig:tG_decomp}
\end{center}
\end{figure}
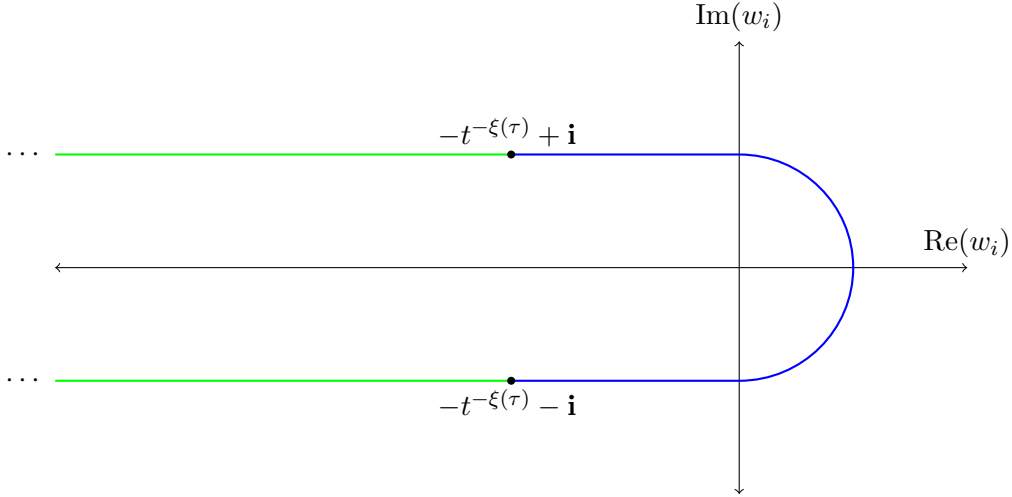
is independent of $\tau$. To complete the proof, we must show that
\begin{multline}
\label{eq:cont_integral_limit}
\lim_{\substack{\tau \to \infty \\ \log_t(\tau) \equiv \zeta }} \text{RHS\eqref{eq:int_w_3}} =  \frac{(t;t)_\infty^{k-1}}{k! (2 \pi \bi)^k} \prod_{i=1}^{k-1} \frac{t^{\binom{L_i-L_k}{2}}}{(t;t)_{L_i-L_{i+1}}} \int_{\tG^k} e^{\frac{t^{L_k+\zeta}}{1-t}(w_1+\ldots+w_k)} \frac{\prod_{1 \leq i \neq j \leq k} (w_i/w_j;t)_\infty}{\prod_{i=1}^k (-w_i^{-1};t)_\infty (-tw_i;t)_{\infty}} \\ 
\times \sum_{j=0}^{L_{k-1}-L_k} t^{\binom{j+1}{2}} \sqbinom{L_{k-1}-L_k}{j}_t  P_{(L_1-L_k,\ldots,L_{k-1}-L_k,j)}(w_1^{-1},\ldots,w_k^{-1};t,0)  \prod_{i=1}^k \frac{dw_i}{w_i}.
\end{multline}
To compress notation we abbreviate the $\tau$-independent part of the integrand as
\begin{multline}
\label{eq:integrand_junk}
f_{\frac{t^{L_k+\zeta}}{1-t}}(w_1,\ldots,w_k) := \frac{(t;t)_\infty^{k-1}}{k!} \prod_{i=1}^{k-1} \frac{t^{\binom{L_i-L_k}{2}}}{(t;t)_{L_i-L_{i+1}}} e^{\frac{t^{L_k+\zeta}}{1-t}(w_1+\ldots+w_k)}\frac{\prod_{1 \leq i \neq j \leq k} (w_i/w_j;t)_\infty}{\prod_{i=1}^k w_i(-w_i^{-1};t)_\infty }  \\ 
\times \sum_{j=0}^{L_{k-1}-L_k} t^{\binom{j+1}{2}} \sqbinom{L_{k-1}-L_k}{j}_t  P_{(L_1-L_k,\ldots,L_{k-1}-L_k,j)}(w_1^{-1},\ldots,w_k^{-1};t,0).
\end{multline}
Since $\zeta$ and $L_k$ are fixed throughout this proof, we will suppress the subscript of $f$ everywhere below. With this notation, we may rewrite the equality \eqref{eq:cont_integral_limit} which we want to show as
\begin{align}
\label{eq:split_integral1}
&\lim_{\substack{\tau \to \infty \\ \log_t(\tau) \equiv \zeta }} \frac{1}{(2 \pi \bi)^k} \int_{\Gamma_1(\tau)^k} f(w_1,\ldots,w_k)\left(\frac{g_{ \tau}(w_1,\ldots,w_k)}{\prod_{i=1}^k(-tw_i;t)_{\eta_k(\tau)}} - \frac{1}{\prod_{i=1}^k(-tw_i;t)_{\infty}}  \right) \prod_{i=1}^k dw_i \\  
&+ \lim_{\substack{\tau \to \infty \\ \log_t(\tau) \equiv \zeta }}\frac{1}{(2 \pi \bi)^k} \int_{\Gamma(\tau)^k \setminus \Gamma_1(\tau)^k} \frac{f(w_1,\ldots,w_k)g_{ \tau}(w_1,\ldots,w_k)}{\prod_{i=1}^k(-tw_i;t)_{\eta_k(\tau)}} \prod_{i=1}^k dw_i \label{eq:split_integral2} \\
&-\lim_{\substack{\tau \to \infty \\ \log_t(\tau) \equiv \zeta }} \frac{1}{(2 \pi \bi)^k} \int_{\tG^k \setminus \Gamma_1(\tau)^k} \frac{f(w_1,\ldots,w_k)}{\prod_{i=1}^k(-tw_i;t)_{\infty}}\prod_{i=1}^k dw_i \label{eq:split_integral3}  =0.
\end{align}
We will show each of the three lines \eqref{eq:split_integral1}, \eqref{eq:split_integral2} and \eqref{eq:split_integral3} above are $0$ separately. For this, we first state several needed asymptotics for the functions $f$ and $g_{\tau}$, the proofs of which we defer to later in the section.

\begin{lemma}\label{thm:f_bound}
For any $w_1,\ldots,w_k$ in $\tG$ or in $\Gamma(\tau)$ for $\tau$ sufficiently large (recall these contours are defined in \Cref{fig:tG_decomp}),
\begin{equation}
\label{eq:f_bound}
|f(w_1,\ldots,w_k)| \leq C \prod_{i=1}^k e^{c_1\Re(w_i) + \frac{k-1}{2}(\log t^{-1}) \floor{\log_t |w_i|}^2 + c_2 \floor{\log_t |w_i|}}
\end{equation}
for some positive constants $C,c_1,c_2$ independent of $\tau$. Here $f$ is as in \eqref{eq:integrand_junk}.
\end{lemma}

\begin{lemma}\label{thm:g_bound_small_w}
In the notation of the above proof,
\begin{equation}
g_{\tau}(w_1,\ldots,w_k) = 1 + O(t^{\eta_k(\tau) - \nu_1'(\tau) - \xi(\tau)})
\end{equation}
as $\tau \to \infty$, with implied constant uniform over $w_1,\ldots,w_k \in \Gamma_1(\tau)$.
\end{lemma}

Outside of $\Gamma_1(\tau)$ we content ourselves with a cruder bound on $g_\tau$:

\begin{lemma}\label{thm:g_bound_large_w}
In the notation of the above proof, 
\begin{equation}
|g_{\tau}(w_1,\ldots,w_k)| = O(t^{-(k+1)(\nu_1'(\tau))^2})
\end{equation}
as $\tau \to \infty$, with implied constant uniform over $w_1,\ldots,w_k$ in $\Gamma(\tau)$.
\end{lemma}

We also use the following elementary bounds on $|\frac{1}{(-z;t)_\infty}|$, which will be proven in \Cref{subsec:bounds}.

\begin{lemma}\label{thm:qp_bound_tG}
There exists a constant such that for any $n \in \Z_{\geq 1} \cup \{\infty\}$ and $w \in \tG$,
\begin{equation}
\abs*{\frac{1}{(-tw;t)_n}} \leq C.
\end{equation}
\end{lemma}

\begin{lemma}\label{thm:qpoc_lower_bound}
For all $n \in \Z_{\geq 1} \cup \{\infty\}$, $\delta > 0$, and $z \in \C$ such that $|\Re(z) - (-t^{-i})| > \delta t^{-i}$, we have the bound
\begin{equation} \label{eq:qpoc_lower_bound}
|(z;t)_n| \geq \delta (t^{1/2};t)_\infty^2.
\end{equation}
\end{lemma}

To prove \eqref{eq:split_integral1} is $0$, we claim that
\begin{align}
\label{eq:prod_diff_O}
\begin{split}
\frac{g_{\tau}(w_1,\ldots,w_k)}{\prod_{i=1}^k(-tw_i;t)_{\eta_k(\tau)}} - \frac{1}{\prod_{i=1}^k(-tw_i;t)_{\infty}}  &= \frac{g_{\tau}(w_1,\ldots,w_k)\prod_{i=1}^k (-t^{\eta_k(\tau)+1}w_i;t)_\infty - 1}{\prod_{i=1}^k(-tw_i;t)_{\infty}} \\ 
&= O(t^{\eta_k(\tau) - \nu_1'(\tau) - \xi(\tau)})
\end{split}
\end{align}
if $w_i \in \Gamma_1(\tau)$ for all $i$. For such $w_i$ we have $|t^{\eta_k(\tau)+1}w_i| \leq t^{\eta_k(\tau)-\xi(\tau)}$ (for $\tau$ large enough so $|-t^{- \xi(\tau)}\pm i| \leq t^{-\xi(\tau) - 1}$). Hence $(-t^{\eta_k(\tau)+1}w_i;t)_{\infty} = 1 + O(t^{\eta_k(\tau) - \xi(\tau)})$, so
\begin{equation}\label{eq:compare_fin_inf_qpoc}
\frac{1}{(-tw_i;t)_{\eta_k(\tau)}} = \frac{1}{(-tw_i;t)_{\infty}}(1+O(t^{\eta_k(\tau) - \xi(\tau)})).
\end{equation}
Combining \Cref{thm:g_bound_small_w} (which has the bigger error term than \eqref{eq:compare_fin_inf_qpoc}, so we neglect the one in \eqref{eq:compare_fin_inf_qpoc}) yields that 
\begin{equation}\label{eq:g_qp_bound}
g_{\tau}(w_1,\ldots,w_k)\prod_{i=1}^k (-t^{\eta_k(\tau)+1}w_i;t)_\infty = 1 + O(t^{\eta_k(\tau) - \nu_1'(\tau) - \xi(\tau)}).
\end{equation}
Furthermore, $|(-tw_i;t)_{\infty}|$ is bounded away from $0$ for $w_i \in \tG$ and $\Gamma_1(\tau) \subset \tG$ for each $\tau$, and combining with \eqref{eq:g_qp_bound} shows \eqref{eq:prod_diff_O}. 

Since $|w_i| \leq t^{-\xi(\tau)-1}$ holds for all $w_i \in \Gamma_1(\tau)$ (for large $\tau$ as above),
\begin{equation}
e^{\frac{k-1}{2}(\log t^{-1}) \floor{\log_t |w_i|}^2} = O(t^{-\frac{k-1}{2}(\xi(\tau)+1)^2}) \quad \quad \text{ uniformly over $w_i \in \Gamma_1(\tau)$.}
\end{equation}
The other terms in the exponent of \eqref{eq:f_bound} are dominated by the $\frac{k-1}{2}(\log t^{-1}) \floor{\log_t |w_i|}^2$ term (using that $\Re(w_i) \leq 1$ on $\Gamma_1(\tau)$), so 
\begin{equation}
\label{eq:cruder_f_bound}
|f(w_1,\ldots,w_k)| = O(t^{-k\frac{k-1}{2}(\xi(\tau)+1)^2}) \quad \quad \text{uniformly over $w_1,\ldots,w_k \in \Gamma_1(\tau)$}.
\end{equation}
Multiplying the bounds \eqref{eq:g_qp_bound} and \eqref{eq:cruder_f_bound} by the length of the contour $\Gamma_1(\tau)$ which is $O(t^{-\xi(\tau)})$, we find that the first line \eqref{eq:split_integral1} is bounded by
\begin{equation}
O(t^{\eta_k(\tau) - \nu_1'(\tau) -  \xi(\tau)}) \cdot O(t^{-k\frac{k-1}{2}(\xi(\tau)+1)^2}) \cdot O(t^{-k\xi(\tau)})
\end{equation}
as $\tau  \to \infty$, and this is $o(1)$ by \eqref{eq:technical_restriction} and \eqref{eq:xi_small}. This shows the vanishing of the first line \eqref{eq:split_integral1}.

For the other two integrals \eqref{eq:split_integral2} and \eqref{eq:split_integral3}, we have by \Cref{thm:qp_bound_tG} that $|(-tw_i;t)_\infty|^{-1}$ and $|(-tw_i;t)_{\eta_k(\tau)}|^{-1}$ are uniformly bounded over $w_i \in \tG$. The only part of $\tG \cup \Gamma(\tau)$ which is not contained in $\tG$ is the vertical segment, and \Cref{thm:qpoc_lower_bound} (applied with any $0<\delta<1/2$) shows that $|(-tw_i;t)_{\eta_k(\tau)}|^{-1}$ is uniformly bounded on these segments as $\tau$ varies; here it is important that we chose the vertical part of $\Gamma(\tau)$ to have real part sufficiently far away from $-t^\Z$. To prove that the limits in \eqref{eq:split_integral2} and \eqref{eq:split_integral3} are $0$ it thus suffices to show 
\begin{equation}
\label{eq:split_2}
\lim_{\substack{\tau \to \infty \\ \log_t(\tau) \equiv \zeta }} \int_{\Gamma(\tau)^k \setminus \Gamma_1(\tau)^k} |f(w_1,\ldots,w_k)| \cdot |g_\tau(w_1,\ldots,w_k)| \prod_{i=1}^k dw_i = 0
\end{equation}
and 
\begin{equation}
\label{eq:split_3}
\lim_{\substack{\tau \to \infty \\ \log_t(\tau) \equiv \zeta }} \int_{\tG^k \setminus \Gamma_1(\tau)^k} |f(w_1,\ldots,w_k)| \prod_{i=1}^k dw_i = 0.
\end{equation}
We first show \eqref{eq:split_2}. Note that
\begin{equation}\label{eq:gamma_union_split}
\Gamma(\tau)^k \setminus \Gamma_1(\tau)^k = \bigcup_{i=1}^k \Gamma(\tau)^{i-1} \times \Gamma_2(\tau) \times \Gamma(\tau)^{k-i}
\end{equation}
(not a disjoint union). Hence by symmetry of the integrand it suffices to show
\begin{equation}
\label{eq:sym_split_2}
\lim_{\substack{\tau \to \infty \\ \log_t(\tau) \equiv \zeta }}  \int_{\Gamma_2(\tau)} \left(\int_{\Gamma(\tau)^{k-1} } |g_{\tau}(w_1,\ldots,w_k)| \cdot |f(w_1,\ldots,w_k)| \prod_{i=1}^{k-1} dw_i\right) dw_k = 0.
\end{equation}
The bound \eqref{eq:f_bound} on $|f|$ factors, so combining \Cref{thm:f_bound} and \Cref{thm:g_bound_large_w} yields
\begin{multline}\label{eq:split2_factor_bound}
\text{LHS\eqref{eq:sym_split_2}} \leq  \lim_{\substack{\tau \to \infty \\ \log_t(\tau) \equiv \zeta }} \left(\int_{\Gamma_2(\tau)}  C  e^{c_1\Re(w) + \frac{k-1}{2}(\log t^{-1}) \floor{\log_t |w|}^2 + c_2 \floor{\log_t |w|}}  dw \right) \\ 
\times \left( \int_{\Gamma(\tau)} C  e^{c_1\Re(w) + \frac{k-1}{2}(\log t^{-1}) \floor{\log_t |w|}^2 + c_2 \floor{\log_t |w|}} dw\right)^{k-1} \times O(t^{-(k+1)\nu_1'(\tau)^2}).
\end{multline}
It is easy to see that 
\begin{equation}
\abs*{\int_{\Gamma(\tau)} C  e^{c_1\Re(w) + \frac{k-1}{2}(\log t^{-1}) \floor{\log_t |w|}^2 + c_2 \floor{\log_t |w|}} dw} < \text{const}
\end{equation}
independent of $\tau$, since the $\Re(w)$ term is negative and dominates because $\Re(w) \approx -|w|$ on the contour. Hence by \eqref{eq:split2_factor_bound} it suffices to show 
\begin{equation}
\label{eq:split2_almost_done}
 \lim_{\substack{\tau \to \infty \\ \log_t(\tau) \equiv \zeta }}  \int_{\Gamma_2(\tau)}  C  e^{c_1\Re(w) + \frac{k-1}{2}(\log t^{-1}) \floor{\log_t |w|}^2 + c_2 \floor{\log_t |w|}}  dw \times O(t^{-(k+1)\nu_1'(\tau)^2}) = 0.
\end{equation}
Since
\begin{equation}
\frac{k-1}{2}(\log t^{-1}) \floor{\log_t |w|}^2 + c_2 \floor{\log_t |w|} = o(\Re(w)) \quad \quad \text{ as $w \to \infty$ along $\bigcup_\tau \Gamma_2(\tau)$,}
\end{equation}
and 
\begin{equation}
\inf_{w \in \Gamma_2(\tau)} |w| \to \infty \quad \quad \text{as $T \to \infty$,}
\end{equation}
there is a constant $0 < c_1' < c_1$ for which 
\begin{equation}
C  e^{c_1\Re(w) + \frac{k-1}{2}(\log t^{-1}) \floor{\log_t |w|}^2 + c_2 \floor{\log_t |w|}} \leq e^{c_1' \Re(w)}
\end{equation}
on $\Gamma_2(\tau)$ for all sufficiently large $\tau$. We may therefore bound
\begin{equation}\label{eq:gamma2_int_bound}
\abs*{\int_{\Gamma_2(\tau)}  C  e^{c_1\Re(w) + \frac{k-1}{2}(\log t^{-1}) \floor{\log_t |w|}^2 + c_2 \floor{\log_t |w|}}  dw } \leq \int_{\Gamma_3(\tau)} e^{c_1' \Re(w)} dw + o(1)
\end{equation}
where the $o(1)$ corresponds to the vertical part of $\Gamma_2(\tau)$. Explicitly,
\begin{equation}
\text{RHS\eqref{eq:gamma2_int_bound}} = 2 \int_{-\infty}^{-t^{-\xi(\tau)}} e^{c_1' x} dx = \frac{2}{c_1'} e^{-t^{-\xi(\tau)}}.
\end{equation}
We have thus shown that the expression inside the limit of \eqref{eq:split2_almost_done} is $O(e^{-t^{-\xi(\tau)}}) \times O(t^{-(k+1)\nu_1'(\tau)^2})$. By the naive bound $\nu_1'(\tau) \leq \log_{t^{-1}}\tau$ (for large enough $\tau$, by \eqref{eq:technical_restriction}), we may rewrite this as 
\begin{equation}
O(e^{-t^{-\xi(\tau)}}) \times O(t^{-(k+1)\nu_1'(\tau)^2}) = O(\exp(-e^{(\log t^{-1})\xi(\tau)} + (k-1)(\log t^{-1}) e^{- 2 \log \log \tau})),
\end{equation} 
which is $o(1)$ by \eqref{eq:xi_large}. This shows \eqref{eq:split_2}, so the limit \eqref{eq:split_integral2} is indeed $0$. The case of \eqref{eq:split_3} is almost exactly the same, except that (a) the analogue of \eqref{eq:split2_factor_bound} yields directly to treating a single integral over $\Gamma_3(\tau)$ rather than having to reduce to this as in \eqref{eq:gamma2_int_bound}, and (b) there is no $g_\tau$ so the final bound is in fact better requires only that $\xi(\tau) \to \infty$ at any rate rather than the growth rate \eqref{eq:xi_large} to finish. This shows \eqref{eq:cont_integral_limit} and completes the proof.
\end{proof}

\subsection{Analytic lemmas on $f$ and $g_\tau$.} We will deduce \Cref{thm:f_bound} from a closely related result \Cref{thm:tf_bound}, which is phrased in such a way as to be useful later in \Cref{sec:alpha} as well.

\begin{defi}\label{def:tf}
We denote by $\tf$ the meromorphic function
\begin{multline}
\label{eq:def_tf}
\tf(w_1,\ldots,w_k) := \frac{(t;t)_\infty^{k-1}}{k!} \prod_{i=1}^{k-1} \frac{t^{\binom{L_i-L_k}{2}}}{(t;t)_{L_i-L_{i+1}}} \frac{\prod_{1 \leq i \neq j \leq k} (w_i/w_j;t)_\infty}{\prod_{i=1}^k w_i(-w_i^{-1};t)_\infty }  \\ 
\times \sum_{j=0}^{L_{k-1}-L_k} t^{\binom{j+1}{2}} \sqbinom{L_{k-1}-L_k}{j}_t  P_{(L_1-L_k,\ldots,L_{k-1}-L_k,j)}(w_1^{-1},\ldots,w_k^{-1};t,0),
\end{multline}
(i.e. $f$ from above but without the exponential factor), which implicitly depends on parameters $t \in (0,1),\vec{L} \in \Sig_k$ as well though we suppress this in the notation.
\end{defi}

\begin{lemma}\label{thm:tf_bound}
For any neighborhood $-1 \in U \subset \C$, there exist positive constants $C,c_2$ such that the bound
\begin{equation}
\label{eq:tf_bound}
|\tf(w_1,\ldots,w_k)| \leq C \prod_{i=1}^k e^{\frac{k-1}{2}(\log t^{-1}) \floor{\log_t |w_i|}^2 + c_2 \floor{\log_t |w_i|}}
\end{equation}
holds for any $w_1,\ldots,w_k \in \C \setminus (\D \cup U)$, where $\D$ is the open unit disc.
\end{lemma}

\begin{proof}[Proof of \Cref{thm:tf_bound}]
If $k \geq 2$ then the sum over $j$ in \eqref{eq:def_tf} is a polynomial in $w_1^{-1},\ldots,w_k^{-1}$, hence if $|w_i| \geq 1$ for all $i$ we have
\begin{equation}\label{eq:poly_bounded}
\abs*{\sum_{j=0}^{L_{k-1}-L_k} t^{\binom{j+1}{2}} \sqbinom{L_{k-1}-L_k}{j}_t  P_{(L_1-L_k,\ldots,L_{k-1}-L_k,j)}(w_1^{-1},\ldots,w_k^{-1};t,0)} < \text{const}.
\end{equation}
If $k=1$, then the sum is 
\begin{equation}
\label{eq:poly_bounded_k=1}
\sum_{j=0}^\infty \frac{t^{\binom{j+1}{2}} (w_1^{-1})^j}{(t;t)_j} = (-tw_1^{-1};t)_\infty
\end{equation}
by the $q$-binomial theorem, and this too is clearly bounded by a constant over all $|w_1| \geq 1$.

For the products in the denominator of \eqref{eq:def_tf},
\begin{equation}\label{eq:inverse_qpoc_bound}
\abs*{\frac{1}{(-w_i^{-1};t)_\infty}} \leq \abs*{\frac{1}{1+w_i^{-1}}} \frac{1}{(t;t)_\infty}
\end{equation}
since $|w_i|^{-1} \leq 1$, and $w_i$ is bounded away from $-1$ the above is bounded by a constant. Similarly, $|1/w_i|$ is clearly bounded above by a constant in $\C \setminus (\D \cup U)$. We now treat nonconstant terms. Since $|w_j| \geq 1$, 
\begin{equation}
|(w_i/w_j;t)_\infty| \leq (-|w_i|;t)_\infty \leq (-t^{\floor{\log_t |w_i|}};t)_\infty.
\end{equation}
By writing
\begin{equation}
(-t^{-b};t)_\infty = (-1;t)_\infty t^{-b}(t^b+1)t^{-b+1}(t^{b-1}+1) \cdots t^{-1}(t+1) \leq t^{-\binom{b+1}{2}} (-1;t)_\infty^2,
\end{equation}
we therefore obtain
\begin{equation}
\label{eq:bound_ij_prod}
\abs*{\prod_{1 \leq i \neq j \leq k} (w_i/w_j;t)_\infty} \leq \text{const} \cdot \prod_{i=1}^k e^{(-\log t)\left(\frac{k-1}{2}\floor{\log_t |w_i|}^2+\frac{3(k-1)}{2}\floor{\log_t |w_i|}\right)}.
\end{equation}
Combining \eqref{eq:bound_ij_prod} with the previous constant bounds completes the proof.
\end{proof}

\begin{proof}[Proof of \Cref{thm:f_bound}]
The contours given do not intersect $\D \cup B_{1/2}(-1)$ provided that $\tau$ is large enough that the vertical part of $\Gamma(\tau)$ lies to the left of the set. Hence the result follows from \Cref{thm:tf_bound} together with the elementary inequality
\begin{equation}
\abs*{e^{\frac{t^{L_k+\zeta}}{1-t}w}} = e^{c_1 \Re(w)}.
\end{equation}
\end{proof}

\begin{proof}[Proof of \Cref{thm:g_bound_small_w}]
First note that
\begin{align}\label{eq:g_branch_explicit}
\begin{split}
g_{\tau}(w_1,\ldots,w_k)&= \sum_{\kappa \subset \nu(\tau)} \frac{P_{\nu(\tau)/\kappa}(1,t,\ldots;0,t) }{P_{\nu(\tau)}(1,t,\ldots;0,t)} t^{\eta_k(\tau)|\kappa|} Q_{\kappa'}(w_1,\ldots,w_k; t,0)\\
&= \sum_{\kappa \subset \nu(\tau)} Q_{\kappa'}(w_1,\ldots,w_k; t,0) \prod_{i=1}^k t^{\binom{\nu_i'(\tau) - \kappa_i'}{2} - \binom{\nu_i'(\tau)}{2} + \eta_k(\tau)\kappa_i'} (t^{1+\nu_i'(\tau)-\kappa_i'};t)_{m_i(\kappa)},
\end{split}
\end{align}
where we have used the branching rule, the fact that $Q_{\kappa'}$ is homogeneous of degree $|\kappa|$, and the explicit formula \eqref{eq:hl_quotient}. We wish to bound
\begin{align}\label{eq:g-1}
\begin{split}
&\abs*{g_{\tau}(w_1,\ldots,w_k) - 1} \\ 
&= \abs*{\sum_{\substack{\kappa \subset \nu(\tau) \\ \kappa \neq \emptyset}} Q_{\kappa'}(w_1,\ldots,w_k; t,0) \prod_{i=1}^k t^{\binom{\nu_i'(\tau) - \kappa_i'}{2} - \binom{\nu_i'(\tau)}{2} + \eta_k(\tau)\kappa_i'} (t^{1+\nu_i'(\tau)-\kappa_i'};t)_{m_i(\kappa)}} \\ 
&\leq \sum_{\substack{\kappa \subset \nu(\tau) \\ \kappa \neq \emptyset}} Q_{\kappa'}(|w_1|,\ldots,|w_k|; t,0) \prod_{i=1}^k t^{\binom{\nu_i'(\tau) - \kappa_i'}{2} - \binom{\nu_i'(\tau)}{2} + \eta_k(\tau)\kappa_i'} (t^{1+\nu_i'(\tau)-\kappa_i'};t)_{m_i(\kappa)} \\ 
&\leq \sum_{\kappa \subset \nu(\tau) } Q_{\kappa'}(|w_1|,\ldots,|w_k|; t,0) \prod_{i=1}^k t^{\binom{\nu_i'(\tau) - \kappa_i'}{2} - \binom{\nu_i'(\tau)}{2} + \eta_k(\tau)\kappa_i'} - 1,
\end{split}
\end{align}
where in the last bound we used that $(t^{1+\nu_i'(\tau)-\kappa_i'};t)_{m_i(\kappa)} \in [0,1]$. We rewrite the exponential as 
\begin{equation}\label{eq:t_power_bound}
t^{\binom{\nu_i'(\tau) - \kappa_i'}{2} - \binom{\nu_i'(\tau)}{2} + \eta_k(\tau)\kappa_i'} = t^{ \kappa_i'(\kappa_i'/2+1/2-\nu_i'(\tau) + \eta_k(\tau))} \leq t^{ \kappa_i'(\kappa_i'/2+1/2-\nu_1'(\tau) + \eta_k(\tau))}.
\end{equation}
Then 
\begin{equation}\label{eq:t_to_Q}
\prod_{i=1}^k t^{ \kappa_i'(\kappa_i'/2+1/2-\nu_1'(\tau) + \eta_k(\tau))} = t^{n(\kappa) + (\eta_k(\tau) - \nu_1'(\tau))|\kappa|} = \left(t^{\eta_k(\tau)-\nu_1'(\tau)}\right)^{|\kappa|} Q_\kappa(1,t,\ldots;0,t).
\end{equation}
Substituting \eqref{eq:t_power_bound} and \eqref{eq:t_to_Q} into \eqref{eq:g-1} yields 
\begin{align}
\begin{split}
\abs*{g_{\tau}(w_1,\ldots,w_k) - 1} &\leq \sum_{\kappa \subset \nu(\tau) } \left(t^{\eta_k(\tau)-\nu_1'(\tau)}\right)^{|\kappa|} Q_{\kappa'}(|w_1|,\ldots,|w_k|; t,0) Q_\kappa(1,t,\ldots;0,t) -1 \\ 
&\leq \sum_{\kappa \in \Y} \left(t^{\eta_k(\tau)-\nu_1'(\tau)}\right)^{|\kappa|} Q_{\kappa'}(|w_1|,\ldots,|w_k|; t,0) Q_\kappa(1,t,\ldots;0,t) -1 \\ 
&= \Pi_{0,t}(1,t,\ldots; \beta(t^{\eta_k(\tau)-\nu_1'(\tau)}|w_1|,\ldots,t^{\eta_k(\tau)-\nu_1'(\tau)}|w_k|)) - 1\\ 
&= \prod_{i=1}^k (-t^{\eta_k(\tau)-\nu_1'(\tau)}|w_i|;t)_\infty - 1
\end{split}
\end{align}
by the Cauchy identity. In our setting, $|w_i| \leq t^{-\xi(\tau)-1}$ since $w_i \in \Gamma_1(\tau)$ for all $i$, so by the $q$-binomial theorem 
\begin{equation}
\prod_{i=1}^k (-t^{\eta_k(\tau)-\nu_1'(\tau)}|w_i|;t)_\infty - 1 \leq \prod_{i=1}^k \left(\sum_{j \geq 0} \frac{\left(t^{\eta_k(\tau)-\nu_1'(\tau)}|w_i|\right)^j t^{\binom{j}{2}}}{(t;t)_j} \right) - 1 = O(t^{\eta_k(\tau) - \nu_1'(\tau) - \xi(\tau)-1}),
\end{equation}
completing the proof.
\end{proof}

\begin{proof}[Proof of \Cref{thm:g_bound_large_w}]
By the same manipulations as in the first part of the proof of \Cref{thm:g_bound_small_w}, 
\begin{align}\label{eq:w_large_1}
\abs*{g_{\tau}(w_1,\ldots,w_k)} \leq \sum_{\kappa \subset \nu(\tau) } Q_{\kappa'}(t^{\eta_k(\tau)}|w_1|,\ldots,t^{\eta_k(\tau)}|w_k|; t,0) \prod_{i=1}^k t^{ \kappa_i'(\kappa_i'/2+1/2-\nu_i'(\tau))}.
\end{align}
We bound this as 
\begin{equation}
\label{eq:naive_w_bound}
\text{RHS\eqref{eq:w_large_1}} \leq \#\{\kappa \in \Y: \kappa \subset \nu(\tau)\} \times \sup_{\kappa \subset \nu(\tau)} Q_{\kappa'}(t^{\eta_k(\tau)}|w_1|,\ldots,t^{\eta_k(\tau)}|w_k|; t,0) \times \sup_{\kappa \subset \nu(\tau)} \prod_{i=1}^k t^{ \kappa_i'(\kappa_i'/2+1/2-\nu_i'(\tau))}.
\end{equation}
Since $\kappa_i' \leq \nu_i'(\tau)$ for each $1 \leq i \leq k$ and $\kappa_k'=0$, a naive bound gives 
\begin{equation}
 \#\{\kappa \in \Y: \kappa \subset \nu(\tau)\} \leq (\nu_1'(\tau)+1) \cdots (\nu_k'(\tau)+1).
\end{equation}
By the branching rule,
\begin{align}\label{eq:bound_qW_kappa'}
\begin{split}
&Q_{\kappa'}(t^{\eta_k(\tau)}|w_1|,\ldots,t^{\eta_k(\tau)}|w_k|; t,0) \\
&= \sum_{\emptyset = \rho^{(0)} \prec \rho^{(1)} \prec \ldots \prec \rho^{(k)} = \kappa'} \prod_{i=1}^k \left(t^{\eta_k(\tau)}|w_i|\right)^{|\rho^{(i)}/\rho^{(i-1)}|} \frac{1}{(t;t)_{\rho^{(i)}_1 - \rho^{(i-1)}_1}}  \prod_{j=1}^i \sqbinom{\rho^{(i-1})_j - \rho^{(i-1)}_{j+1}}{\rho^{(i-1})_j - \rho^{(i)}_{j+1}}_t \\ 
& \leq \text{const}^{|\kappa|} \cdot \sum_{\emptyset = \rho^{(0)} \prec \rho^{(1)} \prec \ldots \prec \rho^{(k)} = \kappa'} (t;t)_\infty^{-2k^2}
\end{split}
\end{align}
where we have used that $t^{\eta_k(\tau)}|w_i| \leq \text{const}$ on $\Gamma(\tau)$, and bounded the branching coefficients by $(t;t)_\infty^{-2k^2}$. Similarly to the proof of \Cref{thm:g_bound_small_w}, we bound the number of Gelfand-Tsetlin patterns $\emptyset = \rho^{(0)} \prec \rho^{(1)} \prec \ldots \prec \rho^{(k)} = \kappa'$ by 
\begin{equation}
\prod_{\ell=1}^{k-1} \#\{(\zeta_1,\ldots,\zeta_\ell): 0 \leq \zeta_i \leq \kappa_1' \text{ for all }1 \leq i \leq \ell\} = (\kappa_1'+1)^{\binom{k}{2}},
\end{equation}
so \eqref{eq:bound_qW_kappa'} becomes
\begin{equation}
\label{eq:final_kappa'_bound}
Q_{\kappa'}(|w_1|,\ldots,|w_k|; t,0) \leq \text{const}^{|\kappa|} (t;t)_\infty^{-2k^2} (\kappa_1'+1)^{\binom{k}{2}}.
\end{equation}
We now bound the power of $t$ in \eqref{eq:naive_w_bound} as
\begin{equation}\label{eq:t_power_ineq}
\prod_{i=1}^k t^{ \kappa_i'(\kappa_i'/2+1/2-\nu_i'(\tau))} \leq t^{\frac{k}{8}-\sum_{i=1}^k (\nu_i')^2} \leq \text{const} \cdot t^{-k(\nu_1')^2}.
\end{equation}
Combining \eqref{eq:w_large_1}, \eqref{eq:final_kappa'_bound}, \eqref{eq:t_power_ineq} and the fact that $\kappa_1' \leq \nu_1'$ and so $|\kappa| \leq k\nu_1'$, we have
\begin{equation}
\abs*{g_{\tau}(w_1,\ldots,w_k)} \leq \text{const}' \cdot \text{const}^{k\nu_1'} \cdot (\nu_1'+1)^{\binom{k}{2}} \cdot t^{-k(\nu_1')^2} = O(t^{-(k+1)(\nu_1'(\tau))^2}),
\end{equation}
completing the proof. 
\end{proof}

\begin{rmk}
We initially proved \Cref{thm:hl_stat_dist} and its auxiliary lemmas in the case of trivial initial condition $\nu(\tau) \equiv \emptyset$ only. It was quite unexpected that the addition of an initial condition merely produces a simple multiplicative factor 
\[
\frac{P_\nu(\beta(z_1,\ldots,z_k),1,t,\ldots;0,t)}{P_\nu(1,t,\ldots;0,t)}
\]
in the integrand of \Cref{thm:use_torus_product}, which may be treated asymptotically as above.
\end{rmk}

\begin{rmk} \label{rmk:why_hypotheses_suboptimal}
Our hypothesis $\log_{t^{-1}}\tau - \nu_1'(\tau) \geq 2\pfrac{k}{\log t^{-1}}^2 \log \log \tau $ in \Cref{thm:hl_stat_dist}, which we believe is slightly suboptimal as mentioned in \Cref{rmk:hypotheses_suboptimal}, comes from the need for existence of a function $\xi(\tau)$ to split the contours as above. The requirement \eqref{eq:xi_large} forces $\xi(\tau)$ to be large, while the requirement \eqref{eq:xi_small} forces $\log_{t^{-1}}\tau - \nu_1'(\tau)$ to be larger than $\xi(\tau)^2$, so improving the bounds in either case could lead to improved versions of the technical hypothesis \eqref{eq:technical_restriction}. 

The requirement \eqref{eq:xi_large} in the proof above essentially comes from the need to dominate the error term bounded in \Cref{thm:g_bound_large_w}. This error term came from the main term in the proof of \Cref{thm:g_bound_large_w}, which comes from bounding
\begin{equation}
\frac{P_{\nu(\tau)/\kappa}(1,t,\ldots;0,t)}{P_{\nu(\tau)}(1,t,\ldots;0,t)},
\end{equation}
and our bound on this term is essentially sharp up to unimportant constants for the case $\kappa_i' \approx \nu_i'/2$, $\nu_k' \approx \nu_1'$. This is why it is not clear to us at the moment how the $\log \log \tau$ growth of $\log_{t^{-1}}\tau - \nu_1'(\tau) $ can be improved beyond improving the constant in front of $\log \log \tau$. It seems likely to us that this can be done by manipulating expressions differently before bounding to take advantage of more cancellations.
\end{rmk}

\begin{rmk}
It seems possible that a more involved version of the above manipulations could yield an explicit joint distribution of $\la_1'(\tau),\ldots,\la_k'(\tau)$ where the initial condition $\nu(\tau)$ have parts $\nu_i'(\tau)$ which grow like $\log_{t^{-1}}(\tau) + c_i$. Such a distribution would in particular be different from the one above, since the fact that $\la_i'(\tau) \geq \nu_i'(\tau)$ would make the conjugate parts $\la_i'(\tau)$ bounded below in the above regime. 
\end{rmk}

\section{Residue expansions} \label{sec:residues}

The probability distribution in \Cref{thm:hl_stat_dist}, which is expressed there by a contour integral, may be residue-expanded to yield formulas for the same limiting probability in terms of certain infinite series in $e^{-t^d}, d \leq L_k$, which lead to series formulas for the weights of $\cL_{k,t,\chi}$.

\begin{prop}\label{thm:hl_residue_formula}
Let $\la(\tau)$ be as in \Cref{thm:hl_stat_dist}, with or without the initial condition. Then for any $\vec{L} = (L_1,\ldots,L_k) \in \Sig_k$, we have
\begin{multline}\label{eq:brpw_residue_formula}
\lim_{\substack{\tau \to \infty \\ \log_{t^{-1}}(\tau) + \zeta \in \Z}} \Pr( (\la_i'(\tau) - \log_{t^{-1}}(\tau) - \zeta)_{1 \leq i \leq k} = \vec{L})   = \frac{1}{(t;t)_\infty}\sum_{d \leq L_k} e^{\frac{t^{d+\zeta}}{1-t}} \frac{t^{\sum_{i=1}^k \binom{L_i-d}{2}}}{ (t;t)_{L_k-d} \prod_{i=1}^{k-1}(t;t)_{L_i-L_{i+1}}}  \\ 
\times \sum_{\substack{\mu \in \Sig_{k-1} \\ \mu \prec \vec{L}}} (-1)^{|\vec{L}| - |\mu|-d}  \prod_{i=1}^{k-1} \sqbinom{L_i-L_{i+1}}{L_i-\mu_i}_t Q_{(\mu - (d[k-1]))'}(\gamma(t^{d+\zeta}),\alpha(1);0,t).
\end{multline}
\end{prop}

\begin{rmk}\label{rmk:mysterious_hl_formula}
By the branching and principal specialization formulas (\Cref{thm:hl_principal_formulas} and \Cref{thm:hl_qw_branch_formulas}), the formula \eqref{eq:brpw_residue_formula} may also be written as 
\begin{equation}\label{eq:mysterious_hl_formula}
\frac{1}{(t;t)_\infty}\sum_{d \leq L_k} e^{\frac{t^{d+\zeta}}{1-t}} P_{\vec{L}(d)'}(1,t,\ldots;0,t) \sum_{\substack{\mu \in \Sig_{k-1} \\ \mu \prec \vec{L}}} P_{\vec{L}(d)/\mu(d)}(-1;t,0) Q_{\mu(d)'}(\gamma(t^{d+\zeta}),\alpha(1);0,t)
\end{equation}
where $\vec{L}(d) := (L_1-d,\ldots,L_k-d), \mu(d) := (\mu_1-d,\ldots,\mu_k-d)$. The fact that the final answer has such a simple expression in terms of symmetric functions seems in no way justified by the complicated intermediate manipulations we have taken, and it would be very interesting to find a simpler proof of \Cref{thm:hl_residue_formula} which explains this. We remark also that at first glance it might appear that the branching rule (for general specialized Macdonald symmetric functions) would simplify the sum over $\mu$ in \eqref{eq:mysterious_hl_formula}. The issue is that the sum is over $\mu \in \Sig_{k-1}$, which is a smaller index set, c.f. \eqref{eq:almost_cauchy} and the discussion after for a similar sum in an earlier prelimit expression which appears to be responsible for the above.
\end{rmk}

The integrand in our previous contour integral formula \eqref{eq:explicit_hl_stat_dist_formula} has poles at $w_i = -t^x, x \in \Z, 1 \leq i \leq k$, all of which lie within $\tG$. To derive \Cref{thm:hl_residue_formula} we wish to residue expand at these poles to obtain the right hand side of \eqref{eq:brpw_residue_formula}, but because $\tG$ is not a closed contour and furthermore the integrand is not meromorphic in a neighborhood of $0$, justifying this takes some care. To this end we state the following lemmas.  \Cref{thm:residue_expansion_independence} is the main algebraic step of computing the residues which arise from shifting contours. \Cref{thm:small_contour_negligible} shows that the contour integral appearing as an error term in \Cref{thm:residue_expansion_independence} is indeed negligible, and hence should be thought of as the statement that the integral in \eqref{eq:explicit_hl_stat_dist_formula} is indeed equal to its naive residue expansion. Recall the function $f$ from \eqref{eq:integrand_junk}.

\begin{lemma}\label{thm:residue_expansion_independence}
Fix $k \in \Z_{\geq 1}$, and let $h \in \Z_{\geq 0}$ and $\Gamma$ be a simple closed contour with interior containing $\{-t^x: x \in \Z, x \geq -h\}$. Then for any $\vec{L} = (L_1,\ldots,L_k) \in \Sig_k$ and any integer $n \geq L_k$,
\begin{multline}\label{eq:residue_expansion_independence}
\frac{1}{(2 \pi \bi)^k} \int_{\Gamma^k} \frac{f(w_1,\ldots,w_k)}{\prod_{i=1}^k (-tw_i;t)_\infty} \prod_{i=1}^k dw_i = \frac{1}{(t;t)_\infty}\sum_{d=L_k-h}^{L_k}  \frac{e^{-\frac{t^{d+\zeta}}{1-t}}t^{\sum_{i=1}^k \binom{L_i-d}{2}}}{ (t;t)_{L_k-d} \prod_{i=1}^{k-1}(t;t)_{L_i-L_{i+1}}}  \\ 
\times \sum_{\substack{\mu \in \Sig_{k-1} \\ \mu \prec \vec{L}}} (-1)^{|\vec{L}| - |\mu|-d} \prod_{i=1}^{k-1} \sqbinom{L_i-L_{i+1}}{L_i-\mu_i}_t Q_{(\mu - (d[k-1]))'}(\gamma(t^{d+\zeta}),\alpha(1);0,t) \\ 
+   \frac{1}{(2 \pi \bi)^k} \int_{(t^{n+1/2}\T)^k} \frac{f(w_1,\ldots,w_k)}{\prod_{i=1}^k (-tw_i;t)_\infty} \prod_{i=1}^k dw_i
\end{multline}
where $f(w_1,\ldots,w_k)$ is as in \eqref{eq:integrand_junk}. In particular, the right hand side is independent of $n \geq L_k$.
\end{lemma}

\begin{lemma}\label{thm:small_contour_negligible}
Fix $k \in \Z_{\geq 1}$ and $\vec{L} = (L_1,\ldots,L_k) \in \Sig_k$. Then for any $n \geq L_k$,
\begin{multline}\label{eq:zero_integral}
 \frac{(t;t)_\infty^{k-1}}{k! (2 \pi \bi)^k} \prod_{i=1}^{k-1} \frac{t^{\binom{L_i-L_k}{2}}}{(t;t)_{L_i-L_{i+1}}} \int_{(t^{n+1/2}\T)^k} e^{\frac{t^{L_k+\zeta}}{1-t}(w_1+\ldots+w_k)} \frac{\prod_{1 \leq i \neq j \leq k} (w_i/w_j;t)_\infty}{\prod_{i=1}^k (-w_i^{-1};t)_\infty (-tw_i;t)_{\infty}} \\ 
\times \sum_{j=0}^{L_{k-1}-L_k} t^{\binom{j+1}{2}} \sqbinom{L_{k-1}-L_k}{j}_t  P_{(L_1-L_k,\ldots,L_{k-1}-L_k,j)}(w_1^{-1},\ldots,w_k^{-1};t,0)  \prod_{i=1}^k \frac{dw_i}{w_i} = 0
\end{multline}
with the sum over $j$ interpreted as in \Cref{thm:hl_stat_dist} when $k=1$.
\end{lemma}

We also record a certain computation used several times below in the following lemma. 

\begin{lemma}\label{thm:theta_transformations}
For any $w \in \C^\times$ and $n \in \Z$, 
\begin{equation}
(-t^{-n}w^{-1};t)_\infty (-t^{n+1}w;t)_{\infty} = w^{-n} t^{-\binom{n+1}{2}} (-w^{-1};t)_\infty (-t w; t)_\infty.
\end{equation}
\end{lemma}
\begin{proof}
A simple direct computation.
\end{proof}

\begin{rmk}
Since 
\begin{equation}
(-w^{-1};t)_\infty (-t w; t)_\infty = \frac{\theta_3(t^{1/2}w;t)}{(t;t)_\infty}
\end{equation}
may be written in terms of the Jacobi theta function
\begin{equation}
\theta_3(z;t) := (t;t)_\infty \prod_{n \in \Z_{\geq 0}} (1+t^{n+1/2}z)(1+t^{n+1/2}/z),
\end{equation}
the above computation is in fact equivalent to the standard transformation law 
\begin{equation}
\theta_3(tz;t) = t^{-1/2}z^{-1}\theta_3(z;t).
\end{equation}
It is worth mentioning that Jacobi theta functions have appeared in the related context of periodic Schur processes introduced in \cite{Bor07}, used further in e.g. \cite{ahn2022lozenge}, \cite{betea2019periodic}, \cite{imamura2021skew,imamura2021identity,imamura2022solvable}, and their above transformation law has been useful there. We are not aware of any closer connection with the present work, however. See e.g. \cite[Chapter 13]{erdelyi1981higher} for more background on theta functions, though the notation there differs from that of \cite{Bor07} which we use above.
\end{rmk}

\begin{proof}[Proof of \Cref{thm:residue_expansion_independence}]
%idea: just shrink last contour to the small circle first. Then this throws a bunch of residues which are what we compute (using inner product, after pushing contours back to circles and noting that all remaining residues are 0). Shrinking the remaining contours throws residues which are not 0 (they're just more copies of the first ones), and yields error term. 
The integrand on the left hand side of \eqref{eq:residue_expansion_independence} is meromorphic away from $0$ and $\infty$, and for $w_1,\ldots,w_{k-1}$ fixed it has poles at $w_k = -t^m, m \in \Z$. Of these, $-t^{h}, -t^{h+1},\ldots$ lie in the interior of $\Gamma$. Hence by deforming the $w_k$ contour to $t^{n+1/2}\T$ we obtain
\begin{align}\label{eq:push_one_contour}
\text{LHS\eqref{eq:residue_expansion_independence}} &= \frac{1}{(2 \pi \bi)^{k-1}} \int_{\Gamma^{k-1}} \left(\sum_{m=-h}^n \Res_{w_k=-t^m} \frac{f(w_1,\ldots,w_k)}{\prod_{i=1}^k (-tw_i;t)_\infty} \right) \prod_{i=1}^{k-1} dw_i \\ 
&+ \frac{1}{(2 \pi \bi)^k} \int_{\Gamma^{k-1} \times t^{n+1/2}\ \T} \frac{f(w_1,\ldots,w_k)}{\prod_{i=1}^k (-tw_i;t)_\infty} \prod_{i=1}^k dw_i. \label{eq:k-fold}
\end{align}
When $w_k = -t^m$, the factor
\[
\prod_{1 \leq i \neq j \leq k} (w_i/w_j;t)_\infty
\] 
in the numerator of $f$ has zeros at all $w_i \in -t^\Z, 1 \leq i \leq k-1$. Hence
\[
\Res_{w_k=-t^m} \frac{f(w_1,\ldots,w_k)}{\prod_{i=1}^k (-tw_i;t)_\infty}
\]
is in fact holomorphic away from $0$, so the contours of the $(k-1)$-fold integral in \eqref{eq:push_one_contour} may be deformed to any simple closed contours containing $0$ without any additional residue terms, in particular they may all be deformed to $\T$. For the $k$-fold integral in \eqref{eq:k-fold} we similarly deform the $w_{k-1}$-contour to $t^{n+1/2}\T$, yielding a term identical (by symmetry of the variables $w_1,\ldots,w_k$) to the $(k-1)$-fold integral of \eqref{eq:push_one_contour} except that one of the integrals is over $t^{n+1/2}\T$. However, since we may deform to any contours around $0$ this makes no difference, and so by pushing each of the $k$ contours in \eqref{eq:residue_expansion_independence} to $t^{n+1/2}\T$ and using symmetry of the variables to equate the $k$ sums of residues (and commuting the finite sum with the integral) we have 
\begin{align}
\begin{split}
\text{LHS\eqref{eq:residue_expansion_independence}} &= k \sum_{m=-h}^n \frac{1}{(2 \pi \bi)^{k-1}}  \int_{\T^{k-1}} \Res_{w_k=-t^m} \frac{f(w_1,\ldots,w_k)}{\prod_{i=1}^k (-tw_i;t)_\infty} \prod_{i=1}^{k-1} dw_i \\ 
&+ \frac{1}{(2 \pi \bi)^k}  \int_{(t^{n+1/2}\T)^k} \frac{f(w_1,\ldots,w_k)}{\prod_{i=1}^k (-tw_i;t)_\infty} \prod_{i=1}^k dw_i.
\end{split}
\end{align}
Hence to show \eqref{eq:residue_expansion_independence} we must show 
\begin{multline}\label{eq:residue_wts}
\frac{1}{(2 \pi \bi)^{k-1}} \int_{\T^{k-1}} \left(\sum_{m=-h}^n \Res_{w_k=-t^m} \frac{k f(w_1,\ldots,w_k)}{\prod_{i=1}^k (-tw_i;t)_\infty} \right) \prod_{i=1}^{k-1} dw_i =  \frac{1}{(t;t)_\infty }\sum_{d=L_k-h}^{L_k} \frac{e^{-\frac{t^{d+\zeta}}{1-t}}  t^{\sum_{i=1}^k \binom{L_i-d}{2}}}{ \prod_{i=1}^{k-1}(t;t)_{L_i-L_{i+1}}} \\ 
\times \frac{1}{(t;t)_{L_k-d}}\sum_{\substack{\mu \in \Sig_{k-1} \\ \mu \prec \vec{L}}} (-1)^{|\vec{L}| - |\mu|-d} \prod_{i=1}^{k-1} \sqbinom{L_i-L_{i+1}}{L_i-\mu_i}_t Q_{(\mu - (d[k-1]))'}(\gamma(t^{d+\zeta}),\alpha(1);0,t)
\end{multline}
(for $k \geq 2$) and 
\begin{equation}
\label{eq:residue_wts_k=1}
\sum_{m=-h}^n \Res_{w_1=-t^m} \frac{f(w_1)}{ (-tw_1;t)_\infty} =\frac{1}{(t;t)_\infty} \sum_{d=L_1-h}^{L_1} e^{-\frac{t^{d+\zeta}}{1-t}} \frac{t^{\binom{L_1-d}{2}}}{ (t;t)_{L_1-d} } (-1)^{L_1-d}
\end{equation}
(for $k=1$). We begin with \eqref{eq:residue_wts_k=1}, where 
\begin{align}
\begin{split}
\frac{f(w_1)}{(-tw_1;t)_\infty} &= e^{\frac{t^{L_1+\zeta}}{1-t}w_1}\frac{1}{ w_1(-w_1^{-1};t)_\infty (-tw_1;t)_\infty}  \sum_{j=0}^{\infty} \frac{t^{\binom{j+1}{2}}}{(t;t)_j}   P_{(j)}(w_1^{-1};t,0) \\ 
&= e^{\frac{t^{L_1+\zeta}}{1-t}w_1}\frac{1}{ w_1(-w_1^{-1};t)_\infty (-tw_1;t)_\infty} (-tw_1^{-1};t)_\infty \\ 
&= e^{\frac{t^{L_1+\zeta}}{1-t}w_1}\frac{1}{ (-w_1;t)_\infty}
\end{split}
\end{align}
by the $q$-binomial theorem, since $P_{(j)}(w_1^{-1};t,0) = w_1^{-j}$ by the branching rule. This function has simple poles at $w_1 = -t^{-m}, m \in \Z_{\geq 0}$, of which $-t^{-h},-t^{-h+1},\ldots,-1$ are contained in our contour. The pole at $w_1 = -t^{-m}$ contributes 
\begin{equation}\label{eq:k=1_residue}
e^{\frac{-t^{L_1+\zeta+m}}{1-t}}\frac{1}{(1-t^{-m}) \cdots (1-t^{-1}) (t^m) (t;t)_\infty} = \frac{1}{(t;t)_\infty} e^{\frac{-t^{L_1+\zeta+m}}{1-t}} \frac{t^{\binom{m}{2}}}{(t;t)_m}
\end{equation}
Summing \eqref{eq:k=1_residue} over $0 \leq m \leq h$ and making the change of variables $d=L_1+m$ yields \eqref{eq:residue_wts_k=1} and completes the $k=1$ case.

We now show the $k\geq 2$ case, \eqref{eq:residue_wts}. It is not hard to check similarly to above (one may use \Cref{thm:theta_transformations} to simplify the computation, though this is not necessary) that
\begin{equation}
\label{eq:main_residue_term}
\Res_{w_k=-t^m} \frac{1}{w_k(-w_k^{-1};t)_\infty(-tw_k;t)_\infty} = \frac{(-1)^m}{(t;t)_\infty^2 t^{-\binom{m+1}{2}}},
\end{equation}
where we let $\binom{m+1}{2} = (m^2+m)/2$ even when $m$ is negative. \Cref{thm:theta_transformations} also implies that
\begin{equation}
\label{eq:theta_residue}
\frac{(-t^{-m}w_i;t)_\infty (-t^m w_i^{-1};t)_\infty}{w_i(-w_i^{-1};t)_\infty (-tw_i;t)_\infty} = (1+t^{-m}w_i)w_i^m t^{-\binom{m+1}{2}}.
\end{equation}
Using \eqref{eq:main_residue_term}, \eqref{eq:theta_residue} and the explicit formula \eqref{eq:integrand_junk} for $f$ we compute
\begin{align}
\begin{split}
\label{eq:single_residue}
&\Res_{w_k=-t^m} \frac{k f(w_1,\ldots,w_k)}{\prod_{i=1}^k (-tw_i;t)_\infty} = \frac{(t;t)_\infty^{k-1}}{(k-1)!} \prod_{i=1}^{k-1} \frac{t^{\binom{L_i-L_k}{2}}}{(t;t)_{L_i-L_{i+1}}} e^{\frac{t^{L_k+\zeta}}{1-t}(w_1+\ldots+w_{k-1})} \\ 
& \times \prod_{1 \leq i \neq j \leq k-1} (w_i/w_j;t)_\infty \left(\sum_{j=0}^{L_{k-1}-L_k} t^{\binom{j+1}{2}} \sqbinom{L_{k-1}-L_k}{j}_t P_{\tLL + j \vec{e}_k}(w_1^{-1},\ldots,w_{k-1}^{-1},-t^{-m};t,0)\right) \\ 
&\times \frac{\prod_{i=1}^{k-1} (1+t^{m}w_i^{-1})w_i^m t^{-\binom{m+1}{2}}}{(-1)^m t^{-\binom{m+1}{2}}(t;t)_\infty^2}e^{-\frac{t^{L_k+\zeta+m}}{1-t}},
\end{split}
\end{align}
where $\tLL = (L_1-L_k,\ldots,L_{k-1}-L_k,0)$ as before. By the branching rule,
\begin{align}
\begin{split}\label{eq:branch_cancel_j}
&\sum_{j=0}^{L_{k-1}-L_k} t^{\binom{j+1}{2}} \sqbinom{L_{k-1}-L_k}{j}_t P_{\tLL + j \vec{e}_k}(w_1^{-1},\ldots,w_{k-1}^{-1},-t^{-m};t,0) \\ 
&= \sum_{j=0}^{L_{k-1}-L_k} t^{\binom{j+1}{2}} \sqbinom{L_{k-1}-L_k}{j}_t \sum_{\substack{\mu \in \Sig_{k-1} \\ \mu \prec \tLL+j\vec{e}_k}} (-t^{-m})^{|\tLL|+j-|\mu|} \sqbinom{L_1-L_2}{L_1-L_k-\mu_1}_t \cdots \sqbinom{L_{k-2}-L_{k-1}}{L_{k-2}-L_k - \mu_{k-2}}_t \\ 
&\times \sqbinom{L_{k-1}-L_k-j}{L_{k-1}-L_k-\mu_{k-1}}_t P_\mu(w_1^{-1},\ldots,w_{k-1}^{-1};t,0) \\ 
&= \sum_{\substack{\mu \in \Sig_{k-1} \\ \mu \prec \tLL}}(-t^{-m})^{|\tLL|-|\mu|} \sqbinom{L_1-L_2}{L_1-L_k-\mu_1}_t \cdots \sqbinom{L_{k-2}-L_{k-1}}{L_{k-2}-L_k - \mu_{k-2}}_t P_\mu(w_1^{-1},\ldots,w_{k-1}^{-1};t,0) \\ 
&\times \sum_{j=0}^{\mu_{k-1}} \sqbinom{L_{k-1}-L_k-j}{L_{k-1}-L_k-\mu_{k-1}}_t \sqbinom{L_{k-1}-L_k}{j}_t  t^{\binom{j+1}{2}}(-t^{-m})^j \\ 
&= \sum_{\substack{\mu \in \Sig_{k-1} \\ \mu \prec \tLL}}(-t^{-m})^{|\tLL|-|\mu|} \sqbinom{L_1-L_2}{L_1-L_k-\mu_1}_t \cdots \sqbinom{L_{k-2}-L_{k-1}}{L_{k-2}-L_k - \mu_{k-2}}_t P_\mu(w_1^{-1},\ldots,w_{k-1}^{-1};t,0) \\ 
&\times \sqbinom{L_{k-1}-L_k}{L_{k-1}-L_k-\mu_{k-1}}_t \sum_{j=0}^{\mu_{k-1}} \sqbinom{\mu_{k-1}}{j}_t   t^{\binom{j+1}{2}}(-t^{-m})^j \\ 
&= \sum_{\substack{\mu \in \Sig_{k-1} \\ \mu \prec \tLL}}(-t^{-m})^{|\tLL|-|\mu|} \prod_{i=1}^{k-1} \sqbinom{L_i-L_{i+1}}{L_i - L_k - \mu_i}_t P_\mu(w_1^{-1},\ldots,w_{k-1}^{-1};t,0)(t^{-m+1};t)_{\mu_{k-1}} 
\end{split}
\end{align}
where in the penultimate equality we used the identity
\begin{equation}
\label{eq:apply_qbin_swap}
\sqbinom{L_{k-1}-L_k-j}{L_{k-1}-L_k-\mu_{k-1}}_t \sqbinom{L_{k-1}-L_k}{j}_t = \sqbinom{\mu_{k-1}}{j}_t \sqbinom{L_{k-1}-L_k}{L_{k-1}-L_k-\mu_{k-1}}_t
\end{equation}
(which is the same as \eqref{eq:trade_qbinom}) and brought the factor which does not depend on $j$ outside the sum, and in the last equality we applied the $q$-binomial theorem. By \eqref{eq:single_residue}, \eqref{eq:branch_cancel_j}, and the fact that $w_i^{-1}=\bar{w_i}$ on $\T$,
\begin{align}
\begin{split}\label{eq:to_use_k-1_product}
&\frac{1}{(2 \pi \bi)^{k-1}} \int_{\T^{k-1}} \Res_{w_k=-t^m} \frac{k f(w_1,\ldots,w_k)}{\prod_{i=1}^k (-tw_i;t)_\infty}  \prod_{i=1}^{k-1} dw_i = (t;t)_\infty^{k-3} \prod_{i=1}^{k-1} \frac{t^{\binom{L_i-L_k}{2}}}{(t;t)_{L_i-L_{i+1}}} (-1)^m t^{\binom{m+1}{2}} e^{-\frac{t^{L_k+\zeta+m}}{1-t}} \\ 
&\times \sum_{\substack{\mu \in \Sig_{k-1} \\ \mu \prec \tLL}}(-t^{-m})^{|\tLL|-|\mu|} (t^{-m+1};t)_{\mu_{k-1}}\prod_{i=1}^{k-1} \sqbinom{L_i-L_{i+1}}{L_i-L_k-\mu_i}  \frac{1}{(k-1)!(2\pi \bi)^{k-1}} \int_{\T^{k-1}} e^{\frac{t^{L_k+\zeta}}{1-t}(w_1+\ldots+w_{k-1})} \\ 
&\times \prod_{1 \leq i \neq j \leq k-1} (w_i/w_j;t)_\infty P_\mu(\bar{w}_1,\ldots,\bar{w}_{k-1};t,0)   \prod_{i=1}^{k-1} (1+t^{-m}w_i)w_i^{m-1} t^{-\binom{m}{2}} dw_i.
\end{split}
\end{align}
We recognize a factor in the above integrand as
\begin{align}\label{eq:use_spec_cauchy}
\begin{split}
\prod_{i=1}^{k-1} e^{\frac{t^{L_k+\zeta}}{1-t}w_i} (1+t^{-m}w_i) &= \Pi_{0,t}(\gamma(t^{L_k+\zeta}),\alpha(t^{-m});\beta(w_1,\ldots,w_{k-1})) \\ 
&= \sum_{\la \in \Y} Q_\la(\gamma(t^{L_k+\zeta}),\alpha(t^{-m});0,t) Q_{\la'}(w_1,\ldots,w_{k-1};t,0),
\end{split}
\end{align}
by \eqref{eq:specialized_cauchy}. Note also that since $w_i \in \T$, by \Cref{thm:signature_shift} 
\begin{equation}
(w_1^{m} \cdots w_{k-1}^{m}) P_\mu(\bar{w}_1,\ldots,\bar{w}_{k-1};t,0) = P_{\mu-(m[k-1])}(\bar{w}_1,\ldots,\bar{w}_{k-1};t,0).
\end{equation}
If $\mu_{k-1} - m < 0$, then the above is not a polynomial but a Laurent polynomial, and by orthogonality of the $q$-Whittaker Laurent polynomials (\Cref{thm:laurent_orthogonality}) the integral in \eqref{eq:to_use_k-1_product} is $0$. Otherwise, if $\mu_{k-1} - m \geq 0$, orthogonality still implies that only the term $\la' = \mu - (m[k-1])$ of the sum \eqref{eq:use_spec_cauchy} contributes to the integral in \eqref{eq:to_use_k-1_product}. Hence we obtain
\begin{multline}
\label{eq:return_scalar_product}
\frac{t^{-(k-1)\binom{m}{2}} }{(k-1)!(2\pi \bi)^{k-1}} \int_{\T^{k-1}} \Pi_{0,t}(\gamma(t^{L_k+\zeta}),\alpha(t^{-m});\beta(w_1,\ldots,w_{k-1})) P_{\mu-(m[k-1])}(\bar{w}_1,\ldots,\bar{w}_{k-1};t,0) \\ 
\times  \prod_{1 \leq i \neq j \leq k-1} (w_i/w_j;t)_\infty \prod_{i=1}^{k-1} \frac{dw_i}{w_i} = t^{-(k-1)\binom{m}{2}} Q_{(\mu-(m[k-1]))'}(\gamma(t^{L_k+\zeta}),\alpha(t^{-m});0,t) \\ 
\times  \bbone(\mu_{k-1}-m \geq 0) \lan Q_{\mu-(m[k-1])}(w_1,\ldots,w_{k-1};t,0),P_{\mu-(m[k-1])}(w_1,\ldots,w_{k-1};t,0)\ran'_{t,0;k-1}.
\end{multline}
By \eqref{eq:QP_IP},
\begin{multline}\label{eq:apply_QPIP}
\lan Q_{\mu-(m[k-1])}(w_1,\ldots,w_{k-1};t,0),P_{\mu-(m[k-1])}(w_1,\ldots,w_{k-1};t,0)\ran'_{t,0;k-1} 
= \frac{1}{(t;t)_{\mu_{k-1}-m} (t;t)_\infty^{k-2}}.
\end{multline}
Combining \eqref{eq:to_use_k-1_product}, \eqref{eq:return_scalar_product} and \eqref{eq:apply_QPIP} and writing
\begin{equation}
\bbone(\mu_{k-1}-m \geq 0)\frac{(t^{-m+1};t)_{\mu_{k-1}}}{(t;t)_{\mu_{k-1}-m}} = \bbone(m \leq 0) \frac{1}{(t;t)_{-m}}
\end{equation}
yields
\begin{align}
\begin{split}
\label{eq:single_residue_2}
&\frac{1}{(2 \pi \bi)^{k-1}}\int_{\T^{k-1}} \Res_{w_k=-t^m} \frac{k f(w_1,\ldots,w_k)}{\prod_{i=1}^k (-tw_i;t)_\infty} \prod_{i=1}^{k-1} dw_i = \frac{\bbone(m \leq 0) }{(t;t)_\infty} \left(\prod_{i=1}^{k-1} \frac{t^{\binom{L_i-L_k}{2}}}{(t;t)_{L_i-L_{i+1}}}\right)  \\ 
&\times \frac{(-1)^m t^{-(k-2)\binom{m}{2}+m}}{(t;t)_{-m}}e^{-\frac{t^{L_k+\zeta+m}}{1-t}} \sum_{\substack{\mu \in \Sig_{k-1} \\ \mu \prec \tLL}} (-t^{-m})^{|\tLL|-|\mu|} \prod_{i=1}^{k-1} \sqbinom{L_i-L_{i+1}}{L_i-L_k-\mu_i}_t \\ 
&\times Q_{(\mu-(m[k-1]))'}(\gamma(t^{L_k+\zeta}),\alpha(t^{-m}); 0,t).
\end{split}
\end{align}
Set $\vec{L} = (L_1,\ldots,L_k) = \tLL + (L_k[k])$ and relabel $\mu \mapsto \mu + (L_k[k-1])$ so that 
\begin{multline}\label{eq:scale_Q}
\sum_{\substack{\mu \in \Sig_{k-1} \\ \mu \prec \tLL}} (-t^{-m})^{|\tLL|-|\mu|} \prod_{i=1}^{k-1} \sqbinom{L_i-L_{i+1}}{L_i-L_k-\mu_i}_t Q_{(\mu-(m[k-1]))'}(\gamma(t^{L_k+\zeta}),\alpha(t^{-m}); 0,t) = \\ 
\sum_{\substack{\mu \in \Sig_{k-1} \\ \mu \prec \vec{L}}} (-t^{-m})^{|\vec{L}| - |\mu| - L_k} \prod_{i=1}^{k-1} \sqbinom{L_i-L_{i+1}}{L_i-\mu_i}_t t^{-m(|\mu|-(k-1)(L_k+m))}Q_{\mu-((L_k+m)[k-1])'}(\gamma(t^{L_k+\zeta+m}),\alpha(1);0,t),
\end{multline}
where we have also used homogeneity of $Q$ to scale the specializations. Setting $d=L_k+m$, the power of $t$ in \eqref{eq:single_residue_2} (together with the one in \eqref{eq:scale_Q}) is
\begin{multline}
-(k-2)\binom{m}{2}+m - m(|\vec{L}|-|\mu|-L_k) -m(|\mu|-(k-1)(L_k+m)) = \sum_{i=1}^k \binom{L_i-d}{2}
\end{multline}
by a tedious computation using \eqref{eq:favorite_binomial_split}. The sign in \eqref{eq:single_residue_2} is $(-1)^{|\vec{L}| - |\mu| - d}$. Putting this together we have 
\begin{align}
\begin{split}
\label{eq:single_residue_final}
&\frac{1}{(2 \pi \bi)^{k-1}}\int_{\T^{k-1}} \Res_{w_k=-t^m} \frac{k f(w_1,\ldots,w_k)}{\prod_{i=1}^k (-tw_i;t)_\infty} \prod_{i=1}^{k-1} dw_i = e^{-\frac{t^{d+\zeta}}{1-t}} \frac{t^{\sum_{i=1}^k \binom{L_i-d}{2}}}{(t;t)_\infty (t;t)_{L_k-d} \prod_{i=1}^{k-1}(t;t)_{L_i-L_{i+1}}} \\ 
&\times \sum_{\substack{\mu \in \Sig_{k-1} \\ \mu \prec \vec{L}}} (-1)^{|\vec{L}| - |\mu|-d} \prod_{i=1}^{k-1} \sqbinom{L_i-L_{i+1}}{L_i-\mu_i}_t Q_{(\mu - (d[k-1]))'}(\gamma(t^{d+\zeta}),\alpha(1);0,t).
\end{split}
\end{align}
This shows \eqref{eq:residue_wts} and hence completes the proof.
\end{proof}

\begin{proof}[Proof of \Cref{thm:small_contour_negligible}]
%idea: we know that for different n the integrals are equal, so it suffices to look at the limit as n goes to infinity
By \Cref{thm:residue_expansion_independence}, the integral which we wish to compute is independent of $n \geq L_k$. Hence it suffices to show 
\begin{multline}\label{eq:zero_integral_limit}
\lim_{\substack{n \to \infty \\ n \in \Z}}   \frac{(t;t)_\infty^{k-1}}{k! (2 \pi \bi)^k} \prod_{i=1}^{k-1} \frac{t^{\binom{L_i-L_k}{2}}}{(t;t)_{L_i-L_{i+1}}} \int_{(t^{n+1/2}\T)^k} e^{\frac{t^{L_k+\zeta}}{1-t}(w_1+\ldots+w_k)} \frac{\prod_{1 \leq i \neq j \leq k} (w_i/w_j;t)_\infty}{\prod_{i=1}^k (-w_i^{-1};t)_\infty (-tw_i;t)_{\infty}} \\ 
\times \sum_{j=0}^{L_{k-1}-L_k} t^{\binom{j+1}{2}} \sqbinom{L_{k-1}-L_k}{j}_t  P_{(L_1-L_k,\ldots,L_{k-1}-L_k,j)}(w_1^{-1},\ldots,w_k^{-1};t,0)  \prod_{i=1}^k \frac{dw_i}{w_i} = 0.
\end{multline}
To simplify expressions we will show \eqref{eq:zero_integral_limit} by showing
\begin{multline}\label{eq:j_summand_zero}
\lim_{\substack{n \to \infty \\ n \in \Z}}  \int_{(t^{n+1/2}\T)^k} e^{\frac{t^{L_k+\zeta}}{1-t}(w_1+\ldots+w_k)} \frac{\prod_{1 \leq i \neq j \leq k} (w_i/w_j;t)_\infty}{\prod_{i=1}^k (-w_i^{-1};t)_\infty (-tw_i;t)_{\infty}} \\ 
\times  P_{(L_1-L_k,\ldots,L_{k-1}-L_k,j)}(w_1^{-1},\ldots,w_k^{-1};t,0)  \prod_{i=1}^k \frac{dw_i}{w_i} = 0
\end{multline}
for each $j$. Letting $w_i = t^n u_i$ we have 
\begin{multline}\label{eq:zero_int_var_change}
\text{LHS\eqref{eq:j_summand_zero}} = \lim_{\substack{n \to \infty \\ n \in \Z}} \int_{(t^{1/2}\T)^k} e^{\frac{t^{L_k+\zeta+n}}{1-t}(u_1+\ldots+u_k)} \frac{\prod_{1 \leq i \neq j \leq k} (u_i/u_j;t)_\infty}{\prod_{i=1}^k (-t^{-n}u_i^{-1};t)_\infty (-t^{n+1}u_i;t)_{\infty}} \\ 
\times t^{-n(|\vec{L}|+j-kL_k)} P_{(L_1-L_k,\ldots,L_{k-1}-L_k,j)}(u_1^{-1},\ldots,u_k^{-1};t,0)  \prod_{i=1}^k \frac{du_i}{u_i} 
\end{multline}
where we have used homogeneity of $P$. By \Cref{thm:theta_transformations},
\begin{equation}
\label{eq:another_theta_manipulation}
(-t^{-n}u_i^{-1};t)_\infty (-t^{n+1}u_i;t)_{\infty} = u_i^{-n} t^{-\binom{n+1}{2}} (-u_i^{-1};t)_\infty (-t u_i; t)_\infty,
\end{equation}
so
\begin{align}\label{eq:zero_int_simplify_1}
&\text{RHS\eqref{eq:zero_int_var_change}} =  \lim_{\substack{n \to \infty \\ n \in \Z}}  \int_{(t^{1/2}\T)^k} \prod_{i=1}^k \left(u_i^{n} t^{\binom{n+1}{2} -n(|\vec{L}|+j-kL_k)} \right)  e^{\frac{t^{L_k+\zeta+n}}{1-t}(u_1+\ldots+u_k)} \\ 
&\times \frac{\prod_{1 \leq i \neq j \leq k} (u_i/u_j;t)_\infty}{\prod_{i=1}^k (-u_i^{-1};t)_\infty (-tu_i;t)_{\infty}} P_{(L_1-L_k,\ldots,L_{k-1}-L_k,j)}(u_1^{-1},\ldots,u_k^{-1};t,0)  \prod_{i=1}^k \frac{du_i}{u_i}. \label{eq:zero_int_simplify_2}
\end{align}
The part of the integrand on the second line \eqref{eq:zero_int_simplify_2} is bounded on $\T^k$ and is independent of $n$, while the part of the integrand on the first line \eqref{eq:zero_int_simplify_1} goes to $0$ in $n$ uniformly over $\T^k$, showing \eqref{eq:j_summand_zero} and hence completing the proof.
\end{proof}

\begin{proof}[Proof of \Cref{thm:hl_residue_formula}]
%idea of proof: take the contour in first lemma to be Gamma(B) and take B to infinity, and then the estimates I showed earlier yield the contour integral while the first lemma (plus second for error term) yields the sum.
The idea is to take a family of simple closed contours which approach the contour $\tG$ defined in \Cref{thm:hl_stat_dist} and also encircle more and more of the poles $-t^x$. The contours $\Gamma(\tau)$ defined in \eqref{eq:contours} work, though we write them as $\Gamma(B), B \in \Z$ to avoid a clash of notation with the $\tau$ in \Cref{thm:hl_residue_formula}. By \Cref{thm:residue_expansion_independence} together with \Cref{thm:small_contour_negligible},
\begin{multline}
\frac{1}{(2 \pi \bi)^k} \int_{\Gamma(B)^k} \frac{f(w_1,\ldots,w_k)}{\prod_{i=1}^k (-tw_i;t)_\infty} \prod_{i=1}^k dw_i = \sum_{d=L_k-\eta_k(B)}^{L_k} e^{-\frac{t^{d+\zeta}}{1-t}} \frac{t^{\sum_{i=1}^k \binom{L_i-d}{2}}}{(t;t)_\infty (t;t)_{L_k-d} \prod_{i=1}^{k-1}(t;t)_{L_i-L_{i+1}}}  \\ 
\times \sum_{\substack{\mu \in \Sig_{k-1} \\ \mu \prec \vec{L}}} (-1)^{|\vec{L}| - |\mu|-d} \prod_{i=1}^{k-1} \sqbinom{L_i-L_{i+1}}{L_i-\mu_i}_t Q_{(\mu - (d[k-1]))'}(\gamma(t^{d+\zeta}),\alpha(1);0,t) + 0.
\end{multline}
It follows immediately that 
\begin{equation}\label{eq:contour_grow_res_limit}
\lim_{\substack{B \to \infty \\ B \in \Z}} \frac{1}{(2 \pi \bi)^k} \int_{\Gamma(B)^k} \frac{f(w_1,\ldots,w_k)}{\prod_{i=1}^k (-tw_i;t)_\infty} \prod_{i=1}^k dw_i = \text{RHS\eqref{eq:brpw_residue_formula}}.
\end{equation}
The limit
\begin{multline}\label{eq:gammaB_limit}
\lim_{\substack{B \to \infty \\ B \in \Z}} \frac{1}{(2 \pi \bi)^k} \int_{\Gamma(B)^k} \frac{f(w_1,\ldots,w_k)}{\prod_{i=1}^k (-tw_i;t)_\infty} \prod_{i=1}^k dw_i \\ 
= \frac{(t;t)_\infty^{k-1}}{k! (2 \pi \bi)^k} \prod_{i=1}^{k-1} \frac{t^{\binom{L_i-L_k}{2}}}{(t;t)_{L_i-L_{i+1}}} \int_{\tG^k} e^{\frac{t^{L_k+\zeta}}{1-t}(w_1+\ldots+w_k)}  \frac{\prod_{1 \leq i \neq j \leq k} (w_i/w_j;t)_\infty}{\prod_{i=1}^k (-w_i^{-1};t)_\infty (-tw_i;t)_{\infty}} \\ 
\times \sum_{j=0}^{L_{k-1}-L_k} t^{\binom{j+1}{2}} \sqbinom{L_{k-1}-L_k}{j}_t  P_{(L_1-L_k,\ldots,L_{k-1}-L_k,j)}(w_1^{-1},\ldots,w_k^{-1};t,0)  \prod_{i=1}^k \frac{dw_i}{w_i}
\end{multline}
follows by the estimate \Cref{thm:f_bound} exactly as with \eqref{eq:split_2} in the proof of \Cref{thm:hl_stat_dist}. Combining \eqref{eq:contour_grow_res_limit} with \eqref{eq:gammaB_limit} completes the proof.
\end{proof}

\section{Tightness and the limiting random variable}\label{sec:tightness_and_limit}

In the Introduction, we stated that the limiting formulas on the right hand side of \eqref{eq:explicit_hl_stat_dist_formula} and \eqref{eq:brpw_residue_formula} define a $\Sig_k$-valued random variable, but \Cref{thm:hl_stat_dist} and \Cref{thm:hl_residue_formula} do not \emph{a priori} imply this because mass may escape to $\pm \infty$. In this section we show that there is no escape of mass and the formulas indeed define a random variable. 

\begin{thm}\label{thm:stat_dist_1pt}
For any $t \in (0,1),\chi \in \mathbb{R}_{>0}$, there is a unique $\operatorname{Sig}_{\infty}$-valued random variable $\cL_{t, \chi} = (\cL_{t,\chi}^{(1)},\cL_{t,\chi}^{(2)},\ldots)$, measurable with respect to the product $\sigma$-algebra on $\Z^\infty$ where each factor has the discrete $\sigma$-algebra, such that the laws of its truncations $\cL_{k,t,\chi} := (\cL_{t,\chi}^{(1)},\ldots,\cL_{t,\chi}^{(k)})$ satisfy
\begin{multline}\label{eq:limit_rv_res_formula}
\Pr(\cL_{k,t,\chi} = (L_1,\ldots,L_k)) = \frac{1}{(t;t)_\infty}\sum_{d \leq L_k} e^{-\chi t^d} \frac{t^{\sum_{i=1}^k \binom{L_i-d}{2}}}{ (t;t)_{L_k-d} \prod_{i=1}^{k-1}(t;t)_{L_i-L_{i+1}}}  \\ 
\times \sum_{\substack{\mu \in \Sig_{k-1} \\ \mu \prec \vec{L}}} (-1)^{|\vec{L}| - |\mu|-d}  \prod_{i=1}^{k-1} \sqbinom{L_i-L_{i+1}}{L_i-\mu_i}_t Q_{(\mu - (d[k-1]))'}(\gamma((1-t)t^d\chi),\alpha(1);0,t)
\end{multline}
for any $\left(L_{1}, \ldots, L_{k}\right) \in \operatorname{Sig}_{k}$. The probabilities also have a contour integral formula
\begin{multline}
\label{eq:limit_rv_int_formula}
\Pr(\cL_{k,t,\chi} = (L_1,\ldots,L_k)) \\ 
=  \frac{(t;t)_\infty^{k-1}}{k! (2 \pi \bi)^k} \prod_{i=1}^{k-1} \frac{t^{\binom{L_i-L_k}{2}}}{(t;t)_{L_i-L_{i+1}}} \int_{\tG^k} e^{\chi t^{L_k}(w_1+\ldots+w_k)} \frac{\prod_{1 \leq i \neq j \leq k} (w_i/w_j;t)_\infty}{\prod_{i=1}^k (-w_i^{-1};t)_\infty (-tw_i;t)_{\infty}} \\ 
\times \sum_{j=0}^{L_{k-1}-L_k} t^{\binom{j+1}{2}} \sqbinom{L_{k-1}-L_k}{j}_t  P_{(L_1-L_k,\ldots,L_{k-1}-L_k,j)}(w_1^{-1},\ldots,w_k^{-1};t,0)  \prod_{i=1}^k \frac{dw_i}{w_i}
\end{multline}
for $k \geq 2$ and 
\begin{equation}
\label{eq:limit_rv_int_formula_k=1}
\Pr(\cL_{t,\chi}^{(1)} = (L)) = \frac{1}{2 \pi \bi} \int_{\tG} \frac{e^{\chi t^L w}}{(-w;t)_\infty} dw,
\end{equation}
with contour $\tG$ as defined in \Cref{fig:tG}. 
\end{thm}

We will often write simply $\cL^{(i)}$ for $\cL_{t,\chi}^{(i)}$ when $t$ and $\chi$ are clear from context. For the proof of \Cref{thm:stat_dist_1pt} we require the following.

\begin{prop}\label{thm:tightness}
In the notation of \Cref{thm:hl_stat_dist}, the sequence of $\Sig_k$-valued random variables
\begin{equation}\label{eq:wts_tight}
(\la_i'(\tau) - \log_{t^{-1}}(\tau) - \zeta)_{1 \leq i \leq k}, \tau \in t^{-\N - \zeta}
\end{equation}
is tight.
\end{prop}
\begin{proof}
We must show that for every $\eps > 0$, there exists $D = D(\eps)$ such that 
\begin{equation}\label{eq:tight_explicit}
\Pr(-D \leq \la_k'(\tau) - \log_{t^{-1}}(\tau) - \zeta \leq \la_1'(\tau) - \log_{t^{-1}}(\tau) - \zeta \leq D) > 1-\eps
\end{equation}
for all $\tau \in t^{-\N - \zeta}$. For the upper bound in \eqref{eq:tight_explicit}, first note that $\la_1'$ is Markov, and if $\la_1'(\tau) = x$ at some $\tau$ then the waiting time before $\la_1'$ jumps to $x+1$ follows an exponential distribution with rate $t^{x}/(1-t)$. This is because $\la_1'$ will increase as soon as one of the clocks $x+1,x+2,\ldots$ rings, and these have rates $t^x, t^{x+1},\ldots$. Hence for $D \in \N$ we have
\begin{equation}
\label{eq:lambda1'_upper}
\Pr(\la_1'(\tau) > D) = \Pr(\sum_{i=0}^D E_i \leq \tau)
\end{equation}
where $E_i \sim \Exp(t^{i}/(1-t))$. Clearly
\begin{equation}
\Pr(\sum_{i=0}^D E_i \leq \tau) \leq \Pr(E_D \leq \tau) = 1-e^{-\frac{t^{D}}{1-t}\tau }.
\end{equation}
Hence
\begin{equation}\label{eq:explicit_lambda1'_bound}
\Pr(\la_1'(\tau) - \log_{t^{-1}}(\tau) - \zeta \leq D) \geq e^{- \frac{t^{\log_{t^{-1}}(\tau)+\zeta+D}}{1-t}\tau} = e^{-\frac{t^{\zeta+D}}{1-t}}.
\end{equation}

Now for the lower bound of \eqref{eq:tight_explicit}. Suppose that $\la_k'(\tau) = x$ at some time $\tau$, and consider the waiting time until $\la_k'$ jumps to $x+1$. If clock $x+1$ rings $k$ times then $\la_k'$ will jump, even in the unfavorable case $\la_k'(\tau) = \ldots = \la_1'(\tau) = x$; note that there are many other ways the clocks can ring to cause $\la_k'$ to jump, we are just choosing this one for the lower bound. So the waiting time until $\la_k'$ jumps to $x+1$ is upper-bounded by a sum of $k$ independent $\Exp(t^x)$ random variables. Denoting this sum of $k$ random variables by $E_{k,x}$, it follows as before that
\begin{equation}
\label{eq:lambdak'_lower}
\Pr(\la_k'(\tau) \leq H) \leq \Pr(\sum_{i=0}^{H-1} E_{k,i} > \tau)
\end{equation}
where in contrast to \eqref{eq:lambda1'_upper} we have an inequality rather than an equality because we have only bounded the waiting time rather than giving it exactly. By Markov's inequality, 
\begin{equation}
\Pr(\sum_{i=0}^{H-1} E_{k,i} > \tau) \leq \frac{\E[\sum_{i=0}^{H-1} E_{k,i}]}{\tau} = \frac{kt^{-H+1}}{\tau} \frac{1-t^H}{1-t} \leq \frac{kt^{-H+1}}{(1-t)\tau}.
\end{equation}
Hence 
\begin{equation}\label{eq:explicit_lambdak'_bound}
\Pr(\la_k'(\tau) - \log_{t^{-1}}(\tau) - \zeta \geq -D) \geq 1 - \frac{kt^{D+2- \log_{t^{-1}}(\tau) - \zeta}}{(1-t)\tau} = 1 - t^D \frac{kt^{2-\zeta}}{1-t}.
\end{equation}
Since the bounds \eqref{eq:explicit_lambda1'_bound} and \eqref{eq:explicit_lambdak'_bound} are both independent of $\tau$ and both go to $1$ as $D \to \infty$, together they show \eqref{eq:tight_explicit}.
\end{proof}

\begin{proof}[Proof of {\Cref{thm:stat_dist_1pt}}]
By the Kolmogorov extension theorem, any family of distributions on $\Sig_k$ for each $k \geq 1$, consistent under the maps $(L_1,\ldots,L_k) \mapsto (L_1,\ldots,L_{k-1})$, defines a distribution on $\Sig_\infty$ (with respect to the product $\sigma$-algebra). Hence to show \Cref{thm:stat_dist_1pt}, it suffices to show that (a) the formulas given for $\Pr(\cL_{k,t,\chi} = \vec{L})$ define valid random variables, and (b) these form a consistent family.

For any $\chi \in \R_{>0}$, taking $\zeta = \log_t((1-t)\chi)$ in \Cref{thm:hl_stat_dist} and combining \Cref{thm:tightness} with Prokhorov's theorem shows that the formula \eqref{eq:limit_rv_int_formula} defines a valid random variable. By \Cref{thm:hl_stat_dist} and \Cref{thm:hl_residue_formula}, the right hand sides of \eqref{eq:limit_rv_res_formula} and \eqref{eq:limit_rv_int_formula} are equal. Consistency follows from consistency of the prelimit random variables, completing the proof.
\end{proof}

% From now on we state limit results in terms of the random variables $\cL_{t,\chi}$ or $\cL_{k,t,\chi}$ rather than limit probabilities, which substantially declutters notation. When stating in terms of $\cL_{t,\chi}$ we will also use the following basic convergence lemma.

% \begin{lemma}\label{thm:check_marginals}
% Let $\vec{X}^{(N)} = (X_1^{(N)},X_2^{(N)},\ldots)$ be any family of $\Sig_\infty$-valued random variables indexed by $N \in \Z_{\geq 1}$ for which 
% \begin{equation}
% (X_1^{(N)},\ldots,X_k^{(N)}) \to \cL_{k,t,\chi}
% \end{equation}
% in distribution as $N \to \infty$, for any $k$. Then $\vec{X}^{(N)} \to \cL_{t,\chi}$ in distribution, with respect to the product topology on $\Z^\infty$ where each factor has the discrete topology.
% \end{lemma}
% \begin{proof}
% Convergence in distribution in the product discrete topology on $\Z^\infty$ exactly means that finite collections of coordinates converge in the discrete topology on $\Z^k$.
% \end{proof}

%As in the Introduction, we will typically abuse notation and identify $\Sig_1$ with $\Z$ to view $\cL_{1,t,\chi}$ as a $\Z$-valued random variable. We prefer this to denoting the random variable by $\cL^{(1)}$, as it keeps the parameters $t,\chi$ explicit in the notation.

\section{An indeterminate moment problem} \label{sec:indeterminate}

In this section we prove the result mentioned in the Introduction, \Cref{thm:k=1_t_moments}, that the random variables $\chi^{-1}t^{-\cL_{t,\chi}^{(1)}}$ solve an indeterminate Stieltjes moment problem, and then discuss the case of $k > 1$. This material is not necessary for reading the rest of the paper.

\subsection{The $k=1$ case.} 

\begin{prop}\label{thm:k=1_t_moments}
In the notation of \Cref{thm:stat_dist_1pt}, for any $t \in (0,1)$ and $\chi \in \R_{>0}$ the $\chi^{-1}t^\Z$-valued random variable $\chi^{-1}t^{-\cL_{t,\chi}^{(1)}}$ has moments
\begin{equation}
\E\left[\left(\chi^{-1}t^{-\cL_{t,\chi}^{(1)}}\right)^m\right] = \frac{t^{-\binom{m+1}{2}}(t;t)_m}{m!}
\end{equation}
for each $m \in \Z_{\geq 0}$. 
\end{prop}

\begin{rmk}\label{rmk:k=1_two_rv_proofs}
Note that the $m=0$ case of \Cref{thm:k=1_t_moments} provides a proof the explicit probabilities for $\cL_{t,\chi}^{(1)}$ sum to $1$, which is independent of the previous section. Of course, one would still have to show that they are nonnegative to obtain a completely independent proof that $\cL_{t,\chi}^{(1)}$ is a valid random variable, but luckily the tightness arguments of the previous section have already established this.
\end{rmk}

We will need the following lemma, proven at the end of this section, to apply dominated convergence in the proof of the above.

\begin{lemma}\label{thm:can_do_dct_on_tg}
For any $m \in \Z_{\geq 0}$, there exists a constant $C$ such that for all $z \in \tG$,
\begin{equation}
\sum_{L \in \Z} \abs*{\frac{t^{-(m+1)L}}{(-t^{-L}z;t)_\infty}} < C.
\end{equation}
\end{lemma}

\begin{proof}[Proof of \Cref{thm:k=1_t_moments}]
We begin by changing variables to $z=t^L w$ rewrite \eqref{eq:limit_rv_int_formula_k=1} as
\begin{equation}\label{eq:change_to_z}
\Pr(\cL_{t,\chi}^{(1)}=L) = \frac{1}{2 \pi \bi} \int_{t^L\tG} \frac{e^{\chi z}}{(-t^{-L}z;t)_\infty} t^{-L} dz. 
\end{equation}
We first argue that the contour in \eqref{eq:change_to_z} may be shifted to $\tG$. By the residue theorem, for any $n \in \Z_{>0}$
\begin{align}\label{eq:four_integrals}
\begin{split}
&\frac{1}{2 \pi \bi} \left( \int_{t^L\tG} -  \int_{\tG} \right)  \frac{t^{-L}e^{\chi z}}{(-t^{-L}z;t)_\infty}  dz = \frac{1}{2 \pi \bi} \left(\int_{t^L\tG \cap \{z: \Re(z) \leq -t^{-n+1/2}\}} - \int_{\tG \cap \{z: \Re(z) \leq -t^{-n+1/2}\}}\right. \\ 
&\left. - \int_{z=-t^{-n+1/2}+t^L \bi}^{-t^{-n+1/2}+\bi} + \int_{z=-t^{-n+1/2}-t^L \bi}^{-t^{-n+1/2}-\bi}\right) \frac{t^{-L}e^{\chi z}}{(-t^{-L}z;t)_\infty}  dz 
\end{split}
\end{align}
Because the integrals over $\tG$ and $t^L \tG$ are finite (e.g. by \Cref{thm:stat_dist_1pt}), the integrals over $\tG \cap \{z: \Re(z) \leq -t^{-n+1/2}\}$ and $t^L\tG \cap \{z: \Re(z) \leq -t^{-n+1/2}\}$ go to $0$ as $n \to \infty$.

The two vertical contours in \eqref{eq:four_integrals} have lengths $|1-t^L|$, the factor $|1/(-t^{-L}z;t)_\infty|$ in the integrand is uniformly bounded above along them, and the exponential factor becomes $e^{-\chi t^{-n+1/2}}$, hence these integrals also go to $0$ as $n \to \infty$. Hence \eqref{eq:change_to_z} and the above yield 
\begin{equation}\label{eq:shift_to_tG}
\Pr(\cL_{t,\chi}^{(1)}=L) =  \frac{1}{2 \pi \bi} \int_{\tG} \frac{t^{-L}e^{\chi z}}{(-t^{-L}z;t)_\infty}  dz. 
\end{equation}

By \Cref{thm:can_do_dct_on_tg} and the fact that $|e^{\chi z}|$ is integrable on $\tG$, dominated convergence applies and
\begin{align}\label{eq:expectation_integral}
\begin{split}
\E\left[\left(\chi^{-1}t^{-\cL_{t,\chi}^{(1)}}\right)^m\right] &= \chi^{-m}\sum_{L \in \Z} t^{-mL}\Pr(\cL_{t,\chi}^{(1)}=L) \\ 
&= \chi^{-m}\lim_{D \to \infty} \int_{\tG} \sum_{L =-D}^D  \frac{t^{-L-mL}}{(-t^{-L}z;t)_\infty}dz \\ 
&=  \chi^{-m}\frac{1}{2 \pi \bi} \int_{\tG} e^{\chi z} \sum_{L \in \Z} \frac{t^{-L-mL}}{(-t^{-L}z;t)_\infty}dz.
\end{split}
\end{align}
By Ramanujan's $\;_1\psi_1$ summation identity (see e.g. \cite[(5.2.1)]{gasper_rahman_2004}),
\begin{align}
\begin{split}\label{eq:1psi1}
\sum_{L \in \Z} \frac{t^{-L-mL}}{(-t^{-L}z;t)_\infty} &= \frac{1}{(-z;t)_\infty} \;_1\psi_1(-z;0;t,t^{m+1}) \\ 
&= \frac{1}{(-z;t)_\infty} \frac{(t;t)_\infty (0;t)_\infty (-zt^{m+1};t)_\infty (-t^{-m}z^{-1};t)_\infty}{(0;t)_\infty (-tz^{-1};t)_\infty (t^{m+1};t)_\infty (0;t)_\infty} \\ 
&= (t;t)_m z^{-m-1} t^{-\binom{m+1}{2}}.
\end{split}
\end{align}
By closing off the contour $\tG$ to a union of the unit circle and a shifted copy of $\tG$, we have
\begin{align}
\begin{split}\label{eq:bubble_contour}
&\chi^{-m}\frac{1}{2 \pi \bi} \int_{\tG} e^{\chi z} (t;t)_m z^{-m-1} t^{-\binom{m+1}{2}}dz \\ 
&= \chi^{-m}\frac{1}{2 \pi \bi} \int_{\T} e^{\chi z} (t;t)_m z^{-m-1} t^{-\binom{m+1}{2}}dz + \chi^{-m}\frac{1}{2 \pi \bi} \int_{\tG-c} e^{\chi z} (t;t)_m z^{-m-1} t^{-\binom{m+1}{2}}dz \\ 
&= \chi^{-m}\frac{1}{2 \pi \bi} \int_{\T} e^{\chi z} (t;t)_m z^{-m-1} t^{-\binom{m+1}{2}}dz,
\end{split}
\end{align}
where the last equality follows because the second term on the right hand side clearly goes to $0$ as $c \to \infty$ by exponential decay of the integrand as $\Re(z) \to -\infty$, so since it is independent of shifting the contour, it must have been $0$ to start with. Combining \eqref{eq:expectation_integral}, \eqref{eq:1psi1}, and \eqref{eq:bubble_contour}, we have
\begin{equation}
\E\left[\left(\chi^{-1}t^{-\cL_{t,\chi}^{(1)}}\right)^m\right] = \chi^{-m}\frac{1}{2 \pi \bi} \int_{\T} e^{\chi z} (t;t)_m z^{-m-1} t^{-\binom{m+1}{2}}dz = \frac{(t;t)_m t^{-\binom{m+1}{2}}}{m!}
\end{equation}
by the residue formula, completing the proof.
\end{proof}

\subsection{General $k$.} \label{sec:indet_gen_k}For any $\kappa = (\kappa_1,\ldots,\kappa_k) \in \Sig_k$ with $\kappa_k \geq 0$, a formula for limiting joint moments
\begin{equation}\label{eq:limit_in_qtasep}
\lim_{\tau \to \infty} \E[(t^{-(\la_1'(\tau) - \log_{t^{-1}}\tau)})^{\kappa_1} \cdots (t^{-(\la_k'(\tau) - \log_{t^{-1}}\tau)})^{\kappa_k}]
\end{equation}
in terms of a certain multiple contour integral\footnote{The notation differs slightly: the $M$ in \cite[Proposition 5.1]{van2022q} corresponds to our $\kappa_1$, and $r_1,\ldots,r_M$ correspond to $\kappa_1',\ldots,\kappa_{\kappa_1}'$. The notation $\la(\tau)$ is, however, the same in both \cite{van2022q} and the present work. Note also that there is a small typo on the left hand side of \cite[(5.1)]{van2022q}, $j$ should be $i$ (or vice versa).}  was derived in \cite[Proposition 5.1]{van2022q}, and found to be independent of the subsequence along which $\tau$ goes to $\infty$. On this basis, we conjectured directly after the proof of \cite[Proposition 5.1]{van2022q} that the random variables $(\la_i'(\tau) - \log_{t^{-1}}\tau)$ did indeed converge in (joint) distribution to random variables on shifts of the integer lattice, depending on the subsequence along which $\tau \to \infty$. The results \Cref{thm:hl_stat_dist} and \Cref{thm:stat_dist_1pt} show that this is the case, namely that $(\la_i'(\tau) - \log_{t^{-1}}\tau)_{1 \leq i \leq k}$ converges along $\tau \in t^{-\N+\zeta}$ to the $(\Z+\zeta)^k$-valued random variable $\cL_{k,t,t^{\zeta}/(1-t)} + (\zeta[k])$. It is very tempting to conclude that the mixed moments of $\cL_{t,\chi} = (\cL^{(i)})_{i \geq 1}$ 
\begin{equation}
\E\left[(\chi^{-1}t^{-\cL^{(1)}})^{\kappa_1} \cdots (\chi^{-1}t^{-\cL^{(k)}})^{\kappa_k}\right]
\end{equation}
exist and are given by \eqref{eq:limit_in_qtasep} (after suitable overall normalization). Indeed, evaluating the contour integral formula \cite[(5.1)]{van2022q} for the limiting moments \eqref{eq:limit_in_qtasep} gives the correct formula from \Cref{thm:k=1_t_moments} for the moments of $t^{-\cL_{t,\chi}^{(1)}}$ in the $k=1$ case. However, proving that the $t$-moments of the limit random variable $\cL_{k,t,\chi}$ are the same as the limiting $t$-moments of the prelimit random variables $(\la_i'(\tau) - \log_{t^{-1}}\tau)_{1 \leq i \leq k}$ requires tighter control on the tails of the prelimit random variables than we establish in \Cref{sec:tightness_and_limit}. We suspect that generalizing the method of the previous subsection to all $k$ may provide a better route to the mixed $t$-moments of $\cL_{k,t,\chi}$, but the formula \cite[(5.1)]{van2022q} at least provides conjectural answers. We remark that the contour integrals \cite[(5.1)]{van2022q} are not particularly transparent to us, and we found the formula \Cref{thm:k=1_t_moments} by evaluating them on Mathematica in small cases. It is not hard to evaluate the mixed moments in this way for general $k$ to obtain rational functions in $t$, though we have not tried to conjecture a general form for them beyond $k=1$.

\subsection{Proofs of bounds on $q$-Pochhammer symbols.} \label{subsec:bounds} In this subsection we collect proofs of all bounds on $q$-Pochhammer symbols which we used earlier, namely \Cref{thm:qpoc_lower_bound}, \Cref{thm:qp_bound_tG}, and \Cref{thm:can_do_dct_on_tg}, of which the last is the most delicate. The computations are similar in all cases, and we begin with \Cref{thm:qpoc_lower_bound}, which illustrates the method, then show \Cref{thm:can_do_dct_on_tg} and \Cref{thm:qp_bound_tG}.

\begin{proof}[Proof of \Cref{thm:qpoc_lower_bound}]
We first show the case where $n$ is finite and $t^{1/2} < \Re(z) < t^{-n+1/2}$. Let $\ell_0$ be the unique index for which $t^{\ell_0+1/2} < \Re(z) \leq t^{\ell_0-1/2}$. Then we have the bound
\begin{equation}\label{eq:cases_qpoc_1}
|1-t^\ell z| \geq \begin{cases}
t^{\ell-\ell_0+1/2} - 1 & 0 \leq \ell < \ell_0 \\ 
\delta & \ell = \ell_0 \\ 
1-t^{\ell-\ell_0-1/2} & \ell > \ell_0
\end{cases}
\end{equation}
where the first and third cases come by bounding the modulus by the real part, and the second from our hypothesis on $z$. It follows by applying the bounds from \eqref{eq:cases_qpoc_1} to each factor that 
\begin{equation}
|(z;t)_n| \geq \prod_{\ell=0}^{\ell_0-1}(t^{\ell-\ell_0+1/2} - 1) \cdot \delta \cdot (t^{1/2};t)_{n-\ell_0-1} = \delta t^{-\frac{\ell_0^2}{2}} (t^{1/2};t)_{\ell_0} (t^{1/2};t)_{n-\ell_0-1},
\end{equation}
showing \eqref{eq:qpoc_lower_bound} in this case. If $n=\infty$ the proof and the bound are the same. If $\Re(z) \geq t^{-n+1/2}$, we simply apply the first bound in \eqref{eq:cases_qpoc_1} to all but one term in $(z;t)_n$ and apply the second bound in \eqref{eq:cases_qpoc_1} to the remaining term; if $\Re(z) \leq t^{1/2}$ we apply the third bound in \eqref{eq:cases_qpoc_1} to all but one term in $(z;t)_n$ and apply the second bound in \eqref{eq:cases_qpoc_1} to the remaining term. In both cases, we obtain a bound $\delta (t^{1/2};t)_\infty$, and this completes the proof.
\end{proof}

To prove \Cref{thm:can_do_dct_on_tg}, we will bound by separate constants for $\Re(z) \leq -t^{1/2}, -t^{1/2} \leq \Re(z) \leq 0$, and $0 \leq \Re(z) \leq 1$, and take the maximum of these three. We use the following lemmas to establish bounds on various parts of the contour $\tG$. Each one is an easy computation, but we found it helpful to highlight them separately to avoid (our own) confusion between the cases.% because there are so many. 

\begin{lemma}\label{thm:qpoc_tmoment_bound}
For any $z = x \pm \sqrt{1-x^2}\bi$, $0 \leq x \leq 1$, we have
\begin{equation}\label{eq:qpoc_tmoment_bound}
\frac{1}{|(-t^{-L}z;t)_\infty|  } \leq \frac{1}{\sqrt{(-t^{-2L};t^2)_\infty}}.
\end{equation}
\end{lemma}
The following lemmas are refinements of \Cref{thm:qpoc_lower_bound}.
\begin{lemma}\label{thm:small_L_z_bound}
For any $n \in \Z,L \in \Z_{<-n}$ and $z = -t^{-n+u} \pm i, u \in [-1/2,1/2]$,
\begin{equation}
\frac{1}{|(-t^{-L}z;t)_\infty|} \leq \frac{1}{(t^{1/2};t)_\infty}.
\end{equation}
\end{lemma}

\begin{lemma}\label{thm:small_horiz_z_bound}
For any $n \in \Z_{<0}, L \in \Z_{\geq 0}$ and $z = -t^{-n+u} \pm i, u \in [-1/2,1/2]$,
\begin{equation}
\frac{1}{|(-t^{-L}z;t)_\infty|} \leq \frac{t^{\binom{L+1}{2}}}{(t^{1/2};t)_\infty}.
\end{equation}
\end{lemma}

\begin{lemma}\label{thm:large_n_large_L_bound}
For any $n \in \Z_{\geq 0},L \in \Z_{\geq -n}$ and $z = -t^{-n+u} \pm i, u \in [-1/2,1/2]$,
\begin{equation}
\frac{1}{|(-t^{-L}z;t)_\infty|} \leq \frac{t^{(L+n)^2-n}}{(t^{1/2};t)_\infty^2}.
\end{equation}
\end{lemma}

\begin{proof}[Proof of \Cref{thm:qpoc_tmoment_bound}]
We have 
\begin{equation}
|1+t^\ell(t^{-L}z)| = \sqrt{1+t^{2\ell-2L}+2x t^{\ell-L}} \geq \sqrt{1+t^{2\ell-2L}}.
\end{equation}
Taking a product over all $\ell \in \Z_{\geq 0}$,
\begin{equation}
|(-t^{-L}z;t)_\infty| \geq \sqrt{(-t^{-2L};t^2)_\infty}.
\end{equation}
\end{proof}

\begin{proof}[Proof of \Cref{thm:small_L_z_bound}]
Bounding the modulus by the real part, we have
\begin{equation}
|1+t^\ell (t^{-L}z)| \geq 1-t^{-L-n+\ell+u}  \geq 1-t^{-L-n+\ell-1/2}
\end{equation}
for each $\ell \in \Z_{\geq 0}$. Combining these,
\begin{equation}
|(-t^{-L}z;t)_\infty| \geq (t^{-L-n-1/2})_\infty \geq (t^{1/2};t)_\infty
\end{equation}
since $-L-n > 0$, and inverting the above completes the proof.
\end{proof}

\begin{proof}[Proof of \Cref{thm:small_horiz_z_bound}]
We obtain 
\begin{equation}
|1+t^\ell(t^{-L}z)| \geq \begin{cases}
t^{\ell-L} & \ell=0,\ldots,L \\ 
1-t^{-L-n-1/2+\ell} & \ell=L+1,L+2,\ldots
\end{cases}
\end{equation}
by bounding by the modulus by the imaginary part and real part respectively. Combining these, 
\begin{equation}
|(-t^{-L}z;t)_\infty| \geq \prod_{\ell=0}^L t^{\ell-L} \prod_{\ell=L+1}^\infty 1-t^{-L-n-1/2+\ell} = t^{-\binom{L+1}{2}} (t^{-n+1/2};t)_\infty.
\end{equation}
Bounding $(t^{-n+1/2};t)_\infty \geq (t^{1/2};t)_\infty$ (since $n  < 0$) and inverting the above completes the proof.
\end{proof}

\begin{proof}[Proof of \Cref{thm:large_n_large_L_bound}]
Our hypotheses yield that $L+n \geq 0$. We obtain three bounds
\begin{equation}
|1+t^\ell(t^{-L}z)| \geq \begin{cases}
t^{-L-n+\ell+1/2}-1 &  0 \leq \ell \leq L+n-1 \\ 
t^{-L+\ell} &  \ell = L+n \\ 
1-t^{-L-n+\ell-1/2} & \ell \geq L+n+1
\end{cases}
\end{equation}
by bounding the modulus by the real, imaginary, and real parts respectively. Taking the product of these,
\begin{align}
\begin{split}
|(-t^{-L}z;t)_\infty| &\geq \prod_{\ell=0}^{L+n-1} (t^{-L-n+\ell+1/2}-1) \cdot t^n \cdot \prod_{\ell=L+n+1}^\infty 1-t^{-L-n+\ell-1/2} \\ 
&= t^{-(L+n)^2+n} (t^{1/2};t)_{L+n} (t^{1/2};t)_\infty.
\end{split}
\end{align}
Bounding $(t^{1/2};t)_{L+n} \geq (t^{1/2};t)_\infty$ and inverting the above completes the proof.
\end{proof}

\begin{proof}[Proof of \Cref{thm:can_do_dct_on_tg}]
We bound by separate constants for $\Re(z) \leq -t^{1/2}, -t^{1/2} \leq \Re(z) \leq 0$, and $0 \leq \Re(z) \leq 1$, and take the maximum of these three. For any fixed $z \in \tG$ with $\Re(z) \leq -t^{1/2}$, there exists $n \in \Z_{\geq 0}$ such that $-t^{-n-1/2} \leq \Re(z) \leq -t^{-n+1/2}$. For this $n$, we have by \Cref{thm:small_L_z_bound} and \Cref{thm:large_n_large_L_bound} that 
\begin{align}
\begin{split}\label{eq:large_z_in_proof}
\sum_{L \in \Z} \abs*{\frac{t^{-(m+1)L}}{(-t^{-(m+1)L}z;t)_\infty}} & \leq  \sum_{L \in \Z_{<-n}} \frac{t^{-(m+1)L}}{(t^{1/2};t)_\infty}  +  t^{mn}\sum_{L \in \Z_{\geq -n}} \frac{t^{(L+n)^2-(m+1)(L+n)}}{(t^{1/2};t)_\infty^2} \\ 
&\leq \frac{1}{(t^{1/2};t)_\infty^2}\left(\frac{t^{-(-n-1)(m+1)}}{1-t^{m+1}}(t^{1/2};t)_\infty + t^{mn}\sum_{\ell \geq 0} t^{\ell^2-\ell}\right).
\end{split}
\end{align}
This is uniformly bounded over all $n,m \in \Z_{\geq 0}$.

For fixed $z \in \tG$ with $-t^{1/2} \leq \Re(z) \leq 0$, letting $n \in \Z_{<0}$ be as before, we have by \Cref{thm:small_horiz_z_bound} and \Cref{thm:small_L_z_bound} that 
\begin{align}
\begin{split}\label{eq:small_z_in_proof}
\sum_{L \in \Z} \abs*{\frac{t^{-(m+1)L}}{(-t^{-L}z;t)_\infty}} & \leq \sum_{L \in \Z_{\geq 0}} \frac{t^{\binom{L+1}{2}-(m+1)L}}{(t^{1/2};t)_\infty} + \sum_{L \in \Z_{<0}} \frac{t^{-(m+1)L}}{(t^{1/2};t)_\infty},
\end{split}
\end{align}
which is independent of $n$. 

For $z \in \tG$ with $\Re(z) \geq 0$, i.e. on the circular arc, first note the elementary bounds
\begin{equation}
(-t^{-2L};t^2)_\infty = t^{-2\binom{L+1}{2}} (-1;t^2)_L (-1;t^2)_\infty \geq t^{-2\binom{L+1}{2}}(-1;t^2)_\infty
\end{equation}
for $L \geq 0$ and
\begin{equation}
(-t^{-2L};t^2)_\infty \geq 1
\end{equation}
for all $L$. These, together with \Cref{thm:qpoc_tmoment_bound}, yield that
\begin{align}
\begin{split}\label{eq:arc_z_in_proof}
\sum_{L \in \Z} \abs*{\frac{t^{-(m+1)L}}{(-t^{-L}z;t)_\infty}} & \leq \sum_{L \in \Z_{\geq 0}} \frac{t^{\binom{L+1}{2}-(m+1)L}}{\sqrt{(-1;t^2)_\infty}} + \sum_{L \in \Z_{<0}} t^{-(m+1)L}.
\end{split}
\end{align}
Combining \eqref{eq:large_z_in_proof}, \eqref{eq:small_z_in_proof}, and \eqref{eq:arc_z_in_proof} establishes constant bounds on all of $\tG$, completing the proof.
\end{proof}

\begin{proof}[Proof of \Cref{thm:qp_bound_tG}]
When $n=\infty$, it follows from \Cref{thm:can_do_dct_on_tg} that the $L=-1$ term of the sum is also uniformly bounded over all $w \in \tG$, and this exactly yields \Cref{thm:qp_bound_tG}. For general $n$ it does not directly follow from that result, but the same method as in Lemmas \ref{thm:qpoc_tmoment_bound}-\ref{thm:large_n_large_L_bound} of bounding each term in the $q$-Pochhammer symbol differently based on the location of $w$ works and we omit the details.
\end{proof}

\section{Examples of residue formula for {$\cL_{k,t,\chi}$}} \label{sec:examples}

In this section we compute more explicitly the infinite series formula in \Cref{thm:stat_dist_1pt} for $k=1$, $k=2$, and an example case of $k=3$. This section is not necessary to understand the rest of the paper, but we hope that for the interested reader it helps illustrate how the residue formula \Cref{thm:stat_dist_1pt} looks and how the complexity increases as $k$ increases.

\begin{cor}\label{thm:k=1_case}
In the notation of \Cref{thm:stat_dist_1pt}, for any $\chi \in \R_{>0}$ the $\Z$-valued random variable $\cL_{1,t,\chi}$ is defined by
\begin{equation}
\Pr(\cL_{1,t,\chi} = (L)) = \frac{1}{(t;t)_\infty} \sum_{m \geq 0} e^{-\chi t^{L-m}} \frac{(-1)^m t^{\binom{m}{2}}}{(t;t)_m}
\end{equation}
for all $L \in \Z$.
\end{cor}
\begin{proof}
Follows immediately from \Cref{thm:stat_dist_1pt}, by noting that the sum over $\mu \in \Sig_{k-1}$ has only one term $\mu = ()$ and is equal to $(-1)^{L-d}$, and changing variables to $m=L-d$.
\end{proof}

We next show the $k=2$ formula from the Introduction.
\begin{proof}[Proof of \Cref{thm:k=2_case}]
We begin with the formula \Cref{thm:stat_dist_1pt} and make the change of variables $m=L-d$ so
\begin{align}
\label{eq:use_branch_once}
\begin{split}
Q_{(\mu-(d))'}(\gamma(t^d (1-t) \chi),\alpha(1);0,t) &= Q_{(1[m+i])}(\gamma(t^{L-m}(1-t)\chi),\alpha(1);0,t) \\ 
&= Q_{(1[m+i])}(\gamma(t^{L-m}(1-t)\chi);0,t) \\ 
&+ \bbone(i+m \geq 1) (1-t^{m+i}) Q_{(1[m+i-1])}(\gamma(t^{L-m}(1-t)\chi);0,t),
\end{split}
\end{align}
where the second equality follows from the branching rule \eqref{eq:specialization_branch} together with explicit formulas in \Cref{thm:hl_qw_branch_formulas} for the branching coefficients. Using the relation between Plancherel and alpha specializations in \Cref{thm:specialize_mac_poly} together with the branching rule (\Cref{thm:branching_formulas} and \Cref{thm:hl_qw_branch_formulas}), we obtain
\begin{equation}\label{eq:easy_planch_case}
Q_{(1[a])}(\gamma(c(1-t));0,t) = \frac{(t;t)_a}{a!}c^a.
\end{equation}
Substituting \eqref{eq:easy_planch_case} into \eqref{eq:use_branch_once} and then into \Cref{thm:stat_dist_1pt} (with our change of variables) yields the desired formula.
\end{proof}

\begin{example}
In the notation of \Cref{thm:stat_dist_1pt}, 
\begin{align}\label{eq:k3_ex}
\begin{split}
&\Pr(\cL_{3,t,\chi} = (L+2,L,L)) = \frac{1}{(t;t)_\infty (t;t)_2} \left( (e^{-t^{L}\chi}t)\left((t;t)_2\frac{(t^L\chi)^2}{2!}- t(1-t^2)\frac{t^L  \chi}{1!} + t^2\right) \right. \\ 
& -  \frac{e^{-t^{L-1}\chi}t^3}{1-t} \left((1-t)^4(1 + t) (3 + 2 t + t^2) \frac{(t^{L-1} \chi)^4}{4!} + (1 - t)^3 (1 + t)t^3 (1 - 2 t - t^2 - t^3) \frac{(t^{L-1}\chi)^3}{3!} \right. \\ 
&\left. + (1-t)^2t(-1+t+t^2) \frac{(t^{L-1} \chi)^2}{2!} - (1-t)^2 t\frac{t^{L-1} \chi}{1!} + (1-t)\right)  \\ 
& +\frac{e^{-t^{L-2}\chi}t^8}{(t;t)_2}\left( (1 - t)^6 (1 + t)^2 (9 + 13 t + 12 t^2 + 7 t^3 + 3 t^4 + t^5) \frac{(t^{L-2}\chi)^6}{6!}     \right. \\ 
&+  (1 - t)^5 (1 + t)^2 (4 - 2 t - 4 t^2 - 6 t^3 - 4 t^4 - 2 t^5 - t^6) \frac{(t^{L-2}\chi)^5}{5!} \\ 
&+ (1 - t)^4 (1-t) t (-3 + 3 t^2 + 2 t^3 + t^4) \frac{(t^{L-2}\chi)^4}{4!} \\ 
&\left. \left.+ (1 - t)^4 (1+t)t (-2 - t) \frac{(t^L \chi)^3}{3!} + (1-t)^3(1-t) \frac{(t^{L-2}\chi)^2}{2!}\right) \right. \\ 
&\left.- \frac{e^{-t^{L-3}\chi}t^{16}}{(t;t)_3}(\ldots) + \frac{e^{-t^{L-4}\chi}t^{27}}{(t;t)_4}(\ldots) + \ldots \right),
\end{split}
\end{align}
where we have only computed the first three terms in the series. 

We now show the computation. By \Cref{thm:stat_dist_1pt} with a change of variables $L-d=m$, 
\begin{align}\label{eq:expand_k3_ex}
\begin{split}
\text{LHS\eqref{eq:k3_ex}} &= \frac{1}{(t;t)_\infty} \sum_{d \leq L} e^{-\chi t^{d}} P_{(L-d+2,L-d,L-d)'}(1,t,\ldots;0,t) \\ 
&\times \sum_{\substack{\mu \in \Sig_{k-1} \\ \mu \prec (L+2,L,L)}} P_{(L-d+2,L-d,L-d)/(\mu_1-d,\mu_2-d)}(-1;t,0) Q_{(\mu_1-d,\mu_2-d)'}(\gamma(t^{d}(1-t)\chi),\alpha(1);0,t) \\ 
&= \frac{1}{(t;t)_\infty} \sum_{m \geq 0} e^{- \chi t^{L-m}}  \frac{t^{\binom{m+2}{2}+2\binom{m}{2}}}{(t;t)_2(t;t)_0(t;t)_m} \left( \sqbinom{2}{0}_t (-1)^m Q_{(m+2,m)'}(\cdots;0,t) \right. \\ 
&\left.+ \sqbinom{2}{1}_t(-1)^{m+1} Q_{(m+1,m)'}(\cdots;0,t)+ \sqbinom{2}{2}_t (-1)^{m+2} Q_{(m,m)'}(\cdots;0,t)\right) \\ 
&= \frac{1}{(t;t)_\infty(t;t)_2} \sum_{m \geq 0} e^{-t^{L-m}\chi} \frac{t^{\frac{3}{2}m^2+\frac{1}{2}m+1}}{(t;t)_m} (-1)^m \\ 
&\times (Q_{(m+2,m)'}(\cdots;0,t) - (1+t)Q_{(m+1,m)'}(\cdots;0,t) + Q_{(m,m)'}(\cdots;0,t)),
\end{split}
\end{align}
where we have used \Cref{thm:hl_qw_branch_formulas} and \Cref{thm:hl_principal_formulas} and written $\cdots$ for the specialization $\gamma(t^{L-m}(1-t)\chi),\alpha(1)$. It remains to compute the three Hall-Littlewood polynomials in the last line of \eqref{eq:expand_k3_ex}, and since there is not a closed form we compute the first few terms $m=0,1$. 

For any $\la \in \Y$ the branching rule yields
\begin{equation}
\label{eq:Qalphagamma_branch}
Q_\la(\gamma(g),\alpha(a);q,t) = \sum_{c = 0}^{\la_1} \phi_{(c)}(q,t) a^c \sum_{B \in SYT(\la/(c))} \frac{\phi_B(q,t) g^{|\la|-c}}{(|\la|-c)!}.
\end{equation}
Here $SYT(\la/(c))$ is the set of standard Young tableaux corresponding to this skew shape, and for a tableau $B$ identified with a sequence of partitions $(c) = \la^{(0)} \prec \ldots \prec \la^{(|\la|-c)} = \la$, we use the shorthand
\begin{equation}
\phi_B(q,t) := \prod_{i=1}^{|\la|-c} \phi_{\la^{(i)}/\la^{(i-1)}}(q,t)
\end{equation}
where $\phi_{\la^{(i)}/\la^{(i-1)}}$ is as in \Cref{def:psi_varphi}. It follows from \Cref{thm:hl_qw_branch_formulas} that 
\begin{equation}
\phi_{(c)}(0,t) = (1-t^c),
\end{equation}
so it remains to compute the sum in \eqref{eq:Qalphagamma_branch}. For each Hall-Littlewood polynomial appearing in \eqref{eq:expand_k3_ex} with $m=0,1,2$ we will enumerate the pairs $(c,B)$ with $B \in SYT(\la/(c))$ and give their coefficients.

\ytableausetup{notabloids}
\underline{\textbf{m=0:}} We compute the coefficient of $e^{-t^{L}\chi}$ in \eqref{eq:k3_ex}. Trivially $Q_{(0,0)'}(\cdots;0,t)=1$. For $Q_{(1,0)'}(\cdots;0,t)$ we may either take $c=0$ or $c=1$: in each case there is one tableau, and the cases contribute $(1-t)(t^{L}\chi)^1/1!$ and $1-t$ respectively. For $Q_{(2,0)'}(\cdots;0,t)$ we again may either take $c=0$ or $c=1$, and in both cases have one skew tableau\footnote{Here and elsewhere we denote the missing boxes, referred to as $(c)$ above, by colored boxes.},
\begin{equation}
\begin{ytableau}1 \\ 2\end{ytableau} \quad \quad \text{      and      } \quad \quad  \begin{ytableau}*(purple) \\ 1\end{ytableau},
\end{equation}
yielding coefficients $(t;t)_2 (t^{L}\chi)^2/2!$ and $(t;t)_2(t^{L}\chi)/1!$ respectively in \eqref{eq:Qalphagamma_branch}. Summing these yields the desired coefficient.

\underline{\textbf{m=1:}} We obtain
\begin{equation}
Q_{(1,1)'}(\cdots;0,t) = Q_{(2)}(\cdots;0,t) = (1-t)^2 \frac{(t^{L-1}\chi)^2}{2!}+ (1-t)^2\frac{(t^{L-1}\chi)^1}{1!} + (1-t)\frac{(t^{L-1}\chi)^0}{0!}
\end{equation}
with the three summands coming from the three skew tableaux
\begin{equation}
\begin{ytableau}1 & 2\end{ytableau} \quad \quad \text{ and } \quad \quad \begin{ytableau}*(purple) & 1\end{ytableau} \quad \quad \text{ and } \quad \quad  \begin{ytableau}*(purple) & *(purple)\end{ytableau}.
\end{equation}
Similarly
\begin{equation}
Q_{(2,1)'}(\cdots;0,t) = (1-t)^3(2+t) \frac{(t^{L-1}\chi)^3}{3!} + (1-t)^3(2+t) \frac{(t^{L-1}\chi)^2}{2!} + (1-t)^2 \frac{(t^{L-1}\chi)^1}{1!}
\end{equation}
with tableaux
\begin{equation}
\begin{ytableau}1 & 2 \\ 3\end{ytableau} \quad \quad \text{ and } \quad \quad  \begin{ytableau}1 & 3 \\ 2\end{ytableau}
\end{equation}
contributing $(1-t)^3$ and $(1-t)^2(1-t^2)$ respectively to the degree $3$ term,
\begin{equation}
\begin{ytableau}*(purple) & 1 \\ 2\end{ytableau} \quad \quad \text{ and } \quad \quad  \begin{ytableau}*(purple) & 2 \\ 1\end{ytableau}
\end{equation}
also contributing $(1-t)^3$ and $(1-t)^2(1-t^2)$ respectively to the degree $2$ term, and 
\begin{equation}
\begin{ytableau}*(purple) & *(purple) \\ 1 \end{ytableau}
\end{equation}
contributing $(1-t)^2$ to the degree $0$. Finally,
\begin{multline}\label{eq:hardest_Q_ex}
Q_{(3,1)'}(\cdots;0,t) = Q_{(2,1,1)}(\cdots;0,t) = (1-t)^2(1-t^2)\frac{(t^{L-1}\chi)^2}{2!}\\ 
+ ((1-t)^3(1-t^2) +(1-t)^2(1-t^2)^2 + (1-t)^2(1-t^2)(1-t^3)) \left( \frac{(t^{L-1}\chi)^3}{3!} + \frac{(t^{L-1}\chi)^4}{4!}\right).
\end{multline}
\underline{\textbf{m=2:}} Here the three relevant tableaux for the degree-$4$ term are
\begin{equation}
\begin{ytableau}1 & 2 \\ 3 \\ 4\end{ytableau}, \quad \quad  \begin{ytableau}1 & 3 \\ 2 \\ 4\end{ytableau}, \quad \quad  \begin{ytableau}1 & 4 \\ 2 \\ 3\end{ytableau}
\end{equation}
\ytableausetup{smalltableaux}
in the same order as their coefficients in \eqref{eq:hardest_Q_ex}. The skew tableaux for the degree $3$ term are the same as for $m=1$, but with the $\begin{ytableau} 1 \end{ytableau}$ box replaced by $\begin{ytableau} *(purple) \end{ytableau}$ and the other indices shifted down by $1$ to yield a skew standard Young tableau, and their coefficients are exactly the same. Only one tableau contributes to the degree $2$ term, namely
\ytableausetup{nosmalltableaux}
\begin{equation}
\begin{ytableau} *(purple) & *(purple) \\ 1  \\ 2 \end{ytableau}.
\end{equation}
Summing the above and simplifying yields the coefficient of the $e^{-t^{L-2}\chi}$ term in \eqref{eq:k3_ex}. 

One may, either by hand or by computer, generate the coefficient of any given remaining $e^{-t^{L-m}\chi}$ term for $m > 2$, but the number of tableaux grows with $m$ and there is no closed form of which we are aware. However, because $e^{-t^{L-m}\chi}$ shrinks very fast as $m$ increases (particularly if $t$ is not close to $1$), it is in fact quite easy to compute good approximations of this probability for any given values of $t,\chi,L$ by computing a small number of terms.
\end{example}

\section{The case of pure $\alpha$ specializations}\label{sec:alpha}

For concreteness we did the computations in \Cref{sec:stationary_law_hl} for the Plancherel process $\la(\tau)$, but for matrix products it is desirable to compute a similar limit distribution of Hall-Littlewood alpha Cauchy dynamics
\begin{equation}
\Pr(\la \to \nu) = Q_{\nu/\la}(\alpha_1,\alpha_2,\ldots;0,t) \frac{P_\nu(1,t,\ldots;0,t)}{P_\la(1,t,\ldots;0,t) \Pi_{0,t}(\alpha_1,\ldots;1,t,\ldots)},
\end{equation}
in light of \Cref{thm:hl_meas_matrices}. Luckily our computations from \Cref{sec:stationary_law_hl} generalize straightforwardly as in the previous section. We use the power sum symmetric polynomial
\begin{equation}
p_1(\alpha_1,\alpha_2,\ldots) = \sum_{i \geq 1} \alpha_i
\end{equation}
in that result to highlight the similarity with \Cref{thm:hl_stat_dist}, where the Plancherel-specialized $p_1$ also appears in the formulas.

\begin{prop}\label{thm:alpha_hl_stat_dist}
Let $k \in \Z_{\geq 1}$ and $\zeta \in \R$, and let $\phi$ be a pure alpha Hall-Littlewood nonnegative specialization determined by $\alpha$ parameters $\alpha_1 \geq \alpha_2 \geq \ldots$ with $0 < \alpha_1 < 1$. Let $\tl(s), s \in \Z_{\geq 0}$ be distributed by the Hall-Littlewood measure
\begin{equation}
\Pr(\tl(s) = \la) = \frac{Q_\la(\alpha_1[s],\alpha_2[s],\ldots;0,t)P_\la(1,t,\ldots;0,t)}{\Pi_{0,t}(\alpha_1[s],\alpha_2[s],\ldots;1,t,\ldots)}.
\end{equation}
Let $(s_n)_{n \in \N}$ be any sequence with $s_n \xrightarrow{n \to \infty} \infty$ such that $\log_{t} s_n$ converges in $\R/\Z$, and let $\zeta$ be a lift of this limit to $\R$. Then
\begin{equation}
(\tl_i'(s_n) - [\log_{t^{-1}}(s_n)+\zeta])_{1 \leq i \leq k} \to \cL_{k,t,t^\zeta p_1(\alpha_1,\alpha_2,\ldots)}
\end{equation}
where $[\cdot]$ is the nearest integer function. 
\end{prop}

\begin{proof}
Since $\Sig_k$ is a discrete set, by the definition of $\cL$ in \Cref{thm:stat_dist_1pt} it suffices to show convergence of probabilities
\begin{multline}
\label{eq:alpha_hl_stat_dist_formula}
\lim_{\substack{n \to \infty}} \Pr(\tl_i'(s_n) - [\log_{t^{-1}}(s_n)+\zeta]  = L_i  \text{ for all }1 \leq i \leq k) \\ 
=  \frac{(t;t)_\infty^{k-1}}{k! (2 \pi \bi)^k} \prod_{i=1}^{k-1} \frac{t^{\binom{L_i-L_k}{2}}}{(t;t)_{L_i-L_{i+1}}} \int_{\tG^k} e^{p_1(\alpha_1,\alpha_2,\ldots)t^{L_k+\zeta}(w_1+\ldots+w_k)} \frac{\prod_{1 \leq i \neq j \leq k} (w_i/w_j;t)_\infty}{\prod_{i=1}^k (-w_i^{-1};t)_\infty (-tw_i;t)_{\infty}} \\ 
\times \sum_{j=0}^{L_{k-1}-L_k} t^{\binom{j+1}{2}} \sqbinom{L_{k-1}-L_k}{j}_t  P_{(L_1-L_k,\ldots,L_{k-1}-L_k,j)}(w_1^{-1},\ldots,w_k^{-1};t,0)  \prod_{i=1}^k \frac{dw_i}{w_i},
\end{multline}
where if $k=1$ we interpret the sum on the last line as in \Cref{thm:hl_stat_dist}.

\Cref{thm:prob_from_observables} holds with $\la(\tau)$ replaced by $\tl(s)$ on the left hand side and $\gamma(\tau)$ replaced by $\phi[s]$ (i.e. $\phi$ repeated $s$ times, recall \Cref{def:spec_notation}) on the right hand side, by the exact same proof. Similarly, \Cref{thm:use_torus_product} holds with $\la(\tau)$ replaced by $\tl(s)$ on the left hand side and 
\begin{equation}
e^{\frac{\tau}{1-t}(z_1+\ldots+z_k)} = \Pi_{0,t}(\gamma(\tau);\beta(z_1,\ldots,z_k))
\end{equation}
on the right hand side replaced by 
\begin{align}\label{eq:gen_cauchy_kernel_in_proof}
\begin{split}
\Pi_{0,t}(\phi[s];\beta(z_1,\ldots,z_k)) &= \prod_{\substack{1 \leq i \leq k \\ j \geq 1}} (1+\alpha_j z_i).
\end{split}
\end{align}
We now explain how the asymptotic analysis used to prove \Cref{thm:hl_stat_dist} carries over. Let
\begin{equation}
a := \sum_{j \geq 1} \alpha_j,
\end{equation}
and
\begin{equation}
\eta(s) := (L_i + [\log_{t^{-1}}(s) + \zeta])_{1 \leq i \leq k}.
\end{equation}
Let
\begin{equation}
\label{eq:def_eps_error}
\eps(s) := (\log_{t^{-1}} (s) + \zeta) - [\log_{t^{-1}}(s) + \zeta],
\end{equation}
so that $\lim_{n \to \infty} \eps(s_n) = 0$ by our choice of the subsequence $s_n$, and $\eta_k(s) = L_k + \zeta + \log_{t^{-1}} s - \eps(s)$. We then make a change of variables to
\begin{equation}
w_i := t^{-\eta_k(s_n)}z_i = s_n t^{-\zeta - L_k + \eps(s_n)}z_i.
\end{equation}
Note that for $s=s_n$ this is $s_n t^{-\zeta - L_k}z_i (1+o(1))$ as $n \to \infty$, but it will be clearer later to have the $t^{\eps(s_n)} = 1+o(1)$ multiplicative error term written explicitly, as the rate at which $\eps(s_n)$ goes to $0$ influences the decomposition of contours which we now define. First fix any choice of a function $\xi'(s)$ for which
\begin{enumerate}
\item $\xi'(s) \to \infty$ as $s \to \infty$,
\item $\xi'(s) = o( \log_{t^{-1}} s)$ as $s \to \infty$, and
\item $\xi'(s_n) = o( -\log \eps(s_n))$ as $n \to \infty$.
\end{enumerate}
This will play the same role as the function $\xi(\tau)$ in the proof of \Cref{thm:hl_stat_dist}, but may in general have to be slower-growing due to the third condition. Fix a constant $\delta > 0$ with $\alpha_1 < (1+\delta)^{-1} < 1$, and define contours 
\begin{align}\label{eq:contours'}
\begin{split}
\Gamma'(s) &= \Gamma_1'(s) \cup \Gamma_2'(s) \\ 
\Gamma_1'(s) &= \{x + \bi: -t^{-\xi'(s)} < x \leq 0 \} \cup \{x - \bi: -t^{-\xi'(s)}  < x \leq 0\} \cup \{x+\bi y: x^2+y^2=1, x > 0\} \\ 
\Gamma'_2(s) &= \{-(1+\delta)t^{-\eta_k(s)} + \bi y: -1 \leq y \leq 1\} \cup \{x + \bi: -(1+\delta)t^{-\eta_k(s)}< x \leq  -t^{-\xi'(s)}  \} \\ 
&  \cup \{x - \bi: -t^{-\eta_k(s)}< x \leq  -t^{-\xi'(s)} \},
\end{split}
\end{align}
as shown in \Cref{fig:alpha_contour}. $\Gamma'(s)$ is very similar to the contour $\Gamma(\tau)$ from the proof of \Cref{thm:hl_stat_dist} with $\tau=s$ but the vertical part of $\Gamma_2'$ lies slightly to the right of the vertical part of $\Gamma_2$ in general (and this is needed in the analysis).

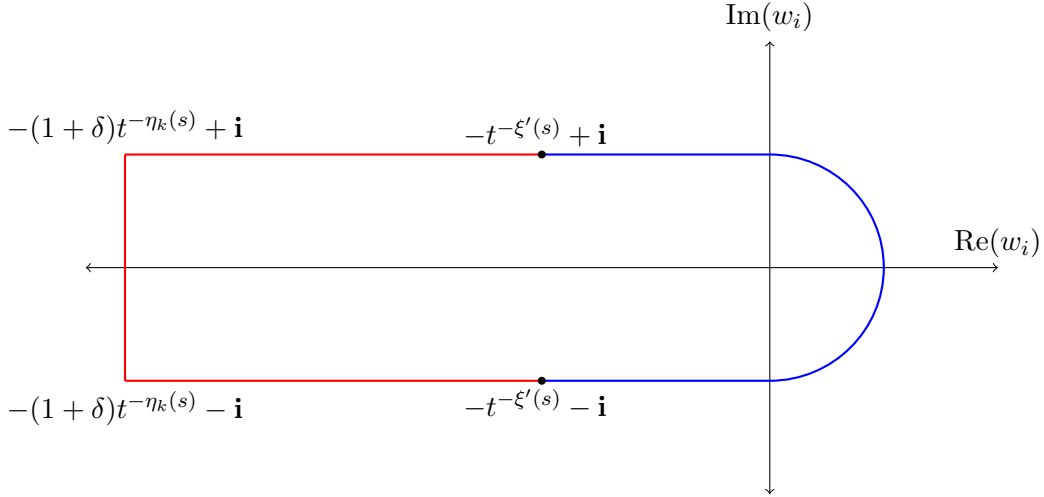
\begin{figure}[htbp]
\begin{center}
\begin{tikzpicture}[scale=1.5]
\def\b{2};
\def\t{.5};
\def\d{-5.656}; %this is -t^{-2.5} but i was having trouble with pow(...) for some reason...
  % Draw the axes
  \draw[<->] (0,-2) -- (0,2) node[above] {$\Im(w_i)$};
  \draw[<->] (-6,0) -- (2,0) node[above] {$\Re(w_i)$};

  % Draw the half circle
  \draw[thick,blue] (0,-1) arc (-90:90:1);

  % Draw the horizontal lines
  \draw[thick,blue] (-\b,1) node[left,black,yshift=3mm,xshift=10mm] {$-t^{-\xi'(s)}+\bi$} -- (0,1) ;
  \draw[thick,blue] (-\b,-1) node[left,black,yshift=-3mm,xshift=10mm] {$-t^{-\xi'(s)}-\bi$} -- (0,-1) ;

    \draw[thick,red] (\d,1)  -- (-\b,1);% node[right,yshift=3mm] {(0,1)};
    \draw[thick,red] (\d,-1)  -- (-\b,-1);% node[right,yshift=-3mm] {(0,-1)};

    \draw[thick,red] (\d,1) node[above,black] {$-(1+\delta)t^{-\eta_k(s)}+\bi$} -- (\d,-1) node[below,black] {$-(1+\delta)t^{-\eta_k(s)}-\bi$};

\fill (-\b,1) circle (1pt);
\fill (-\b,-1) circle (1pt);

 % to draw poles, loop through powers of t
 % \foreach \n in {-2,...,20} {
    % Calculate x-coordinate
  %  \pgfmathsetmacro\x{-pow(\t,\n)}
    % Draw a dot at coordinate (\x,0)
   % \fill (\x,0) circle (.5pt);
  %}

\end{tikzpicture}
\caption{The contour $\Gamma'(s)$ with the analogous decomposition to \eqref{eq:contours} into $\Gamma_1'(s)$ in blue and $\Gamma'_2(s)$ in red.
}\label{fig:alpha_contour}
\end{center}
\end{figure}

Now, the same manipulations as in the proof of \Cref{thm:hl_stat_dist} to consolidate the powers of $t$ after the change of variables give
\begin{align}
\label{eq:gen_nice_prelim_int}
\begin{split}
&\Pr((\tl(s)'_i)_{1 \leq i \leq k} = \eta(s)) = \frac{(t;t)_\infty^{k-1}}{k! (2 \pi \bi)^k} \prod_{i=1}^{k-1} \frac{t^{\binom{L_i-L_k}{2}}}{(t;t)_{L_i-L_{i+1}}} \int_{\Gamma(s)'^k} \prod_{\substack{1 \leq i \leq k \\ j \geq 1}} (1+t^{\zeta+L_k-\eps(s)}\alpha_j w_i/s)^s  \\ &\times\frac{\prod_{1 \leq i \neq j \leq k} (w_i/w_j;t)_\infty}{\prod_{i=1}^k (-w_i^{-1};t)_\infty (-tw_i;t)_{\eta_k(s)}}
 \sum_{j=0}^{L_{k-1}-L_k} t^{\binom{j+1}{2}} \sqbinom{L_{k-1}-L_k}{j}_t  \\ 
 &\times P_{(L_1-L_k,\ldots,L_{k-1}-L_k,j)}(w_1^{-1},\ldots,w_k^{-1};t,0)  \prod_{i=1}^k \frac{dw_i}{w_i}.
\end{split}
\end{align}

The reason for the slightly different contour from earlier is that now by our choice of $\delta$,
\begin{equation}\label{eq:Rew_big}
\Re(t^{\zeta+L_k-\eps(s_n)}\alpha_j w_i/s_n) \geq -(1+\delta)\alpha_j t^{-\eps(s_n)} > -1
\end{equation}
for all $w_i \in \Gamma'(s_n)$ and $n$ large enough that $t^{-\eps(s_n)}$ is sufficiently close to $1$. This implies that for such large enough $n$ and $w_i \in \Gamma'(s_n)$, 
\begin{align}\label{eq:why_new_contour}
\begin{split}
\abs*{1+\frac{t^{\zeta+L_k-\eps(s_n)}\alpha_j}{s_n} w_i} &= \sqrt{\left(1+\frac{t^{\zeta+L_k-\eps(s_n)}\alpha_j}{s_n}\Re(w_i)\right)^2+\left(\frac{t^{\zeta+L_k-\eps(s_n)}\alpha_j}{s_n}\Im(w_i)\right)^2 } \\
& \leq \abs*{1+\frac{t^{\zeta+L_k-\eps(s_n)}\alpha_j}{s_n}\Re(w_i)} + \abs*{\frac{t^{\zeta+L_k-\eps(s_n)}\alpha_j}{s_n}\Im(w_i)} \\ 
& \leq 1+\frac{t^{\zeta+L_k-\eps(s_n)}\alpha_j}{s_n}(\Re(w_i)+1),
\end{split}
\end{align}
where the last line follows by \eqref{eq:Rew_big} (to remove the absolute value on the first factor) and the bound $\Im(w_i) \leq 1$.

To simplify notation we express the integrand in terms of the function $\tf$ of \Cref{def:tf}, which we recall is given by 
\begin{multline}
\label{eq:def_tf_recall}
\tf(w_1,\ldots,w_k) := \frac{(t;t)_\infty^{k-1}}{k!} \prod_{i=1}^{k-1} \frac{t^{\binom{L_i-L_k}{2}}}{(t;t)_{L_i-L_{i+1}}} \frac{\prod_{1 \leq i \neq j \leq k} (w_i/w_j;t)_\infty}{\prod_{i=1}^k w_i(-w_i^{-1};t)_\infty }  \\ 
\times \sum_{j=0}^{L_{k-1}-L_k} t^{\binom{j+1}{2}} \sqbinom{L_{k-1}-L_k}{j}_t  P_{(L_1-L_k,\ldots,L_{k-1}-L_k,j)}(w_1^{-1},\ldots,w_k^{-1};t,0),
\end{multline}
and is the function $f$ of \eqref{eq:integrand_junk} without the exponential factor. Then similarly to \eqref{eq:split_integral1}, \eqref{eq:split_integral2} and \eqref{eq:split_integral3}, we must show 
\begin{align}
\label{eq:alpha_split_integral1}
&\lim_{n \to \infty} \frac{1}{(2 \pi \bi)^k} \int_{\Gamma_1'(s_n)^k} \tf(w_1,\ldots,w_k)\left(\frac{\prod_{\substack{1 \leq i \leq k, j \geq 1}} (1+\frac{t^{\zeta+L_k-\eps(s_n)}\alpha_j w_i}{s_n})^{s_n}}{\prod_{i=1}^k(-tw_i;t)_{\eta_k(s_n)}}- \frac{e^{t^{\zeta+L_k}\sum_{i=1}^k w_i}}{\prod_{i=1}^k(-tw_i;t)_{\infty}}   \right)\prod_{i=1}^k dw_i \\  
&+ \lim_{n \to \infty}\frac{1}{(2 \pi \bi)^k} \int_{\Gamma'(s_n)^k \setminus \Gamma_1'(s_n)^k} \frac{\tf(w_1,\ldots,w_k)\prod_{\substack{1 \leq i \leq k \\ j \geq 1}} (1+\frac{t^{\zeta+L_k-\eps(s_n)}\alpha_j w_i}{s_n})^{s_n}}{\prod_{i=1}^k(-tw_i;t)_{\eta_k(s_n)}} \prod_{i=1}^k dw_i \label{eq:alpha_split_integral2} \\
&-\lim_{n \to \infty} \frac{1}{(2 \pi \bi)^k} \int_{\tG^k \setminus \Gamma_1'(s_n)^k} \frac{\tf(w_1,\ldots,w_k)e^{t^{\zeta+L_k}\sum_{i=1}^k w_i}}{\prod_{i=1}^k(-tw_i;t)_{\infty}}\prod_{i=1}^k dw_i=0.\label{eq:alpha_split_integral3}
\end{align}
We will show each line is $0$ separately. The third line \eqref{eq:alpha_split_integral3} is exactly the same, up to changing the value of $\zeta$ and renaming the parameter in the limit, as \eqref{eq:split_integral3}, which was shown to be $0$ in the proof of \Cref{thm:hl_stat_dist}. For the first line \eqref{eq:alpha_split_integral1}, we have by writing $(1+x)^{s_n} = e^{s_n \log(1+x)}$ and Taylor expanding the logarithm that
\begin{align}\label{eq:gamma1_prod_bound}
\begin{split}
\prod_{\substack{1 \leq i \leq k \\ j \geq 1}} \left(1+\frac{t^{\zeta+L_k-\eps(s_n)}\alpha_j w_i}{s_n}\right)^{s_n} &= e^{at^{\zeta+L_k-\eps(s_n)}\sum_{i=1}^k w_i + O(w_i^2/s_n)} \\ 
&= e^{at^{\zeta+L_k}\sum_{i=1}^k w_i}\left(1+O\left(\eps(n)\sum_i w_i\right) + O\left(\sum_i w_i^2/s_n\right)\right)
\end{split}
\end{align}
as $n \to \infty$. Because $|w_i| \leq \const \cdot t^{-\xi'(s_n)}$ on $\Gamma_1'(s_n)$, and the above yields
\begin{multline}\label{eq:gamma1_prod_bound_2}
\prod_{\substack{1 \leq i \leq k \\ j \geq 1}} \left(1+\frac{t^{\zeta+L_k-\eps(s_n)}\alpha_j w_i}{s_n}\right)^{s_n}  = e^{at^{\zeta+L_k}\sum_{i=1}^k w_i}(1+O(t^{-\xi'(s_n)} \eps(s_n)) + O(t^{-2\xi'(s_n)}/s_n)).
\end{multline}
For $w_i \in \Gamma_1'(s_n)$ we have $|t^{\eta_k(s_n)+1}w_i| \leq t^{\eta_k(s_n)-\xi'(s_n)}$ (for $n$ large enough so $|-t^{- \xi'(s_n)}\pm i| \leq t^{-\xi'(s_n) - 1}$). Hence $(-t^{\eta_k(s_n)+1}w_i;t)_{\infty} = 1 + O(t^{\eta_k(s_n) - \xi'(s_n)})$ and so
\begin{equation}
\label{eq:qp_inf_fin_bound}
\frac{1}{(-tw_i;t)_{\eta_k(s_n)}} = \frac{1}{(-tw_i;t)_\infty}(1+O(t^{\eta_k(s_n) - \xi'(s_n)})).
\end{equation}
Since $\Gamma_1'(s_n) \subset \tG$, \Cref{thm:qp_bound_tG} shows that $|1/(-tw_i;t)_\infty|$ is bounded above by a constant. By combining \eqref{eq:gamma1_prod_bound_2}, the bound \Cref{thm:tf_bound} applied to $\tf$, \eqref{eq:qp_inf_fin_bound}, and this constant bound, the integrand in \eqref{eq:alpha_split_integral1} is 
\begin{equation}\label{eq:big_O}
O\left(e^{at^{\zeta+L_k}\sum_{i=1}^k \Re(w_i)  }(t^{\eta_k(s_n) - \xi'(s_n)} + t^{-\xi'(s_n)} \eps(s_n) + t^{-2\xi'(s_n)}/s_n\right).
\end{equation}
Since $\Re(w_i)$ is bounded above, the exponential factor in \eqref{eq:big_O} is bounded above by a constant, so multiplying by the volume of $\Gamma_1'(s_n)^k$ we have that the integral \eqref{eq:split_integral1} is 
\begin{equation}
O(t^{-k\xi'(s_n)}(t^{\eta_k(s_n) - \xi'(s_n)} + t^{-\xi'(s_n)} \eps(s_n) + t^{-2\xi'(s_n)}/s_n)).
\end{equation}
This is $o(1)$ by our choice of $\xi'$, i.e. the limit \eqref{eq:split_integral1} is indeed $0$.

For the second line \eqref{eq:alpha_split_integral2}, we need an analogue of \Cref{thm:f_bound} for
\begin{equation}
\tf(w_1,\ldots,w_k)\prod_{\substack{1 \leq i \leq k, j \geq 1}} (1+t^{\zeta+L_k-\eps(s_n)}\alpha_j w_i/s_n)^{s_n}
\end{equation}
by combining \Cref{thm:tf_bound} with a bound on the product. By \eqref{eq:why_new_contour} and the elementary inequality
\begin{equation}
\left(1+\frac{x}{n}\right)^n \leq e^{x} \quad \quad \quad \text{for $x> -n$},
\end{equation}
we have
\begin{equation}
\left|\prod_{\substack{j \geq 1}} \left(1+\frac{t^{\zeta+L_k-\eps(s_n)}\alpha_j w_i}{s_n}\right)^s\right| \leq \prod_{\substack{j \geq 1}} e^{t^{\zeta+L_k-\eps(s_n)}\alpha_j(\Re(w_i)+1)} = e^{at^{\zeta+L_k-\eps_n}(\Re(w_i)+1)},
\end{equation}
for all large enough $n$ and $w_i \in \Gamma'(s_n)$. Combining with \Cref{thm:tf_bound} and dropping the unimportant $t^{-\eps(s_n)} = 1+o(1)$ yields
\begin{multline}\label{eq:new_f_bound}
\abs*{\tf(w_1,\ldots,w_k)\prod_{\substack{1 \leq i \leq k \\ j \geq 1}} (1+t^{\zeta+L_k-\eps(s_n)}\alpha_j w_i/s_n)^{s_n}} \\ 
\leq C\prod_{1 \leq i \leq k} e^{at^{\zeta+L_k}(\Re(w_i)+1) + \frac{k-1}{2}(\log t^{-1})\floor{\log_t |w_i|}^2 + c_2 \floor{\log_t |w_i|}}
\end{multline}
for all sufficiently large $n$. The rest of the proof of the vanishing of \eqref{eq:alpha_split_integral2} is the same as for \eqref{eq:split_integral2}, with \eqref{eq:new_f_bound} in place of \Cref{thm:f_bound}. The only other difference, which is very slight, is that above we were forced to choose contours with vertical portion passing closer to the poles $-t^\Z$ than before. However, the hypothesis of \Cref{thm:qpoc_lower_bound} (with $\delta$ as above) is still satisfied along the vertical portions.
\end{proof}

\begin{rmk}\label{rmk:what_goes_wrong_for_beta}
One may try to carry through the above argument with dual $\beta$ parameters in the specialization, but the Cauchy kernel 
\begin{equation}
\Pi_{0,t}(\beta(b_1,b_2,\ldots);\beta(z_1,\ldots,z_k)) = \Pi_{t,0}(b_1,\ldots;z_1,\ldots,z_k) = \frac{1}{\prod_{\substack{1 \leq i \leq k \\ j \geq 1}} (b_jz_i;t)_\infty}
\end{equation}
creates extra poles in the integrand, in contrast to the alpha and Plancherel cases treated in \Cref{thm:hl_stat_dist} and \Cref{thm:alpha_hl_stat_dist}. These probably can be dealt with, but we did not attempt to carry this out since we only need the alpha and Plancherel cases for our random matrix results.
\end{rmk}

We note a technical extension of \Cref{thm:alpha_hl_stat_dist} which we need for \Cref{thm:corner_product_bulk}. 

\begin{prop}\label{thm:alpha_hl_stat_dist_gen}
Let $k \in \Z_{\geq 1}$ and $\zeta \in \R$, and fix $\delta > 0$. Let $(\phi^{(n)})_{n \in \N}$ be sequence of pure alpha Hall-Littlewood nonnegative specializations, with $\phi^{(n)}$ having parameters $\alpha_1^{(n)} \geq \alpha_2^{(n)} \geq \ldots$, with $0 < \alpha_1^{(n)} < (1+\delta)^{-1}$. Let $(s_n)_{n \in \N}$ be any sequence with $s_n \xrightarrow{n \to \infty} \infty$ such that $\log_{t} (s_n p_1(\alpha_1^{(n)},\alpha_2^{(n)},\ldots))$ converges in $\R/\Z$, and let $\zeta$ be a lift of this limit to $\R$. Let $\tl(s_n)$ be distributed by the Hall-Littlewood measure
\begin{equation}
\Pr(\tl(s_n) = \la) = \frac{Q_\la(\alpha_1^{(n)}[s_n],\alpha_2^{(n)}[s_n],\ldots;0,t)P_\la(1,t,\ldots;0,t)}{\Pi_{0,t}(\alpha_1^{(n)}[s_n],\alpha_2^{(n)}[s_n],\ldots;1,t,\ldots)}.
\end{equation}
Then
\begin{equation}
(\tl_i'(s_n) - [\log_{t^{-1}}(s_n p_1(\alpha_1^{(n)},\alpha_2^{(n)},\ldots))+\zeta])_{1 \leq i \leq k} \to \cL_{k,t,t^\zeta}
\end{equation} 
where $[\cdot]$ is the nearest integer function. 
\end{prop}
\begin{proof}
The proof is essentially the same as that of \Cref{thm:alpha_hl_stat_dist}, which is a special case. The only place the values of the $\alpha_i$ mattered there was the choice of $\delta$ in the proof, which we have packaged into the hypotheses above. Note that $\zeta$ above corresponds to $\zeta + \log_t p_1(\alpha_1,\ldots)$ in \Cref{thm:alpha_hl_stat_dist}.
\end{proof}

\section{From processes of infinitely many particles to finite matrix bulk limits}\label{sec:pois_mat_prod}

In this section we finally prove the main results \Cref{thm:matrix_product_bulk_metric_intro} and \Cref{thm:corner_product_bulk_metric_intro}, as well as stating the main result \Cref{thm:matrix_stat_dist}. We will deduce \Cref{thm:matrix_product_bulk_metric_intro} and \Cref{thm:corner_product_bulk_metric_intro} from the following equivalent versions, which treat the issue of nonuniqueness of the limit random variable by passing to subsequences rather than considering a metric on probability measures as was done in \Cref{thm:matrix_product_bulk_metric_intro} and \Cref{thm:corner_product_bulk_metric_intro}.

\begin{thm}\label{thm:matrix_product_bulk}%this was the old main theorem, and most proofs will refer to it
Fix $p$ prime, and for each $N \in \mathbb{Z}_{\geq 1}$ let $A_{i}^{(N)}, i \geq 1$ be iid matrices with iid entries distributed by the additive Haar measure on $\mathbb{Z}_{p}$. Let $(s_{N})_{N \geq 1}$ be a sequence of natural numbers such that $s_{N}$ and $N-\log _{p} s_{N}$ both go to $\infty$ as $N \rightarrow \infty$. Let $(s_{N_{j}})_{j \geq 1}$ be any subsequence for which $-\log _{p} s_{N_{j}}$ converges in $\mathbb{R} / \mathbb{Z}$, and let $\zeta$ be any preimage in $\mathbb{R}$ of this limit. Then for any $k \in \Z_{\geq 1}$,
\begin{equation}\label{eq:rmt_bulk}
(\mathrm{SN}(A_{s_{N_{j}}}^{(N_{j})} \cdots A_{1}^{(N_{j})})_{i}'-[\log _{p}(s_{N_{j}})+\zeta])_{1 \leq i \leq k} \rightarrow \cL_{k,p^{-1}, p^{-\zeta}/(p-1)}
\end{equation}
in distribution as $j \rightarrow \infty$, where $[\cdot]$ is the nearest integer function and $\cL_{k,t,\chi}$ is as in \Cref{thm:stat_dist_1pt}.
\end{thm}

\begin{thm}\label{thm:corner_product_bulk} %old main theorem, most proofs still refer to it
Fix $p$ prime, and for each $N \in \mathbb{Z}_{\geq 1}$ let $D_N \in \Z_{\geq 1}$ be an integer and $A_{i}^{(N)}, i \geq 1$ be iid $N \times N$ corners of matrices distributed by the Haar probability measure on $\GL_{N+D_N}(\Z_p)$. Let $(s_{N})_{N \geq 1}$ be a sequence of natural numbers such that $s_{N}$ and $N-\log _{p} s_{N}$ both go to $\infty$ as $N \rightarrow \infty$. Let $(s_{N_{j}})_{j \geq 1}$ be any subsequence for which $-\log _{p} ((1-p^{-D_{N_j}})s_{N_{j}})$ converges in $\mathbb{R} / \mathbb{Z}$, and let $\zeta$ be any preimage in $\mathbb{R}$ of this limit. Then for any $k \in \Z_{\geq 1}$,
\begin{equation}\label{eq:corner_bulk_intro}
(\mathrm{SN}(A_{s_{N_{j}}}^{(N_{j})} \cdots A_{1}^{(N_{j})})_{i}'-[\log _{p}((1-p^{-D_{N_j}})s_{N_{j}})+\zeta])_{1 \leq i \leq k} \rightarrow \cL_{k,p^{-1}, p^{-\zeta}/(p-1)}
\end{equation}
in distribution as $j \rightarrow \infty$, where $[\cdot]$ is the nearest integer function and $\cL_{k,t,\chi}$ is as in \Cref{thm:stat_dist_1pt}.
\end{thm}

However, before proving these, we state the result \Cref{thm:matrix_stat_dist} mentioned in the Introduction, and deduce it from \Cref{thm:hl_stat_dist}. Recall the definitions of $\Pois^{(N)}$ and $\Pois^{(\infty)}$ from \Cref{def:poisson_walks}. As mentioned previously, the process $\Pois^{(N)}$ is the natural $p$-adic analogue of (multiplicative) Dyson Brownian motion in the complex case \cite{van2023p}, so the appearance of the same limit in \Cref{thm:matrix_product_bulk}, \Cref{thm:corner_product_bulk} and \Cref{thm:matrix_stat_dist} is not surprising from this perspective.
%For the next result, we work instead with $\Pois^{(N)}$, and we must give probabilistic arguments that $\Pois^{(N)}$ and $\Pois^{(\infty)}$ behave similarly enough in the limit in order to apply previous results as above.

\begin{thm}\label{thm:matrix_stat_dist}
Fix $\zeta \in \R$, and let $\tau_N \in t^{\zeta+\Z}, N \geq 1$ be a sequence of real numbers such that 
\begin{enumerate} 
\item $\tau_N \to \infty$ as $N \to \infty$, and
\item $N - \log_{t^{-1}} \tau_N \to \infty$.
\end{enumerate}
Then for any $k \in \Z_{\geq 1}$,
\begin{equation}
(\Pois^{(N)}(\tau_N)'_i - \log_{t^{-1}} \tau_N - \zeta)_{1 \leq i \leq k} \to \cL_{k,t,t^{\zeta+1}/(1-t)}
\end{equation}
in distribution, where $\Pois^{(N)}$ is as in \Cref{def:poisson_walks} and $\cL_{k,t,\chi}$ is as in \Cref{thm:stat_dist_1pt}.
\end{thm}

\begin{rmk}
A version of \Cref{thm:matrix_stat_dist} stated in terms of $D_\infty$-convergence in the sense of \Cref{thm:matrix_product_bulk_metric_intro}, rather than passing to subsequences, can be proven from \Cref{thm:matrix_stat_dist} exactly as \Cref{thm:matrix_product_bulk_metric_intro} is deduced from \Cref{thm:matrix_product_bulk} below. We omit the details.
\end{rmk}

%In this section we state the result \Cref{thm:matrix_stat_dist} mentioned in the Introduction, and deduce it from \Cref{thm:hl_stat_dist}. We then similarly use \Cref{thm:alpha_hl_stat_dist_gen} to deduce Theorems \ref{thm:matrix_product_bulk} and \ref{thm:corner_product_bulk}. We begin by stating the first theorem and making the first deduction, which is slightly easier technically but uses the same method as the argument for Theorems \ref{thm:matrix_product_bulk} and \ref{thm:corner_product_bulk}.

\subsection{Bulk limits for $\Pois^{(N)}(\tau)$: proof of \Cref{thm:matrix_stat_dist}.} The convergence of $\Pois^{(\infty)}$ to $\cL_{t,\chi}$ has essentially already been established, but let us show it anyway so that constant factors are kept track of and so that the differences in \Cref{thm:matrix_stat_dist} are more apparent.

\begin{prop}\label{thm:pois_infty_cvg}
Fix $\zeta \in \R$, and let $\tau_N \in t^{\zeta+\Z}, N \geq 1$ be a sequence of real numbers such that $\tau_N \to \infty$ as $N \to \infty$. Then for any $k \in \Z_{\geq 1}$,
\begin{equation}\label{eq:pois_infty_cvg}
(\Pois^{(\infty)}(\tau_N)'_i - \log_{t^{-1}} \tau_N - \zeta)_{1 \leq i \leq k} \to \cL_{k,t,t^{\zeta+1}/(1-t)}
\end{equation}
in distribution, where $\Pois^{(\infty)}$ is as in \Cref{def:poisson_walks} and $\cL_{k,t,\chi}$ is as in \Cref{thm:stat_dist_1pt}.
\end{prop}
\begin{proof}
By \Cref{thm:stat_dist_1pt} it suffices to show 
\begin{multline}\label{eq:pois_infty_suffices_cvg}
\lim_{N \to \infty} \Pr((\Pois^{(\infty)}(\tau_N)'_i - \log_{t^{-1}} \tau_N - \zeta)_{1 \leq i \leq k}  = (L_1,\ldots,L_k))  \\ 
=  \frac{(t;t)_\infty^{k-1}}{k! (2 \pi \bi)^k} \prod_{i=1}^{k-1} \frac{t^{\binom{L_i-L_k}{2}}}{(t;t)_{L_i-L_{i+1}}} \int_{\tG^k} e^{\frac{t^{L_k+\zeta+1}}{1-t}(w_1+\ldots+w_k)} \frac{\prod_{1 \leq i \neq j \leq k} (w_i/w_j;t)_\infty}{\prod_{i=1}^k (-w_i^{-1};t)_\infty (-tw_i;t)_{\infty}} \\ 
\times \sum_{j=0}^{L_{k-1}-L_k} t^{\binom{j+1}{2}} \sqbinom{L_{k-1}-L_k}{j}_t  P_{(L_1-L_k,\ldots,L_{k-1}-L_k,j)}(w_1^{-1},\ldots,w_k^{-1};t,0)  \prod_{i=1}^k \frac{dw_i}{w_i},
\end{multline}
where if $k=1$ we interpret the sum on the last line as in \Cref{thm:hl_stat_dist}. By \Cref{thm:HL_poisson}, $\Pois^{(\infty)}(\tau/t) = \la(\tau)$ where $\la(\tau)$ is as in \Cref{thm:hl_stat_dist}, and so \eqref{eq:pois_infty_cvg} is exactly the case of trivial initial condition in \Cref{thm:hl_stat_dist}.
\end{proof}

\begin{proof}[{Proof of \Cref{thm:matrix_stat_dist}}]
By \Cref{thm:stat_dist_1pt} it suffices to show that for any $k \in \Z_{\geq 1}$ and integers $L_1 \geq \ldots \geq L_k$,
\begin{multline}\label{eq:matrix_stat_dist}
\lim_{N \to \infty} \Pr((\Pois^{(N)}(\tau_N)'_i - \log_{t^{-1}} \tau_N - \zeta)_{1 \leq i \leq k}  = (L_1,\ldots,L_k))  \\ 
=  \frac{(t;t)_\infty^{k-1}}{k! (2 \pi \bi)^k} \prod_{i=1}^{k-1} \frac{t^{\binom{L_i-L_k}{2}}}{(t;t)_{L_i-L_{i+1}}} \int_{\tG^k} e^{\frac{t^{L_k+\zeta+1}}{1-t}(w_1+\ldots+w_k)} \frac{\prod_{1 \leq i \neq j \leq k} (w_i/w_j;t)_\infty}{\prod_{i=1}^k (-w_i^{-1};t)_\infty (-tw_i;t)_{\infty}} \\ 
\times \sum_{j=0}^{L_{k-1}-L_k} t^{\binom{j+1}{2}} \sqbinom{L_{k-1}-L_k}{j}_t  P_{(L_1-L_k,\ldots,L_{k-1}-L_k,j)}(w_1^{-1},\ldots,w_k^{-1};t,0)  \prod_{i=1}^k \frac{dw_i}{w_i},
\end{multline}
where if $k=1$ we interpret the sum on the last line as in \Cref{thm:hl_stat_dist}. We claim that it suffices to show 
\begin{multline}\label{eq:matrix_dbm_wts}
\lim_{N \to \infty} \Pr((\la^{(N)}(\tau_N)_i'- \log_{t^{-1}} \tau_N - \zeta)_{1 \leq i \leq k} = (L_1,\ldots,L_k))  \\ 
=  \frac{(t;t)_\infty^{k-1}}{k! (2 \pi \bi)^k} \prod_{i=1}^{k-1} \frac{t^{\binom{L_i-L_k}{2}}}{(t;t)_{L_i-L_{i+1}}} \int_{\tG^k} e^{\frac{t^{L_k+\zeta}}{1-t}(w_1+\ldots+w_k)} \frac{\prod_{1 \leq i \neq j \leq k} (w_i/w_j;t)_\infty}{\prod_{i=1}^k (-w_i^{-1};t)_\infty (-tw_i;t)_{\infty}} \\ 
\times \sum_{j=0}^{L_{k-1}-L_k} t^{\binom{j+1}{2}} \sqbinom{L_{k-1}-L_k}{j}_t  P_{(L_1-L_k,\ldots,L_{k-1}-L_k,j)}(w_1^{-1},\ldots,w_k^{-1};t,0)  \prod_{i=1}^k \frac{dw_i}{w_i}
\end{multline}
where $\la^{(N)}$ is as in \Cref{def:lambda_hl_planch}. By \Cref{thm:HL_poisson} 
\begin{equation}
\Pois^{(N)}(\tau) = \la^{(N)}\left(\frac{1-t^N}{t}\tau\right)
\end{equation}
in (multi-time) distribution, Since $N - \log_{t^{-1}} \tau_N \to \infty$ as $N \to \infty$, we have that 
\begin{equation}
\frac{1-t^N}{t}\tau_N = \frac{1}{t}\tau_N + o(1).
\end{equation}
Since $\la^{(N)}$ is a Poisson jump process with the exit rate from any state bounded above, if $(\la^{(N)}(\tau_N)_i'- \log_{t^{-1}}(\tau_N) - \zeta)_{1 \leq i \leq k}$ has a limiting distribution then 
\begin{equation}
(\la^{(N)}((1-t^N)\tau_N)_i'- \log_{t^{-1}}(\tau_N) - \zeta)_{1 \leq i \leq k} = (\Pois^{(N)}(\tau_N/t)_i'- \log_{t^{-1}}(\tau_N/t) - \zeta + 1)_{1 \leq i \leq k}
\end{equation}
must have the same limiting distribution. Hence \eqref{eq:matrix_stat_dist} (with $\tau_N$ replaced by $\tau_N/t$ and $\zeta$ by $\zeta+1$) follows from \eqref{eq:matrix_dbm_wts}, so it suffices to show the latter. The remainder of the proof consists of showing \eqref{eq:matrix_dbm_wts} by arguing that the Hall-Littlewood processes $\la^{(N)}(\tau)$ and $\la(\tau)$ (the latter of which was analyzed in \Cref{thm:hl_stat_dist}) are not so different on the timescale we consider.

Define stopping times
\begin{align}
\begin{split}
T^{(N)}_N &:= \inf(\{\tau \in \R_{\geq 0}: \la^{(N)}(\tau)_1' = N\}) \\ 
T_N &:= \inf(\{\tau \in \R_{\geq 0}: \la(\tau)_1' = N\}) \\ 
\Xi_N &:= \inf_{j \geq N+1} (\text{time at which clock }j\text{ rings for $\la(\tau)$}).
\end{split}
\end{align}
Conditionally on the event that clocks $N+1,N+2,\ldots$ do not ring on a given time interval, both $\la^{(N)}_i(\tau), 1 \leq i \leq N$ and $\la_i(\tau), 1 \leq i \leq N$ have the same local dynamics controlled by $N$ Poisson clocks on that interval, by \Cref{thm:HL_poisson}. Taking the time interval $[0,\tau_N]$, since $\min(\tau_N,T_N^{(N)})$ and $\min(\tau_N,T_N)$ are measurable with respect to the $\sigma$-algebras generated by $\la^{(N)}([0,\tau_N])$ and $\la([0,\tau_N])$ respectively, this implies that
\begin{equation}\label{eq:min_law_eq}
\Law(\min(\tau_N,T_N^{(N)})) = \Law(\min(\tau_N,T_N)| \Xi_N > \tau_N)
\end{equation}
and 
\begin{equation}\label{eq:dist_condl_eq}
\Law((\la^{(N)}(\tau_N)_1',\ldots,\la^{(N)}(\tau_N)_k')|T^{(N)}_N > \tau_N) = \Law((\la(\tau_N)_1',\ldots,\la(\tau_N)_k')|T_N > \tau_N \text{ and }\Xi_N > \tau_N).
\end{equation}
The explicit description of our dynamics implies the distributional equality
\begin{equation}
\Law(T_N) = \Law\left(\sum_{i=0}^{N-1} Y_{t^i/(1-t)}\right)
\end{equation}
where $Y_{r}$ is an exponential distribution with rate $r$. Because $N - \log_{t^{-1}} \tau_N \to \infty$, 
\begin{equation}
\E[Y_{t^{N-1}/(1-t)}] = (1-t)t^{1-N} \gg \tau_N,
\end{equation}
and the fluctuations of $Y_{t^{N-1}/(1-t)}$ are of lower order than its mean, hence
\begin{equation}\label{eq:stop_time_big}
 \lim_{N \to \infty} \Pr(T_N > \tau_N) = 1.
\end{equation}
Furthermore, since the first time one of the clocks $N+1,N+2,\ldots$ rings follows an exponential distribution with rate $t^{N+1}/(1-t)$, the hypothesis $N - \log_{t^{-1}} \tau_N \to \infty$ is exactly what is needed to guarantee that the probability that any of clocks $N+1,N+2,\ldots$ rings on the interval we are concerned with is asymptotically negligible, i.e.
\begin{equation}\label{eq:Xi_big}
\lim_{N \to \infty} \Pr(\Xi_N \leq \tau_N) = 0.
\end{equation}
From \eqref{eq:stop_time_big} and \eqref{eq:Xi_big} it follows that
\begin{equation}\label{eq:tau_N_small_enough}
\lim_{N \to \infty} \Pr(T_N > \tau_N \text{ and }\Xi_N > \tau_N) = 1.
\end{equation}
We thus finally have
\begin{align}\label{eq:TNN_big}
\begin{split}
\lim_{N \to \infty} \Pr(T^{(N)}_N > \tau_N) &= \lim_{N \to \infty} \Pr(T_N > \tau_N | \Xi_N > \tau_N) \\ 
&= \lim_{N \to \infty} \frac{\Pr(T_N > \tau_N \text{ and }\Xi_N > \tau_N)}{\Pr(\Xi_N > \tau_N)} \\ 
&= 1,
\end{split}
\end{align}
where the first equality follows from \eqref{eq:min_law_eq}, the second is trivial, and the third follows by applying \eqref{eq:Xi_big} and \eqref{eq:TNN_big} to the denominator and numerator respectively. From \eqref{eq:TNN_big} it follows that 
\begin{align}
\begin{split}\label{eq:reduce_to_cond_exp}
\text{LHS\eqref{eq:matrix_dbm_wts}} &= \lim_{N \to \infty} \Pr((\la^{(N)}(\tau_N)'_i - \log_{t^{-1}}(\tau_N))_{1 \leq i \leq k}  = (L_i+\zeta)_{1 \leq i \leq k}| T_N^{(N)} > \tau_N) \cdot \Pr(T_N^{(N)} > \tau_N) \\ 
&+ \lim_{N \to \infty} \Pr((\la^{(N)}(\tau_N)'_i - \log_{t^{-1}}(\tau_N))_{1 \leq i \leq k}  = (L_i+\zeta)_{1 \leq i \leq k}| T_N^{(N)} \leq \tau_N) \cdot \Pr(T_N^{(N)} \leq \tau_N) \\ 
&= \lim_{N \to \infty} \Pr((\la^{(N)}(\tau_N)'_i - \log_{t^{-1}}(\tau_N))_{1 \leq i \leq k}  = (L_i+\zeta)_{1 \leq i \leq k}| T_N^{(N)} > \tau_N).
\end{split}
\end{align}
By \eqref{eq:min_law_eq},
\begin{align}
\begin{split}\label{eq:another_condl_split}
\text{RHS\eqref{eq:reduce_to_cond_exp}}& = \lim_{N \to \infty} \Pr((\la(\tau_N)'_i - \log_{t^{-1}}(\tau_N))_{1 \leq i \leq k}  = (L_i+\zeta)_{1 \leq i \leq k}| T_N > \tau_N \text{ and }\Xi_N > \tau_N) \\ 
&= \lim_{N \to \infty} \frac{1}{\Pr(T_N > \tau_N \text{ and }\Xi_N > \tau_N)} \left(\Pr((\la(\tau_N)'_i - \log_{t^{-1}}(\tau_N))_{1 \leq i \leq k}  = (L_i+\zeta)_{1 \leq i \leq k} ) \right.\\ 
&\left.- \Pr((\la(\tau_N)'_i - \log_{t^{-1}}(\tau_N))_{1 \leq i \leq k}  = (L_i+\zeta)_{1 \leq i \leq k}\text{ and } (T_N \leq \tau_N \text{ or }\Xi_N \leq \tau_N)\right).
\end{split}
\end{align}
Since $\Pr(T_N > \tau_N \text{ and }\Xi_N > \tau_N) = 1-o(1)$ by \eqref{eq:tau_N_small_enough}, we have 
\begin{equation}
\text{RHS\eqref{eq:another_condl_split}} = \lim_{N \to \infty}\Pr((\la(\tau_N)'_i - \log_{t^{-1}}(\tau_N))_{1 \leq i \leq k}  = (L_i+\zeta)_{1 \leq i \leq k} ).
\end{equation}
By \Cref{thm:hl_stat_dist}, the above is equal to the right hand side of \eqref{eq:matrix_stat_dist}, and this completes the proof.
\end{proof}

\subsection{Products of random matrices: proofs of \Cref{thm:matrix_product_bulk} and \Cref{thm:corner_product_bulk}.} Before proving \Cref{thm:matrix_product_bulk} we must also give a more explicit description of the Hall-Littlewood processes arising from matrix products via \Cref{thm:hl_meas_matrices}. This description plays the role of the explicit dynamics of \Cref{def:poisson_walks} in the previous proof. It was given previously in \cite[Proposition 5.3]{van2020limits} for finite principal specializations (and slightly different conventions on Hall-Littlewood polynomials); the background and statement below have been only slightly modified from \cite{van2020limits} to include the infinite principal specialization case.

\begin{defi}\label{def:interacting_insertion}
For $n \in \N$, define the insertion map
\begin{align}
\begin{split}
\iota^{(n)}: \Z_{\geq 0}^n \times \Sig_n & \to \Sig_n \\ 
((a_1,\ldots,a_n),\la) &\mapsto (\iota^{(n)}(a_1,\ldots,a_n;\la)_i)_{1 \leq i \leq n}
\end{split}
\end{align}
 where the parts of $\iota^{(n)}(a_1,\ldots,a_n; \la)$ are given explicitly by
\begin{align}\label{eq:iota_uniform_def}
\iota^{(n)}(a_1,\ldots,a_n; \la)_{i} &:= \min(\la_{i-1}, \max(\la_{i}+a_{i},\la_{i+1}+a_{i}+a_{i+1},\ldots,\la_n+a_{i}+\ldots+a_n))
\end{align}
for each $i=1,\ldots,n$, where we formally take $\la_{i-1} = \la_0 = \infty$ in the edge case $i=1$. For $\vec{a} = (a_i)_{i \in \N} \in \Z_{\geq 0}^\N$ with $\sum_i a_i < \infty$ and $\la \in \Y$, we further define $\iota^{(\infty)}$ by
\begin{equation}
\iota^{(\infty)}(\vec{a};\la)_i = \min(\la_{i-1}, \max_{j \geq i}(\la_j + \sum_{\ell=i}^j a_i))
\end{equation}
for any $i \in \Z_{\geq 1}$, with the same convention $\la_0=\infty$. We will sometimes write $\iota$ without the superscript when it is clear from context.
\end{defi}

The following algorithm is equivalent and was stated as the definition of $\iota$ in \cite{van2020limits}. Here and in the rest of this subsection, we will identify signatures $\la \in \Sig_n$ (resp. $\la \in \Y$) with configurations of $n$ (resp. $\infty$) particles on $\Z$ by placing $m_i(\la)$ particles at each position $i \in \Z$. Each particle corresponds to a part of $\la$, and we will refer to them as the $1^{\text{st}},\ldots,n\tth$ particle or `particle $1,\ldots,$ particle $n$' to reflect this, even when some are in the same location. In this numbering, particle $j$ will correspond to a particle at position $\la_j$. 

\begin{prop}\label{thm:sampling_alg_def_iota}
For any $n \in \N \cup \{\infty\}$, the above definition of $\iota^{(n)}$ is equivalent to the following. First assign to each particle $j$ an `impulse' $a_j$, and let $\ell = \max(i \in \Z_{\geq 1}: a_i > 0)$ (which is finite by hypothesis even if $n=\infty$). Particle $\ell$ then moves to the right until it has either moved $a_\ell$ steps or encountered particle $\ell-1$. If it encounters particle $\ell-1$, then it is `blocked' by particle $\ell-1$ and donates the remainder $a_\ell - (\la_{\ell-1}-\la_\ell)$ of its impulse to particle $\ell-1$. Particle $\ell-1$ now has impulse $a_{\ell-1} + \max(0, a_\ell - (\la_{\ell-1}-\la_\ell))$, and moves in the same manner, possibly donating some of its impulse to particle $\ell-2$; all further particle evolve in the same manner.
\end{prop}
\begin{proof}
Follows from the definitions.
\end{proof}

\begin{example}\label{ex:particle_evolution}
To compute $\iota(1,4,2; (5,3,-1)) = (8,5,1)$ the particles jump as above. The numbers above the particles represent their impulses; note that impulse-donation from particle $2$ to particle $1$ occurs at the third step shown.
\vspace{3ex}
\begin{center}
\begin{tikzpicture}[
dot/.style = {circle, fill, minimum size=#1,
              inner sep=0pt, outer sep=0pt},
dot/.default = 6pt  % size of the circle diameter 
                    ] %dot commands taken from https://tex.stackexchange.com/questions/445946/how-set-tikz-circle-radius-in-nodecircle
\def\h{0}
\draw[latex-latex] (-3.5,\h) -- (9.5,\h) ; %edit here for the axis
\foreach \x in  {-3,-2,-1,0,1,2,3,4,5,6,7,8,9} % edit here for the vertical lines
\draw[shift={(\x,\h)},color=black] (0pt,3pt) -- (0pt,-3pt);
\foreach \x in {-3,-2,-1,0,1,2,3,4,5,6,7,8,9} % edit here for the numbers
\draw[shift={(\x,\h)},color=black] (0pt,0pt) -- (0pt,-3pt) node[below] 
{$\x$};
\node[dot,label=above:$2$] at (-1,\h) {};
\node[dot,label=above:$4$] at (3,\h) {};
\node[dot,label=above:$1$] at (5,\h) {};

\def\h{-2}
\draw[latex-latex] (-3.5,\h) -- (9.5,\h) ; %edit here for the axis
\foreach \x in  {-3,-2,-1,0,1,2,3,4,5,6,7,8,9} % edit here for the vertical lines
\draw[shift={(\x,\h)},color=black] (0pt,3pt) -- (0pt,-3pt);
\foreach \x in {-3,-2,-1,0,1,2,3,4,5,6,7,8,9} % edit here for the numbers
\draw[shift={(\x,\h)},color=black] (0pt,0pt) -- (0pt,-3pt) node[below] 
{$\x$};
\node[dot,label=above:$0$] at (1,\h) {};
\node[dot,label=above:$4$] at (3,\h) {};
\node[dot,label=above:$1$] at (5,\h) {};
\draw [->] (-1,\h) to [out=30,in=150] ((1-.15,\h+.15);

\def\h{-4}
\draw[latex-latex] (-3.5,\h) -- (9.5,\h) ; %edit here for the axis
\foreach \x in  {-3,-2,-1,0,1,2,3,4,5,6,7,8,9} % edit here for the vertical lines
\draw[shift={(\x,\h)},color=black] (0pt,3pt) -- (0pt,-3pt);
\foreach \x in {-3,-2,-1,0,1,2,3,4,5,6,7,8,9} % edit here for the numbers
\draw[shift={(\x,\h)},color=black] (0pt,0pt) -- (0pt,-3pt) node[below] 
{$\x$};
\node[dot,label=above:$0$] at (1,\h) {};
\node[dot,label=above:$0$] at (5-.25,\h) {};
\node[dot,label=above:$3$] at (5,\h) {};
\draw [->] (3,\h) to [out=30,in=150] ((5-.15-.25,\h+.15);

\def\h{-6}
\draw[latex-latex] (-3.5,\h) -- (9.5,\h) ; %edit here for the axis
\foreach \x in  {-3,-2,-1,0,1,2,3,4,5,6,7,8,9} % edit here for the vertical lines
\draw[shift={(\x,\h)},color=black] (0pt,3pt) -- (0pt,-3pt);
\foreach \x in {-3,-2,-1,0,1,2,3,4,5,6,7,8,9} % edit here for the numbers
\draw[shift={(\x,\h)},color=black] (0pt,0pt) -- (0pt,-3pt) node[below] 
{$\x$};
\node[dot,label=above:$0$] at (1,\h) {};
\node[dot,label=above:$0$] at (5,\h) {};
\node[dot,label=above:$0$] at (8,\h) {};
\draw [->] (5+.25,\h) to [out=30,in=150] ((8-.15,\h+.15);
\end{tikzpicture}
\end{center}
\vspace{3ex}
\end{example}

The following fact, related to the consistency property of Hall-Littlewood polynomials under setting variables to $0$, is useful for comparing different $\iota^{(n)}$.

\begin{prop}\label{thm:iotas_consistent}
For any $\la \in \Y_n$, 
\begin{equation}
\iota^{(n)}(a_1,\ldots,a_n;\la) = \iota^{(n+1)}(a_1,\ldots,a_n,0;\la) = \iota^{(\infty)}(a_1,\ldots,a_n,0,0,\ldots;\la).
\end{equation}
\end{prop}
\begin{proof}
Obvious from either \Cref{def:interacting_insertion} or \Cref{thm:sampling_alg_def_iota}.
\end{proof}

It is obvious from \Cref{thm:sampling_alg_def_iota} that the interlacing 
\begin{equation}
 \iota(a_1,\ldots,a_n; \la) \geq \la_1 \geq \iota(a_1,\ldots,a_n; \la)_2 \geq \ldots \geq \iota(a_1,\ldots,a_n; \la)_n \geq \la_n
\end{equation}
holds for any $\ba \in \Z_{\geq 0}^n$. We now use the insertion $\iota$ with random inputs $a_i$ to define random signatures, which we will show in \Cref{thm:alpha_sampling_alg} yields the `Cauchy' Markov transition dynamics of \Cref{def:cauchy_dynamics}. First we define the measures which will be the distributions of the $a_i$. 

\begin{defi}\label{def:geom_diff_rvs}
Let $G_x$ be the measure on $\Z_{\geq 0}$ which is the distribution of $\max(X-T,0)$ where $X \sim \Geom(x), T \sim \Geom(t)$. Explicitly,
\begin{equation}\label{eq:trunc_geom_formula}
    G_x(\ell) = \frac{1-x}{1-t x}(1-t)^{\bbone(\ell>0)}x^\ell.
\end{equation}
Equivalently $G_x$ is defined by the generating function 
\begin{equation}\label{eq:trunc_geom_pgf}
    \sum_{\ell \geq 0} G_x(\ell) z^\ell = \frac{1-x}{1-t x} \frac{1-t x z}{1-x z} = \frac{\Pi_{0,t}(z;x)}{\Pi_{0,t}(1;x)}.
\end{equation}
\end{defi}

\begin{prop}\label{thm:alpha_sampling_alg}
Fix $n \in \N \cup \{\infty\}$ and $x \in (0,1)$, and let $X_i, 1 \leq i \leq n$ be independent with $X_i \sim G_{xt^{i-1}}$. Let $\la,\nu \in \Y_n$ with $\la \prec \nu$. Then
\begin{align}\label{eq:sampling_alg}
    \Pr(\iota(X_1,\ldots,X_n;\la) = \nu) &= \frac{1-x}{1-t^nx} \prod_{j: m_j(\la)=m_j(\nu)+1}(1-t^{m_j(\la)}) \prod_{i=1}^n (xt^{i-1})^{\nu_i-\la_i} \\
    &=  \frac{Q_{\nu/\la}(x;0,t)P_\nu(1,\ldots,t^{n-1};0,t)}{P_\la(1,\ldots,t^{n-1};0,t) \Pi_{0,t}(x;1,\ldots,t^{n-1})},\label{eq:sampling_alg2}
\end{align}
where we take $t^nx = 0$ in the $n=\infty$ case.
\end{prop}

\begin{proof}
For $n \in \N$ this is exactly \cite[Proposition 5.3]{van2020limits}, and we will deduce the $n=\infty$ case from this. It is easy to check from \Cref{def:geom_diff_rvs} that 
\begin{equation}\label{eq:likely_high_X_0}
\lim_{n \to \infty} \Pr(0=X_n=X_{n+1}=\ldots) = 1.
\end{equation}
Hence with probability $1$, all but finitely many of the $X_i$ are $0$, so indeed $\iota(X_1,X_2,\ldots;\la)$ is well-defined with probability $1$. By \Cref{thm:iotas_consistent},
\begin{align}\label{eq:split_prob}
\begin{split}
&\Pr(\iota^{(\infty)}(X_1,X_2,\ldots;\la) = \nu)  \\ 
&= \Pr(\iota^{(\infty)}(X_1,X_2,\ldots;\la) = \nu| 0 = X_{n+1} = X_{n+2} = \ldots) \cdot \Pr(0=X_{n+1}=X_{n+2}=\ldots) \\ 
&+ \Pr\left(\iota^{(\infty)}(X_1,X_2,\ldots;\la) = \nu \left| \sum_{i=1}^\infty X_{n+i} > 0\right.\right) \cdot \Pr\left(\sum_{i=1}^\infty X_{n+i} > 0\right)\\ 
&= \Pr(\iota^{(n)}(X_1,\ldots,X_n;\la) = \nu)\cdot  \Pr(0=X_{n+1}=X_{n+2}=\ldots)  \\ 
&+ \Pr\left(\iota^{(\infty)}(X_1,X_2,\ldots;\la) = \nu\left| \sum_{i=1}^\infty X_{n+i} > 0\right.\right) \cdot \Pr\left(\sum_{i=1}^\infty X_{n+i} > 0\right)
\end{split}
\end{align}
By \eqref{eq:likely_high_X_0} and \eqref{eq:split_prob}, 
\begin{equation}\label{eq:n_infty_sampling}
\lim_{n \to \infty}\Pr(\iota^{(n)}(X_1,\ldots,X_n;\la) = \nu) = \Pr(\iota^{(\infty)}(X_1,X_2,\ldots;\la) = \nu).
\end{equation}
But applying the finite $n$ case, we have
\begin{align}\label{eq:hl_convergence}
\begin{split}
\lim_{n \to \infty}\Pr(\iota^{(n)}(X_1,\ldots,X_n;\la) = \nu) &=\lim_{n \to \infty} \frac{Q_{\nu/\la}(x;0,t)P_\nu(1,\ldots,t^{n-1};0,t)}{P_\la(1,\ldots,t^{n-1};0,t) \Pi_{0,t}(x;1,\ldots,t^{n-1})} \\ 
&= \frac{Q_{\nu/\la}(x;0,t)P_\nu(1,t,\ldots;0,t)}{P_\la(1,t,\ldots;0,t) \Pi_{0,t}(x;1,t,\ldots)},
\end{split}
\end{align}
where the second line follows for instance by the explicit formulas for principal specializations (\Cref{thm:hl_principal_formulas}) and for the Cauchy kernel \eqref{eq:def_cauchy_kernel}. Combining \eqref{eq:n_infty_sampling} and \eqref{eq:hl_convergence} completes the proof.
\end{proof}

\begin{proof}[Proof of {\Cref{thm:matrix_product_bulk}}]
To control subscripts we abuse notation and write $N$ for $N_j$ below, so all limits should be interpreted as along our subsequence $(N_j)_{j \geq 1}$. By \Cref{thm:stat_dist_1pt} it suffices to show 
\begin{multline}\label{eq:unif_matrix_products}
\lim_{N \to \infty} \Pr((\SN(A^{(N)}_{s_N} \cdots A^{(N)}_1)_i' - [\log_{t^{-1}}(s)+\zeta])_{1 \leq i \leq k} = (L_1,\ldots,L_k)) \\ 
=  \frac{(t;t)_\infty^{k-1}}{k! (2 \pi \bi)^k} \prod_{i=1}^{k-1} \frac{t^{\binom{L_i-L_k}{2}}}{(t;t)_{L_i-L_{i+1}}} \int_{\tG^k} e^{\frac{t^{L_k+\zeta+1}}{1-t}(w_1+\ldots+w_k)} \frac{\prod_{1 \leq i \neq j \leq k} (w_i/w_j;t)_\infty}{\prod_{i=1}^k (-w_i^{-1};t)_\infty (-tw_i;t)_{\infty}} \\ 
\times \sum_{j=0}^{L_{k-1}-L_k} t^{\binom{j+1}{2}} \sqbinom{L_{k-1}-L_k}{j}_t  P_{(L_1-L_k,\ldots,L_{k-1}-L_k,j)}(w_1^{-1},\ldots,w_k^{-1};t,0)  \prod_{i=1}^k \frac{dw_i}{w_i},
\end{multline}
where if $k=1$ we interpret the sum on the last line as in \Cref{thm:hl_stat_dist}. For any $s \in \N$, let $\tl(s)$ be a Hall-Littlewood measure with one specialization $1,t,\ldots$ and one $\alpha(t,t^2,\ldots)[s]$. Then \Cref{thm:alpha_hl_stat_dist} applies with $a=t+t^2+\ldots=t/(1-t)$, yielding
\begin{multline}\label{eq:infinite_matrix_product_works}
\lim_{\substack{N \to \infty}} \Pr(\tl_i'(s_N) - [\log_{t^{-1}}(s_N)+\zeta]  = L_i \text{ for all }1 \leq i \leq k) \\ 
=  \frac{(t;t)_\infty^{k-1}}{k! (2 \pi \bi)^k} \prod_{i=1}^{k-1} \frac{t^{\binom{L_i-L_k}{2}}}{(t;t)_{L_i-L_{i+1}}} \int_{\tG^k} e^{\frac{t^{L_k+\zeta+1}}{1-t}(w_1+\ldots+w_k)} \frac{\prod_{1 \leq i \neq j \leq k} (w_i/w_j;t)_\infty}{\prod_{i=1}^k (-w_i^{-1};t)_\infty (-tw_i;t)_{\infty}} \\ 
\times \sum_{j=0}^{L_{k-1}-L_k} t^{\binom{j+1}{2}} \sqbinom{L_{k-1}-L_k}{j}_t  P_{(L_1-L_k,\ldots,L_{k-1}-L_k,j)}(w_1^{-1},\ldots,w_k^{-1};t,0)  \prod_{i=1}^k \frac{dw_i}{w_i}.
\end{multline}
By \Cref{thm:hl_meas_matrices}, $\SN(A_s^{(N)} \cdots A_1^{(N)}), s \geq 0$ is a Hall-Littlewood process $\tl^{(N)}(s), s \in \Z_{\geq 0}$ with transition probabilities
\begin{equation}
\Pr(\tl^{(N)}(s+1) = \nu|\tl^{(N)}(s)=\kappa) = Q_{\nu/\kappa}(t,t^2,\ldots;0,t) \frac{P_\nu(1,\ldots,t^{N-1};0,t)}{\Pi_{0,t}(t,t^2,\ldots;1,\ldots,t^{N-1})P_\kappa(1,\ldots,t^{N-1};0,t)}
\end{equation}
(and initial condition $\emptyset \in \Y$). By \Cref{thm:alpha_sampling_alg}, both $\tl$ and $\tl^{(N)}$ have a sampling algorithm for which we briefly recall the important points. First, the random step $\tl(s) \mapsto \tl(s+1)$ involves an infinite number of substeps, indexed by the alpha variables $t,t^2,\ldots$, of which with probability $1$ only finitely many are nontrivial. Second, each such substep involves sampling random variables $X_1,\ldots,X_N$ (for $\tl^{(N)}$) or $X_1,X_2,\ldots$ (for $\tl$) and applying an `insertion map' 
\begin{equation}
\tl^{(N)}(s+1) = \iota^{(N)}(X_1,\ldots,X_N;\tl^{(N)}(s))
\end{equation}
(for $\tl^{(N)}$) or 
\begin{equation}
\tl(s+1) = \iota^{(\infty)}(X_1,X_2,\ldots;\tl(s))
\end{equation}
(for $\tl$). Now define stopping times
\begin{align}
\begin{split}
T^{(N)}_N &:= \min(\{s \in \Z_{\geq 0}: \tl^{(N)}(s)_1' = N\}) \\ 
T_N &:= \min(\{s \in \Z_{\geq 0}: \tl(s)_1' = N\}) \\ 
\Xi_N &:= \min\{s \in \Z_{\geq 0}: \text{at some substep of $\tl(s) \mapsto \tl(s+1)$, $\max_{j \geq N+1} X_j \geq 1$}\}
\end{split}
\end{align}
The rest of the proof proceeds exactly as for \Cref{thm:matrix_stat_dist} by first showing that with probability converging to $1$, the variables $X_{N+1},X_{N+2},\ldots$ will all be $0$ with high probability for all steps $s=1,2,\ldots,s_N$. Hence by \Cref{thm:iotas_consistent}, $\tl(s_N) = \tl^{(N)}(s_N)$ with probability going to $1$. The proof then finishes as before, using \eqref{eq:infinite_matrix_product_works} in place of \eqref{eq:matrix_dbm_wts}. 
\end{proof}

\begin{proof}[Proof of \Cref{thm:corner_product_bulk}]
The proof is essentially the same as the previous one. \Cref{thm:hl_meas_matrices} yields that $\SN(A_s^{(N)} \cdots A_1^{(N)}), s \geq 0$ is a Hall-Littlewood process (which we abuse notation and also denote by $\tl^{(N)}$) with transition probabilities
\begin{equation}
\Pr(\tl^{(N)}(s+1) = \nu|\tl^{(N)}(s)=\kappa) = Q_{\nu/\kappa}(t,\ldots,t^{D_N};0,t) \frac{P_\nu(1,\ldots,t^{N-1};0,t)}{\Pi_{0,t}(t,\ldots,t^{D_N};1,\ldots,t^{N-1})P_\kappa(1,\ldots,t^{N-1};0,t)}.
\end{equation}
Now, $D_N$ may vary with $N$, but we apply the more general \Cref{thm:alpha_hl_stat_dist_gen} with
\begin{equation}
\alpha_i^{(N)} = \begin{cases}
t^i & 1 \leq i \leq D_N \\ 
0 & i > D_N 
\end{cases}
\end{equation}
and 
\begin{equation}\label{eq:p1_corners}
p_1(\alpha_1^{(N)},\ldots) = \frac{t-t^{D_N+1}}{1-t}.
\end{equation}
By hypothesis we have some $\zeta \in \R$ for which
\begin{equation}
\log_t((1-t^{D_{N_j}})s_{N_j}) \to \zeta \quad \quad \quad \text{in $\R/\Z$, as $j \to \infty$.}
\end{equation}
Hence letting $\tzeta = \zeta + 1 - \log_t(1-t)$, by \eqref{eq:p1_corners} we have
\begin{equation}
\log_t (p_1(\alpha_1^{(N_j)},\ldots)s_{N_j}) \to \tzeta \quad \quad \quad \text{in $\R/\Z$, as $j \to \infty$.}
\end{equation}
By \Cref{thm:alpha_hl_stat_dist_gen} with $\tzeta$ playing the role of $\zeta$ in that theorem, we therefore have
\begin{multline}\label{eq:infinite_matrix_product_works_corners}
\lim_{\substack{N \to \infty}} \Pr(\tl_i'(s_N) - [\log_{t^{-1}}(s_N (t-t^{D_N+1})/(1-t))+\tzeta]  = L_i \text{ for all }1 \leq i \leq k) \\ 
=  \frac{(t;t)_\infty^{k-1}}{k! (2 \pi \bi)^k} \prod_{i=1}^{k-1} \frac{t^{\binom{L_i-L_k}{2}}}{(t;t)_{L_i-L_{i+1}}} \int_{\tG^k} e^{t^{L_k+\tzeta}(w_1+\ldots+w_k)} \frac{\prod_{1 \leq i \neq j \leq k} (w_i/w_j;t)_\infty}{\prod_{i=1}^k (-w_i^{-1};t)_\infty (-tw_i;t)_{\infty}} \\ 
\times \sum_{j=0}^{L_{k-1}-L_k} t^{\binom{j+1}{2}} \sqbinom{L_{k-1}-L_k}{j}_t  P_{(L_1-L_k,\ldots,L_{k-1}-L_k,j)}(w_1^{-1},\ldots,w_k^{-1};t,0)  \prod_{i=1}^k \frac{dw_i}{w_i},
\end{multline}
where as in the proof of \Cref{thm:matrix_product_bulk} we interpret the limit as along our subsequence of $N_j$'s to avoid writing more subscripts. Rewriting \eqref{eq:infinite_matrix_product_works_corners} in terms of $\zeta$ completes the proof.
\end{proof}

\subsection{From subsequential convergence to $D_\infty$-convergence: proofs of \Cref{thm:matrix_product_bulk_metric_intro} and \Cref{thm:corner_product_bulk_metric_intro}.}
\begin{proof}[Proof of {\Cref{thm:matrix_product_bulk_metric_intro}}]
Suppose for the sake of contradiction that \eqref{eq:haar_metric_cvg} does not hold. Then there exists some $\eps > 0$, $k \in \Z_{\geq 1}$ and some subsequence $(N_j)_{j \geq 1}$ of $\Z_{\geq 1}$ such that 
\begin{equation}\label{eq:distances_too_large}
D_\infty\left((X_{N_j}^{(i)})_{1 \leq i \leq k}, (\cL^{(i)}_{p^{-1},p^{r(j)}/(p-1)})_{1 \leq i \leq k}\right) > \eps
\end{equation}
for all $j$, where we use notation $r(j) := \{\log_p s_{N_j}\}$ to control subscripts. Since the fractional parts $\{\log_p s_{N_j}\}$ always lie in the compact set $[0,1]$, there is some $\zeta\in [-1,0]$ and further subsequence $(\tN_j)_{j \geq 1}$ of $(N_j)_{j \geq 1}$ such that 
\begin{equation}\label{eq:frac_zeta_cvg}
\lim_{j \to \infty} \{\log_p s_{\tN_j}\} = -\zeta,
\end{equation}
and in particular $-\log_p s_{\tN_j}$ converges to $\zeta$ in $\R/\Z$. Hence by \Cref{thm:matrix_product_bulk}, for all $k \geq 1$ we have
\begin{equation}\label{eq:apply_matrix_bulk}
((\mathrm{SN}(A_{s_{\tN_{j}}}^{(\tN_{j})} \cdots A_{1}^{(\tN_{j})})_{i}'-[\log _{p}(s_{\tN_{j}})+\zeta])_{1 \leq i \leq k}  \to (\cL^{(i)}_{p^{-1},p^{-\zeta}/(p-1)})_{1 \leq i \leq k}
\end{equation}
in distribution as $j \to \infty$,. By \eqref{eq:frac_zeta_cvg}, $[\log _{p}(s_{\tN_{j}})+\zeta] = \floor{\log _{p}(s_{\tN_{j}})}$ for all $j$ sufficiently large, hence \eqref{eq:apply_matrix_bulk} implies that for all $k \geq 1$,
\begin{equation}\label{eq:apply_matrix_bulk2}
(X_{\tN_j}^{(i)})_{1 \leq i \leq k} \to (\cL^{(i)}_{p^{-1},p^{-\zeta}/(p-1)})_{1 \leq i \leq k}
\end{equation}
in distribution as $j \to \infty$, where $X_N^{(i)}$ is as in the statement of the theorem being proven. Equivalently,
\begin{equation}\label{eq:D_to_zero_zeta}
\lim_{j \to \infty} D_\infty\left((X_{\tN_j}^{(i)})_{1 \leq i \leq k},(\cL^{(i)}_{p^{-1},p^{-\zeta}/(p-1)})_{1 \leq i \leq k}\right) = 0.
\end{equation}
The integral representation in \Cref{thm:stat_dist_1pt} and the integrand bound \Cref{thm:f_bound} together imply that for each $\vec{L} \in \Sig_k$, the probability
\begin{equation}
\Pr((\cL^{(i)}_{p^{-1},p^{-\zeta}/(p-1)})_{1 \leq i \leq k} = \vec{L})
\end{equation}
depends continuously on $\zeta$. Hence by \eqref{eq:frac_zeta_cvg}, writing $\tilde{r}(j) := \{\log_p s_{\tN_j}\}$ similarly to before, we have
\begin{equation}\label{eq:probs_unif_cont}
\lim_{j \to \infty} D_\infty\left((\cL^{(i)}_{p^{-1},p^{\tilde{r}(j)}/(p-1)})_{1 \leq i \leq k},(\cL^{(i)}_{p^{-1},p^{-\zeta}/(p-1)})_{1 \leq i \leq k}\right) = 0
\end{equation}
(this requires uniform continuity of the probabilities over all $\vec{L}$, but this follows from the stated continuity of each individual probability since the sum of probabilities is $1$). The triangle inequality for $D_\infty$ and the equations \eqref{eq:D_to_zero_zeta}, \eqref{eq:probs_unif_cont} thus imply
\begin{align}
\begin{split}
\lim_{j \to \infty} D_\infty\left((X_{\tN_j}^{(i)})_{1 \leq i \leq k}, (\cL^{(i)}_{p^{-1},p^{\tilde{r}(j)}/(p-1)})_{1 \leq i \leq k}\right) = 0,
\end{split}
\end{align}
but this contradicts our assumption \eqref{eq:distances_too_large}. Hence this assumption is false, i.e. the conclusion \eqref{eq:haar_metric_cvg} of \Cref{thm:matrix_product_bulk_metric_intro} holds, and this completes the proof.
\end{proof}

\begin{proof}[Proof of {\Cref{thm:corner_product_bulk_metric_intro}}]
Follows from \Cref{thm:corner_product_bulk} exactly as \Cref{thm:matrix_product_bulk_metric_intro} followed from \Cref{thm:matrix_product_bulk}.
\end{proof}

\appendix

\section{Parallels with complex matrix products}\label{appendix:existing_work} 

This appendix is a longer continuation of the discussion in the Introduction concerning analogies with the complex matrix product literature, which is not necessary to understand the results but which we feel is somewhat illuminating.

At a structural level, singular value decomposition and Smith normal form are identical. It was further shown in \cite{van2020limits} that many exact formulas for the joint distribution of singular numbers of $p$-adic random matrix ensembles are structurally identical to formulas for their complex analogues, by exchanging the relevant special functions in the $p$-adic case (Hall-Littlewood polynomials) for their complex analogues (type $A$ Heckman-Opdam hypergeometric functions). The probabilistic behaviors, on the contrary, can be quite different: as was observed in \cite{van2020limits}, the additive Haar matrices in the $p$-adic setting are analogous to iid Gaussian (Ginibre) matrices in the complex setting, but the fact that almost all singular numbers are $0$ in the $N \to \infty$ limit is quite unlike anything in complex random matrix theory. Nonetheless, we find that there are close probabilistic parallels between the two settings regarding the `amounts of universality' between the hard edge on the one hand and the bulk/soft edge on the other.

Let us first reiterate: The limit we take does \emph{not} correspond to the hard edge in classical random matrix theory. Though \Cref{fig:cplx_svs} and \Cref{fig:1matrix_sns} might naively suggest this comparison, the points in \Cref{fig:1matrix_sns} are singular numbers while those in \Cref{fig:cplx_svs} are singular values. Because the $p$-adic norm $|p^m|_p = p^{-m}$, we in fact have that $p^{-\SN(A)_1}$ corresponds to the smallest singular value, while $p^{-\SN(A)_N}$ corresponds to the largest singular value. 

For a single $N \times N$ complex matrix, say with iid Gaussian entries, if one takes $N \to \infty$ and zooms in near the expected location of the largest singular value one obtains an infinite collection of points with a rightmost point (corresponding to the largest singular value). This random point configuration is the Airy point process, the correlation functions of which were computed earlier by Forrester \cite{forrester1993spectrum}. The distribution of this rightmost point, the limit of the largest singular value, is the eponymous distribution studied by Tracy and Widom \cite{tracy1994level}. 

These limits are distinct from those obtained in the bulk, which are governed by the sine kernel. However, for a matrix in $\Mat_N(\Z_p)$ with e.g. additive Haar measure, almost all singular numbers are $0$. Hence, for instance, $\SN(A)_{\floor{N/2}}$ and $\SN(A)_N$ will both be equal with probability going to $1$ in the $N \to \infty$ limit. For a complex random matrix, a local limit at the expected location of the $(N/2)\tth$ singular value will yield the sine kernel, while a local limit near the $N\tth$ singular value will yield the Airy kernel; at the risk of making an obvious point, these limits are allowed to be distinct because the $(N/2)\tth$ singular value is far away from the $N\tth$ singular value. In the $p$-adic setting this is not the case: looking at the asymptotic distribution of the singular numbers in an $[0,b]$ containing $0$ will not only completely characterize the behavior of the singular numbers $\SN(A)_N,\SN(A)_{N-1},\ldots$ close to the edge, but also the behavior of $\ldots,\SN(A)_{\floor{cN}-1},\SN(A)_{\floor{cN}},\SN(A)_{\floor{cN}+1},\ldots$ for $c \in (0,1)$, which might reasonably be termed the bulk. The limit for matrix products in \Cref{thm:matrix_product_bulk} characterizes the joint distribution of all singular numbers $\leq k-1$, i.e. the smallest $\approx N-\log_p s_N$ singular numbers. Hence it is reasonable to consider \Cref{thm:matrix_product_bulk}, \Cref{thm:corner_product_bulk} and \Cref{thm:matrix_stat_dist} as giving both the bulk and soft edge limits in the sense of classical random matrix theory.

In the setting of complex matrix products, the bulk and soft edge are much more universal than the hard edge. For products of a fixed finite number of Ginibre matrices, the local limits at the soft edge and bulk are the same Airy and sine kernels as for a single matrix, as shown by Liu-Wang-Zhang \cite{liu2016bulk}. The hard edge limit was computed by Kuijlaars-Zhang \cite{kuijlaars-zhang2014singular} for a product of $\tau$ Ginibre matrices, yielding a kernel introduced earlier by Bertola-Gekhtman-Szmigielski \cite{bertola2014cauchy}. Unlike the bulk and soft edge, this kernel depended on the number $\tau$ of products. Replacing one Ginibre in the product by a corner of a unitary matrix, Kuijlaars-Stivigny \cite{kuijlaars-stivigny2014singular} found the same limit, while Kieburg-Kuijlaars-Stivigny \cite{kieburg2016singular} found limits not agreeing with this one for products of truncated unitary matrices in a slightly different limit on the size of the ambient unitary matrix. In other words, the basic picture is that the hard edge limit depends both on the matrix distribution and the number of products, while the soft edge and bulk are insensitive to these.

The same pattern is visible for products of a fixed number of $p$-adic matrices. It follows directly from the exact formula \cite[Corollary 3.4]{van2020limits} that for a product of $\tau$ additive Haar matrices in $\Mat_N(\Z)$, the $N \to \infty$ limiting distribution of singular numbers depends on $\tau$; this limit was subsequently shown for products of $\tau$ matrices with generic iid entries in \cite{nguyen2022universality} (the $\tau=1$ case was shown earlier by Wood \cite{wood2015random}). However, for each $\tau$ the singular numbers $\SN(A)_{\floor{cN}},\SN(A)_{\floor{cN}+1},\ldots,\SN(A)_N$ will all be $0$ with probability going to $1$ as $N \to \infty$ for any $c \in (0,1)$ by these results, so the bulk/soft edge behavior is insensitive to $\tau$ exactly as in the complex case. 

In the interpolating regime $N \to \infty, \tau/N \to \text{const}$, the story is similar. We have already mentioned that in the complex setting, the bulk limits of singular values is governed by the interpolating kernel of \cite{akemann2019integrable} rather than the sine kernel. The same is true on the soft edge, namely there is another limit (also introduced in \cite{akemann2019integrable} and also treated in the mathematical literature in \cite{liu2018lyapunov}) which interpolates between the Airy process and independent Gaussian statistics; we mention also work of Berezin-Strahov \cite{berezin2022gap}, which computed the gap probabilities of the limit. This soft edge limit was further shown by Ahn for truncated unitary matrices \cite{ahn2019fluctuations} and later for a broad class of invariant ensembles \cite{ahn2022extremal}. In this regime, again, the soft edge is universal, and the bulk is expected to be so. The hard edge yields a deterministic limit \cite[Section V]{akemann2020universality} and the fluctuations of each log singular value around its expected location are expected to yield independent Gaussians; this limit is additionally nonuniversal in the sense that the limiting rescaled locations of singular values depend on the ensemble.

\Cref{thm:matrix_product_bulk} and \Cref{thm:corner_product_bulk} complement this picture on the $p$-adic side, showing that in the $N,\tau \to \infty$ for $p$-adic matrix products, the bulk/soft edge limit agrees for those examples. We mention also that in the related regime where first $\tau \to \infty$ and then $N \to \infty$, \cite[Theorem 1.2]{van2020limits} computed the limits of the smallest singular numbers for the ensembles treated in \Cref{thm:matrix_product_bulk} and \Cref{thm:corner_product_bulk}, and similarly found that the answer was independent of the ensemble. At the hard edge (largest few singular numbers), the exact formulas for the fixed $N$, $\tau \to \infty$ limits computed in \cite[Theorem 1.1]{van2020limits} do not become universal as $N \to \infty$.

In summary:
\begin{enumerate}
\item The limit of \Cref{thm:matrix_product_bulk}, \Cref{thm:corner_product_bulk} and \Cref{thm:matrix_stat_dist} should be viewed as determining both bulk and soft edge local limits, but not the hard edge.
\item The examples known so far indicate the following picture: For both complex and $p$-adic matrix products, the bulk and soft edge are universal in both the finite-$\tau$ and interpolating regime, while the hard edge is nonuniversal and depends on the number of products in the finite-$\tau$ regime.
\end{enumerate}
While nonuniversality at the hard edge has been established by the non-matching of examples mentioned above, the claimed universality at the bulk/soft edge is in general simply indicated by agreement of multiple examples rather than rigorously shown in the literature. Except for \cite{ahn2022extremal} (complex soft edge, interpolating regime) and \cite{nguyen2022universality} ($p$-adic bulk/soft edge, finite $\tau$ regime), we are not aware of true universality results. 

\section{A Hall-Littlewood proof of \cite[Theorem 1.2]{van2023p}}\label{sec:dbm_pois_appendix}

\Cref{thm:dbm_is_hl} was proven in \cite[Theorem 1.2]{van2023p} by elementary means, motivated by justifying that $\Pois^{(N)}$ is a natural $p$-adic analogue of multiplicative Dyson Brownian motion; we refer to that work for details on this. The purpose of this appendix is to give an independent proof using Hall-Littlewood machinery. We believe the direct proof in \cite{van2023p} gives somewhat more explanation why the result is true, which is why we chose to give it in that work. However, the proof here is shorter once the relevant Hall-Littlewood background is established, and indeed this is the first proof of \cite[Theorem 1.2]{van2023p} which we found originally.

\begin{defi}\label{def:mat_prod_dbm_proc}
We define the process $X^{(N)}(\tau), \tau \in \R_{\geq 0}$ on $\GL_N(\Q_p)$ by 
\begin{equation}
X^{(N)}(\tau) := U_{P(\tau)} \diag(p,1[N-1]) V_{P(\tau)}  \cdots U_1 \diag(p,1[N-1]) V_1 U_0
\end{equation}
where $P(\tau)$ is a Poisson process on $\Z_{\geq 0}$ with rate $1$, and $U_i,V_i \in \GL_N(\Z_p)$ are iid and Haar-distributed.
\end{defi}

\begin{thm}[{\cite[Theorem 1.2]{van2023p}}]\label{thm:dbm_is_hl}
For any $N \in \Z_{\geq 1}$, 
\begin{equation}
\SN(X^{(N)}(\tau)) = \Pois^{(N)}\left(\left(\frac{1}{t}\frac{1-t}{1-t^N}\right)\tau\right)
\end{equation}
in multi-time distribution.
\end{thm}

The proof goes by comparing Markov generators, so we must know that this suffices. This comes from the following standard result which is essentially due to Feller \cite{feller2015integro}; at the referee's request we provide a proof of the version we will use.

\begin{prop}\label{thm:cite_feller}
Let $Y_\tau, \tau \in \R_{\geq 0}$ be a Markov process on a countable state space $\mathcal{X}$ with well-defined generator $Q$ satisfying
\begin{equation}\label{eq:diag_gen_finite}
\sup_{a \in \mathcal{X}} |Q(a,a)| < \infty.
\end{equation}
Then the law of $Y_\tau, \tau \in \R_{\geq 0}$ is uniquely determined by $Q$ and its initial condition $Y_0$. 
\end{prop}
\begin{proof}
Combining \cite[Corollary 2.34 (a)]{liggett2010continuous} with \cite[Theorem 2.37]{liggett2010continuous}, there exists a unique Markov chain $X_\tau$ with transition probabilities $p_\tau(x,y) = \Pr(X_\tau=y|X_0=x)$ satisfying the Kolmogorov backward equation 
\begin{equation}\label{eq:kol_back}
\frac{d}{d\tau} p_\tau(x,y) = \sum_{z \in \mathcal{X}} Q(x,z)p_\tau(z,y)
\end{equation}
for all $x,y \in \mc{X}$. Now we let $\tilde{p}_\tau$ be any Markov semigroup with generator $Q$, and we wish to show it satisfies the equivalent form
\begin{equation}\label{eq:massage_kol_back}
\lim_{\eps \searrow 0} \sum_{z \in \mc{X}} \left(\frac{\tilde{p}_\eps(x,z) - \tilde{p}_0(x,z)}{\eps} - Q(x,z)\right) \tilde{p}_\tau(z,y) = 0
\end{equation}
of \eqref{eq:kol_back}; here the equivalence comes from writing
\begin{equation}
\tilde{p}_{\tau+\eps}(x,y) = \sum_{z \in \mc{X}} \tilde{p}_\eps(x,z) \tilde{p}_\tau(z,y).
\end{equation}
We will argue similarly to \cite[Theorem 2.14(c)]{liggett2010continuous}.
%Note that each term in the sum in \eqref{eq:massage_kol_back} goes to $0$ by the definition of generator as $\eps \searrow 0$, so it suffices to show

Let $T \subset \mc{X}$ be a finite subset containing $x$. Then by triangle inequality and the fact that $\sum_{z \in \mc{X}} Q(x,z) = 0$, we have
\begin{align}\label{eq:from_liggett}
\begin{split}
\abs{ \sum_{z \in \mc{X} \setminus T} \left(\frac{\tilde{p}_\eps(x,z) - \tilde{p}_0(x,z)}{\eps} - Q(x,z)\right) \tilde{p}_\tau(z,y)} &\leq \sum_{z \in \mc{X} \setminus T} \frac{\tilde{p}_\eps(x,z)}{\eps} + \sum_{z \in \mc{X} \setminus T} Q(x,z) \\ 
&= \eps^{-1}\left(1-\sum_{z \in T}\frac{\tilde{p}_\eps(x,z)}{\eps} \right) - \sum_{z \in T} Q(x,z).
\end{split}
\end{align}
Then
\begin{equation}
\lim_{\eps \searrow 0} \text{RHS\eqref{eq:from_liggett}} = -2 \sum_{z \in T} Q(x,z)
\end{equation}
by interchanging the limit with the sum. Because the $\eps \searrow 0$ limit of each term in the sum \eqref{eq:massage_kol_back} is $0$, 
\begin{equation}
\text{LHS\eqref{eq:massage_kol_back}} = \sum_{z \in \mc{X} \setminus T} \left(\frac{\tilde{p}_\eps(x,z) - \tilde{p}_0(x,z)}{\eps} - Q(x,z)\right) \tilde{p}_\tau(z,y) 
\end{equation}
for any finite $T$. Choosing $T$ to be large, $\sum_{z \in T} Q(x,z)$ may be made arbitrarily small, and \eqref{eq:massage_kol_back} (and hence the proof) follows.
\end{proof}

The next result requires Hall-Littlewood structure constants $c_{\la,\mu}^\nu(0,t)$. We note that the set $\{P_\l(\bx;q,t): \l \in \Sig_n\}$ forms a basis for the ring of symmetric Laurent polynomials $\L_n[(x_1 \cdots x_n)^{-1}]$. Hence for any $\la,\mu \in \Sig_n$ one has
\begin{equation}\label{eq:mac_func_struct_coefs}
    P_\la(\bx;q,t) \cdot P_\mu(\bx;q,t) = \sum_{\nu \in \Sig_n} c_{\la,\mu}^\nu(q,t) P_\nu(\bx;q,t)
\end{equation}
for some structure coefficients $c_{\la,\mu}^\nu(q,t)$. By matching degrees it is clear that these coefficients are nonzero only if $|\l|+|\mu|=|\nu|$. These multiplicative structure coefficients for the $P$ polynomials are related to the `comultiplicative' structure constants of the $Q$ polynomials.

\begin{prop}[{\cite[Proposition 2.4]{van2020limits}}]\label{thm:q_coproduct_coefs}
Let $m,n \in \N$ and $\l,\nu \in \Y_n$. Then
\begin{equation*}
    Q_{\l/\nu}(x_1,\ldots,x_m;q,t) = \sum_{\mu \in \Y_m} c_{\nu,\mu}^\la(q,t) Q_\mu(x_1,\ldots,x_m;q,t).
\end{equation*}
\end{prop}

\begin{prop}\label{thm:product_conv_cite}
Let $\la,\mu \in \Sig_N$ and let $U$ be a Haar-distributed element of $\GL_N(\Z_p)$. Then 
\begin{equation}\label{eq:hl_mat_prod}
\Pr(\SN(\diag(p^\mu)U\diag(p^\la)) = \nu) = c_{\la,\mu}^\nu(0,t) \frac{P_\nu(1,\ldots,t^{N-1};0,t)}{P_\mu(1,\ldots,t^{N-1};0,t)P_\la(1,\ldots,t^{N-1};0,t)}.
\end{equation}
\end{prop}
\begin{proof}
This is essentially \cite[Theorem 1.3 Part 3]{van2020limits}, though let us remark on slight differences in setup. That result was stated for two matrices $A,B \in \GL_N(\Q_p)$ with fixed singular numbers $\la,\mu$ respectively, and distribution invariant under left- and right-multiplication by $\GL_N(\Z_p)$. Such matrices are given by $U \diag(p^\la)V$ and $U' \diag(p^\mu)V'$ for $U,V,U',V'$ independent Haar-distributed elements of $\GL_N(\Z_p)$. Hence 
\begin{equation}
\SN(AB) = \SN(U\diag(p^\la)VU' \diag(p^\mu)V') = \SN(\diag(p^\la)VU' \diag(p^\mu)),
\end{equation}
and $VU'$ has Haar distribution. Hence 
\begin{equation}
\Pr(\SN(\diag(p^\mu)U\diag(p^\la)) = \nu) = \Pr(\SN(AB) = \nu),
\end{equation}
and \eqref{eq:hl_mat_prod} now follows by \cite[Theorem 1.3 Part 3]{van2020limits}.
\end{proof}

\begin{lemma}\label{thm:hl_gen_computation}
The generator of the Markov process $\la^{(n)}(\tau)$ defined in \Cref{def:lambda_hl_planch} exists and is given by 
\begin{equation}
B_{HL}(\mu,\nu) = \begin{cases} 
-\frac{1-t^n}{1-t} & \mu = \nu \\ 
\frac{\varphi_{\nu/\mu}(0,t)}{1-t}\frac{P_\nu(1,\ldots,t^{n-1};0,t)}{P_\mu(1,\ldots,t^{n-1};0,t)} & \mu \prec \nu \text{ and }|\nu| = |\mu|+1 \\
0 & \text{otherwise}
\end{cases}
\end{equation}
for $\mu,\nu \in \Y_N$. 
\end{lemma}
\begin{proof}
This is computed in \cite[Proof of Theorem 3.4]{van2022q}.
\end{proof}

\begin{proof}[Proof of \Cref{thm:dbm_is_hl}]
By \Cref{def:mat_prod_dbm_proc}
\begin{multline}\label{eq:small_time_jumps}
\Pr(\SN(X^{(N)}(\tau+\eps) = \nu | \SN(X^{(N)}(\tau)) = \mu) \\ 
= \bbone(\nu=\mu) \cdot (1-\eps) + \eps \Pr(\SN(\diag(p^\mu)U\diag(p,1,\dots,1)) = \nu) + O(\eps^2)
\end{multline}
as $\eps \to 0$. Hence the generator of the process $\SN(X^{(N)}(\tau))$ on $\Sig_N$ is given by 
\begin{equation}\label{eq:BSN}
B_{SN}(\mu,\nu) := -\bbone(\mu=\nu) + \Pr(\SN(\diag(p^\mu)U\diag(p,1,\dots,1)) = \nu)
\end{equation}
for any $\mu,\nu \in \Sig_N$.

By \Cref{thm:product_conv_cite},
\begin{multline}\label{eq:apply_prod_conv}
\Pr(\SN(\diag(p^\mu)U\diag(p,1,\dots,1)) = \nu) \\ 
=  c_{\mu,(1,0[N-1])}^\nu(0,t) \frac{P_\nu(1,\ldots,t^{N-1};0,t)}{P_\mu(1,\ldots,t^{N-1};0,t)P_{(1,0[N-1])}(1,\ldots,t^{N-1};0,t)}.
\end{multline}
By \Cref{thm:q_coproduct_coefs}, when $\nu \succ \mu$ and $|\nu|-|\mu|=1$ we have
\begin{equation}\label{eq:Q_LR_coefs}
\varphi_{\nu/\mu}(0,t) = Q_{\nu/\mu}(1;0,t) = c_{\mu,(1,0[N-1])}^\nu(0,t) Q_{(1,0[N-1])}(1;0,t) = c_{\mu,(1,0[N-1])}^\nu(0,t)(1-t).
\end{equation}
Additionally, 
\begin{equation}\label{eq:simple_hl_princ_case}
P_{(1,0[N-1])}(1,\ldots,t^{N-1};0,t) = \frac{1-t^N}{1-t}
\end{equation}
by \Cref{thm:hl_principal_formulas} (one may also simply use that this Hall-Littlewood polynomial is the elementary symmetric polynomial $e_1$). Substituting \eqref{eq:simple_hl_princ_case} and \eqref{eq:Q_LR_coefs} into \eqref{eq:apply_prod_conv} yields
\begin{equation}\label{eq:compute_X_gen}
\text{RHS\eqref{eq:apply_prod_conv}}  = \left(\frac{1-t}{1-t^N}\right) \frac{\varphi_{\nu/\mu}}{1-t} \frac{P_\nu(1,\ldots,t^{N-1};0,t)}{P_\mu(1,\ldots,t^{N-1};0,t)}.
\end{equation}
Combining with \eqref{eq:BSN} and \Cref{thm:hl_gen_computation} yields that
\begin{equation}
B_{SN}(\mu,\nu) = \left(\frac{1-t}{1-t^N}\right)B_{HL}(\mu,\nu)
\end{equation}
for all $\mu,\nu \in \Sig_N^+$. Since the diagonal entries $B_{HL}(\mu,\mu)$ are bounded independent of $\mu$ and similarly for $B_{SN}$, it follows by \Cref{thm:cite_feller} that $\SN(X^{(N)}(\tau))$ and $\la^{(N)}\left(\frac{1-t}{1-t^N}\tau\right)$ are uniquely determined by their (equal) generators, hence are equal in multi-time distribution. The result follows by combining this with \Cref{thm:HL_poisson} to relate $\la^{(N)}$ to $\Pois^{(N)}$.
\end{proof}

% \bibliographystyle{plain}
% \bibliography{references.bib}

\begin{thebibliography}{10}

\bibitem{ahn2022extremal}
Andrew Ahn.
\newblock Extremal singular values of random matrix products and {B}rownian
  motion on $\mathrm{GL}_n(\mathbb{C})$.
\newblock {\em arXiv preprint arXiv:2201.11809}, 2022.

\bibitem{ahn2019fluctuations}
Andrew Ahn.
\newblock Fluctuations of $\beta$-{J}acobi product processes.
\newblock {\em Probability Theory and Related Fields}, 183(1-2):57--123, 2022.

\bibitem{ahn2022lozenge}
Andrew Ahn, Marianna Russkikh, and Roger Van~Peski.
\newblock Lozenge tilings and the {G}aussian free field on a cylinder.
\newblock {\em Communications in Mathematical Physics}, 396(3):1221--1275,
  2022.

\bibitem{ahn2020product}
Andrew Ahn and Eugene Strahov.
\newblock Product matrix processes with symplectic and orthogonal invariance
  via symmetric functions.
\newblock {\em International Mathematics Research Notices},
  2022(14):10767--10821, 2022.

\bibitem{ahn2023lyapunov}
Andrew Ahn and Roger Van~Peski.
\newblock Lyapunov exponents for truncated unitary and {G}inibre matrices.
\newblock {\em Ann. Inst. Henri Poincar\'{e} Probab. Stat.}, 59(2):1029--1039,
  2023.

\bibitem{akemann2019integrable}
Gernot Akemann, Zdzislaw Burda, and Mario Kieburg.
\newblock From integrable to chaotic systems: {U}niversal local statistics of
  {L}yapunov exponents.
\newblock {\em EPL (Europhysics Letters)}, 126(4):40001, 2019.

\bibitem{akemann2020universality}
Gernot Akemann, Zdzislaw Burda, and Mario Kieburg.
\newblock Universality of local spectral statistics of products of random
  matrices.
\newblock {\em arXiv preprint arXiv:2008.11470}, 2020.

\bibitem{akemann2013singular}
Gernot Akemann, Mario Kieburg, and Lu~Wei.
\newblock Singular value correlation functions for products of {W}ishart random
  matrices.
\newblock {\em Journal of Physics A: Mathematical and Theoretical},
  46(27):275205, 2013.

\bibitem{bellman1954limit}
Richard Bellman.
\newblock Limit theorems for non-commutative operations. {I}.
\newblock {\em Duke Mathematical Journal}, 21(3):491--500, 1954.

\bibitem{berezin2022gap}
Sergey Berezin and Eugene Strahov.
\newblock Gap probability for products of random matrices in the critical
  regime.
\newblock {\em Journal of Approximation Theory}, 274:105687, 2022.

\bibitem{bertola2014cauchy}
Marco Bertola, M~Gekhtman, and J~Szmigielski.
\newblock Cauchy--{L}aguerre two-matrix model and the {Meijer-G} random point
  field.
\newblock {\em Communications in Mathematical Physics}, 326:111--144, 2014.

\bibitem{betea2019periodic}
Dan Betea and J{\'e}r{\'e}mie Bouttier.
\newblock The periodic {S}chur process and free fermions at finite temperature.
\newblock {\em Mathematical Physics, Analysis and Geometry}, 22(1):3, 2019.

\bibitem{bhargava2013modeling}
Manjul Bhargava, Daniel~M Kane, Hendrik~W Lenstra, Bjorn Poonen, and Eric
  Rains.
\newblock Modeling the distribution of ranks, {S}elmer groups, and
  {S}hafarevich--{T}ate groups of elliptic curves.
\newblock {\em Cambridge Journal of Mathematics}, 3(3):275--321, 2015.

\bibitem{borodin1995limit}
Alexei Borodin.
\newblock Limit {J}ordan normal form of large triangular matrices over a finite
  field.
\newblock {\em Funktsional. Anal. i Prilozhen.}, 29(4):72--75, 1995.

\bibitem{borodin1999lln}
Alexei Borodin.
\newblock The law of large numbers and the central limit theorem for the
  {J}ordan normal form of large triangular matrices over a finite field.
\newblock {\em Journal of Mathematical Sciences}, 96(5):3455--3471, 1999.

\bibitem{Bor07}
Alexei Borodin.
\newblock Periodic {S}chur process and cylindric partitions.
\newblock {\em Duke Math. J.}, 140(3):391--468, 2007.

\bibitem{borodin2016between}
Alexei Borodin, Alexey Bufetov, and Michael Wheeler. 
\newblock Between the stochastic six vertex model and Hall-Littlewood processes.
\newblock {\em arXiv preprint arXiv:1611.09486}, 2016.


\bibitem{borodin2014macdonald}
Alexei Borodin and Ivan Corwin.
\newblock Macdonald processes.
\newblock {\em Probability Theory and Related Fields}, 158(1-2):225--400, 2014.

\bibitem{borodin2020dynamic}
Alexei Borodin and Ivan Corwin.
\newblock Dynamic {ASEP}, duality, and continuous $q^{-1}$-{H}ermite
  polynomials.
\newblock {\em International Mathematics Research Notices}, 2020(3):641--668,
  2020.

\bibitem{borodin2014free}
Alexei Borodin, Ivan Corwin, and Patrik Ferrari.
\newblock Free energy fluctuations for directed polymers in random media in 1+
  1 dimension.
\newblock {\em Communications on Pure and Applied Mathematics},
  67(7):1129--1214, 2014.

\bibitem{borodin2018anisotropic}
Alexei Borodin, Ivan Corwin, and Patrik~L Ferrari.
\newblock {Anisotropic (2+1){D} growth and {G}aussian limits of q-{W}hittaker
  processes}.
\newblock {\em Probability Theory and Related Fields}, 172(1-2):245--321, 2018.

\bibitem{borodin2015general}
Alexei Borodin and Vadim Gorin.
\newblock {General $\beta$-{J}acobi Corners Process and the {G}aussian Free
  Field}.
\newblock {\em Communications on Pure and Applied Mathematics},
  68(10):1774--1844, 2015.

\bibitem{borodin2018product}
Alexei Borodin, Vadim Gorin, and Eugene Strahov.
\newblock Product matrix processes as limits of random plane partitions.
\newblock {\em International Mathematics Research Notices}, 2018.

\bibitem{borodin2017representations}
Alexei Borodin and Grigori Olshanski.
\newblock {\em Representations of the infinite symmetric group}, volume 160.
\newblock Cambridge University Press, 2017.

\bibitem{borodin2020observables}
Alexei Borodin and Michael Wheeler.
\newblock Observables of coloured stochastic vertex models and their polymer
  limits.
\newblock {\em Probability and Mathematical Physics}, 1(1):205--265, 2020.

\bibitem{bufetov2018hall}
Alexey Bufetov and Konstantin Matveev.
\newblock {Hall--{L}ittlewood {RSK} field}.
\newblock {\em Selecta Mathematica}, 24(5):4839--4884, 2018.

\bibitem{cheong2021cohen}
Gilyoung Cheong and Yifeng Huang.
\newblock Cohen--lenstra distributions via random matrices over complete
  discrete valuation rings with finite residue fields.
\newblock {\em Illinois Journal of Mathematics}, 65(2):385--415, 2021.

\bibitem{cheong2022generalizations}
Gilyoung Cheong and Nathan Kaplan.
\newblock {Generalizations of results of Friedman and Washington on cokernels
  of random p-adic matrices}.
\newblock {\em Journal of Algebra}, 604:636--663, 2022.

\bibitem{cheong2023cokernel}
Gilyoung Cheong and Myungjun Yu.
\newblock The cokernel of a polynomial of a random integral matrix.
\newblock {\em arXiv preprint arXiv:2303.09125}, 2023.

\bibitem{clancy2015cohen}
Julien Clancy, Nathan Kaplan, Timothy Leake, Sam Payne, and Melanie~Matchett
  Wood.
\newblock On a {C}ohen--{L}enstra heuristic for {J}acobians of random graphs.
\newblock {\em Journal of Algebraic Combinatorics}, 42(3):701--723, 2015.

\bibitem{cohen-lenstra}
Henri Cohen and Hendrik~W Lenstra.
\newblock Heuristics on class groups of number fields.
\newblock In {\em Number Theory Noordwijkerhout 1983}, pages 33--62. Springer,
  1984.

\bibitem{cohen1984stability}
Joel~E Cohen and Charles~M Newman.
\newblock The stability of large random matrices and their products.
\newblock {\em The Annals of Probability}, pages 283--310, 1984.

\bibitem{crisanti2012products}
Andrea Crisanti, Giovanni Paladin, and Angelo Vulpiani.
\newblock {\em Products of random matrices: in Statistical Physics}, volume
  104.
\newblock Springer Science \& Business Media, 2012.

\bibitem{dimitrov2018kpz}
Evgeni Dimitrov.
\newblock {KPZ} and {A}iry limits of {H}all--{L}ittlewood random plane
  partitions.
\newblock In {\em Annales de l'Institut Henri Poincar{\'e}, Probabilit{\'e}s et
  Statistiques}, volume~54, pages 640--693. Institut Henri Poincar{\'e}, 2018.

\bibitem{dyson1962statistical}
Freeman~J Dyson.
\newblock Statistical theory of the energy levels of complex systems {I}, {II}
  and {III}.
\newblock {\em Journal of Mathematical Physics}, 3(1):140--156, 157--165,
  166--175, 1962.

\bibitem{erdelyi1981higher}
Arthur Erd\'{e}lyi, Wilhelm Magnus, Fritz Oberhettinger, and Francesco~G.
  Tricomi.
\newblock {\em Higher transcendental functions. {V}ol. {II}}.
\newblock Robert E. Krieger Publishing Co., Inc., Melbourne, Fla., 1981.
\newblock Based on notes left by Harry Bateman, Reprint of the 1953 original.

\bibitem{erdHos2017dynamical}
L{\'a}szl{\'o} Erd{\H{o}}s and Horng-Tzer Yau.
\newblock {\em A dynamical approach to random matrix theory}, volume~28.
\newblock American Mathematical Soc., 2017.

\bibitem{evans2002elementary}
Steven~N Evans.
\newblock Elementary divisors and determinants of random matrices over a local
  field.
\newblock {\em Stochastic processes and their applications}, 102(1):89--102,
  2002.

\bibitem{feller2015integro}
William Feller.
\newblock On the integro-differential equations of purely discontinuous
  {M}arkoff processes.
\newblock In {\em Selected Papers I}, pages 539--566. Springer, 2015.

\bibitem{forrester1993spectrum}
Peter~J Forrester.
\newblock The spectrum edge of random matrix ensembles.
\newblock {\em Nuclear Physics B}, 402(3):709--728, 1993.

\bibitem{friedman-washington}
Eduardo Friedman and Lawrence~C Washington.
\newblock On the distribution of divisor class groups of curves over a finite
  field.
\newblock {\em Th{\'e}orie des Nombres/Number Theory Laval}, 1987.

\bibitem{furstenberg1960products}
Harry Furstenberg and Harry Kesten.
\newblock Products of random matrices.
\newblock {\em The Annals of Mathematical Statistics}, 31(2):457--469, 1960.

\bibitem{gasper_rahman_2004}
George Gasper and Mizan Rahman.
\newblock {\em Basic Hypergeometric Series}.
\newblock Encyclopedia of Mathematics and its Applications. Cambridge
  University Press, 2 edition, 2004.

\bibitem{gorin2018gaussian}
Vadim Gorin and Yi~Sun.
\newblock Gaussian fluctuations for products of random matrices.
\newblock {\em American Journal of Mathematics}, 144(2):287--393, 2022.

\bibitem{hanin2020products}
Boris Hanin and Mihai Nica.
\newblock Products of many large random matrices and gradients in deep neural
  networks.
\newblock {\em Communications in Mathematical Physics}, 376(1):287--322, 2020.

\bibitem{imamura2021identity}
Takashi Imamura, Matteo Mucciconi, and Tomohiro Sasamoto.
\newblock Identity between restricted {C}auchy sums for the $ q $-{W}hittaker
  and skew {S}chur polynomials.
\newblock {\em arXiv preprint arXiv:2106.11913}, 2021.

\bibitem{imamura2021skew}
Takashi Imamura, Matteo Mucciconi, and Tomohiro Sasamoto.
\newblock Skew {RSK} dynamics: Greene invariants, affine crystals and
  applications to $ q $-{W}hittaker polynomials.
\newblock {\em arXiv preprint arXiv:2106.11922}, 2021.

\bibitem{imamura2022solvable}
Takashi Imamura, Matteo Mucciconi, and Tomohiro Sasamoto.
\newblock Solvable models in the {KPZ} class: approach through periodic and
  free boundary {S}chur measures.
\newblock {\em arXiv preprint arXiv:2204.08420}, 2022.

\bibitem{ismail2005classical}
Mourad Ismail.
\newblock {\em Classical and quantum orthogonal polynomials in one variable},
  volume~13.
\newblock Cambridge university press, 2005.

\bibitem{ismail1994q}
Mourad Ismail and David Masson.
\newblock $q$-{H}ermite polynomials, biorthogonal rational functions, and
  $q$-beta integrals.
\newblock {\em Transactions of the American Mathematical Society},
  346(1):63--116, 1994.

\bibitem{kerov1992generalized}
SV~Kerov.
\newblock Generalized {H}all-{L}ittlewood symmetric functions and orthogonal
  polynomials.
\newblock {\em Representation theory and dynamical systems Advances in Soviet
  Math}, 9:67--94, 1992.

\bibitem{kieburg2016singular}
Mario Kieburg, Arno~BJ Kuijlaars, and Dries Stivigny.
\newblock Singular value statistics of matrix products with truncated unitary
  matrices.
\newblock {\em International Mathematics Research Notices},
  2016(11):3392--3424, 2016.

\bibitem{kirillov1995variations}
A.~A. Kirillov.
\newblock Variations on the triangular theme.
\newblock In {\em Lie groups and {L}ie algebras: {E}. {B}. {D}ynkin's
  {S}eminar}, volume 169 of {\em Amer. Math. Soc. Transl. Ser. 2}, pages
  43--73. Amer. Math. Soc., Providence, RI, 1995.

\bibitem{kirillov1998new}
Anatol~N Kirillov.
\newblock New combinatorial formula for modified {H}all-{L}ittlewood
  polynomials.
\newblock {\em arXiv preprint math/9803006}, 1998.

\bibitem{kovaleva2020note}
Valeriya Kovaleva.
\newblock A note on the distribution of equivalence classes of random symmetric
  p-adic matrices.
\newblock {\em arXiv preprint arXiv:2008.10732}, 2020.

\bibitem{kuijlaars-stivigny2014singular}
Arno~BJ Kuijlaars and Dries Stivigny.
\newblock Singular values of products of random matrices and polynomial
  ensembles.
\newblock {\em Random Matrices: Theory and Applications}, 3(03):1450011, 2014.

\bibitem{kuijlaars-zhang2014singular}
Arno~BJ Kuijlaars and Lun Zhang.
\newblock Singular values of products of {G}inibre random matrices, multiple
  orthogonal polynomials and hard edge scaling limits.
\newblock {\em Communications in Mathematical Physics}, 332:759--781, 2014.

\bibitem{lee2022universality}
Jungin Lee.
\newblock Universality of the cokernels of random $ p $-adic hermitian
  matrices.
\newblock {\em arXiv preprint arXiv:2205.09368}, 2022.

\bibitem{lee2023joint}
Jungin Lee.
\newblock Joint distribution of the cokernels of random p-adic matrices.
\newblock {\em Forum Mathematicum}, 35(4):1005--1020, 2023.

\bibitem{liggett2010continuous}
Thomas M. Liggett.
\newblock Continuous time Markov processes: an introduction. 
\newblock {\em American Mathematical Soc.}, 2010.


\bibitem{lipnowski2020cohen}
Michael Lipnowski, Will Sawin, and Jacob Tsimerman.
\newblock Cohen-lenstra heuristics and bilinear pairings in the presence of
  roots of unity.
\newblock {\em arXiv preprint arXiv:2007.12533}, 2020.

\bibitem{liu2018lyapunov}
Dang-Zheng Liu, Dong Wang, and Yanhui Wang.
\newblock Lyapunov exponent, universality and phase transition for products of
  random matrices.
\newblock {\em Communications in Mathematical Physics}, 399(3):1811--1855,
  2023.

\bibitem{liu2016bulk}
Dang-Zheng Liu, Dong Wang, and Lun Zhang.
\newblock Bulk and soft-edge universality for singular values of products of
  {G}inibre random matrices.
\newblock {\em Annales de L'Institut Henri Poincare Section (B) Probability and
  Statistics}, 52(4):1734--1762, 2016.

\bibitem{mac}
Ian~Grant Macdonald.
\newblock {\em Symmetric functions and Hall polynomials}.
\newblock Oxford university press, 1998.

\bibitem{matveev2019macdonald}
Konstantin Matveev.
\newblock Macdonald-positive specializations of the algebra of symmetric
  functions: Proof of the {K}erov conjecture.
\newblock {\em Annals of Mathematics}, 189(1):277--316, 2019.

\bibitem{mehta1960density}
Madan~Lal Mehta and Michel Gaudin.
\newblock On the density of eigenvalues of a random matrix.
\newblock {\em Nuclear Physics}, 18:420--427, 1960.

\bibitem{meszaros2020distribution}
Andr{\'a}s M{\'e}sz{\'a}ros.
\newblock The distribution of sandpile groups of random regular graphs.
\newblock {\em Transactions of the American Mathematical Society},
  373(9):6529--6594, 2020.

\bibitem{meszaros2023cohen}
Andr{\'a}s M{\'e}sz{\'a}ros.
\newblock Cohen-{L}enstra distribution for sparse matrices with determinantal
  biasing.
\newblock {\em arXiv preprint arXiv:2307.04741}, 2023.

\bibitem{nguyen2022universality}
Hoi~H Nguyen and Roger Van~Peski.
\newblock Universality for cokernels of random matrix products.
\newblock {\em Advances in Mathematics}, 438(109451), 2024.

\bibitem{nguyen2022local}
Hoi~H Nguyen and Melanie~Matchett Wood.
\newblock Local and global universality of random matrix cokernels.
\newblock {\em arXiv preprint arXiv:2210.08526}, 2022.

\bibitem{nguyen2022random}
Hoi~H Nguyen and Melanie~Matchett Wood.
\newblock Random integral matrices: universality of surjectivity and the
  cokernel.
\newblock {\em Inventiones mathematicae}, 228(1):1--76, 2022.

\bibitem{sawin2022moment}
Will Sawin and Melanie~Matchett Wood.
\newblock The moment problem for random objects in a category.
\newblock {\em arXiv preprint arXiv:2210.06279}, 2022.

\bibitem{tracy1994level}
Craig~A Tracy and Harold Widom.
\newblock Level-spacing distributions and the {A}iry kernel.
\newblock {\em Communications in Mathematical Physics}, 159:151--174, 1994.

\bibitem{van2023+dynamical}
Roger Van~Peski.
\newblock Reflecting {P}oisson walks and dynamical universality in $p$-adic
  random matrix theory.
\newblock \textit{In preparation}. [Appeared after first version of present text, see \url{https://arxiv.org/abs/2312.11702}.]

\bibitem{van2023+rank}
Roger Van~Peski.
\newblock The rank of a random triangular matrix over $\F_q$.
\newblock In preparation.

\bibitem{van2023+symmetric}
Roger Van~Peski.
\newblock Symmetric functions and the inverse moment problem for abelian
  groups.
\newblock In preparation. [Appeared after first version of present text, see \url{https://arxiv.org/abs/2402.16625}.] 

\bibitem{van2020limits}
Roger Van~Peski.
\newblock Limits and fluctuations of $p$-adic random matrix products.
\newblock {\em Selecta Mathematica}, 27(5):1--71, 2021.

\bibitem{vanpeski2021halllittlewood}
Roger Van~Peski.
\newblock Hall–{L}ittlewood polynomials, boundaries, and $p$-adic random
  matrices.
\newblock {\em International Mathematics Research Notices},
  2023(13):11217--11275, 2022.

\bibitem{van2022q}
Roger Van~Peski.
\newblock q-{TASEP} with position-dependent slowing.
\newblock {\em Electronic Journal of Probability}, 27:1--35, 2022.

\bibitem{van2023asymptotics}
Roger Van~Peski.
\newblock {\em Asymptotics, exact results, and analogies in $p$-adic random
  matrix theory}.
\newblock PhD thesis, Massachusetts Institute of Technology, 2023.

\bibitem{van2023p}
Roger Van~Peski.
\newblock What is a $ p $-adic {Dyson B}rownian motion?
\newblock {\em arXiv preprint arXiv:2309.02865}, 2023.

\bibitem{wig1}
Eugene~P Wigner.
\newblock On the statistical distribution of the widths and spacings of nuclear
  resonance levels.
\newblock In {\em Mathematical Proceedings of the Cambridge Philosophical
  Society}, volume~47, pages 790--798. Cambridge University Press, 1951.

\bibitem{wig3}
Eugene~Paul Wigner.
\newblock {\em Statistical properties of real symmetric matrices with many
  dimensions}.
\newblock Princeton University, 1957.

\bibitem{wood2017distribution}
Melanie~Matchett Wood.
\newblock The distribution of sandpile groups of random graphs.
\newblock {\em Journal of the American Mathematical Society}, 30(4):915--958,
  2017.

\bibitem{wood2018cohen}
Melanie~Matchett Wood.
\newblock Cohen-{L}enstra heuristics and local conditions.
\newblock {\em Research in Number Theory}, 4(4):41, 2018.

\bibitem{wood2015random}
Melanie~Matchett Wood.
\newblock Random integral matrices and the {Cohen-L}enstra heuristics.
\newblock {\em American Journal of Mathematics}, 141(2):383--398, 2019.

\end{thebibliography}

\end{document}